\documentclass[11pt]{article}
\usepackage{amsmath}
\usepackage{amssymb}
\usepackage[german,english]{babel}
\usepackage{url}
\usepackage{graphicx}
\usepackage{gastex}
\usepackage{longtable}
\usepackage{lscape}
\usepackage{multicol}
\usepackage{latexsym}
\usepackage{verbatim}
\usepackage{multirow}

\usepackage{amssymb,latexsym,amsmath,epsfig,amsthm,mathrsfs}
\usepackage{amsfonts}
\usepackage{amscd}
\usepackage{dsfont}
\usepackage{bbm}
\usepackage{rotating}
\usepackage{newlfont}
\usepackage{enumitem}
\usepackage{lineno,hyperref}
\usepackage[english]{babel}
\usepackage{kbordermatrix}
\usepackage{tabularx}
\usepackage{tikz}
\usepackage{tkz-graph}
\usepackage{longtable}
\usepackage{tabularx,ragged2e,booktabs,caption}
\tikzstyle{vertex}=[ draw, inner sep=0pt, minimum size=0pt]

\modulolinenumbers[5]

\usepackage{epstopdf}
\usepackage{url}
\usepackage{hyperref}
\usepackage{float}
\date{}
\theoremstyle{plain}
\newtheorem{theorem}{Theorem}
\newtheorem{lemma}[theorem]{Lemma}

\theoremstyle{definition}

\newtheorem{case}{Case}
\newtheorem{subcase}{Subcase}
\numberwithin{subcase}{case}

\makeatletter
\renewcommand{\@seccntformat}[1]{\csname the#1\endcsname. }
\renewcommand{\arraystretch}{1.5}
\makeatother

\def\h25{\hspace{-.3cm}}

\usepackage{graphicx}
\usepackage{epsfig,graphics,graphicx}
\usepackage{geometry}
\usepackage{color}
\usepackage{tabularx}
\usepackage{tikz}
\usepackage{tkz-graph}
\tikzstyle{vertex}=[ draw, inner sep=0pt, minimum size=0pt]
\hbadness=\maxdimen
\newtheorem{lem}{{\sc Lemma}}[section]
\newtheorem{thm}{{\sc Theorem}}[section]
\newtheorem{prop}{{\sc Proposition}}[section]

\newcolumntype{C}[1]{>{\centering}m{#1}}
\begin{document}
%\raggedbottom
%\raggedleft
\title{\sc Semi-equivelar maps of Euler characteristics -2 with few vertices}
\author{{Debashis Bhowmik and Ashish Kumar Upadhyay}\\[2mm]
{\normalsize Department of Mathematics}\\{\normalsize Indian Institute of Technology Patna }\\{\normalsize Bihta, Bihar 801103, India.}\\
{\small \{debashis.pma15, upadhyay\}@iitp.ac.in}}
%\date{May 2009}
\maketitle

\vspace{-5mm}

\hrule

\begin{abstract}  We enumerate and classify all the semi equivelar maps on the surface of $ \chi=-2 $ with upto 12 vertices. We also determine which of these are vertex transitive and which are not.
\end{abstract}

{\small

{\bf AMS classification\,: 52B70, 52C20, 57M20, 57N05}

{\bf Keywords\,:} Semi equivelar maps, Vertex transitive maps, Archimedean solids
}

\bigskip

\hrule

\section{Introduction}\label{intro}

 Recall that a compact, connected, 2-dimensional manifold without boundary is called a closed surface or surface. A \textit{Polyhedral complex} is a finite collection $ X $ of convex polytopes, in Euclidean space $ \mathbb{E}^n $ such that following two conditions are satisfied: First, if $ \sigma\in X $ and $ \tau $ is a face of $ \sigma $, then $ \tau \in X $ and Second, if $ \sigma_{1},\sigma_{2}\in X $ then $ \sigma_{1}\cap\sigma_{2} $ is a face of $ \sigma_{1} $ and $ \sigma_{2} $. The \textit{dimension} $ d $ of $ X $ is the maximum of dimensions (of the affine hulls) of the elements in $ X $. We also call $ X $ a \textit{polyhedral d-complex} or \textit{combinatorial d-complex}. For a polyhedral complex $X$, let $v$ be a vertex of a $X$. Then \textit{star} of $v$, denoted by $st(v)$, is a polyhedral complex $\{\sigma\in X:\{v\}\leq \sigma \}$ and the \textit{link} of a vertex $ v $ in $ X $, denoted by $lk(v)$ or $lk_{X}(v)$, is the polyhedral complex $\{\sigma\in Cl(st(v)):\sigma\cap\{v\}=\phi\}$. The Automorphism group \textit{Aut(X)} of a polyhedral complex $ X $ is a collection of all automorphism $ f:X \rightarrow  X$, naturally it forms a group under composition of maps and is also said to be the group of automorphisms of $ X $. If \textit{Aut(X)} acts transitively on the set of all vertices $V(X)$ of $X$ then $X$ is called a \textit{vertex transitive map}.  For other related definition, one may refer to \cite{west}, \cite{armstrong}, \cite{aku.akt} \cite{bd.aku1} \cite{handbook} \cite{jk.rn} \cite{bd.nn} \cite{bd}. \nocite{GAP4} \nocite{MATLAB:2014}\par
	%\section{Literature review}
	For a polyhedral complex $ X $ with $ v\in V(X) $ , let $ N(v)=\{u\in V(X):\exists\sigma\in X \text{ such that }\allowbreak u,v\in\sigma \} $. For all $ i\geq 0 $, we define the graph $ G_{i}(X) $ as follows: $ V(G_{i}(X))=V(EG(X)) $ and $ [u,v]\in EG(G_{i}(X))  $ if $ |N(u)\cap N(v)|=i $. Clearly, if $ X $ and $ Y $ are isomorphic, then $ G_{i}(X) $ and $ G_{i}(Y) $ are isomorphic for all $ i\geq 0 $.
	The \textit{Face sequence} (see figure \ref{tab:fc}) of a vertex $v$ in a map is a finite cyclically ordered sequence $(a_{1}^{n_{1}},a_{2}^{n_{2}},...,a_{m}^{n_{m}})$ of powers of positive integers $a_{1},...,a_{m}\geq 3$ and $n_{1},...,n_{m}\geq 1$, such that through the vertex $v$, $n_{1}$ numbers of $C_{a_{1}}$, $n_{2}$ numbers of $C_{a_{2}}$, ... , $n_{m}$ numbers of $C_{a_{m}}$ are incident in the given cyclic order where $ C_i $ is a cycle of length $ i $. A polyhedral complex $X$ is called \textit{Semi-Equivelar Map} if the face sequence of each vertex of $X$ is same. We write Semi-Equivelar Map as SEM. A SEM with face sequence $(a_{1}^{n_{1}},a_{2}^{n_{2}},...,a_{m}^{n_{m}})$ is also called SEM of type $(a_{1}^{n_{1}},a_{2}^{n_{2}},...,a_{m}^{n_{m}})$. For more details, see \cite{aku.akt}.
	For $\chi =0$ with 12 vertices, SEM are classified in \cite{bd.aku2}. In this section we present the existence of maps in the surface of Euler characteristics($ \chi $) $ -2 $.
	\begin{figure}
		\begin{center}
			
			\begin{tikzpicture}[scale=1,line width=.5pt]{gull}
			\draw (0,0) node[anchor=north west]{\footnotesize{$0$}}
			-- (1,0) node[anchor=west]{\footnotesize{$1$}}
			-- (1,1) node[anchor=west]{\footnotesize{$2$}}
			-- cycle;
			\draw (0,0)
			-- (1,1)
			-- (0,1) node[anchor=south]{\footnotesize{$3$}}
			-- cycle;
			\draw (0,0)
			-- (0,1)
			-- (-1,1) node[anchor=east]{\footnotesize{$4$}}
			-- cycle;
			\draw (0,0)
			-- (-1,1)
			-- (-1,0) node[anchor=east]{\footnotesize{$5$}}
			-- cycle;
			\draw (0,0)
			-- (-1,0)
			-- (-1,-1) node[anchor=east]{\footnotesize{$6$}}
			-- (0,-1) node[anchor=north]{\footnotesize{$7$}}
			-- cycle;
			\draw (0,0)
			-- (0,-1)
			-- (1,-1) node[anchor=west]{\footnotesize{$8$}}
			-- (1,0)
			-- cycle;
			\node at (0,-2){$(3^4,4^2)$};
			\end{tikzpicture}
			\begin{tikzpicture}[scale=1,line width=.5pt]
			\draw (0,0) node[anchor=north west]{\footnotesize{$0$}}
			-- (1,0) node[anchor=west]{\footnotesize{$1$}}
			-- (1,1) node[anchor=west]{\footnotesize{$2$}}
			-- cycle;
			\draw (0,0)
			-- (1,1)
			-- (0,1) node[anchor=south]{\footnotesize{$3$}}
			-- cycle;
			\draw (0,0)
			-- (0,1)
			-- (-1,1) node[anchor=east]{\footnotesize{$4$}}
			-- cycle;
			\draw (0,0)
			-- (-1,1)
			-- (-1.5,0) node[anchor=east]{\footnotesize{$5$}}
			-- (-1,-1) node[anchor=east]{\footnotesize{$6$}}
			-- cycle;
			\draw (0,0)
			-- (-1,-1)
			-- (0,-1) node[anchor=north]{\footnotesize{$7$}}
			-- cycle;
			\draw (0,0)
			-- (0,-1)
			-- (1,-1) node[anchor=west]{\footnotesize{$8$}}
			-- (1,0)
			-- cycle;
			\node at (0,-2){$(3^3,4,3,4)$};
			
			\end{tikzpicture}
		\end{center}
		\caption{Some examples of face sequence}
		\label{tab:fc}
	\end{figure}
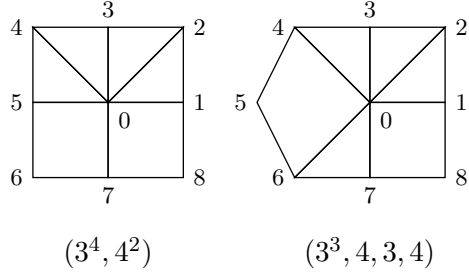
In the present article we show\,:
\begin{prop}\label{prop2}
	A semi equivelar map does not exixt on the surface of double torus with vertices less than or equal to 11.
\end{prop}
\par It can be proved by easy computation using Eular characteristic equation.\par
For $ v=12 $ all possible type of face sequences for this surfaces are
$(3^2,4^2,6)$, $(3^3,6^2)$, $(4^2,6^2)$, $(3^7)$, $(3^4,4^2)$, $(3^3,4,3,4)$, $(3,4^4)$. It is easy to show that corresponding to face sequences $(3^2,4^2,6)$, $(3^3,6^2)$, $(4^2,6^2)$, a semi equivelar map does not exist and the map $(3^7)$ is classified in \cite{bd.aku2} for $leq 12$ vertices.

It has been shown further that the maps described in figure \ref{3(4)4(2)o} and \ref{3(4)4(2)no} are non-isomorphic, (see Lemma \ref{lem3}). Further we will show that\,:
% In this article we also show\,:
\begin{thm}\label{thm4}
	Let $ K $ be a SEM of type $ (3,4^4) $ on the surface of Euler characteristic ($ \chi $) -2 with 12 vertices. If $ K $ is non-orientable then it is isomorphic to one of $ KNO_{1[(3,4^4)]} or KNO_{2[(3,4^4)]} $, given in figure \ref{344no}. If $ K $ is orientable then it is isomorphic to $ KO_{1[(3,4^4)]} $, given in figure \ref{344o}.
\end{thm}
\begin{thm}\label{thm5}
	If $ K $ is a SEM of type $ (3^4,4^2) $ on the surface of Euler characteristic$ (\chi) $ $ -2 $ with $12$ vertices, then $ K $ is isomorphic to one of $ KO_{1[(3^4,4^2)]}, KO_{2[(3^4,4^2)]} $ or $ KNO_{(3^4,4^2)} $, given in figure \ref{3(4)4(2)o} and \ref{3(4)4(2)no}. The maps $ KO_{1[(3^4,4^2)]} $ and $ KO_{2[(3^4,4^2)]} $ are orientable and $ KNO_{[(3^4,4^2)]} $ is non-orientable.
\end{thm}
\begin{thm}\label{thm6}
	Let $ K $ be a SEM of type $ (3^3,4,3,4) $ on the surface of Euler characteristic$ (\chi) $ $ -2 $ with $12$ vertices If $ K $ is orientable then it is isomorphic to one of $ KO_{i[(3^3,4,3,4)]} $ given in figure \ref{33434or}  and if it is non-orientable then it is isomorphic to one of $ KNO_{j[(3^3,4,3,4)]} $ given in figure \ref{33434non}.
\end{thm}

We further show that\,:
\begin{lemma}\label{lem2}
	$ Aut(KNO_{1[(3,4^4)]})=\langle\alpha_{1},\alpha_{2}\rangle\cong\mathbb{Z}_2\times\mathbb{Z}_2 $, $ Aut(KNO_{2[(3,4^4)]})=\langle\beta_{1},\beta_{2},\beta_{3},\beta_{4} \rangle\cong S_4 $, $ Aut(KO_{[(3,4^4)]})=\langle\gamma \rangle\cong\mathbb{Z}_{12} $, where $ \alpha_{1}=(1,6)(2,5)(3,4)(7,9)(10,11) $, $ \alpha_{2}=(0,8)(1,7)\allowbreak(2,3)(4,5)(6,9)(10,11) $, $ \beta_{1}=(1,6)(2,5)(3,4)(7,9)(10,11) $, $ \beta_{2}=(0,6,1)(2,8,5)(3,7,10)(4,\allowbreak11,9) $, $ \beta_{3}=(0,8)(1,7)\allowbreak (2,4)(3,5)(6,9)(10,11) $, $ \beta_{4}=(0,10)(1,4)(2,7)(3,9)(5,6)(8,11) $ and $ \gamma=(0,2,9,10,4,6,1,8,\allowbreak 3,11,5,7) $.
\end{lemma}

\begin{lemma}\label{lem3}
	\begin{sloppypar}
	$ Aut({KO_{1[(3^4,4^2)]}})=\langle\alpha_{1},\alpha_{2},\alpha_{3}\rangle\cong \mathbb{Z}_{2}\times\mathbb{Z}_{2}\times\mathbb{Z}_{3} $, $ Aut({KO_{2[(3^4,4^2)]}})=\langle\beta_{1},\beta_{2},\beta_{3} \rangle\allowbreak\cong D_{6} $ and $ Aut({KNO_{[(3^4,4^2)]}})=\langle\gamma_{1},\gamma_{2} \rangle\cong \mathbb{Z}_{2}\times\mathbb{Z}_{2} $, where $ \alpha_{1}=(0,3, 2)(1, 4, 5)(6, 8, 9)(7, 10,\allowbreak 11) $,
	$ \alpha_{2}=(0,4)(1,2)(3,5)(6,10)(7,9)(8,11) $,
	$ \alpha_{3}=(0,9)(1,11)(2,8)(3,6)(4,7)(5,10) $, $ \beta_{1} =(0,2)(1,3)\allowbreak(4,5)(6,10)(7,9)(8,11) $,$ \beta_{2} =(0, 4,3)(1,2,5)(6,8,9)(7,10,11) $, $ \beta_{3}=(0,6)(1,11)\allowbreak(2,10)(3,8)\allowbreak(4,9)(5,7) $, $\gamma_{1} =(0,4)(1,2)(3,5)(6,10)(7,9)(8,11) $ and $ \gamma_{2}=(0,7)(1,8)(2,11)\allowbreak(3,10)(4,9)\allowbreak(5,6) $.
\end{sloppypar}
\end{lemma}

\begin{lemma}\label{lem4}
	\begin{sloppypar}
	$ Aut(KO_{1[(3^3,4,3,4)]})=\langle\alpha\rangle\cong \mathbb{Z}_2 $, $ Aut(KO_{2[(3^3,4,3,4)]})=\langle\alpha_{1},\alpha_{2}\rangle\cong \mathbb{Z}_2\times\mathbb{Z}_2 $, $ Aut(KNO_{1[(3^3,4,3,4)]})=\langle\beta\rangle\cong\mathbb{Z}_2 $, $ Aut(KNO_{2[(3^3,4,3,4)]})=\langle\beta_{1},\beta_{2}\rangle\cong\mathbb{Z}_2\times\mathbb{Z}_2 $,\allowbreak $ Aut(KNO_{3[(3^3,4,3,4)]})=\langle\gamma\rangle\cong\mathbb{Z}_4 $, where $ \alpha=(0,3)(1,9)(2,4)(5,8)(6,10)(7,11) $, $ \alpha_{1}=(0,2)(1,3)(4,9)(5,11)(6,8)\allowbreak(7,10) $, $ \alpha_{2}=(0,4)(2,9)(5,6)(7,10)(8,11) $, $ \beta=(0,3)(1,9)(2,4)(5,8)(6,11) $, $ \beta_{1}=(0,1)(3,9)\allowbreak(5,10)(6,11)(7,8) $, $ \beta_{2}=(0,3)(1,9)(2,4)(5,8)(7,10) $, $ \gamma=(0,1,10,6)(2,11,5,7)(3,4,8,9) $.
\end{sloppypar}
\end{lemma}

%--------------------------------------------
The article is organized in the following manner. In the next section, we present examples of semi-equivelar maps on the surface of Euler characteristic $ -2 $. In the section following it, we describe the results and their proofs.
The technique of the proofs involve exhaustive search for the desired objects. This leads to a case by case considerations and enumeration of these objects.

\newpage

\section{Examples: Semi-equivelar maps on the surface of Euler characteristic $ -2 $}\label{example 4.1}
\begin{center}
	\textbf{\captionof{table}{Tabulated list of SEMs on the surface of $ \chi=-2 $ with 12 verticex}} \label{tab1}
	\begin{longtable}{ |c|c|C{3cm}|C{6cm}|c| }
%		\caption {Table Title} \label{tab1}
		\hline
		Type & Maps($ M $) & $ Aut(M) $& Orbits & comment \\
		\hline
		& $ KO_{1[(3^4,4^2)]} $ & & &\\
		\cline{2-5}
		$ (3^4,4^2) $ & $ KO_{2[(3^4,4^2)]} $ & $ D_6 $ & $ [0,1,2]_{12}, [0,3,4]_{4},[0,5,6,7]_{6} $ & 3-isohedral \\
		\cline{2-5}
		& $ KNO_{[(3^4,4^2)]} $ & $ \mathbb{Z}_2\times\mathbb{Z}_2 $ &$ [0,1,2]_4 $, $ [0,2,3]_4 $, $ [0,3,4]_4 $, $ [1,5,10]_4 $, $ [0,5,6,7]_2 $, $ [0,1,8,7]_2 $, $ [1,8,3,10]_2 $ & 7-isohedral \\
		\hline
		$ (3,4^4) $ & $ KNO_{1[(3,4^4)]} $ & $ \mathbb{Z}_2\times\mathbb{Z}_2 $ & $ [0,3,4]_2 $, $ [1,9,10]_2 $, $ [0,1,2,3]_4 $, $ [0,6,7,8]_2 $, $ [1,2,11,6]_4 $, $ [2,5,10,11]_2 $ & 5-isohedral \\
		\cline{2-5}
		& $ KNO_{2[(3,4^4)]} $ & $ S_{4} $ & $ [0,3,4]_4 $, $ [0,1,2,3]_{12} $ & 2-isohedral\\
		\cline{2-5}
		& $ KO_{[(3,4^4)]} $ & $ \mathbb{Z}_{12} $& $ [0,3,4]_4 $, $ [0,1,2,3]_{12} $ & 2-isohedral\\
		\hline
		$ (3^3,4,3,4) $ & $ KO_{1[(3^3,4,3,4)]} $ & $ \mathbb{Z}_2 $ & $ [0,1,2]_2 $, $ [0,2,3]_2 $, $ [1,2,9]_2 $, $ [0,6,7]_2 $, $ [5,7,9]_2 $, $ [2,6,8]_2 $, $ [5,7,10]_2 $, $ [5,6,11]_2 $, $ [0,4,5,6]_2 $, $ [0,1,8,7]_2 $, $ [2,6,7,9]_2 $ & 11-isohedral \\
		\cline{2-5}
		& $ KO_{2[(3^3,4,3,4)]} $ & $ \mathbb{Z}_2\times\mathbb{Z}_2 $ & $ [0,1,2]_4 $, $ [0,3,4]_2 $, $ [0,6,7]_4 $, $ [3,5,6]_2 $, $ [5,8,20]_2 $, $ [5,7,8]_2 $, $ [0,4,5,6]_2 $, $ [0,1,8,7]_4 $ & 8-isohedral\\
		\cline{2-5}
		& $ KNO_{1[(3^3,4,3,4)]} $ & $ \mathbb{Z}_2 $ & $ [0,1,2]_2 $, $ [0,2,3]_2 $, $ [0,6,7]_2 $, $ [1,2,9]_2 $, $ [6,7,11]_1 $, $ [1,8,11]_2 $, $ [2,8,10]_2 $, $ [5,8,10]_1 $, $ [5,7,8]_1 $, $ [6,10,11]_1 $, $ [0,4,5,6]_2 $, $ [0,1,8,7]_2 $, $ [1,4,10,11]_2 $ & 13-isohedral\\
		\cline{2-5}
		& $ KNO_{2[(3^3,4,3,4)]} $ & $ \mathbb{Z}_2\times\mathbb{Z}_2 $ & $ [1,2,3]_2 $, $ [0,2,3]_4 $, $ [0,6,7] $, $ [2,7,8]_2 $, $ [5,6,8]_2 $, $ [6,7,10]_2 $, $ [0,4,5,6]_4 $, $ [0,1,8,7]_2 $ &8-isohedral \\
		\cline{2-5}
		& $ KNO_{3[(3^3,4,3,4)]} $ & $ \mathbb{Z}_4 $ & &\\
		\hline
	\end{longtable}
\end{center}
\subsection{Figure of existing SEMs}
%\textbf{\captionof{Examples}{Tabulated list of SEMs on the surface of $ \chi=-1 $ and $ \chi=-2 $ with 12 verticex}} \label{exm1}
\begin{figure}[H]
	
	\begin{center}
		\begin{tikzpicture}[scale=0.9,line width=.5pt]
		\draw (0,0) node[anchor=south east]{\footnotesize{$0$}}
		-- (1,0) node[anchor=south west]{\footnotesize{$1$}}
		-- (1,1) node[anchor=north west]{\footnotesize{$9$}}
		-- (0,1) node[anchor=north east]{\footnotesize{$8$}}
		-- cycle;
		\draw (0,0)
		-- (0,1)
		-- (-1,1) node[anchor=north east]{\footnotesize{$7$}}
		-- (-1,0) node[anchor=south east]{\footnotesize{$6$}}
		--cycle;
		\draw (0,0)
		-- (-0.5,-1) node[anchor=north east]{\footnotesize{$4$}}
		-- (-1.5,-1) node[anchor=north]{\footnotesize{$5$}}
		-- (-1,0)
		-- cycle;
		\draw (0,0)
		-- (0.5,-1) node[anchor=north west]{\footnotesize{$3$}}
		-- (1.5,-1) node[anchor=north]{\footnotesize{$2$}}
		-- (1,0)
		-- cycle;
		\draw (0,0)
		-- (0.5,-1)
		-- (-0.5,-1)
		-- cycle;
		\draw (-1,0)
		-- (-2,0) node[anchor=south]{\footnotesize{$1$}}
		-- (-2.5,-1) node[anchor=east]{\footnotesize{$10$}}
		-- (-1.5,-1)
		-- cycle;
		\draw (1,0)
		-- (2,0) node[anchor=south]{\footnotesize{$6$}}
		-- (2.5,-1) node[anchor=west]{\footnotesize{$11$}}
		-- (1.5,-1)
		-- cycle;
		\draw (-2,0)
		-- (-2.5,-1)
		-- (-2.5,0) node[anchor=east]{\footnotesize{$ 9 $}}
		--cycle;
		\draw (2,0)
		-- (2.5,-1)
		-- (2.5,0) node[anchor=west]{\footnotesize{$ 7 $}}
		--cycle;
		\draw (0,1)
		-- (0.5,2) node[anchor=south west]{\footnotesize{$5$}}
		-- (1.5,2) node[anchor=south]{\footnotesize{$4$}}
		-- (1,1)
		-- cycle;
		\draw (0,1)
		-- (-0.5,2) node[anchor=south east]{\footnotesize{$2$}}
		-- (-1.5,2) node[anchor=south]{\footnotesize{$3$}}
		-- (-1,1)
		-- cycle;
		\draw (0,1)
		-- (0.5,2)
		-- (-0.5,2)
		-- cycle;
		\draw (-1,1)
		-- (-2,1) node[anchor=east]{\footnotesize{$9$}}
		-- (-2.5,2) node[anchor=east]{\footnotesize{$10$}}
		-- (-1.5,2)
		-- cycle;
		\draw (1,1)
		-- (2,1) node[anchor=west]{\footnotesize{$7$}}
		-- (2.5,2) node[anchor=west]{\footnotesize{$11$}}
		-- (1.5,2)
		-- cycle;
		\draw (-0.5,2)
		-- (0.5,2)
		-- (0.5,3) node[anchor=south]{\footnotesize{$10$}}
		-- (-0.5,3) node[anchor=south]{\footnotesize{$11$}}
		-- cycle;
		\draw (-0.5,-1)
		-- (0.5,-1)
		-- (0.5,-2) node[anchor=north]{\footnotesize{$10$}}
		-- (-0.5,-2) node[anchor=north]{\footnotesize{$11$}}
		-- cycle;
		\node at (0,-2.7){$\boldsymbol{KNO_{1[(3,4^4)]}}$};%(Aut(KNO_{1[(3,4^4)]})=\mathbb{Z}_2\times\mathbb{Z}_2)$};
		\end{tikzpicture}
		\begin{tikzpicture}[scale=0.9,line width=.5pt]
		\draw (0,0) node[anchor=west]{\footnotesize{$ 0 $}}
		-- (-0.5,1) node[anchor=south west]{\footnotesize{$ 3 $}}
		-- (0.5,1) node[anchor=south east]{\footnotesize{$ 4 $}}
		-- cycle;
		\draw (0,0)
		-- (-1,-0.5) node[anchor=south]{\footnotesize{$ 1 $}}
		-- (-1.5,0.5) node[anchor=south]{\footnotesize{$ 2 $}}
		-- (-0.5,1)
		--cycle;
		\draw (-1,-0.5)
		-- (-2,-1) node[anchor=north]{\footnotesize{$ 6 $}}
		-- (-2.5,0) node[anchor=south]{\footnotesize{$ 11 $}}
		-- (-1.5,0.5)
		-- cycle;
		\draw (-2,-1)
		-- (-2.5,0)
		-- (-3,-1) node[anchor=east]{\footnotesize{$ 7 $}}
		-- cycle;
		\draw (0,0)
		-- (1,-0.5) node[anchor=south]{\footnotesize{$ 6 $}}
		-- (1.5,0.5) node[anchor=south]{\footnotesize{$ 5 $}}
		-- (0.5,1)
		--cycle;
		\draw (1,-0.5)
		-- (2,-1) node[anchor=north]{\footnotesize{$ 1 $}}
		-- (2.5,0) node[anchor=south]{\footnotesize{$ 10 $}}
		-- (1.5,0.5)
		-- cycle;
		\draw (2,-1)
		-- (2.5,0)
		-- (3,-1) node[anchor=west]{\footnotesize{$ 9 $}}
		-- cycle;
		\draw (-0.5,1)
		-- (0.5,1)
		-- (0.5,2) node[anchor=north west]{\footnotesize{$ 10 $}}
		-- (-0.5,2) node[anchor=north east]{\footnotesize{$ 11 $}}
		-- cycle;
		\draw (-0.5,2)
		-- (0.5,2)
		-- (0.5,3) node[anchor=west]{\footnotesize{$ 5 $}}
		-- (-0.5,3) node[anchor=east]{\footnotesize{$ 2 $}}
		-- cycle;
		\draw (-0.5,3)
		-- (0.5,3)
		-- (0,4) node[anchor=south]{\footnotesize{$ 8 $}}
		-- cycle;
		\draw (-0.5,1)
		-- (-1.5,0.5)
		-- (-2.5,1) node[anchor=east]{\footnotesize{$ 8 $}}
		-- (-1.5,1.5) node[anchor=south east]{\footnotesize{$ 9 $}}
		-- cycle;
		\draw (-0.5,1)
		-- (-1.5,1.5)
		-- (-1.5,2.5) node[anchor=south]{\footnotesize{$ 7 $}}
		-- (-0.5,2)
		-- cycle;
		\draw (0.5,1)
		-- (1.5,0.5)
		-- (2.5,1) node[anchor=west]{\footnotesize{$ 8 $}}
		-- (1.5,1.5) node[anchor=south west]{\footnotesize{$ 7 $}}
		-- cycle;
		\draw (0.5,1)
		-- (1.5,1.5)
		-- (1.5,2.5) node[anchor=south]{\footnotesize{$ 9 $}}
		-- (0.5,2)
		-- cycle;
		\draw (0,0)
		-- (-1,-0.5)
		-- (-1,-1.5) node[anchor=north]{\footnotesize{$ 9 $}}
		-- (0,-1) node[anchor=north]{\footnotesize{$ 8 $}}
		-- cycle;
		\draw (0,0)
		-- (1,-0.5)
		-- (1,-1.5) node[anchor=north]{\footnotesize{$ 7 $}}
		-- (0,-1)
		-- cycle;
		\node at (0,-2.5){$\boldsymbol{KNO_{2[(3,4^4)]}}$};%(Aut(KNO_{2[(3,4^4)]})=S_4)$};
		\end{tikzpicture}
	\end{center}
	\caption{Non-rientable SEM of type $ (3,4^4) $}
	\label{344no}
\end{figure}
\begin{figure}[H]
	
	\begin{center}
		\begin{tikzpicture}[scale=1,line width=.5pt]
		\draw (0,0) node[anchor=west]{\footnotesize{$ 0 $}}
		-- (-0.5,1) node[anchor=south west]{\footnotesize{$ 3 $}}
		-- (0.5,1) node[anchor=south east]{\footnotesize{$ 4 $}}
		-- cycle;
		\draw (0,0)
		-- (-1,-0.5) node[anchor=south]{\footnotesize{$ 1 $}}
		-- (-1.5,0.5) node[anchor=south]{\footnotesize{$ 2 $}}
		-- (-0.5,1)
		--cycle;
		\draw (-1,-0.5)
		-- (-2,-1) node[anchor=north]{\footnotesize{$ 7 $}}
		-- (-2.5,0) node[anchor=south]{\footnotesize{$ 6 $}}
		-- (-1.5,0.5)
		-- cycle;
		\draw (0,0)
		-- (1,-0.5) node[anchor=south]{\footnotesize{$ 6 $}}
		-- (1.5,0.5) node[anchor=south]{\footnotesize{$ 5 $}}
		-- (0.5,1)
		--cycle;
		\draw (1,-0.5)
		-- (2,-1) node[anchor=north]{\footnotesize{$ 11 $}}
		-- (2.5,0) node[anchor=south]{\footnotesize{$ 10 $}}
		-- (1.5,0.5)
		-- cycle;
		\draw (-0.5,1)
		-- (0.5,1)
		-- (0.5,2) node[anchor=north west]{\footnotesize{$ 11 $}}
		-- (-0.5,2) node[anchor=north east]{\footnotesize{$ 9 $}}
		-- cycle;
		\draw (-0.5,2)
		-- (0.5,2)
		-- (0.7,2.8) node[anchor=west]{\footnotesize{$ 6 $}}
		-- (0,3.3) node[anchor=south]{\footnotesize{$ 2 $}}
		-- (-0.7,2.8) node[anchor=east]{\footnotesize{$ 8 $}}
		-- cycle;
		\draw (0,3.3)
		-- (0.5,2)
		-- cycle;
		\draw (-0.5,1)
		-- (-1.5,0.5)
		-- (-2.5,1) node[anchor=east]{\footnotesize{$ 8 $}}
		-- (-1.5,1.5) node[anchor=north]{\footnotesize{$ 10 $}}
		-- cycle;
		\draw (-0.5,1)
		-- (-1.5,1.5)
		-- (-1.5,2) node[anchor=east]{\footnotesize{$ 5 $}}
		-- (-0.5,2)
		-- cycle;
		\draw (-1.5,2)
		-- (-1.5,2.5) node[anchor=south]{\footnotesize{$ 1 $}}
		-- (-0.5,2)
		-- cycle;
		\draw (0.5,1)
		-- (1.5,0.5)
		-- (2.5,1) node[anchor=west]{\footnotesize{$ 1 $}}
		-- (1.5,1.5) node[anchor=south west]{\footnotesize{$ 7 $}}
		-- cycle;
		\draw (0.5,1)
		-- (1.5,1.5)
		-- (1.5,2.5) node[anchor=south]{\footnotesize{$ 10 $}}
		-- (0.5,2)
		-- cycle;
		\draw (0,0)
		-- (-1,-0.5)
		-- (-1,-1.5) node[anchor=north]{\footnotesize{$ 9 $}}
		-- (0,-1) node[anchor=north]{\footnotesize{$ 8 $}}
		-- cycle;
		\draw (0,0)
		-- (1,-0.5)
		-- (1,-1) node[anchor=west]{\footnotesize{$ 7 $}}
		-- (0,-1)
		-- cycle;
		\draw (0,-1)
		-- (1,-1.5) node[anchor=north]{\footnotesize{$ 10 $}}
		-- (1,-1)
		-- cycle;
		\node at (0,-2.5){$\boldsymbol{KO_{[(3,4^4)]}}$};%(Aut(KO_{[(3,4^4)]})=\mathbb{Z}_{12})$};
		\end{tikzpicture}
	\end{center}
	\caption{Orientable SEM of type $ (3,4^4) $}
	\label{344o}
\end{figure}
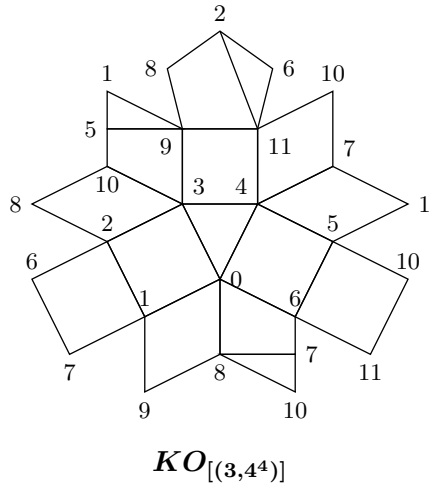
\begin{figure}[H]
	\begin{center}
		\begin{tikzpicture}[scale=0.9,line width=.5pt]
		\draw (0,0) node[anchor=north east]{\footnotesize{$0$}}
		-- (1,0) node[anchor=north east]{\footnotesize{$1$}}
		-- (1,-1) node[anchor=south east]{\footnotesize{$8$}}
		-- (0,-1) node[anchor=south east]{\footnotesize{$7$}}
		-- cycle;
		\draw (0,0)
		-- (-1,0) node[anchor=north west]{\footnotesize{$5$}}
		-- (-1,-1) node[anchor=south west]{\footnotesize{$6$}}
		-- (0,-1)
		-- cycle;
		\draw (0,0)
		-- (-1,0)
		-- (-1,1) node[anchor=south]{\footnotesize{$4$}}
		-- cycle;
		\draw (0,0)
		-- (1,0)
		-- (1,1) node[anchor=south west]{\footnotesize{$2$}}
		-- cycle;
		\draw (0,0)
		-- (-1,1)
		-- (0,1) node[anchor=south east]{\footnotesize{$3$}}
		-- cycle;
		\draw (0,0)
		-- (1,1)
		-- (0,1)
		-- cycle;
		\draw (-1,0)
		-- (-1,-1)
		-- (-2,-1) node[anchor=north]{\footnotesize{$11$}}
		-- (-2,0) node[anchor=south]{\footnotesize{$2$}}
		-- cycle;
		\draw (1,0)
		-- (1,-1)
		-- (2,-1) node[anchor=north]{\footnotesize{$10$}}
		-- (2,0) node[anchor=south]{\footnotesize{$3$}}
		-- cycle;
		\draw (2,0)
		-- (2,-1)
		-- (3,-1) node[anchor=west]{\footnotesize{$9$}}
		-- (3,0) node[anchor=west]{\footnotesize{$4$}}
		-- cycle;
		\draw (-2,0)
		-- (-2,-1)
		-- (-3,-1) node[anchor=east]{\footnotesize{$9$}}
		-- (-3,0) node[anchor=east]{\footnotesize{$4$}}
		-- cycle;
		\draw (0,-1)
		-- (-1,-1)
		-- (-1,-2) node[anchor=north]{\footnotesize{$9$}}
		-- cycle;
		\draw (0,-1)
		-- (1,-1)
		-- (1,-2) node[anchor=north west]{\footnotesize{$11$}}
		-- cycle;
		\draw (0,-1)
		-- (-1,-2)
		-- (0,-2) node[anchor=north east]{\footnotesize{$10$}}
		-- cycle;
		\draw (0,-1)
		-- (1,-2)
		-- (0,-2)
		-- cycle;
		\draw (0,1)
		-- (1,1)
		-- (1,2) node[anchor=south]{\footnotesize{$5$}}
		-- cycle;
		\draw (0,1)
		-- (0,2) node[anchor=south]{\footnotesize{$1$}}
		-- (1,2)
		-- cycle;
		\draw (0,-2)
		-- (1,-2)
		-- (1,-3) node[anchor=north]{\footnotesize{$6$}}
		-- cycle;
		\draw (0,-2)
		-- (0,-3) node[anchor=north]{\footnotesize{$8$}}
		-- (1,-3)
		-- cycle;
		\draw (-1,0)
		-- (-1,1)
		-- (-2,1) node[anchor=east]{\footnotesize{$1$}}
		-- cycle;
		\draw (1,0)
		-- (1,1)
		-- (2,1) node[anchor=west]{\footnotesize{$4$}}
		-- cycle;
		\draw (-1,-1)
		-- (-1,-2)
		-- (-2,-2) node[anchor=east]{\footnotesize{$8$}}
		-- cycle;
		\draw (1,-1)
		-- (1,-2)
		-- (2,-2) node[anchor=west]{\footnotesize{$9$}}
		-- cycle;
		\node at (0,-4){$\boldsymbol{KO_{1[(3^4,4^2)]}}$};%[Aut(\boldsymbol{KO_{1[(3^4,4^2)]}})\cong\mathbb{Z}_{2}\times \mathbb{Z}_{2}\times\mathbb{Z}_{3}]$};
		\end{tikzpicture}
		\begin{tikzpicture}[scale=0.9,line width=.5pt]
		\draw (0,0) node[anchor=north east]{\footnotesize{$0$}}
		-- (1,0) node[anchor=north east]{\footnotesize{$1$}}
		-- (1,-1) node[anchor=south east]{\footnotesize{$8$}}
		-- (0,-1) node[anchor=south east]{\footnotesize{$7$}}
		-- cycle;
		\draw (0,0)
		-- (-1,0) node[anchor=north west]{\footnotesize{$5$}}
		-- (-1,-1) node[anchor=south west]{\footnotesize{$6$}}
		-- (0,-1)
		-- cycle;
		\draw (0,0)
		-- (-1,0)
		-- (-1,1) node[anchor=south east]{\footnotesize{$4$}}
		-- cycle;
		\draw (0,0)
		-- (1,0)
		-- (1,1) node[anchor=south]{\footnotesize{$2$}}
		-- cycle;
		\draw (0,0)
		-- (-1,1)
		-- (0,1) node[anchor=south west]{\footnotesize{$3$}}
		-- cycle;
		\draw (0,0)
		-- (1,1)
		-- (0,1)
		-- cycle;
		\draw (-1,0)
		-- (-1,-1)
		-- (-2,-1) node[anchor=north]{\footnotesize{$11$}}
		-- (-2,0) node[anchor=south]{\footnotesize{$3$}}
		-- cycle;
		\draw (1,0)
		-- (1,-1)
		-- (2,-1) node[anchor=north]{\footnotesize{$10$}}
		-- (2,0) node[anchor=south]{\footnotesize{$4$}}
		-- cycle;
		\draw (2,0)
		-- (2,-1)
		-- (3,-1) node[anchor=west]{\footnotesize{$9$}}
		-- (3,0) node[anchor=west]{\footnotesize{$2$}}
		-- cycle;
		\draw (-2,0)
		-- (-2,-1)
		-- (-3,-1) node[anchor=east]{\footnotesize{$9$}}
		-- (-3,0) node[anchor=east]{\footnotesize{$2$}}
		-- cycle;
		\draw (0,-1)
		-- (-1,-1)
		-- (-1,-2) node[anchor=north]{\footnotesize{$9$}}
		-- cycle;
		\draw (0,-1)
		-- (1,-1)
		-- (1,-2) node[anchor=north west]{\footnotesize{$11$}}
		-- cycle;
		\draw (0,-1)
		-- (-1,-2)
		-- (0,-2) node[anchor=north east]{\footnotesize{$10$}}
		-- cycle;
		\draw (0,-1)
		-- (1,-2)
		-- (0,-2)
		-- cycle;
		\draw (0,1)
		-- (-1,1)
		-- (-1,2) node[anchor=south]{\footnotesize{$1$}}
		-- cycle;
		\draw (0,1)
		-- (0,2) node[anchor=south]{\footnotesize{$5$}}
		-- (-1,2)
		-- cycle;
		\draw (0,-2)
		-- (1,-2)
		-- (1,-3) node[anchor=north]{\footnotesize{$6$}}
		-- cycle;
		\draw (0,-2)
		-- (0,-3) node[anchor=north]{\footnotesize{$8$}}
		-- (1,-3)
		-- cycle;
		\draw (-1,0)
		-- (-1,1)
		-- (-2,1) node[anchor=east]{\footnotesize{$2$}}
		-- cycle;
		\draw (1,0)
		-- (1,1)
		-- (2,1) node[anchor=west]{\footnotesize{$5$}}
		-- cycle;
		\draw (-1,-1)
		-- (-1,-2)
		-- (-2,-2) node[anchor=east]{\footnotesize{$8$}}
		-- cycle;
		\draw (1,-1)
		-- (1,-2)
		-- (2,-2) node[anchor=west]{\footnotesize{$9$}}
		-- cycle;
		\node at (0,-4){$\boldsymbol{KO_{2[(3^4,4^2)]}}$};%[Aut(\boldsymbol{KO_{2[(3^4,4^2)]}})\cong D_{6}]$};
		\end{tikzpicture}
	\end{center}
	\caption{Orientable SEM of type $ (3^4,4^2) $}
	\label{3(4)4(2)o}
\end{figure}
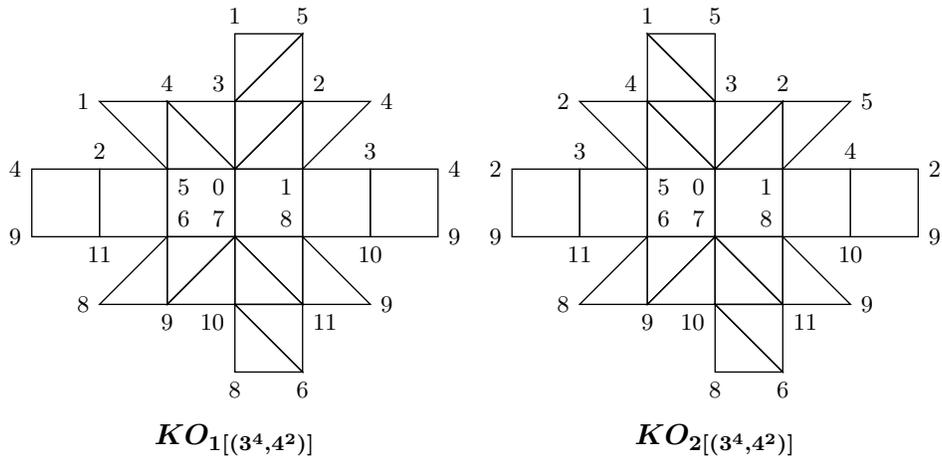
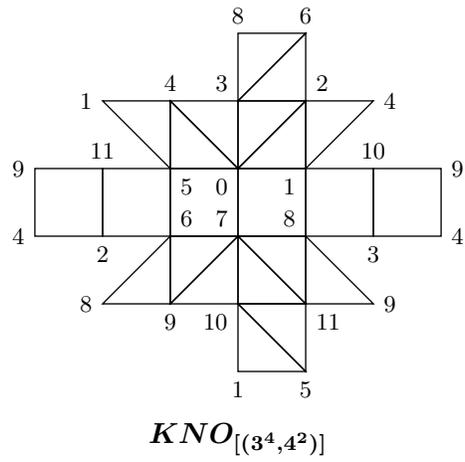
\begin{figure}[H]
	
	\begin{center}
		\begin{tikzpicture}[scale=0.9,line width=.5pt]
		\draw (0,0) node[anchor=north east]{\footnotesize{$0$}}
		-- (1,0) node[anchor=north east]{\footnotesize{$1$}}
		-- (1,-1) node[anchor=south east]{\footnotesize{$8$}}
		-- (0,-1) node[anchor=south east]{\footnotesize{$7$}}
		-- cycle;
		\draw (0,0)
		-- (-1,0) node[anchor=north west]{\footnotesize{$5$}}
		-- (-1,-1) node[anchor=south west]{\footnotesize{$6$}}
		-- (0,-1)
		-- cycle;
		\draw (0,0)
		-- (-1,0)
		-- (-1,1) node[anchor=south]{\footnotesize{$4$}}
		-- cycle;
		\draw (0,0)
		-- (1,0)
		-- (1,1) node[anchor=south west]{\footnotesize{$2$}}
		-- cycle;
		\draw (0,0)
		-- (-1,1)
		-- (0,1) node[anchor=south east]{\footnotesize{$3$}}
		-- cycle;
		\draw (0,0)
		-- (1,1)
		-- (0,1)
		-- cycle;
		\draw (-1,0)
		-- (-1,-1)
		-- (-2,-1) node[anchor=north]{\footnotesize{$2$}}
		-- (-2,0) node[anchor=south]{\footnotesize{$11$}}
		-- cycle;
		\draw (1,0)
		-- (1,-1)
		-- (2,-1) node[anchor=north]{\footnotesize{$3$}}
		-- (2,0) node[anchor=south]{\footnotesize{$10$}}
		-- cycle;
		\draw (2,0)
		-- (2,-1)
		-- (3,-1) node[anchor=west]{\footnotesize{$4$}}
		-- (3,0) node[anchor=west]{\footnotesize{$9$}}
		-- cycle;
		\draw (-2,0)
		-- (-2,-1)
		-- (-3,-1) node[anchor=east]{\footnotesize{$4$}}
		-- (-3,0) node[anchor=east]{\footnotesize{$9$}}
		-- cycle;
		\draw (0,-1)
		-- (-1,-1)
		-- (-1,-2) node[anchor=north]{\footnotesize{$9$}}
		-- cycle;
		\draw (0,-1)
		-- (1,-1)
		-- (1,-2) node[anchor=north west]{\footnotesize{$11$}}
		-- cycle;
		\draw (0,-1)
		-- (-1,-2)
		-- (0,-2) node[anchor=north east]{\footnotesize{$10$}}
		-- cycle;
		\draw (0,-1)
		-- (1,-2)
		-- (0,-2)
		-- cycle;
		\draw (0,1)
		-- (1,1)
		-- (1,2) node[anchor=south]{\footnotesize{$6$}}
		-- cycle;
		\draw (0,1)
		-- (0,2) node[anchor=south]{\footnotesize{$8$}}
		-- (1,2)
		-- cycle;
		\draw (0,-2)
		-- (1,-2)
		-- (1,-3) node[anchor=north]{\footnotesize{$5$}}
		-- cycle;
		\draw (0,-2)
		-- (0,-3) node[anchor=north]{\footnotesize{$1$}}
		-- (1,-3)
		-- cycle;
		\draw (-1,0)
		-- (-1,1)
		-- (-2,1) node[anchor=east]{\footnotesize{$1$}}
		-- cycle;
		\draw (1,0)
		-- (1,1)
		-- (2,1) node[anchor=west]{\footnotesize{$4$}}
		-- cycle;
		\draw (-1,-1)
		-- (-1,-2)
		-- (-2,-2) node[anchor=east]{\footnotesize{$8$}}
		-- cycle;
		\draw (1,-1)
		-- (1,-2)
		-- (2,-2) node[anchor=west]{\footnotesize{$9$}}
		-- cycle;
		\node at (0,-4){$\boldsymbol{KNO_{[(3^4,4^2)]}}$};%[Aut(\boldsymbol{KNO_{[(3^4,4^2)]}})\cong\mathbb{Z}_{2}\times \mathbb{Z}_{2}]$};
		\end{tikzpicture}
		
	\end{center}
	\caption{Non-orientable SEM of type $ (3^4,4^2) $}
	\label{3(4)4(2)no}
\end{figure}
\begin{figure}[H]
	
	\begin{center}
		%\begin{minipage}{.4\textwidth}
		
		%\end{minipage}
		%\begin{minipage}{.4\textwidth}
		\begin{tikzpicture}[scale=0.9,line width=.5pt]
		%\tikzpicture[jd]
		%\begin{scope}
		\draw (0,0) node[anchor=south west]{\footnotesize{$0$}}
		-- (1,0) node[anchor=north east]{\footnotesize{$1$}}
		-- (0,1) node[anchor=north east]{\footnotesize{$2$}}
		-- cycle;

		\draw (0,0)
		-- (-1,1) node[anchor=east]{\footnotesize{$3$}}
		-- (0,1)
		-- cycle;
		\draw (0,0)
		-- (-1,0) node[anchor=east]{\footnotesize{$4$}}
		-- (-1,1)
		-- cycle;
		\draw (0,0)
		-- (-1,0)
		-- (-1.2,-1) node[anchor=east]{\footnotesize{$5$}}
		-- (-.3,-1) node[anchor=north]{\footnotesize{$6$}}
		-- cycle;
		\draw (0,0)
		-- (-.3,-1)
		-- (0.2,-1) node[anchor=north]{\footnotesize{$7$}}
		-- cycle;
		\draw (0,0)
		-- (0.2,-1)
		-- (0.7,-1) node[anchor=north]{\footnotesize{$8$}}
		-- (1,0)
		-- cycle;
		\draw (1,0)
		-- (0.7,-1)
		-- (1.2,-1) node[anchor=south west]{\footnotesize{$11$}}
		-- cycle;
		\draw (1,0)
		-- (1.2,-1)
		-- (2,-1) node[anchor=north]{\footnotesize{$10$}}
		-- (2,0) node[anchor=north east]{\footnotesize{$4$}}
		-- cycle;
		\draw (2.7,0) node[anchor=west]{\footnotesize{$5$}}
		-- (2,-1)
		-- (2,0)
		-- cycle;
		\draw (1,0)
		-- (0,1)
		-- (1,1) node[anchor=north east]{\footnotesize{$9$}}
		-- cycle;
		\draw (1,0)
		-- (1,1)
		-- (2,0)
		-- cycle;
		\draw (0,1)
		-- (-1,1)
		-- (-1.2,2) node[anchor=east]{\footnotesize{$10$}}
		-- (-.3,2) node[anchor=south]{\footnotesize{$8$}}
		-- cycle;
		\draw (-.3,2)
		-- (-1.2,2)
		-- (-1.2,2.7) node[anchor=south]{\footnotesize{$7$}}
		-- cycle;
		\draw (0,1)
		-- (0.2,2) node[anchor=south]{\footnotesize{$6$}}
		-- (-.3,2)
		-- cycle;
		\draw (0,1)
		-- (0.2,2)
		-- (0.7,2) node[anchor=south]{\footnotesize{$7$}}
		-- (1,1)
		-- cycle;
		\draw (0.7,2)
		-- (1.2,2)
		-- (1.5,2.7) node[anchor=south]{\footnotesize{$10$}}
		-- cycle;
		\draw (1,1)
		-- (0.7,2)
		-- (1.2,2) node[anchor=north west]{\footnotesize{$5$}}
		-- cycle;
		\draw (1,1)
		-- (1.2,2)
		-- (2,2) node[anchor=south]{\footnotesize{$11$}}
		-- (2,1) node[anchor=north east]{\footnotesize{$3$}}
		-- cycle;
		\draw (2.7,1) node[anchor=west]{\footnotesize{$10$}}
		-- (2,2)
		-- (2,1)
		-- cycle;
		\draw (1,1)
		-- (2,1)
		-- (2,0)
		-- cycle;
		\draw (0.7,-1)
		-- (1.2,-1)
		-- (1.5,-1.7) node[anchor=north]{\footnotesize{$6$}}
		-- cycle;
		\draw (-.3,-1)
		-- (-1.2,-1)
		-- (-1.2,-1.7) node[anchor=north]{\footnotesize{$11$}}
		-- cycle;
		%\end{scope}
		\node at (.7,-2.5){$\boldsymbol{KO_{1[(3^3,4,3,4)]}}$};
		\end{tikzpicture}
		\begin{tikzpicture}[scale=0.9,line width=.5pt]
		\draw (0,0) node[anchor=south west]{\footnotesize{$0$}}
		-- (1,0) node[anchor=north east]{\footnotesize{$1$}}
		-- (0,1) node[anchor=north east]{\footnotesize{$2$}}
		-- cycle;

		\draw (0,0)
		-- (-1,1) node[anchor=east]{\footnotesize{$3$}}
		-- (0,1)
		-- cycle;
		\draw (0,0)
		-- (-1,0) node[anchor=east]{\footnotesize{$4$}}
		-- (-1,1)
		-- cycle;
		\draw (0,0)
		-- (-1,0)
		-- (-1.2,-1) node[anchor=east]{\footnotesize{$5$}}
		-- (-.3,-1) node[anchor=north east]{\footnotesize{$6$}}
		-- cycle;
		\draw (0,0)
		-- (-.3,-1)
		-- (0.2,-1) node[anchor=north]{\footnotesize{$7$}}
		-- cycle;
		\draw (0,0)
		-- (0.2,-1)
		-- (0.7,-1) node[anchor=north]{\footnotesize{$8$}}
		-- (1,0)
		-- cycle;
		\draw (1,0)
		-- (0.7,-1)
		-- (1.2,-1) node[anchor=south west]{\footnotesize{$11$}}
		-- cycle;
		\draw (1,0)
		-- (1.2,-1)
		-- (2,-1) node[anchor=west]{\footnotesize{$10$}}
		-- (2,0) node[anchor=north east]{\footnotesize{$4$}}
		-- cycle;
		\draw (2.7,0) node[anchor=west]{\footnotesize{$5$}}
		-- (2,-1)
		-- (2,0)
		-- cycle;
		\draw (1,0)
		-- (0,1)
		-- (1,1) node[anchor=north east]{\footnotesize{$9$}}
		-- cycle;
		\draw (1,0)
		-- (1,1)
		-- (2,0)
		-- cycle;
		\draw (0,1)
		-- (-1,1)
		-- (-1.2,2) node[anchor=east]{\footnotesize{$6$}}
		-- (-.3,2) node[anchor=south east]{\footnotesize{$10$}}
		-- cycle;
		
		\draw (0,1)
		-- (0.2,2) node[anchor=south]{\footnotesize{$8$}}
		-- (-.3,2)
		-- cycle;
		\draw (0,1)
		-- (0.2,2)
		-- (0.7,2) node[anchor=south]{\footnotesize{$11$}}
		-- (1,1)
		-- cycle;
		
		\draw (1,1)
		-- (0.7,2)
		-- (1.2,2) node[anchor=south]{\footnotesize{$7$}}
		-- cycle;
		\draw (1,1)
		-- (1.2,2)
		-- (2,2) node[anchor=west]{\footnotesize{$5$}}
		-- (2,1) node[anchor=north east]{\footnotesize{$3$}}
		-- cycle;
		\draw (2.7,1) node[anchor=west]{\footnotesize{$6$}}
		-- (2,2)
		-- (2,1)
		-- cycle;
		\draw (1,1)
		-- (2,1)
		-- (2,0)
		-- cycle;
		\draw (-.3,2)
		-- (.2,2)
		-- (-.5,2.7) node[anchor=south]{\footnotesize{$5$}}
		-- cycle;
		\draw (2,2)
		-- (1.2,2)
		-- (2,2.7) node[anchor=south]{\footnotesize{$8$}}
		-- cycle;
		\draw (2,-1)
		-- (1.2,-1)
		-- (2,-1.7) node[anchor=north]{\footnotesize{$6$}}
		-- cycle;
		\draw (-.3,-1)
		-- (.2,-1)
		-- (-.5,-1.7) node[anchor=north]{\footnotesize{$11$}}
		-- cycle;
		
		\node at (.7,-2.5){$\boldsymbol{KO_{2[(3^3,4,3,4)]}}$};
		\end{tikzpicture}
		
		%\end{minipage}
		%\end{tabular}
	\end{center}
	\caption{Orientable SEMs of type $ (3^3,4,3,4) $ }
	\label{33434or}
\end{figure}
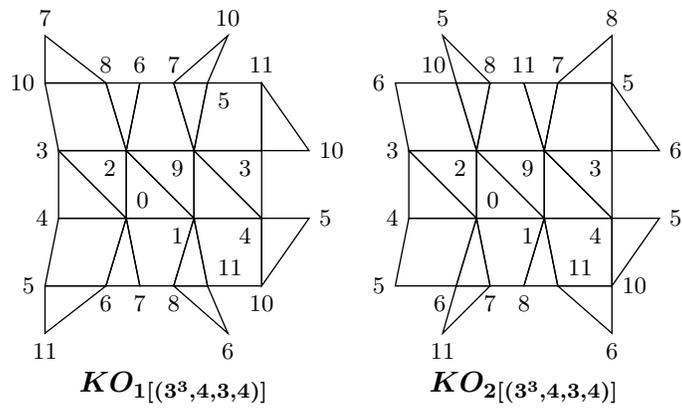
\begin{figure}[H]
	\begin{center}
		\begin{tikzpicture}[scale=0.9,line width=.5pt]
		\draw (0,0) node[anchor=south west]{\footnotesize{$0$}}
		-- (1,0) node[anchor=north east]{\footnotesize{$1$}}
		-- (0,1) node[anchor=north east]{\footnotesize{$2$}}
		-- cycle;
		\draw (0,0)
		-- (-1,1) node[anchor=east]{\footnotesize{$3$}}
		-- (0,1)
		-- cycle;
		\draw (0,0)
		-- (-1,0) node[anchor=east]{\footnotesize{$4$}}
		-- (-1,1)
		-- cycle;
		\draw (0,0)
		-- (-1,0)
		-- (-1.2,-1) node[anchor=east]{\footnotesize{$5$}}
		-- (-.3,-1) node[anchor=north east]{\footnotesize{$6$}}
		-- cycle;
		\draw (0,0)
		-- (-.3,-1)
		-- (0.2,-1) node[anchor=north]{\footnotesize{$7$}}
		-- cycle;
		\draw (0,0)
		-- (0.2,-1)
		-- (0.7,-1) node[anchor=north]{\footnotesize{$8$}}
		-- (1,0)
		-- cycle;
		\draw (1,0)
		-- (0.7,-1)
		-- (1.2,-1) node[anchor=south west]{\footnotesize{$11$}}
		-- cycle;
		\draw (1,0)
		-- (1.2,-1)
		-- (2,-1) node[anchor=west]{\footnotesize{$10$}}
		-- (2,0) node[anchor=north east]{\footnotesize{$4$}}
		-- cycle;
		\draw (2.7,0) node[anchor=west]{\footnotesize{$5$}}
		-- (2,-1)
		-- (2,0)
		-- cycle;
		\draw (1,0)
		-- (0,1)
		-- (1,1) node[anchor=north east]{\footnotesize{$9$}}
		-- cycle;
		\draw (1,0)
		-- (1,1)
		-- (2,0)
		-- cycle;
		\draw (0,1)
		-- (-1,1)
		-- (-1.2,2) node[anchor=east]{\footnotesize{$11$}}
		-- (-.3,2) node[anchor=south east]{\footnotesize{$8$}}
		-- cycle;
		
		\draw (0,1)
		-- (0.2,2) node[anchor=south]{\footnotesize{$10$}}
		-- (-.3,2)
		-- cycle;
		\draw (-.3,2)
		-- (0.2,2)
		-- (-0.3,2.7) node[anchor=south]{\footnotesize{$5$}}
		-- cycle;
		\draw (0,1)
		-- (0.2,2)
		-- (0.7,2) node[anchor=south]{\footnotesize{$6$}}
		-- (1,1)
		-- cycle;
		
		\draw (1,1)
		-- (0.7,2)
		-- (1.2,2) node[anchor=south]{\footnotesize{$5$}}
		-- cycle;
		\draw (1,1)
		-- (1.2,2)
		-- (2,2) node[anchor=west]{\footnotesize{$7$}}
		-- (2,1) node[anchor=north east]{\footnotesize{$3$}}
		-- cycle;
		\draw (1.2,2)
		-- (2,2)
		-- (2,2.7) node[anchor=south]{\footnotesize{$8$}}
		-- cycle;
		\draw (2.7,1) node[anchor=west]{\footnotesize{$11$}}
		-- (2,2)
		-- (2,1)
		-- cycle;
		\draw (1,1)
		-- (2,1)
		-- (2,0)
		-- cycle;
		\draw (2,-1)
		-- (1.2,-1)
		-- (2,-1.7) node[anchor=north]{\footnotesize{$6$}}
		-- cycle;
		\draw (-.3,-1)
		-- (.2,-1)
		-- (-.3,-1.7) node[anchor=north]{\footnotesize{$11$}}
		-- cycle;
		%\end{scope}
		\node at (.7,-2.5){$\boldsymbol{KNO_{1[(3^3,4,3,4)]}}$};
		\end{tikzpicture}
		\begin{tikzpicture}[scale=0.9,line width=.5pt]
		\draw (0,0) node[anchor=south west]{\footnotesize{$0$}}
		-- (1,0) node[anchor=north east]{\footnotesize{$1$}}
		-- (0,1) node[anchor=north east]{\footnotesize{$2$}}
		-- cycle;

		\draw (0,0)
		-- (-1,1) node[anchor=east]{\footnotesize{$3$}}
		-- (0,1)
		-- cycle;
		\draw (0,0)
		-- (-1,0) node[anchor=east]{\footnotesize{$4$}}
		-- (-1,1)
		-- cycle;
		\draw (0,0)
		-- (-1,0)
		-- (-1.2,-1) node[anchor=east]{\footnotesize{$5$}}
		-- (-.3,-1) node[anchor=north east]{\footnotesize{$6$}}
		-- cycle;
		\draw (0,0)
		-- (-.3,-1)
		-- (0.2,-1) node[anchor=north]{\footnotesize{$7$}}
		-- cycle;
		\draw (0,0)
		-- (0.2,-1)
		-- (0.7,-1) node[anchor=north]{\footnotesize{$8$}}
		-- (1,0)
		-- cycle;
		\draw (1,0)
		-- (0.7,-1)
		-- (1.2,-1) node[anchor=south west]{\footnotesize{$11$}}
		-- cycle;
		\draw (1,0)
		-- (1.2,-1)
		-- (2,-1) node[anchor=west]{\footnotesize{$10$}}
		-- (2,0) node[anchor=north east]{\footnotesize{$4$}}
		-- cycle;
		\draw (2.7,0) node[anchor=west]{\footnotesize{$5$}}
		-- (2,-1)
		-- (2,0)
		-- cycle;
		\draw (1,0)
		-- (0,1)
		-- (1,1) node[anchor=north east]{\footnotesize{$9$}}
		-- cycle;
		\draw (1,0)
		-- (1,1)
		-- (2,0)
		-- cycle;
		\draw (0,1)
		-- (-1,1)
		-- (-1.2,2) node[anchor=east]{\footnotesize{$6$}}
		-- (-.3,2) node[anchor=south]{\footnotesize{$8$}}
		-- cycle;
		\draw (-.3,2)
		-- (-1.2,2)
		-- (-1.2,2.7) node[anchor=south]{\footnotesize{$5$}}
		-- cycle;
		\draw (0,1)
		-- (0.2,2) node[anchor=south]{\footnotesize{$7$}}
		-- (-.3,2)
		-- cycle;
		\draw (0,1)
		-- (0.2,2)
		-- (0.7,2) node[anchor=south]{\footnotesize{$11$}}
		-- (1,1)
		-- cycle;
		\draw (0.7,2)
		-- (1.2,2)
		-- (1.5,2.7) node[anchor=south]{\footnotesize{$8$}}
		-- cycle;
		\draw (1,1)
		-- (0.7,2)
		-- (1.2,2) node[anchor=north west]{\footnotesize{$5$}}
		-- cycle;
		\draw (1,1)
		-- (1.2,2)
		-- (2,2) node[anchor=south]{\footnotesize{$10$}}
		-- (2,1) node[anchor=north east]{\footnotesize{$3$}}
		-- cycle;
		\draw (2.7,1) node[anchor=west]{\footnotesize{$6$}}
		-- (2,2)
		-- (2,1)
		-- cycle;
		\draw (1,1)
		-- (2,1)
		-- (2,0)
		-- cycle;
		\draw (2,-1)
		-- (1.2,-1)
		-- (2,-1.7) node[anchor=north]{\footnotesize{$7$}}
		-- cycle;
		\draw (-.3,-1)
		-- (.2,-1)
		-- (-.3,-1.7) node[anchor=north]{\footnotesize{$10$}}
		-- cycle;
		%\end{scope}
		\node at (.7,-2.5){$\boldsymbol{KNO_{2[(3^3,4,3,4)]}}$};
		\end{tikzpicture}
		\begin{tikzpicture}[scale=0.9,line width=.5pt]
		\draw (0,0) node[anchor=west]{\footnotesize{$0$}}
		-- (-0.5,1) node[anchor=north]{\footnotesize{$6$}}
		-- (0.5,1) node[anchor=north]{\footnotesize{$7$}}
		-- cycle;
		\draw (0,0)
		-- (-1,-0.5) node[anchor=north]{\footnotesize{$4$}}
		-- (-1,0.2) node[anchor=west]{\footnotesize{$5$}}
		-- (-0.5,1)
		-- cycle;
		\draw (0,0)
		-- (1,-0.5) node[anchor=south east]{\footnotesize{$1$}}
		-- (1,0.2) node[anchor=east]{\footnotesize{$8$}}
		-- (0.5,1)
		-- cycle;
		\draw (1,-0.5)
		-- (1,0.2)
		-- (1.5,-0.5) node[anchor=north west]{\footnotesize{$10$}}
		-- cycle;
		\draw (1.5,-0.5)
		-- (1,-0.5)
		-- (1,-1) node[anchor=north]{\footnotesize{$4$}}
		-- cycle;
		\draw (1.5,-0.5)
		-- (1,-1)
		-- (1.5,-1) node[anchor=north]{\footnotesize{$11$}}
		-- cycle;
		\draw (-1,-0.5)
		-- (-1,0.2)
		-- (-1.5,-0.5) node[anchor=north]{\footnotesize{$1$}}
		-- cycle;
		\draw (1,0.2)
		-- (1.5,-0.5)
		-- (2,-0.5) node[anchor=west]{\footnotesize{$3$}}
		-- (2,0.2) node[anchor=west]{\footnotesize{$9$}}
		-- cycle;
		\draw (1,0.2)
		-- (2,0.2)
		-- (2,0.8) node[anchor=west]{\footnotesize{$5$}}
		-- cycle;
		\draw (1,0.2)
		-- (2,0.8)
		-- (1.5,1)
		-- cycle;
		
		\draw (-1,0.2)
		-- (-1.5,-0.5)
		-- (-2,-0.5) node[anchor=east]{\footnotesize{$2$}}
		-- (-2,0.2) node[anchor=east]{\footnotesize{$9$}}
		-- cycle;
		\draw (0,0)
		-- (-1,-0.5)
		-- (-0.5,-1) node[anchor=north]{\footnotesize{$3$}}
		-- cycle;
		\draw (0,0)
		-- (1,-0.5)
		-- (0.5,-1) node[anchor=north]{\footnotesize{$2$}}
		-- cycle;
		\draw (-0.5,-1)
		-- (0.5,-1)
		-- (0,-1.5) node[anchor=north]{\footnotesize{$10$}}
		-- cycle;
		\draw (0,0)
		-- (-0.5,-1)
		-- (0.5,-1)
		-- cycle;
		\draw (-0.5,1)
		-- (0.5,1)
		-- (0,2) node[anchor=south]{\footnotesize{$9$}}
		-- cycle;
		\draw (-0.5,1)
		-- (0,2)
		-- (-0.7,2) node[anchor=south]{\footnotesize{$2$}}
		-- cycle;
		\draw (0.5,1)
		-- (0,2)
		-- (0.7,2) node[anchor=south]{\footnotesize{$3$}}
		-- cycle;
		\draw (0.5,1)
		-- (0.7,2)
		-- (1.5,1.5) node[anchor=south]{\footnotesize{$4$}}
		-- (1.5,1) node[anchor=south west]{\footnotesize{$11$}}
		-- cycle;
		\draw (-0.5,1)
		-- (-0.7,2)
		-- (-1.5,1.5) node[anchor=south]{\footnotesize{$10$}}
		-- (-1.5,1) node[anchor=east]{\footnotesize{$11$}}
		-- cycle;
		\draw (-1,0.2)
		-- (-0.5,1)
		-- (-1.5,1)
		-- cycle;
		\draw (1,0.2)
		-- (0.5,1)
		-- (1.5,1)
		-- cycle;

		\node at (-0.1,-2.5){$\boldsymbol{KNO_{3[(3^3,4,3,4)]}}$};
		
		\end{tikzpicture}
		
	\end{center}
	\caption{Non-orientable SEMs of type $ (3^3,4,3,4) $}
	\label{33434non}
\end{figure}
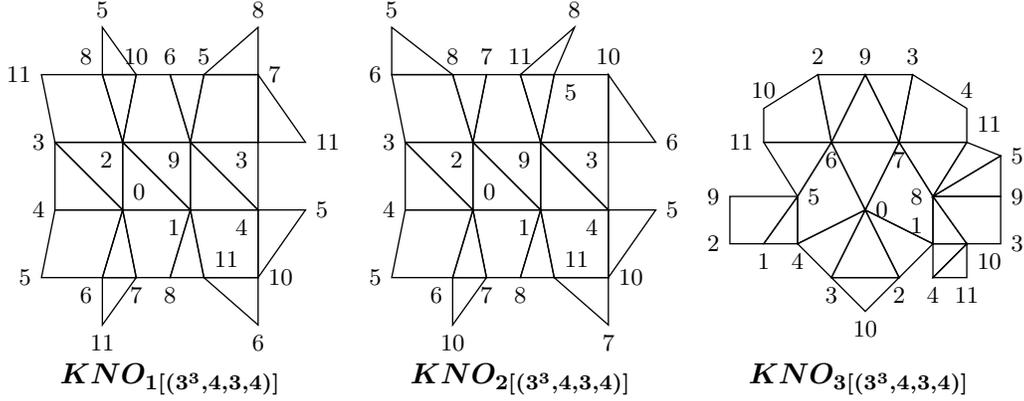

\section{Proof of results}
\begin{lem}\label{figlem1}
$ KNO_{1[(3,4^4)]}\ncong KNO_{2[(3,4^4)]} $, $ KO_{[(3,4^4)]}\ncong KNO_{1[(3,4^4)]},KNO_{2[(3,4^4)]} $, in figure \ref{344no} and \ref{344o}.
\end{lem}
\begin{proof}
	From figure \ref{344no} and \ref{344o}, we see that $ KNO_{1[(3,4^4)]},KNO_{2[(3,4^4)]} $ are non-orientable and $ KO_{[(3,4^4)]} $ is orientable, therefore by Lemma \ref{lem2}, it is clear.
\end{proof}
\begin{lemma}\label{figlem2}
	All $ KO_{1[(3^4,4^2)]},KO_{2[(3^4,4^2)]} $ and $ KNO_{[(3^4,4^2)]} $, in figure \ref{3(4)4(2)o} and \ref{3(4)4(2)no}, are non-isomorphic and $ KO_{1[(3^4,4^2)]},KO_{2[(3^4,4^2)]} $ are orientable, $ KNO_{[(3^4,4^2)]} $ is non-orientable.
\end{lemma}
\begin{proof}
	From figure we see that $ KO_{1[(3^4,4^2)]} $ and $ KO_{2[(3^4,4^2)]} $ are orientable and $ KNO_{[(3^4,4^2)]} $ is non-orientable. Let $ poly({M}) $ be the characteristic polynomial of edge graph of $ {M} $ where $ {M} $ is a SEM. Then(by MATLAB \cite{MATLAB:2014}) \\
	$ poly({KO_{1[(3^4,4^2)]}})=x^{12} - 35x^{10} - 80x^9 + 204x^8 + 1024x^7 + 1456x^6 + 768x^5 + 64x^4 $\\
	$ poly({KO_{2[(3^4,4^2)]}})=x^{12} - 36x^{10} - 80x^9 + 240x^8 + 1152x^7 + 1600x^6 + 768x^5 $.
	
	If $KO_{1[(3^4,4^2)]}$ and $KO_{2[(3^4,4^2)]}$ are isomorphic then its edge graph $EG(KO_{1[(3^4,4^2)]})$ and $EG(KO_{2[(3^4,4^2)]})$ are also isomorphic and then $poly(KO_{1[(3^4,4^2)]})=poly(KO_{2[(3^4,4^2)]})$, \cite{bd.nn}. We see that $poly(KO_{1[(3^4,4^2)]}) \neq poly(KO_{2[(3^4,4^2)]})$. Therefore $KO_{1[(3^4,4^2)]}\ncong KO_{2[(3^4,4^2)]}$.
\end{proof}
\begin{lemma}\label{figlem3}
	$ KO_{1[(3^3,4,3,4)]}\ncong KO_{1[(3,^3,4,3,4)]} $, $ KNO_{i[(3^3,4,3,4)]}\ncong KNO_{j[(3^3,4,3,4)]} $, for all $ i\neq j $, $ i,j=1,2,3 $ in figure \ref{33434or} and \ref{33434non}.
\end{lemma}
\begin{proof}
	\begin{sloppypar}
	We have $ G_7(KO_{1[(3^3,4,3,4)]})=C(0,9,10)\allowbreak\cup C(1,3,6) $, $ G_7(KO_{2[(3^4,4,3,4)]})=C(0,9,10)\cup C(2,4,7) $, $ G_{7}(KNO_{1[(3^3,4,3,4)]})=C(0,9,10)\cup C(1,3,6)\cup C(2,4,7) $, $ G_{7}(KNO_{2[(3^3,4,3,4)]})=\emptyset $, $ G_7(KNO_{3[(3^3,4,3,4)]})=C(0,9,11)\cup C(1,3,5)\cup C(2,6,8) $ and $ poly(KNO_{1[(3^3,4,3,4)]})=x^{12}-48x^{10}-146x^9+72x^8+576x^7+81x^6-648x^5 $, $ poly(KNO_{3[(3^3,4,3,4)]})=x^{12}-48x^{10}-144x^9+66x^8+50x^7+8x^6-57x^5-27x^4+216x^3 $.
	
	We see that $ poly(KNO_{1[(3^3,4,3,4)]})\neq poly(KNO_{3[(3^3,4,3,4)]}) $ and $ G_{7}(KNO_{2[(3^3,4,3,4)]})\neq G_{7}(KNO_{1[(3^3,4,3,4)]}),G_{7}(KNO_{4[(3^3,4,3,4)]}) $, therefore $ KNO_{1[(3^3,4,3,4)]}\ncong KNO_{3[(3^3,4,3,4)]} $, $ KNO_{2[(3^3,4,3,4)]}\allowbreak\ncong KNO_{1[(3^3,4,3,4)]},KNO_{3[(3^3,4,3,4)]} $.
	
	If $ KO_{1[(3,^3,4,3,4)]}\allowbreak\cong KO_{1[(3,^3,4,3,4)]} $, then there exist an isomorphism, say $ f:KO_{1[(3^3,4,3,4)]}\rightarrow KO_{2[(3^3,4,3,4)]} $. Then from $ G_7(KO_{1[(3^3,4,3,4)]}) $ and $ G_7(KO_{2[(3^3,4,3,4)]}) $, $ f(0)\in \{0,2,4,7,9,10\} $. If $ f(0)=0 $ then from $ G_7(KO_{1[(3^3,4,3,4)]}) $, $ G_7(KO_{2[(3^3,4,3,4)]}) $, $ lk(0_{KO_{1[(3^3,4,3,4)]}}) $\footnote{$ lk(0_{M})= lk(0) $ of the polyhedron $ M $}, $ lk(0_{KO_{2[(3^3,4,3,4)]}}) $, $ lk(1_{KO_{1[(3^3,4,3,4)]}}) $ and $ lk(4_{KO_{2[(3^3,4,3,4)]}}) $, we get $ f=(1,4)(2,3)(5,8)(6,7)(10,11) $. But $ [2,9,7,\allowbreak6] $ is a face of $ KO_{1[(3^3,4,3,4)]} $ and $ f([2,9,7,6])=[3,9,6,7] $, $ [3,9,6,7] $ is not a face of $ KO_{2[(3^3,4,3,4)]} $, make a contradiction. If $ f(0)=2 $ then from $ G_7(KO_{1[(3^3,4,3,4)]}) $, $ G_7(KO_{2[(3^3,4,3,4)]}) $, $ lk(0_{KO_{1[(3^3,4,3,4)]}}) $ and $ lk(2_{KO_{2[(3^3,4,3,4)]}}) $, we get $ f(1)=9,f(2)=1,f(3)=0,f(4)=3,f(5)=6,f(6)=10,f(7)=8,f(8)=11 $. Now from $ lk(0_{KO_{2[(3^3,4,3,4)]}}) $ and $ lk(3_{KO_{1[(3^3,4,3,4)]}}) $, we get $ f(5)=5 $, which make contradiction. If $ f(0)=4 $ then from $ G_7(KO_{1[(3^3,4,3,4)]}) $, $ G_7(KO_{2[(3^3,4,3,4)]}) $, $ lk(0_{KO_{1[(3^3,4,3,4)]}}) $, $ lk(0_{KO_{2[(3^3,4,3,4)]}}) $, $ lk(1_{KO_{1[(3^3,4,3,4)]}}) $, $ lk(4_{KO_{2[(3^3,4,3,4)]}}) $, we get $ f=(0,4,1)(2,3,9)\allowbreak (5,11,7)(6,10,8) $. Now $ C(0,9,10) $ is a cycle in $ G_7(KO_{1[(3^3,4,3,4)]}) $, but $ f(C(0,9,10))=C(4,2,8) $ is not a cycle in $ G_7(KO_{2[(3^3,4,3,4)]}) $, make contradiction. If $ f(0)=7 $ then from $ lk(0_{KO_{1[(3^3,4,3,4)]}}) $ and $ lk(7_{KO_{2[(3^3,4,3,4)]}}) $, we have $ f(4)\in \{0,9\} $, but from $ G_7(KO_{1[(3^3,4,3,4)]}) $ and $ G_7(KO_{2[(3^3,4,3,4)]}) $, it is not possible. If $ f(0)=9 $ then from $ G_7(KO_{1[(3^3,4,3,4)]}) $, $ G_7(KO_{2[(3^3,4,3,4)]}) $, $ lk(0_{KO_{[(3^3,4,3,4)]}}) $ and $ lk(9_{KO_{2[(3^3,4,3,4)]}}) $, we get $ f(1)=2, f(2)=1,f(3)=4,f(4)=3,f(5)=5,f(6)=7,f(7)=11,f(8)=8 $ and then $ f([8,2,6])=[8,1,7] $ where $ [8,2,6] $ is a face of $ KO_{1[(3^3,4,3,4)]} $ but $ [8,1,7] $ is not a face of $ KO_{1[(3^3,4,3,4)]} $, which make contradiction. If $ f(0)=10 $ then from $ lk(0_{KO_{1[(3^3,4,3,4)]}}) $ and $ lk(10_{KO_{2[(3^3,4,3,4)]}}) $, we have $ f(4)\in \{2,4\} $ which is not possible, from $ G_7(KO_{1[(3^3,4,3,4)]}) $, $ G_7(KO_{2[(3^3,4,3,4)]}) $. Therefore $ KO_{1[(3,^3,4,3,4)]}\ncong KO_{1[(3,^3,4,3,4)]} $.	
\end{sloppypar}
\end{proof}
%\vskip1in

\begin{proof}[\textbf{Proof of Lemma \ref{lem2}}]
	We have if $ f $ be an automorphism of $ K $, then $ f $ is also an automorphism of $ G_i(K) $ for $ i=1,2,3,...,12 $.
	\begin{sloppypar}
	Now for $ KNO_{1[(3,4^4)]} $, $ G_7(KNO_{1[(3,4^4)]})=C(1,2,6,5)\cup C(3,7,4,9) $ and $ G_8(KNO_{1[(3,4^4)]})=C(1,4,2,9,6,3,5,7)\cup\{[0,8],[0,10],[0,11],[8,10],[8,11],[10,11]\} $. Let $ \alpha\in Aut(KNO_{1[(3,4^4)]}) $, then $ \alpha $ is also an automorphism of $ G_i(KNO_{1[(3,4^4)]}),i=7,8 $, therefore $ \alpha(0)\in\{0,8,10,11\} $. Now if $ \alpha(0)=0 $ then $ \alpha(8)=8 $ and $ \alpha(1)\in\{1,6\} $. If $ \alpha(1)=1 $ then from $ lk(0),lk(1) $ we get $ \alpha(i)=i $ for all $ i\in V $ i.e. $ \alpha=Id_{KNO_{1[(3,4^4)]}} $. If $ \alpha(1)=6 $ then form $ lk(0),lk(1),lk(6) $ of $ KNO_{1[(3,4^4)]} $, we get $ \alpha=\alpha_{1} $. If $ \alpha(0)=8 $ then $ \alpha(1)\in\{7,9\} $. When $ \alpha(1)=7 $ then from $ lk(0),lk(1),lk(7) $ of $KNO_{1[(3,4^4)]} $, we get $ \alpha=\alpha_{2} $ and when $ \alpha(1)=9 $ then form $ lk(0),lk(1),lk(9) $ of $ KNO_{1[(3,4^4)]} $, we get $ \alpha=(0,8)(1,9)(2,4)(3,5)(6,7)=\alpha_1\circ\alpha_{2} $. Now if $ \alpha(0)=10 $ then from $ lk(0) $ we get $ \alpha(1)\in\{3,5\} $. If $ \alpha(1)=3 $ then from $ lk(0),lk(1),lk(3),lk(10) $ we get $ \alpha=(0,10)(1,3,9,4)(2,7)(5,6)(8,11) $ and then $ \alpha(C(1,2,6,\allowbreak5))=C(3,7,5,6)\notin \allowbreak G_7(KNO_{1[(3,4^4)]}) $, therefore $ \alpha\notin Aut(KNO_{1[(3,4^4)]}) $. If $ \alpha(1)=5 $ then similarly we get $ \alpha=(0,10)(1,5,7,4,9,2,6,3)(8,11) $ which implies $ \alpha(C(1,2,6,5))=C(5,6,3,7)\notin G_7(KNO_{1[(3,4^4)]}) $, therefore $ \alpha\notin Aut(KNO_{1[(3,4^4)]}) $. Now if $ \alpha(0)=11 $ then $ \alpha(1)\in\{1,4\} $. If $ \alpha(1)=1 $ then form $ lk(0),lk(1),lk(11) $ of $ KNO_{1[(3,4^4)]} $, we get $ \alpha(2)=2,\alpha(3)=6,\alpha(4)=7,\alpha(5)=9,\alpha(6)=4,\alpha(7)=3,\alpha(8)=10,\alpha(9)=5 $ which implies $ \alpha(C(1,2,6,5))=C(1,2,4,9)\notin G_7(KNO_{1[(3,4^4)]}) $, therefore $ \alpha\notin Aut(KNO_{1[(3,4^4)]}) $. Similarly, if $ \alpha(1)=4 $ we get $ \alpha=(0,11)(1,4,6)(2,9,3,7,5)(8,10) $ and then $ \alpha([4,9,7,11])=[6,3,5,0] $ where $ [4,9,7,11] $ is a face of $ KNO_{1[(3,4^4)]} $, but $ [6,3,5,0] $ is not a face of $ KNO_{1[(3,4^4)]} $, therefore $ \alpha\notin Aut(KNO_{1[(3,4^4)]}) $. Therefore $ Aut(KNO_{1[(3,4^4)]})=\langle\alpha_{1},\alpha_{2}\rangle $ and using GAP \cite{GAP4}, we get $ \langle\alpha_{1},\alpha_{2}\rangle\cong \mathbb{Z}_2 \times\mathbb{Z}_2$.
\end{sloppypar}
	
	Let $ \beta\in Aut(KNO_{2[(3,4^4)]}) $, then $ \beta(0)\in V $. In this paragraph, we consider $ lk(v) $ as link of $ v $ in $ KNO_{2[(3,4^4)]} $. If $ \beta(0)=0 $ then from $ lk(0) $, we get $ \beta(3)\in\{3,4\} $. If $ \beta(3)=3 $ then from $ lk(0),lk(3) $, we get $ \beta=Id_{KNO_{2[(3,4^4)]}} $. If $ \beta(3)=4 $ then from $ lk(0),lk(3),lk(4) $, we get $ \beta=\beta_{1} $. Now if $ \beta(0)=1 $ then from $ lk(0),lk(1) $ we get $ \beta(3)\in \{9,10\} $. If $ \beta(3)=9 $ then from $ lk(0),lk(1),lk(3),lk(9) $, $ \beta=(0,1)(2,8)(3,9)(4,10)(7,11)=\beta_{2}\circ\beta_{1} $ and if $ \beta(3)=10 $ then from $ lk(0),lk(1),lk(3),lk(10) $, we get $ \beta=(0,1,6)(2,5,8)(3,10,7)(4,9,11)=\beta_2\circ\beta_{1}\circ\beta_{2} $. Now if $ \beta(0)=2 $ then from $ lk(0),lk(2) $, we see that $ \beta(3)\in\{5,8\} $. From $ lk(0),lk(1),lk(2),lk(3) $, we get, when $ \beta(3)=5 $ then $ \beta=(0,2,10,7)(1,11,4,8)(3,5,9,6)=\beta_{4}\circ\beta_{1}\circ\beta_{4}\circ\beta_{2}^2 $ and when $ \beta(3)=9 $ then $ \beta=(0,2,9)(1,3,8)(4,5,10)(6,11,7)=\beta_{3}\circ\beta_{2}^2 $. Now if $ \beta(0)=3 $ then from $ lk(0),lk(3) $, we get $ \beta(3)\in\{0,4\} $. From $ lk(0),lk(3),lk(4) $, we get, when $ \beta(3)=0 $ then $ \beta=(0,3)(1,2)(5,10)(6,11)(8,9)=\beta_{1}\circ\beta_{4}\circ\beta_{2} $ and when $ \beta(3)=4 $ then $ \beta=(0,3,4)(1,11,5)(2,10,6)(7,9,8)=\beta_{4}\circ\beta_{2} $. If $ \beta(0)=4 $ then $ lk(0),lk(4) $, we get $ \beta(3)\in\{0,3\} $. From $ lk(0),lk(3),lk(4) $, we get, when $ \beta(3)=0 $ then $ \beta=(0,4,3)(1,5,11)(2,6,10)(7,9,8)=\beta_3\circ\beta_{2}^2\circ\beta_{3} $ and when $ \beta(3)=4 $ then $ \beta=(0,4)(1,10)(2,11)(5,6)(7,8)=\beta_{4}\circ\beta_{2}\circ\beta_{1} $. If $ \beta(0)=5 $ then from $ lk(0),lk(5) $, we get $ \beta(3)\in\{2,8\} $. From $ lk(0),lk(2),lk(3),lk(5),lk(8) $, we get, when $ \beta(3)=2 $ then $ \beta=(0,5,7)(1,10,9)(2,11,3)(4,8,6)=\beta_{4}\circ\beta_{1}\circ\beta_{4}\circ\beta_{2}^2\circ\beta_{1} $ and when $ \beta(3)=8 $ then $ \beta=(0,5,11,9)(1,4,2,7)(3,8,6,10)=\beta_{4}\circ\beta_{1}\circ\beta_{4}\circ\beta_{2} $. If $ \beta(0)=6 $ then from $ lk(0),lk(6) $, we get $ \beta(3)\in\{7,11\} $. From $ lk(0)$, $lk(3)$, $lk(6)$, $lk(7)$, $lk(11) $, we get, when $ \beta(3)=7 $ then $ \beta=\beta_{2} $ and when $ \beta(3)=11 $ then $ \beta=(0,6)(3,11)(4,7)(5,8)(9,10)=\beta_{1}\circ\beta_{2} $. If $ \beta(0)=7 $ then from $ lk(0),lk(7) $, we get $ \beta(3)\in\{6,11\} $. From $ lk(0),lk(3),lk(6),lk(7),lk(11) $, we get, when $ \beta(3)=6 $ then $ \beta=(0,7,10,2)(1,\allowbreak8,4,11)(3,6,9,5)=\beta_{2}\circ\beta_{3}\circ\beta_{1} $ and when $ \beta(3)=11 $ then $ \beta=(0,7,5)(1,9,10)(2,3,11)(4,6,\allowbreak8)=\beta_{3}\circ\beta_{2}\circ\beta_{3}\circ\beta_{2} $. If $ \beta(0)=8 $ then from $ lk(0),lk(8) $, we get $ \beta(3)\in\{2,5\} $. From $ lk(0),lk(2),lk(3),lk(5),lk(8) $, we get, when $ \beta(3)=2 $ then $ \beta=(0,8)(1,9)(2,3)(4,5)(6,7)=\beta_{4}\circ\beta_{1}\circ\beta_{4} $ and when $ \beta(3)=5 $ then $ \beta=\beta_{3} $. If $ \beta(0)=9 $ then from $ lk(0),lk(9) $, we get $ \beta(3)\in\{1,10\} $and then from $ lk(0),lk(1),lk(3),lk(9),lk(10) $, we get, when $ \beta(3)=1 $ then $ \beta=(0,9,2)(1,8,3)(4,10,5)(6,\allowbreak7,11)=\beta_{2}\circ\beta_{3} $ and when $ \beta(3)=10 $ then $ \beta=(0,9,11,5)(1,7,2,4)(3,10,6,8)=\beta_{4}\circ\beta_{1}\circ\beta_{2} $. If $ \beta(0)=10 $ then form $ lk(0),lk(10) $, we get $ \beta(3)\in\{1,9\} $ and then from $ lk(0),lk(1),lk(3),lk(9),lk(10) $, we get, when $ \beta(3)=1 $ then $ \beta=(0,10,8,11)(1,5,7,3)(2,6,\allowbreak4,9)=\beta_{1}\circ\beta_{4} $ and when $ \beta(3)=9 $ then $ \beta=\beta_{4} $. If $ \beta(0)=11 $ then from $ lk(0),lk(11) $, we get $ \beta(3)\in\{6,7\} $ and then from $ lk(0),lk(3),lk(6),lk(7),lk(11) $, we get, when $ \beta(3)=6 $ then $ \beta=(0,11)(1,2)(3,6)(4,7)(5,9)(8,10)=\beta_{4}\circ\beta_{3} $ and when $ \beta(3)=7 $ then $ \beta=(0,11,8,10)(1,3,7,5)(2,9,4,6)=\beta_{4}\circ\beta_{1} $. Therefore, we see that $ Aut(KNO_{2[(3,4^4)]})=\langle\beta_{1},\beta_{2},\beta_{3},\beta_{4}\rangle $ and by GAP\cite{GAP4}, we get $ \langle\beta_{1},\beta_{2},\beta_{3},\beta_{4}\rangle\cong S_4 $.
	
	In this paragraph, we consider $ lk(v) $ link of $ v $ in $ KO_{[(3,4^4)]} $. We have $ G_8(KO_{[(3,4^4)]})=\{$ [0,1], [0,10], [0,11], [1,10], [1,11], [10,11], [2,4], [2,5], [2,8], [4,5], [4,8], [5,8], [3,6], [3,7], [3,9], [6,7], [6,9], [7,9]$\} $. Let $ f\in Aut(KO_{[(3,4^4)]}) $, then $ f(0)\in V $. Now if $ f(0)=0 $ then from $ G_8(KO_{[(3,4^4)]}) $ and $ lk(0) $, we get $ f(1)=1 $ and then from $ lk(0),lk(1) $ we have $ f=Id_{KO_{[(3,4^4)]}}=\gamma^{12} $. If $ f(0)=1 $ then from $ G_8(KO_{[(3,4^4)]}) $ and $ lk(0),lk(1) $, we get $ f(1)=0 $ and therefore from $ lk(0),lk(1),lk(6),lk(7) $, $ f=(0,1)(2,8)(3,9)(4,5)(6,7)(10,11)=\gamma^6 $. If $ f(0)=2 $ then from $ G_8(KO_{[(3,4^4)]}) $ and $ lk(0),lk(2) $, we get $ f(1)=8 $ and therefore from $ lk(0),lk(1),lk(2),lk(8) $, $ f=\gamma $. If $ f(0)=3 $ then from $ G_8(KO_{[(3,4^4)]}) $ and $ lk(0),lk(3) $, we get $ f(1)=9 $ and therefore from $ lk(0),lk(1),lk(3),lk(9) $, $ f=(0,3,4)(1,9,5)(2,11,6)(7,8,10)=\gamma^8 $. Similarly, if $ f(0)=4,5 $ then from $ G_8(KO_{[(3,4^4)]})$,  $lk(0),lk(1),\allowbreak lk(4),lk(5) $, we get $ f=(0,4,3)(1,5,9)(2,6,11)(7,\allowbreak10,8)=\gamma^4 $ and $ f=(0,5,3,1,4,9)(2,7,11,\allowbreak 8,6,10)=\gamma^{10} $, respectively. If $ f(0)=6,7 $ then from $ G_8(KO_{[(3,4^4)]}),lk(0),lk(1),lk(6),lk(7) $, we get $ f=(0,6,5,10,3,2,1,7,4,11,9,8)=\gamma^5 $ and $ f=(0,7,5,11,3,8,1,6,4,10,9,2)=\gamma^{11} $, respectively. If $ f(0)=8 $ then from $ G_8(KO_{[(3,4^4)]}),\allowbreak lk(0),lk(1),lk(2),lk(8) $, we get $ f=(0,8,9,11,\allowbreak4,7,1,\allowbreak2,3,\allowbreak 10,5,6)=\gamma^7 $. If $ f(0)=9 $ then from $ G_8(KO_{[(3,4^4)]}),lk(0),lk(1),lk(3),lk(9) $, we get $ f=(0,9,4,1,3,\allowbreak 5)(2,10,6,8,11,7)=\gamma^2 $. If $ f(0)=10,11 $ then from $ G_8(KO_{[(3,4^4)]}),lk(0)$, $lk(1)$, $lk(10)$, $lk(11) $, we get $ f=(0,10,1,11)\allowbreak(2,4,8,5)(3,7,9,6)=\gamma^3 $ and $ f=(0,11,1,10)\allowbreak(2,5,8,\allowbreak4)(3,6,9,7)=\gamma^9 $, respectively. Therefore, $ Aut(KO_{[(3,4^4)]})=\langle\gamma\rangle $ and by GAP\cite{GAP4}, we get $ \langle\gamma\rangle\cong\mathbb{Z}_{12} $.	
\end{proof}
\begin{proof}[\textbf{Proof of Lemma \ref{lem3}}]
	
	$ G_{5}(KO_{1[(3^4,4^2)]})=\{$[0,2], [0,3], [2,3], [1,4], [1,5], [4,5], [6,8], [6,9], [8,9], [7,10], [7,11], [10,11]$\}$, $ G_{5}(KO_{2[(3^4,4^2)]})=\{$[0,3], [0,4], [3,4], [1,2], [1,5], [2,5], [6,8], [6,9], [8,9], [7,10], [7, 11], [10,11]$\} $ and $ G_{4}(KNO_{[(3^4,4^2)]})=\{$[0,4], [0,7], [0,8], [1,2], [1,7], [1,8], [2,9], [2,11], [3,6], [3,10], [4,9], [4,11], [5,6], [5,10], [7,9], [8,11]$\} $\\
	We see that for $ 1\leq i\leq 3 $ and $ 1\leq j \leq 2 $, $ \alpha_{i} $'s, $ \beta_{i} $'s and $ \gamma_{j} $'s are automorphisms of $ KO_{1[(3^4,4^2)]} $, $ KO_{2[(3^4,4^2)]} $ and $ KNO_{[(3^4,4^2)]} $, respectively. Now let $ \alpha ,\beta , \gamma $ be any automorphism of $ KO_{1[(3^4,4^2)]}, KO_{2[(3^4,4^2)]},KNO_{[(3^4,4^2)]} $, respectively. Then $ \alpha , \beta , \gamma $ induced an automorphism (also denoted by $ \alpha ,\beta ,\gamma $) on $ G_{i}(KO_{1[(3^4,4^2)]}),\allowbreak G_{i}(KO_{2[(3^4,4^2)]}), G_{i}(KNO_{[(3^4,4^2)]}) $ respectively. Now $ \alpha(0)=0,1,2,3,4,5,6,7,8,9,10 $ or $ 11 $. If $ \alpha(0)=3,4 $ or $ 10 $ then $ \alpha = \alpha_{1},\alpha_{2} $ or $ \alpha_{3} $, respectively. Now, if $ \alpha(0)=0 $ then $ lk(0) $ map to $ lk(0) $ under $ \alpha $ and then from induced graph of $ G_{i}(KO_{1[(3^4,4^2)]}) $, we get $ \alpha(2)=2 $ and then $ \alpha(i)=i $ for $ 1\leq i\leq 11 $, which implies $ \alpha =identity= \alpha_{2}\circ \alpha_{2} $. If $ \alpha(0)=1 $ then from $ E_{5,1} $ and to complete $ lk(2) $, $ \alpha(2) $ has to be $ 5 $ and then $ \alpha(1)=3,\alpha(3)=4,\alpha(4)=2,\alpha(5)=0,\alpha(6)=7,\alpha(8)=10,\alpha(9)=11 $. As $ \alpha(1)=3 $, therefore $ \alpha(10)=9 $ and $ \alpha(11)=6 $. Therefore $ \alpha=(0,1,3,4,2,5)(6,7,8,10,9,11)=\alpha_{1}^2\circ\alpha_{2} $. Similarly, if $ \alpha(0)=2 $ then $ \alpha=(0,2,3)(1,5,4)(6,9,8)(7,11,10)=\alpha_{1}^2 $. If $ \alpha(0)=5 $ then $ \alpha=(0,5,2,4,3,1)(6,11,9,\allowbreak 10,8,7)=\alpha_{1}\circ\alpha_{2} $. If $ \alpha(0)=6 $ then $ \alpha=(0,6,2,9,3,8)(1,7,5,11,4,10)=\alpha_{1}\circ\alpha_{3} $. If $ \alpha(0)=7 $ then $ \alpha=(0,7)(1,8)(2,11)(3,\allowbreak 10)(4,9)(5,6)=\alpha_{2}\circ\alpha_{3} $. If $ \alpha(0)=8 $ then $ \alpha=(0,8,3,9,2,6)(1,10,4,11,\allowbreak 5,7)=\alpha_{3}\circ\alpha_{1}^2 $. If $ \alpha(0)=10 $ then $ \alpha=(0,10,2,7,3,11)(1,9,5,8,4,\allowbreak6)\allowbreak=\alpha_{3}\circ\alpha_{1}\circ\alpha_{2} $. If $ \alpha(0)=11 $ then $ \alpha=(0,11,3,7,2,10)\allowbreak (1,6,4,8,5,9)=\alpha_{3}\circ\alpha_{1}^2\circ\alpha_{2} $. These are the all automorphism of $ KO_{1[(3^4,4^2)]} $ and can express as the composition of $ \alpha_{1},\alpha_{2},\alpha_{3} $. Therefore $ Aut(KO_{1[(3^4,4^2)]})=\langle\alpha_{1},\alpha_{2},\alpha_{3}\rangle\cong \mathbb{Z}_{2}\times\mathbb{Z}_{2}\times\mathbb{Z}_{3} $.\\
	Now $ \beta(0)=i $ for $ 0\leq i\leq 11 $.  $ \beta=\beta_{1},\beta_{2},\beta_{3} $ for $ \beta(0)=2,4,6 $, respectively. If $ \beta(0)=0 $ then from induced map on $ G_{i}(KO_{2[(3^4,4^2)]}) $ and $ lk(0) $ we see that $ \beta(i)=i $ for $ 1\leq i\leq 11 $ which implies $ \beta=identity=\beta_{1}^2 $. Similarly if $ \beta(0)=1 $ then $ \beta=(0,1,4,2,3,5)(6,7,8,10,9,11)=\beta_{1}\circ\beta_{2}^2 $. If $ \beta(0)=3 $ then $ \beta =(0,3,4)(1,5,2)(6,9,8)\allowbreak (7,11,10)=\beta_{2}^2 $. If $ \beta(0)=5 $ then $ \beta=(0,5,3,2,4,1)(6,11,9,10,8,7)=\beta_{1}\circ\beta_{2} $. If $ \beta(0)=7 $ then $ \beta=(0,7)(1,6)(2,9)(3,10)(4,11)\allowbreak (5,8)=\beta_{2}\circ\beta_{3} $. If $ \beta(0)=8 $ then $ \beta=(0,8)(1,7)(2,11)(3,9)(4,6)(5,10)=\beta_{2}^2\circ\beta_{3}\circ\beta_{1} $. If $ \beta(0)=9 $ then $ \beta=(0,9)(1,10)(2,7)(3,6)(4,8)(5,11)=\beta_{2}^2\circ \beta_{3} $. If $ \beta(0)=10 $ then $ \beta=(0,10)(1,8)(2,6)(3,\allowbreak 11)(4,7)(5,9)=\beta_{1}\circ\beta_{3} $. If $ \beta(0)=11 $ then $ \beta=(0,11)(1,9)(2,8)(3,7)(4,\allowbreak 10)(5,6)=\beta_{3}\circ\beta_{1}\circ\beta_{2}^2 $. These are all automorphism maps of $ KO_{2[(3^4,4^2)]} $ and can express as the composition of $ \beta_{1},\beta_{2} $ and $ \beta_{3} $. Therefore $ Aut(KO_{2[(3^4,4^2)]})=\langle\beta_{1},\beta_{2},\beta_{3}\rangle \cong D_{6} $ (of order $ 12 $).\\
	Now for induced map $ \gamma $ on $ G_{i}(KNO_{[(3^4,4^2)]}) $, we see that $ \gamma(0)=0,1,2,4,7,8,9,11 $. $ \gamma=\gamma_{1},\gamma_{2} $ for $ \gamma(0)=4,7 $ respectively. If $ \gamma(0)=0 $ then from $ KN_{1,1} $ and $ lk(0) $, we see that $ \gamma(i)=i $ for $ 1\leq i\leq 11 $ which implies $ \gamma=identity=\gamma_{1}^2 $. If $ \gamma(0)=1,2,8,11 $ then $ lk(0) $ map to $ lk(1), lk(2),lk(8),lk(11) $ respectively, under $ \gamma $ and then $ \gamma(1)=10,6,3,5 $ respectively. In each case, we see that $ G_{4}(KNO_{[(3^4,4^2)]}) $ is not automorphic. If $ \gamma(0)=9 $ then $ lk(0) $ map to $ lk(9) $ and then we get $ \gamma=(0,9)(1,11)(2,8)(3,6)(4,7)(5,10)=\gamma_{1}\circ\gamma_{2} $. Therefore we can see that every automorphism map of $ KNO_{[(3^4,4^2)]} $ can express as the composition of $ \gamma_{1} $ and $ \gamma_{2} $. Hence $ Aut(KNO_{[(3^4,4^2)]})=\langle\gamma_{1},\gamma_{2}\rangle\cong \mathbb{Z}_{2}\times\mathbb{Z}_{2} $.
\end{proof}
\begin{proof}[\textbf{Proof of Lemma \ref{lem4}}]
	We have $ G_7(KO_{1[(3^3,4,3,4)]})=C(0,9,10)\cup C(1,3,6) $. Let $ \phi $ be an automorphism map of $ KO_{1[(3^3,4,3,4)]} $, therefore from $ G_7(KO_{1[(3^3,4,3,4)]}) $ we get $ \phi(0),\phi(1)\in\{0,1,3,6,9,10\} $ and $ \phi(4)\in\{2,4,5,7,8,11\} $. Now if $ \phi(0)=0 $, then from $ lk(0) $, we get $ \phi(4)=4 $ which implies $ \phi=Id_{KO_{1[(3^3,4,3,4)]}} $. If $ \phi(0)=1 $ then from $ lk(0) $ and $ lk(4) $, we get $ \phi(1)=0,\phi(3)=9, \phi(4)=4,\phi(6)=11 $. Now $ \phi(C(1,3,6))=C(0,9,11) $, but $ C(1,3,5) $ is a cycle and $ C(0,9,11) $ is not a cycle in $ G_7(KO_{1[(3^3,4,3,4)]}) $, is a contradiction. If $ \phi(0)=3 $ then from $ lk(0), lk(2),lk(3),lk(4) $, we get $ \phi=(0,3)(1,9)(2,4)(5,8)(6,10)(7,11)=\alpha $. If $ \phi(0)=6 $ then from $ lk(0) $ and $ lk(6) $ we get $ \phi(1)\in\{2,5\} $, is a contradiction. If $ \phi(0)=9 $ then from $ lk(0),lk(9) $, we get $ \phi(4)=2,\phi(1)=3,\phi(3)=1,\phi(5)=6 $ which implies $ \phi(C(1,3,5))=C(1,3,6) $. But $ C(1,3,5) $ is a cycle and $ C(1,3,6) $ is not a cycle in $ G_7(KO_{1[(3^3,4,3,4)]}) $, is a contradiction. If $ \phi(0)=10 $ then from $ lk(0),lk(10) $, we get $ \phi(1)\in\{4,10\} $, is a contradiction. Therefore $ Aut(KO_{1[(3^3,4,3,4)]})=\langle \alpha\rangle=\mathbb{Z}_2 $.
		
\begin{sloppypar}

	$ G_7(KO_{2[(3^3,4,3,4)]})=C(0,9,10)\cup C(2,4,7) $. Let $ \phi $ be an automorphism map of $ KO_{2[(3^3,4,3,4)]} $, therefore from $ G_7(KO_{2[(3^3,4,3,4)]}) $ we get $ \phi(0),\phi(4)\in\{0,2,4,7,9,10\} $ and $ \phi(1)\in\{1,3,5,6,\allowbreak8,11\} $. If $ \phi(0)=0 $, then from $ lk(0),lk(1) $, we get $ \phi=Id_{KO_{2[(3^3,4,3,4)}} $. If $ \phi(0)=2 $ then from $ lk(0),lk(2) $, we get $ \phi=(0,2)(1,3)(4,9)(5,11)(6,8)(7,10)=\alpha_{1} $. If $ \phi(0)=4 $ then from $ lk(0),lk(4) $, we get $ \phi=(0,4)(2,9)(5,6)(7,10)(8,11)=\alpha_{2} $. If $ \phi(0)=7 $ then from $ lk(0),lk(7) $, we get $ \phi(1)\in\{0,9\} $, is a contradiction. If $ \phi(0)=9 $, then from $ lk(0),lk(1),lk(3),lk(9) $, we get $ \phi=(0,9)(1,3)(2,4)(5,8)(6,11)=\alpha_{1}\circ\alpha_{2} $. If $ \phi(0)=10 $ then from $ lk(0),lk(10) $, we get $ \phi(1)\in\{2,4\} $, is a contradiction. Therefore $ Aut(KO_{2[(3^3,4,3,4)]})=\langle \alpha_{1},\alpha_{2}\rangle =\mathbb{Z}_2\times\mathbb{Z}_2$.
\end{sloppypar}
	
\begin{sloppypar}
	$ G_7(KNO_{1[(3^3,4,3,4)]})=C(0,9,11)\cup C(1,3,6)\cup C(2,4,7) $ and $ G_6(KNO_{1[(3^3,4,3,4)]})=\{[0,10],\allowbreak[3,10],[7,10],[1,5],[2,5],[5,11],[4,8],[6,8],[8,9]\} $. Let $ \phi $ be an automorphism map of $ KO_{2[(3^3,4,3,4)]} $, therefore from $ G_6(KNO_{1[(3^3,4,3,4)]}) $, we get $ \phi(10)\in\{5,8,10\} $. If $ \phi(10)=5 $ then $ \phi(2)\in\{4,7\} $. $ \phi(2)=4 $ implies $ \phi=(0,11,9)(1,3)(2,4,7)(5,8,10) $, but $ \phi(KNO_{1[(3^3,4,3,4)]})\neq KNO_{1[(3^3,4,3,4)]} $, is a contradiction. $ \phi(2)=7 $ implies $ \phi=(0,1)(2,7)(3,11,6,9)(5,10) $, but $\phi(KNO_{1[(3^3,4,3,4)]})\neq KNO_{1[(3^3,4,3,4)]} $, is a contradiction. If $ \phi(10)=8 $ then $ \phi(2)\in\{2,7\} $. $ \phi(2)=2 $ implies $ \phi=(0,6,11,1)(3,9)(4,7)(8,10) $, but $ \phi(KNO_{1[(3^3,4,3,4)]})\neq KNO_{1[(3^3,4,3,4)]} $, is a contradiction. $ \phi(2)=7 $ implies $ \phi=(0,9)(1,3,6)(4,2,7)(5,10,8) $, but $ \phi(KNO_{1[(3^3,4,3,4)]})\neq KNO_{1[(3^3,4,3,4)]} $, is a contradiction. If $ \phi(0)=10 $ then $ \phi(2)\in\{2,4\} $. If $ \phi(2)=2 $, then from $ lk(2),lk(10) $, we get $ \phi=Id_{KNO_{1[(3^3,4,3,4)]}} $. If $ \phi(2)=4 $, then from $ lk(2),lk(10) $, we get $ \phi=(0,3)(1,9)(2,4)(5,8)(6,11)=\beta $. Therefore $ Aut(KNO_{1[(3^3,4,3,4)]})=\langle\beta\rangle=\mathbb{Z} $.
\end{sloppypar}
	
\begin{sloppypar}
	$ G_5(KNO_{2[(3^3,4,3,4)]})=C(0,2,1,9,4,3)\cup C(5,6,8,7,11,10) $ and let $ \phi $ be an automorphism map of $ KNO_{2[(3^3,4,3,4)]} $. Therefore $ \phi (0)\in V $. If $ \phi(0)=0 $ then from $ G_5(KNO_{2[(3^3,4,3,4)]}) $, we get $ \phi(1)\in\{1,4\} $. If $ \phi(1)=1 $, then from $ lk(1) $ and $ G_5(KNO_{2[(3^3,4,3,4)]}) $, we get $ \phi=Id_{KNO_{2[(3^3,4,3,4)]}}$ and if $ \phi(1)=4 $ then from $ lk(1),lk(4) $ and $ G_5(KNO_{2[(3^3,4,3,4)]}) $ we get $ \phi = (1,4)(2,3)(5,8)(6,7)(10,11)=\phi_1 $(say). But $ \phi_1([5,6,8])=[8,7,5] $, which is not a face of $ KNO_{2[(3^3,4,3,4)]} $. Therefore $ \phi_1 $ is not an automorphism map. If $ \phi(0)=1 $, then from $ G_5(KNO_{2[(3^3,4,3,4)]}) $, we get $ \phi(1)\in\{0,4\} $. If $ \phi(1)=0 $ then from $ lk(0),lk(1),lk(4) $ and $ KNO_{2[(3^3,4,3,4)]} $, we get $ \phi = (0,1)(3,9)(5,10)(6,11)(7,8)=\beta_{1} $. If $ \phi(1)=4 $ then from $ lk(0),lk(1),lk(4) $ and $ G_5(KNO_{2[(3^3,4,3,4)]}) $, we get $ \phi=(0,1,4)(2,9,3)(5,7,11)(6,8,10)=\phi_2 $(say). Then $ \phi_2([5,6,8])=[7,8,10] $, which is not a face, make a contradiction. Therefore $ \phi_2 $ is not an automorphism map. If $ \phi(0)=2 $ then from $ G_5(KNO_{2[(3^3,4,3,4)]}) $ we get $ \phi(1)\in\{3,9\} $. If $ \phi(1)=3 $ then from $ lk(0),lk(1),lk(2),lk(3) $ and $ G_5(KNO_{2[(3^3,4,3,4)]}) $, we get $ \phi=(0,2)(1,3)(4,9)(5,11,10)(6,7,8)=\phi_3 $(say). But $ \phi_3([5,6,8])=[11,7,6] $, which is not a face, make a contradiction. Therefore $ \phi_3 $ is not an automorphism map. Similarly if $ \phi(1)=9 $ then we get $ \phi=(0,2,1,9,4,3)(5,6,8,11) $, which is also not an automorphism map. If $ \phi(0)=3 $ then from $ G_5(KNO_{2[(3^3,4,3,4)]}) $, we get $ \phi(1)\in\{2,9\} $. If $ \phi(1)=2 $ then from $ lk(0),lk(1),lk(2),lk(3) $ and $ G_5(KNO_{2[(3^3,4,3,4)]}) $ we get $ \phi=(0,3,4,9,1,2)(5,6,10,11,7)=\phi_4 $(say). But $ \phi_4([5,6,8])=[6,10,8] $, which is not a face, make a contradiction. Therefore $ \phi_4 $ is not an automorphism map. If $ \phi(1)=9 $, then from $ lk(0),lk(1),lk(3),lk(9) $ and $ G_5(KNO_{2[(3^3,4,3,4)]}) $, we get $ \phi=(0,3)(1,9)(2,4)(5,8)(7,10)=\beta_{2} $. If $ \phi(0)=4 $ then from $ G_5(KNO_{2[(3^3,4,3,4)]}) $, we get $ \phi(1)\in\{0,1\} $. By similar to above, if $ \phi(1)=0 $ then $ \phi=(0,4,1)(2,3,9)(5,11,7)(6,10,8)=\phi_5 $(say) and if $ \phi(1)=1 $ then $ \phi=(0,4)(2,9)(5,6)(7,10)\allowbreak(8,11)=\phi_6 $(say) and both $ \phi_5,\phi_6 $ are not an automorphism maps. If $ \phi(0)=5 $, then from $ G_5(KNO_{2[(3^3,4,3,4)]}) $ we get $ \phi(1)\in\{8,11\} $ and from $ lk(0),lk(5) $ we get $ \phi(1)\in\{6,9\} $, which make a contradiction. If $ \phi(0)=6 $ then from $ G_5(KNO_{2[(3^3,4,3,4)]}) $ we get $ \phi(1)\in\{7,10\} $ and from $ lk(0),lk(6) $, we get $ \phi(1)\in\{0,3\} $, make a contradiction. If $ \phi(0)=7 $, then from $ G_5(KNO_{2[(3^3,4,3,4)]}) $ we get $ \phi(1)\in\{6,10\} $ and from $ lk(0),lk(7) $, we get $ \phi(1)\in\{0,11\} $, is a contradiction. If $ \phi(0)=8 $ then from $ G_5(KNO_{2[(3^3,4,3,4)]}) $, we get $ \phi(1)\in\{5,11\} $ and from $ lk(0),lk(8) $, we get $ \phi(1)\in\{1,6\} $, is a contradiction. If $ \phi(0)=10 $, then from $ G_5(KNO_{2[(3^3,4,3,4)]}) $, we get $ \phi(1)\in\{6,7\} $ and from $ lk(0),lk(10) $, we get $ \phi(1)\in\{3,11\} $, is a contradiction. If $ \phi(0)=11 $, the from $ G_5(KNO_{2[(3^3,4,3,4)]}) $, we get $ \phi(1)\in\{5,8\} $ and from $ lk(0),lk(11) $, we get $ \phi(1)\in \{1,9\} $, is a contradiction. If $ \phi(0)=9 $ then from $ G_5(KNO_{2[(3^3,4,3,4)]}) $, we get $ \phi(1)\in\{2,3\} $. If $ \phi(1)=2 $, then $ \phi=(0,9)(1,2)(3,4)(5,10,6)(7,11,8) $, which is not an automorphism map and if $ \phi(1)=3 $ then $ \phi=(0,9)(1,3)(2,4)(5,7)(6,11)(8,10)=\beta_{1}\not\beta_{2} $. Therefore $ Aut(KNO_{2[(3^3,4,3,4)]})=\langle\beta_1,\beta_{2}\rangle=\mathbb{Z}_2\times\mathbb{Z}_2 $.
\end{sloppypar}
	
\begin{sloppypar}
	$ G_7(KNO_{3[(3^3,4,3,4)]})=C(0,9,11)\cup C(1,3,5)\cup C(2,6,8)\cup C(4,7,10) $. Let $ \phi $ be an automorphism map of $ KNO_{3[(3^3,4,3,4)]} $. Then $ \phi(0)\in V $. If $ \phi(0)=0 $ then from $ lk(0) $, we get $ \phi(1)\in\{1,4\} $. If $ \phi(1)=1 $, then $ \phi=Id_{KNO_{3[(3^3,4,3,4)]}} $. If $ \phi(1)=4 $ then from $ lk(0) $, we get $ \phi(2)=3, \phi(3)=2 $, which implies $ \phi([1,3])=[4,2] $. But $ [1,3] $ is an edge in $ G_7(KNO_{3[(3^3,4,3,4)]}) $, $ [4,2] $ is not an edge in $ G_7(KNO_{3[(3^3,4,3,4)]}) $, is a contradiction. If $ \phi(0)=1 $ then from $ lk(0),lk(1) $, we get $ \phi(1)\in\{8,10\} $. If $ \phi(1)=8 $ then $ \phi(2)=4,\phi(3)=11 $, which implies $ \phi([1,3])=[8,11] $. $ [1,3] $ is an edge in $ G_7(KNO_{3[(3^3,4,3,4)]}) $, but $ [8,11] $ is not an edge in $ G_7(KNO_{3[(3^3,4,3,4)]}) $, is a contradiction. If $ \phi(1)=10 $, then $ \phi=(0,1,10,6)(2,11,5,7)(3,4,8,9)=\alpha_{1} $. If $ \phi(0)=2 $, then from $ lk(0),lk(2) $ and $ G_7(KNO_{3[(3^3,4,3,4)]}) $, we have $ \phi(1)\in\{1,11\} $. $ \phi(1)=1 $ implies $ \phi(9)=6,\phi(11)=4 $, from $ lk(0),lk(1),lk(2) $. Therefore $ \phi(C(0,9,11))=C(2,6,4) $. $ C(0,9,11) $ is a cycle in $ G_7(KNO_{3[(3^3,4,3,4)]}) $, but $ C(2,6,4) $ is not a cycle in $ G_7(KNO_{3[(3^3,4,3,4)]}) $, is a contradiction. $ \phi(1)=11 $ implies $ \phi(2)=3,\phi(6)=9 $, therefore $ \phi([2,6])=[3,9] $. $ [2,6] $ is an edge of $ G_7(KNO_{3[(3^3,4,3,4)]}) $, but $ [3,9] $ is not an edge of $ G_7(KNO_{3[(3^3,4,3,4)]}) $, is a contradiction. When $ \phi(0)=3 $ then from $ lk(0),lk(3) $, we get $ \phi(1)\in\{4,11\} $. $ \phi(1)=4 $ implies $ \phi(2)=0,\phi(6)=6 $, therefore $ \phi([2,6])=[0,6] $. $ [2,6] $ is an edge but $ [0,6] $ is not an edge of $ G_7(KNO_{3[(3^3,4,3,4)]}) $, is a contradiction. $ \phi(1)=11 $ implies $ \phi(2)=2,\phi(6)=9 $, therefore $ \phi([2,6])=[2,9] $. $ [2,6] $ is an edge but $ [2,9] $ is not an edge of $ G_7(KNO_{3[(3^3,4,3,4)]}) $, is a contradiction. When $ \phi(0)=4 $ then $ lk(0),lk(4) $, we get $ \phi(1)\in\{5,8\} $. $ \phi(1)=5 $ implies $ \phi(2)=11,\phi(6)=3 $, therefore $ \phi([2,6])=[11,3] $. $ [2,6] $ is an edge but $ [11,3] $ is not an edge of $ G_7(KNO_{3[(3^3,4,3,4)]}) $, is a contradiction. $ \phi(1)=8 $ implies $ \phi(3)=11 $, therefore $ \phi([1,3])=[8,11] $. $ [1,3] $ is an edge but $ [8,11] $ is not an edge of $ G_7(KNO_{3[(3^3,4,3,4)]}) $, is a contradiction. When $ \phi(0)=5 $ then  from $ lk(0),lk(5) $, we get $ \phi(1)\in\{6,7\} $. $ \phi(1)=6 $ implies $ \phi(2)=10,\phi(6)=11 $, therefore $ \phi([2,6])=[10,11] $. $ [2,6] $ is an edge but $ [10,11] $ is not an edge of $ G_7(KNO_{3[(3^3,4,3,4)]}) $, is a contradiction. $ \phi(1)=7 $ implies $ \phi(2)=8,\phi(6)=4 $, therefore $ \phi([2,6])=[8,4] $. $ [2,6] $ is an edge but $ [8,4] $ is not an edge of $ G_7(KNO_{3[(3^3,4,3,4)]}) $, is a contradiction. When $ \phi(0)=6 $ then from $ lk(0),lk(6) $, we get $ \phi(1)\in\{0,3\} $. $ \phi(1)=0 $ implies $ \phi=(0,6,10,1)(2,7,5,11)(3,9,8,4)=\alpha_{1}^3 $. $ \phi(1)=3 $ implies $ \phi(3)=7 $, therefore $ \phi([1,3])=[3,7] $. $ [1,3] $ is an edge but $ [3,7] $ is not an edge of $ G_7(KNO_{3[(3^3,4,3,4)]}) $, is a contradiction. When $ \phi(0)=7 $ then from $ lk(0),lk(7) $, we get $ \phi(1)\in\{0,2\} $. $ \phi(1)=0 $ implies $ \phi(2)=6,\phi(6)=5 $, therefore $ \phi([2,6])=[6,5] $. $ [2,6] $ is an edge but $ [6,5] $ is not an edge of $ G_7(KNO_{3[(3^3,4,3,4)]}) $, is a contradiction. $ \phi(1)=2 $ implies $ \phi(3)=6,\phi(5)=1 $, therefore $ \phi([3,5])=[6,1] $. $ [3,5] $ is an edge but $ [6,1] $ is not an edge of $ G_7(KNO_{3[(3^3,4,3,4)]}) $, is a contradiction. When $ \phi(0)=8 $ then from $ lk(0),lk(8) $, we get $ \phi(1)\in\{7,9\} $. $ \phi(1)=7 $ implies $ \phi(2)=5,\phi(6)=4 $, therefore $ \phi([2,6])=[5,4] $. $ [2,6] $ is an edge but $ [5,4] $ is not an edge, is a contradiction. $ \phi(1)=9 $ implies $ \phi(3)=5 $, therefore $ \phi([1,3])=[9,5] $. $ [1,3] $ is an edge but $ [9,4] $ is not an edge of $ G_7(KNO_{3[(3^3,4,3,4)]}) $, is a contradiction. When $ \phi(0)=9 $ then from $ lk(0),lk(9) $, we get $ \phi(1)\in\{2,3\} $. $ \phi(1)=2 $ implies $ \phi(3)=9 $, therefore $ \phi([1,3])=[2,9] $. $ [1,3] $ is an edge but $ [2,9] $ is not an edge of $ G_7(KNO_{3[(3^3,4,3,4)]}) $, is a contradiction. $ \phi(1)=3 $ implies $ \phi(3)=7 $, therefore $ \phi([1,3])=[3,7] $. $ [1,3] $ is an edge but $ [3,7] $ is not an edge of $ G_7(KNO_{3[(3^3,4,3,4)]}) $, is a contradiction. When $ \phi(0)=10 $ then from $ lk(0),lk(10) $, we get $ \phi(1)\in\{6,9\} $. $ \phi(1)=6 $ implies $ \phi=(0,10)(1,6)(2,5)(3,8)(4,9)(7,11)=\alpha_{1}^2 $. $ \phi(1)=9 $ implies $ \phi(3)=5 $, therefore $ \phi([1,3])=[9,5] $. $ [1,3] $ is an edge but $ [9,5] $ is not an edge of $ G_7(KNO_{3[(3^3,4,3,4)]}) $, is a contradiction. When $ \phi(0)=11 $ then from $ lk(0),lk(11) $, we get $ \phi(1)\in\{5,10\} $. $ \phi(1)=5 $ implies $ \phi(2)=4,\phi(6)=3 $, therefore $ \phi([2,6])=[4,3] $. $ [2,6] $ is an edge but $ [4,3] $ is not an edge of $ G_7(KNO_{3[(3^3,4,3,4)]}) $, is a contradiction. $ \phi(1)=10 $ implies $ \phi(2)=1,\phi(6)=7 $, therefore $ \phi([2,6])=[1,7] $. $ [2,6] $ is an edge but $ [1,7] $ is not an edge of $ G_7(KNO_{3[(3^3,4,3,4)]}) $, is a contradiction. Therefore $ Aut(KNO_{3[(3^3,4,3,4)]})=\langle\alpha_{1}\rangle=\mathbb{Z}_4 $.
\end{sloppypar}
\end{proof}
\begin{proof}[\textbf{Proof of theorem \ref{thm4}}]
		Let $K$ be a map of type $(3,4^4)$ on the surface of $\chi =-2$. Let $V(K)=\{0,1,2,3,4,5,6,\allowbreak 7,8,9,10,11\}$. We express $lk(v)$ by the notation $lk(v)=C_{8}([v_1,v_2,v_3],[v_3,v_4,v_5],[v_5,v_6,v_7],\allowbreak [v_7,v_8,v_9])$ where $[v,v_{1},v_{9}]$ form 3-gon face and $[v,v_{1},\allowbreak v_{2},v_{3}],[v,v_{3},v_{4},v_{5}], [v,v_5,v_6,v_7],[v,v_7,\allowbreak v_8,v_9]$ form 4-gon faces. With out loss of generality let $ lk(0)=C_{9}([4,5,6],[6,7,8],[8,9,1],[1,\allowbreak 2,3]) $. Then either $ lk(8)=C_{9}([c,d,7],[7,6,0],[0,\allowbreak1,9],[9,a,b]) $ or  $ lk(8)=C_{9}([a,b,c],[c,d,9],\allowbreak [9,1, 0],[0,6,7]) $ or $ lk(8)=C_{9}([a,b,c],[c,d,7],[7,\allowbreak6,0],[0,1,9]) $.
	
	Now for \underline{$\boldsymbol{lk(8)=C_{9}([c,d,7],[7,6,0],[0,1,9],[9,a,b])}  $}, all possible value of $ (a,b,c,d)$ are $ (2,5,10,\allowbreak 3)\approx \{(2,5,11,3), (3,2,10,11), (3,2,11,10)\} $, $ (2,5,10,4)\approx \{(2,5,11,4)$, $(4,2,\allowbreak10,5)$, $ (4,2,11, 5)$, $  (3,10,2,11)$, $(3,11,2,10)$, $(10,11,5,3)$, $ (11,10,5,3)\} $, $ (2,5,10,\allowbreak 11)\approx (2,5,11,10) $,  $ (2,10,11,3)\approx (2,11,10,3) $, $ (2,10,11,4)\approx (2,11,10,4) $, $ (2,10,11,\allowbreak 5)\approx (2,11,\allowbreak10,5) $, $ (2,10,5,\allowbreak 3)\approx\{ (2,11,5,3),(3,11,2,5),(3,10,2,5), (10,5,11,4), (11,5,10,4) \}$, $ (2,10,\allowbreak5,4)\approx\{(2,11,5,\allowbreak 4), (4,10,2,11), (4,11,2,10)\} $, $ (2,10,5,11)\allowbreak \approx (2,11,5,10) $, $ (3,2,5,4) $, $ (3,2,\allowbreak5,10)\approx (3,2,5,\allowbreak 11) $, $ (3,2,10,4)\approx (3,2,11,4) $, $ (3,2,10,\allowbreak 5)\approx \{(3,2,11,5), (10,5,11,3)\} $, $ (3,5,2,\allowbreak 4) $, $ (3,5,2,10)\approx (3,5,2,11) $, $ (3,5,10,2)\approx (3,5,11,2) $, $ (3,5,10,4)\approx (3,5,11,4) $, $ (3,5,10,11)\allowbreak\approx (3,5,11,10) $, $ (3,10,2,4)\approx (3,11,2,4) $, $ (3,10,5,2)\approx (3,11,5,2) $, $ (3,10,5,4)\allowbreak\approx (3,11,5,4) $, $ (3,10,\allowbreak 5,11)\approx (3,11,5,10) $, $ (3,10,11,2)\approx (3,11,10,2) $, $ (3,10,11,4)\approx \{(3,11,\allowbreak10,4), (10,2,\allowbreak 5,11), (11,2,5,10)\} $, $ (3,10,11,5)\approx (3,11,10,5) $, $ (4,2,5,3) $, $ (4,2,5,10)\allowbreak\approx (4,2,\allowbreak5,11) $, $ (4,2,\allowbreak 10,3)\approx (4,2,11,3) $, $ (4,2,10,11)\approx (4,2,11,10) $, $ (4,5,2,3) $, $ (4,5,2,10)\allowbreak\approx (4,5,\allowbreak 2,11) $, $ (4,5,\allowbreak 10,2)\approx (4,5,11,2) $, $ (4,5,10,3)\approx (4,5,11,3) $, $ (4,5,10,11)\approx (4,5,11,\allowbreak 10) $, $ (4,10,2,3)\approx (4,11,2,3) $, $ (4,10,2,5)\approx (4,11,2,5) $, $ (4,10,5,2)\approx(4,11,5,2) $, $ (4,10,5,\allowbreak3)\approx \{(4,11,5,3),\allowbreak (5,2,10,\allowbreak11),\allowbreak (5,2,11,10)\} $, $ (4,10,5,11)\approx (4,11,5,10) $, $ (4,10,11,2)\approx (4,11,10,2) $, $ (4,10,11,3)\allowbreak\approx \{(4,\allowbreak11,10,3), (10,5,\allowbreak 2,11), (11,5,2,10)\} $, $ (4,10,11,5)\approx (4,11,\allowbreak10,5) $, $ (5,2,10,3)\approx (5,2,11,\allowbreak3) $, $ (5,2,\allowbreak10,4)\approx (5,2,11,4) $, $ (5,10,2,3)\approx (5,11,2,3) $, $ (5,10,\allowbreak 2,4)\approx (5,11,2,4) $, $ (5,10,2,\allowbreak11)\approx (5,11,2,10) $, $ (5,10,11,2)\approx (5,11,10,2) $, $ (5,10,\allowbreak 11,3)\approx (5,11,10,3) $, $ (5,10,11,4)\approx (5,11,\allowbreak10,5) $, $ (10,2,5,\allowbreak 3)\approx (11,2,5,3) $, $ (10,2,5,\allowbreak 4)\approx (11,2,5,4) $, $ (10,2,11,3)\approx (11,2,10,\allowbreak3) $, $ (10,2,11,4)\approx (11,2,10,4) $, $ (10,2,11,\allowbreak 5)\approx (11,2,10,5) $, $ (10,5,\allowbreak2,3)\approx (11,5,2,3) $, $ (10,5,2,\allowbreak4)\allowbreak\approx (11,5,2,4) $, $ (10,5,11,2)\approx (11,5,10,2) $, $ (10,11,2,3)\approx (11,10,2,3) $, $ (10,11,\allowbreak2,4)\approx (11,\allowbreak10,2,4) $, $ (10,11,\allowbreak 2,5)\approx (11,10,2,5) $, $ (10,11,5,2)\approx (11,10,\allowbreak5,2) $, $ (10,11,5,4)\approx(11,10,5,4) $.
	
	\hspace{-0.6cm}\underline{\textbf{If} $ \boldsymbol{(a,b,c,d)=(2,5,10,3)} $} then either $ lk(3)=C_{9}([4,a_{1},10],[10,8,7],[7,a_{2},2],[2,\allowbreak 1, 0]) $ or $ lk(3)=C_{9}([4,a_{1},7],[7,8,10],[10.a_{2},2],[2,1,0]) $.
	
	If $ lk(3)=C_{9}([4,a_{1},10],[10,\allowbreak 8,7],[7,a_{2},2],[2,1,0]) $ then $ a_{1},a_{2}\in \{9,11\} $. For $ \boldsymbol{a_{2}=9} $ then $ lk(2)=C_{9}([11,a_{3},5],\allowbreak [5,8,9],[9,7,3],[3,0,1]) $ and $ lk(5)=C_{9}([10,a_{4},4],\allowbreak [4,0,a_{3}],[a_{3},\allowbreak11,2]) $. Therefore from $ lk(3) $ and $ lk(5) $ we see that $ 4 $ and $ 10 $ occur in two 4-gon and does not form an edge, which is a contradiction. For $ \boldsymbol{a_{2}=11} $ then $ lk(2)=C_{9}([9,8,5],[5,a_{3},1],[1,0,\allowbreak3],[3,7,\allowbreak 11]) $ and $ lk(7)=C_{9}([1,a_{4},11],[11,2,3],[3,10,8],[8,\allowbreak 0,6]) $. Then $ lk(1)=C_{9}([7,11,9],\allowbreak [9,8,0],\allowbreak [0,3,2],[2,5,6]) $ and $ a_{3}=6,a_{4}=9 $ which is a contradiction as $ [2,9,11] $ is a 3-gon.
	
	If $ lk(3)=C_{9}([4,a_{1},7],[7,8,10],[10.a_{2},2],[2,1,0]) $ then $ a_{2}\in \{6,11\} $. If $ \boldsymbol{a_{2}=6} $ then $ lk(5)=C_{9}([10,1,4],[4,0,6],[6,11,2],[2,9,8]) $ which implies $ [2,6] $ form edge and non-edge in two 4-gon, which is not possible. If $ \boldsymbol{a_{2}=11} $ then $ a_{1}=9 $ and either $ lk(2)=C_{9}([11,\allowbreak10,3],[3,\allowbreak0, 1],[1,6,5],[5,8,9]) $ or $ lk(2)=C_{9}([1,0,3],[3,10,11],[11,6,5],[5,8,9]) $. If $ lk(2)=C_{9}([11,10,3],\allowbreak [3,0,1],[1,6,5],[5,8,9]) $ then $ lk(1)=C_{9}([7,a_{3},\allowbreak 9],[9,8,0],[0,3,2],[2,5,6]) $ and then $ [7,9] $ form edge and non-edge in two 4-gon which is not possible. If $ lk(2)=C_{9}([1,0,3],\allowbreak[3,10,11],[11,6,\allowbreak 5],[5,8,9]) $ then $ 7 $ occur two times in $ lk(7) $ which is not possible.
	
	\hspace{-0.6cm}\underline{\textbf{If} $ \boldsymbol{(a,b,c,d)=(4,10,5,2)} $} then $ lk(5)=C_{9}([10,a_{2},6],[6,0,4],[4,a_{1},2],[2,7,8]),\allowbreak  lk(4)=C_{9}([3,a_{3},10],[10,8,a_{1}],[a_{1},2,5],[5,6,0]) $ which implies $ a_{1}=9 $. Now $ lk(10)=C_{9}([8,9,4],[4,\allowbreak3, a_{3}],[a_{3},a_{4},a_{2}],[a_{2},6,5]) $ where $ a_{2}\in \{1,11\} $ and $ a_{3}\in \{7,11\} $. If $ a_{3}=7 $ then $ a_{4}\in \{1,11\} $ which is not possible as then $ lk(7) $ will not complete. Therefore $ a_{3}=11 $ which implies $ a_{2}=1 $ and $ a_{4}\in\{2,7\} $. But if $ a_{4}=7 $ then $ lk(1) $ will not possible, therefore $ a_{4}=2 $ and then $ lk(1)=C_{9}([6,5,10],[10,11,2],[2,3,0],[0,8,9]) $ and then $ lk(2) $ will not possible.
	
	\hspace{-0.6cm}\underline{\textbf{If} $ \boldsymbol{(a,b,c,d)=(11,10,5,2)\approx(10,11,5,2)} $} then $ lk(5)=C_{9}([10,a_{2},6],[6,0,4],[4,a_{1},\allowbreak2],[2,7,8]) $ or $ lk(5)\allowbreak =C_{9}([10,a_{2},4],[4,0,6],[6,a_{1},2],[2,7,8]) $.\\
	If $ lk(5)=C_{9}([10,a_{2},6],[6,0,4],[4,a_{1},2],[2,7,8]) $ then $ lk(4)\allowbreak =C_{9}([3,a_{3},a_{4}],[a_{4},a_{5},a_{1}],[a_{1},\allowbreak 2,5],[5,6,0]) $ which implies $ a_{1}\in \{1,9,11\} $. \textbf{If} $ \boldsymbol{a_{1}=1} $ then $ lk(2)=C_{9}([b_{1},b_{2},3],[3,0,1],[1,\allowbreak4,5],[5,8,7]) $ and $ lk(1)=C_{9}([b_{3},b_{4},4],[4,5,2],[2,3,0],\allowbreak [0,8,9]) $. Now two incomplete 3-gons are $ [2,7,b_{1}] $ and $ [1,9,b_{3}] $ which implies $ b_{1},b_{3}\in\{6,11\} $. $ (b_{1},b_{3})\neq (6,11),(11,6) $ as in first case, $ 1 $ or $ 11 $ will occur two times in $ lk(9) $ and in second case, $ \{4,6\} $ will occur in two 4-gon which is not possible. \textbf{If} $ \boldsymbol{a_{1}=9} $ $ a_{5}\in\{1,11\} $. For $ a_{5}=11 $ one 3-gon will be $ [1,2,9] $ which is not possible. For $ a_{5}=1 $, $ lk(9)=C_{9}([11,10,8],[8,0,1],[1,a_{4},\allowbreak 4],[4,5,2]) $ then one 3-gon is $ [1,6,7] $ which implies $ a_{4}=7 $ and then $ lk(1)=C_{9}([7,4,9],[9,8,0],\allowbreak [0,3,2],[2,11,6]) $ and $ a_{2}=3 $ which implies $ lk(6) $ is not possible. \textbf{If} $ \boldsymbol{a_{1}=11} $ then $ a_{2}\in\{1,3\} $. Now $ lk(2)=C_{9}([b_{1},b_{2},5],[5,b_{3},b_{4}],[b_{4},b_{5},3],[3,0,1]) $ which implies two incomplete 3-gons are $ [1,2,b_{1}] $ and $ [6,9,c_{1}] $. As $ a_{2}\in\{1,3\} $ then $ [6,a_{2}] $ will be an adjacent edge of two 4-gons and then $ lk(6)=C_{9}([7,8,0],[0,4,5],[5,10,a_{2}],[a_{2},a_{3},9]) $ which implies $ b_{1}=11,c_{1}=7 $ and $ a_{2}=3 $ and then $ b_{2}=4,b_{3}=8,b_{4}=7, b_{5}=10 $. Now $ lk(7)=C_{9}([9,b_{6},10],[10,3,2],[2,5,8],[8,0,6]) $ which implies $ \{9,10\} $ does not form an edge and occur in two 4-gon which is not possible.\\
	If $ lk(5)\allowbreak =C_{9}([10,a_{2},4],[4,0,6],[6,a_{1},2],[2,7,8]) $ then $ lk(4)=C_{9}([3,a_{3},a_{4}],[a_{4},a_{5},a_{2}],[a_{2},\allowbreak10,5],[5,6,0]) $ and $ lk(10)=C_{9}([8,9,11],[11,c_{1},c_{2}],[c_{2},c_{3},a_{2}],[a_{2},4,5]) $. From these we see that $ a_{2}=1 $ which implies $ a_{5}\in \{2,9\} $. We see that $ [1,10,b_{1}] $ will form a 3-gon which is not possible for any $ a_{5} $.
	
	\hspace{-0.6cm}\underline{\textbf{If} $ \boldsymbol{(a,b,c,d)=(3,5,10,2)} $} then $ a_{1}\in \{1,11\} $ and $ a_{2}\in\{2,11\} $. For $ a_{1}=1 $, after completing $ lk(1) $ we will see that $ 10 $ will occur in two 3-gon which is not possible. Therefore $ a_{1}=11 $ and then $ a_{2}=2 $. Now $ lk(3)=C_{9}([4,7,9],[9,8,5],[5,6,2],[2,1,0]) $, $ lk(4)=C_{9}([3,9,7],[7,a_{3},11],[11,10,5],[5,6,0]) $ and $ lk(7)=C_{9}([9,3,4],[4,11,2],[2,10,8],[8,0,6]) $ which implies $ a_{3}=2 $ and $ [1,2,11] $ will form a 3-gon and then $ lk(2) $ will not possible.
	
	\hspace{-0.6cm}\underline{\textbf{If} $ \boldsymbol{(a,b,c,d)=(11,5,10,2)\approx(10,5,11,2)} $} then either $ lk(5)=C_{9}([10,1,4],[4,0,6],[6,\allowbreak a_{1},11],[11,9,8]) $ or $ lk(5)=C_{9}([10,a_{1},6],[6,0,4],[4,a_{2},11],[11,9,8]) $. If $ lk(5)=C_{9}([10,1,\allowbreak4],[4,0,6],[6,a_{1},\allowbreak11],\allowbreak [11,9,8]) $ then $ [1,10] $ will be an adjacent edge of two 4-gon, which implies $ 4 $ will occur in two 4-gon which is a contradiction. If $ lk(5)=C_{9}([10,a_{1},6],[6,0,4],[4,\allowbreak a_{2},11],[11,9,8]) $ then $ a_{1}\in\{1,3\} $. If $ \boldsymbol{a_{1}=1} $ then $ [1,10] $ will be an adjacent edge two 4-gon, let the other adjacent 4-gon be $ [1,10,a_{3},a_{4}] $. Then $ a_{4}=9 $ and $ lk(1)=C_{9}([6,5,10],[10,a_{3},\allowbreak9],[9,8,0],[0,3,\allowbreak2]) $ which implies $ [7,9,11] $ will be one 3-gon and then $ a_{3}=4 $. Now $ lk(9)=C_{9}([7,a_{5},4],[4,\allowbreak10,1],[1,0,8],[8,\allowbreak 5, 11]) $ and $ lk(7)=C_{9}([11,a_{6},6],[6,0,8],[8,10,a_{5}],[a_{5},4,9]) $ which implies $ a_{5}=2 $ and then $ lk(6)=C_{9}([2,11,7],[7,8,0],[0,4,5],[5,10,1]) $. Therefore $ a_{6}=2 $, which is not possible. If $ \boldsymbol{a_{1}=3} $ then $ lk(3)=C_{9}([4,a_{3},10],[10,5,6],[6,a_{4},2],[2,1,\allowbreak0]) $, $ lk(10)=C_{9}([8,7,2],[2,a_{5},a_{3}],\allowbreak [a_{3},4,3],[3,6,5]) $, $ lk(6)=C_{9}([a_{4},2,3],[3,10,5],[5,4,0],\allowbreak[0,8,7]) $ and $ lk(4)=C_{9}([3,10,a_{3}],\allowbreak [a_{3},a_{6},a_{2}],[a_{2},11,5],[5,6,0]) $ which implies $ a_{3}=9, a_{4}=11 $ and then $ [1,2,9] $ will form a 3-gon which is not possible.
	
	\hspace{-0.6cm}\underline{\textbf{If} $ \boldsymbol{(a,b,c,d)=(3,11,10,2)\approx(3,10,11,2)} $} then $ lk(3)=C_{9}([4,a_{1},9],[9,8,11],[11,a_{2},\allowbreak2],[2,1,0]) $ or $ lk(3)\allowbreak =C_{9}([4,a_{1},11],[11,8,9],[9,a_{2},2],[2,1,0]) $. For $ lk(3)=C_{9}([4,a_{1},9],[9,8,\allowbreak11],[11,a_{2},2],[2,1,\allowbreak 0]) $, $ a_{1}\in\{7,10\} $. \textbf{If} $ \boldsymbol{a_{1}=7} $ then $ lk(7)=C_{9}([9,3,4],[4,a_{3},6],[6,0,8],[8,\allowbreak10,2]) $ and $ lk(6)=C_{9}([1,a_{4},a_{3}],[a_{3},4,7],[7,8,0],[0,4,5]) $ which implies $ a_{3}=11 $ and then from $ lk(2) $ we see that $ a_{2}\in\{9,10\} $ which make contradiction. \textbf{If} $ \boldsymbol{a_{1}=10} $ then $ lk(9)=C_{9}([1,0,8],[8,11,3],[3,\allowbreak4, 10],[4,a_{4},a_{3}]) $ and $ lk(10)=C_{9}([11,a_{5},4],[4,3,9],[9,a_{3},2],[2,7,8]) $ which implies $ a_{4}=2 $ and $ a_{3}=a_{2} $ and then $ lk(2)=C_{9}([7,8,10],[10,9,a_{2}],[a_{2},11,3],[3,0,1]) $ which implies $ 1 $ occur in two 3-gon which is not possible. For $ lk(3) =C_{9}([4,a_{1},11],[11,8,9],\allowbreak[9,a_{2},2],[2,1,\allowbreak0]) $, $ a_{1}\in\{7,10\} $. \textbf{If} $ \boldsymbol{a_{1}}=7 $ then $ 11 $ occur two 3-gon and \textbf{if} $ \boldsymbol{a_{1}=10} $ then $ 8 $ will occur two times in $ lk(11) $.
	
	\begin{sloppypar}
	\hspace{-0.6cm}\underline{\textbf{If} $ \boldsymbol{(a,b,c,d)=(4,11,10,2)\approx(4,10,11,2)} $} then we can see that $ [10,11] $ form an edge in a 4-gon, therefore let us assume that six incomplete 4-gons are $ [10,11,a_{2},a_{1}],[2,10,a_{3},a_{4}],\allowbreak[4,11,a_{5},a_{6}],\allowbreak[10,\allowbreak a_{7}, a_{8},a_{9}],[11,a_{10},a_{11},a_{12}],[a_{13},a_{14},a_{15},a_{16}] $. As $ [7,10] $ form a non-edge in a 4-gon and $ [7,11] $ does not form an edge in a 3-gon, therefore $ [7,11] $ will be an adjacent edge of two 4-gon, therefore $ a_{5}=7 $ and $ a_{12}=7 $. Similarly as $ [9,11] $ is a non-edge in a 4-gon, therefore $ [9,10] $ will form an adjacent edge of two 4-gon and then $ a_{3}=9,a_{9}=9 $. Now  we can see that $ a_{6}\in\{3,5\} $. As $ [0,1,2,3] $ and $ [2,7,8,10] $ are disjoint, therefore either $ [1,2] $ or $ [2,7] $ will form an edge in a 3-gon and as $ [4,9] $ is an adjacent edge of two 4-gon, therefore either $ a_{4}=4 $ or $ a_{8}=4 $. But if $ a_{4}=4 $ then $ lk(4) $ will not possible, therefore $ a_{8}=4 $ which implies $ a_{7}\in\{3,5\} $. Now four 4-gons at $ 7 $ are $ [0,6,7,8],[2,7,8,10],[4,11,7,a_{6}],[7,11,a_{10},a_{11}] $ which implies $ [7,a_{6}] $ will form an edge in a 3-gon, therefore $ a_{6}=5,a_{7}=3 $ and $ a_{4}=5 $. Now $ lk(4)=C_{9}([3,10,9],[9,8,11],[11,7,5],[5,6,0]), lk(7)=C_{9}([2,10,8],[8,0,6],[6,a_{10},11],[11,\allowbreak4,5]),\allowbreak lk(\allowbreak 5) =C_{9}([2,10,9],[9,b_{1},6],[6,0,4],[4,11,7]) $ and $ lk(9)=C_{9}([6,b_{2},10],[10,3,4],[4,\allowbreak11,8],[8,0,\allowbreak 1]) $ which implies $ [6,9,10,b_{2}] $ is a 4-gon which in not any one of the above six incomplete 4-gon, which make contradiction. Similarly for $ \boldsymbol{(a,b,c,d)=(5,11,10,2)\allowbreak\approx(5,10,11,2)} $, map will not exist.
\end{sloppypar}
	
\begin{sloppypar}
    \hspace{-0.6cm}\underline{\textbf{If} $ \boldsymbol{(a,b,c,d)=(4,5,2,3)} $} then $ lk(2)=C_{9}([5,a_{1},a_{2}],[a_{2},a_{3},1],[1,0,3],[3,7,8]),lk(4)=C_{9}([3,a_{4},a_{5}],[a_{5},a_{6},9],[9,8,5],[5,6,0]),lk(5)=C_{9}([2,a_{2},a_{1}],[a_{1},a_{7},6],[6,0,4],[4,9,8]) $ and $ lk(3)=C_{9}([4,a_{5},a_{4}],[a_{4},a_{8},7],[7,8,2],[2,1,0]) $. From these we see that if $ a_{1},a_{2}, a_{3}\in\{6,10,\allowbreak 11\};a_{4},a_{5},a_{6}\in\{7,10,11\};a_{7}\in\{1,10,11\} $ and $ a_{8}\in\{9,10,11\} $.
	As $ a_{1},a_{2}\neq 6 $, therefore $ a_{3}=6 $. Similarly we get $ a_{6}=7,a_{7}=1,a_{8}=9 $. With out loss go generality, let $ a_{1}=10 $ and $ a_{2}=11 $. Now $ lk(6)=C_{9}([11,2,1],[1,10,5],[5,4,0],[0,8,7]) $ which implies $ [1,9,10] $ form a 3-gon. Therefore $ lk(1)=C_{9}([10,5,6],[6,11,2],[2,3,0],[0,8,9]) $ and $ lk(7)=C_{9}([11,b_{1},9],[9,a_{4},3],[3,2,8],[8,0,6]) $. From here we see that $ a_{4}=10,a_{5}=11 $ and $ a_{1}=4 $. Therefore $ lk(9)=C_{9}([10,3,7],[7,11,4],[4,5,8],[8,0,1]),lk(10)=C_{9}([1,6,5],[5,2,11],[11,\allowbreak4,\allowbreak 3],[3,7,9]) $ and $ lk(11)=C_{9}([1,6,2],[2,5,10],[10,3,4],[4,9,7]) $. Let this be the map $ \boldsymbol{KNO_{1[(3,4^4)]}} $.
\end{sloppypar}	
	
\begin{sloppypar}
\hspace{-0.6cm}\underline{\textbf{If} $ \boldsymbol{(a,b,c,d)=(10,5,2,3)} $} then $ lk(2)=C_{9}([5,a_{1},a_{2}],[a_{2},a_{3},1],[1,0,3],[3,7,8]) $ and $ lk(3)\allowbreak =C_{9}([4,a_{4},a_{5}],[a_{5},a_{6},7],[7,8,2],[2,1,0]) $. As $ 11\notin lk(0),lk(8) $ therefore one of $ a_{1},a_{2},a_{3} $ is $ 11 $ and one of $ a_{4},a_{5},a_{6} $ is $ 11 $. Now $ a_{2}\notin\{4,6,10\} $ therefore $ a_{2}=11 $ and $ a_{1},a_{3}\in\{4,6,10\} $. From $ lk(5) $ we get $ a_{1}\neq 10 $, therefore for any $ a_{1}\in\{4,6\} $, $ a_{3}=10 $. \textbf{If} $ \boldsymbol{a_{1}=4} $ then $ lk(5)=C_{9}([2,11,4],[4,0,6],[6,a_{7},10],[10,9,8]) $ and $ lk(4)=C_{8}([3,a_{5},a_{4}],[a_{4},a_{8},11],[11,2,5],[5,6,\allowbreak 0]) $ whish implies $ a_{6}=11 $ and then $ a_{4},a_{5}\in\{9,10\} $. Now $ lk(2) $ implies $ a_{4}=9,a_{5}=10 $ and $ a_{8}=1 $ which implies $ lk(9) $ is not possible. \textbf{If} $ \boldsymbol{a_{1}=6} $ then $ lk(5)=C_{9}([2,11,6],[6,0,4],[4,a_{7},\allowbreak 10],[10,9,8]) $ and $ lk(4)=C_{9}([3,a_{5},a_{4}],[a_{4},a_{8},a_{7}],[a_{7},10,5],[5,6,0]) $ which implies $ a_{4},a_{5}\in\{9,11\} $ and $ a_{7}\in\{1,7\} $. If $ a_{4}=11 $ then $ a_{8}\in\{6,10\} $, which is not possible and if $ a_{4}=9 $ then $ a_{6}\in\{6,10\} $ and $ a_{7}=7 $ which implies $ a_{8}\in\{3,6\} $ which is also not possible.
\end{sloppypar}	
	
	\hspace{-0.6cm}\underline{\textbf{If} $ \boldsymbol{(a,b,c,d)=(5,10,2,3)} $} then $ lk(2)=C_{9}([10,a_{1},a_{2}],[a_{2},a_{3},1],[1,0,3],[3,7,8]) $ and $ lk(3)\allowbreak =C_{9}([4,a_{4},a_{5}],[a_{5},a_{6},7],[7,8,2],[2,1,0]) $ where $ a_{1},a_{2},a_{3}\in\{4,5,6,11\} $ and $ a_{4},a_{5},a_{6}\allowbreak\in\{5,9,10,11\} $. If $ a_{6}=5 $ then $ a_{5}\in\{9,10\} $ as at any vertex, three 4-gon can not be distinct and therefore $ a_{4}=11 $. If $ a_{5}=9 $ then from $ lk(5) $ we can see that $ [5,6,7] $ will form a 3-gon which is not possible and if $ a_{5}=10 $ then $ lk(10)=C_{8}([8,9,5],[5,7,3],[3,4,11],[11,a_{7},2]) $. Now $ a_{7}\neq 1 $ as then $ a_{2}=10,a_{3}=11 $ which is not possible. Therefore $ a_{7}=6 $ and then $ a_{1}=11,a_{2}=6,a_{3}=5 $. Therefore from $ lk(6) $ we can see that $ [1,5,9] $ will be a 3-gon which is not possible. Therefore $ a_{6}\neq 5 $ i.e. $ a_{6}\in\{9,10,11\} $. As $ [9,10] $ is a non-edge in a 4-gon, therefore $ a_{5}=11 $ and $ (a_{4},a_{5})\in\{(9,10),(10,9)\} $.\\
	\textbf{If} $ \boldsymbol{(a_{4},a_{5})=(9,10)} $ then $ a_{1}=11 $ which implies $ lk(10)=C_{9}([8,9,5],[5,a_{7},7],[7,3,11],[11,\allowbreak a_{2},2]) $, $ lk(7)=C_{9}([a_{7},5,10],[10,11,3],[3,2,8],[8,0,6]) $ and then $ a_{7}=1 $, therefore $ lk(1) $ is not possible.\\
	\textbf{If} $ \boldsymbol{(a_{4},a_{5})=(10,9)} $ then $ a_{1}=4 $, therefore $ lk(10)=C_{9}([8,9,5],[5,a_{7},11],[11,3,4],[4,a_{2},\allowbreak2]) $ and $ a_{7}\in\{1,6\} $. If $ a_{7}=1 $ then $ [1,9] $ will be an adjacent edge of a 3-gon and a 4-gon and we have $ [2,10] $ is an edge, therefore $ lk(1)=C_{9}([11,10,5],[5,a_{2},2],[2,3,0],[0,8,9]) $ which implies $ a_{3}=5,a_{1}=11 $ and $ [5,6,7] $ form a 3-gon which is not possible. If $ a_{7}=6 $ then $ [5,9] $ will form an adjacent edge of a 3-gon and a 4-gon, therefore $ lk(5)=C_{9}([a_{8},a_{9},4],[4,0,6],[6,\allowbreak11,10],[10,\allowbreak 8,9]) $ which implies $ a_{8}=7 $ and then $ lk(7)=C_{9}([5,4,6],[6,0,8],[8,2,3],[3,11,\allowbreak9]) $ i.e. $ [5,6] $ form an edge and a non-edge in two different 4-gon, which is a contradiction.

	\hspace{-0.6cm}\underline{\textbf{If} $ \boldsymbol{(a,b,c,d)=(4,11,2,3)\approx(4,10,2,3)} $} then $ lk(2)=C_{9}([11,a_{1},a_{2}],[a_{2},a_{3},1],[1,0,3],\allowbreak[3,7,8]) $, $ lk(3)\allowbreak=C_{9}([4,a_{4},a_{5}],[a_{5},a_{6},7],[7,8,2],[2,1,0]) $ and $ lk(4)=C_{9}([3,a_{5},a_{4}],[a_{4},8,\allowbreak a_{7}],[a_{7},a_{8},5],[5,\allowbreak6,0]) $ where $ a_{1},a_{2},a_{3}\in\{4,5,6,10\} $, $ a_{5}\in\{9,10,11\} $, $ a_{6}\in\{5,9,10,11\} $ and $ a_{4}\in\{9,11\} $. $ \boldsymbol{a_{4}=9}\Rightarrow a_{7}=11,a_{5}=10,a_{6}=11 $. Now $ lk(11)=C_{9}([2,a_{9},3],[3,10,a_{8}],[a_{8},\allowbreak5,4],[4,9,8]) $ which implies $ [2,3] $ form an edge and non-edge in two 4-gon which is not possible. $ \boldsymbol{a_{4}=11}\allowbreak\Rightarrow a_{7}=9,a_{5}=10 $. Now $ lk(11)=C_{9}([2,a_{2},a_{1}],[a_{1},a_{9},10],[10,3,4],[4,9,8]) $. Therefore $ a_{1},a_{2}\in\{5,6\} $, $ a_{9}=1 $ and $ a_{8}\in\{2,7\} $. If $ a_{1}=6 $ then $ a_{2}=5 $ and $ [1,6,7],[5,9,10] $ form 3-gons. Now $ lk(6)=C_{9}([1,10,11],[11,2,5],[5,4,0],[0,8,7]) $. As $ [1,9] $ and $ [9,a_{8}] $ is not an edge in a 3-gon, therefore $ lk(9) $ is not possible. If $ a_{1}=5,a_{2}=6 $ then $ lk(5)=C_{9}([a_{8},9,4],[4,0,6],[6,2,11],\allowbreak [11,10,1]) $ which implies $ a_{8}=7 $ and $ [6,9,10] $ will be a 3-gon and then $ lk(9) $ will not possible as $ [1,9],[7,9] $ will not form an edge in a 3-gon.
	
	\hspace{-0.6cm}\underline{\textbf{If} $ \boldsymbol{(a,b,c,d)=(4,2,5,3)} $} then $ lk(2)=C_{8}([5,a_{1},3],[3,0,1],[1,a_{2},4],[4,9,8]) $ which implies $ [3,5] $ form an edge and non-edge, which is not possible.
	
	\hspace{-0.6cm}\underline{\textbf{If} $ \boldsymbol{(a,b,c,d)=(10,2,5,3)} $} then $ lk(2)=C_{9}([5,a_{1},1],[1,0,3],[3,a_{2},10],[10,9,8]) $ and $ lk(3)\allowbreak =C_{9}([4,a_{3},7],[7,8,5],[5,a_{4},2],[2,1,0]) $ which implies $ a_{2}=5,a_{4}=10 $ and then $ 5 $ will occur two times in $ lk(2) $ which is a contradiction.
	
	\begin{sloppypar}
	\hspace{-0.6cm}\underline{\textbf{If} $ \boldsymbol{(a,b,c,d)=(2,10,5,3)} $} then $ lk(3)=C_{9}([4,a_{1},7],[7,8,5],[5,a_{2},2],[2,1,0]) $ and $ lk(5)=C_{9}([10,a_{3},4],[4,0,6],[6,a_{4},3],[3,7,8]) $ which implies $ a_{1}=11=a_{3},a_{2}=6,a_{4}=2 $ and then $ 3 $ will occur two times in $ lk(4) $ which make contradiction. Similarly for $ \textbf{(a,b,c,d)=(4,10,5,3)} $, map will not exist.
\end{sloppypar}
	
	\hspace{-0.6cm}\underline{\textbf{If} $ \boldsymbol{(a,b,c,d)=(11,10,5,3)\approx(10,11,5,3)} $} then $ lk(3)=C_{9}([4,a_{1},7],[7,8,5],[5,a_{2},2],\allowbreak[2,1,0])$ and $lk(5)\allowbreak=C_{9}([10,a_{3},4],[4,0,6],[6,a_{4},3],[3,7,8]) $. From here we see that $ a_{2}=6,a_{3}=1,a_{4}=2$ and $ a_{1}=11 $. Now $ lk(4)=C_{9}([3,7,a_{1}],[a_{1},a_{5},1],[1,10,5],[5,6,0]) $ where $ a_{5}\in\{2,9\} $. Therefore for any $ a_{5} $, $ [1,10,b_{1}] $ will be a 3-gon which is not possible.
	
	\hspace{-0.6cm}\underline{\textbf{If} $ \boldsymbol{(a,b,c,d)=(4,2,10,3)} $} then $ lk(2)=C_{9}([10,a_{1},3],[3,0,1],[1,a_{2},4],[4,9,8]) $ which implies $ [3,10] $ form an edge and non-edge in a 4-gon which is not possible.
	
	\hspace{-0.6cm}\underline{\textbf{If} $ \boldsymbol{(a,b,c,d)=(4,5,10,3)} $} then $ lk(3)=C_{9}([4,a_{1},b_{1}],[b_{1},8,b_{2}],[b_{2},a_{2},2],[2,1,0]) $. We see that $ a_{1}\neq 9 $ as then $ 9 $ will occur two times in $ lk(4) $, therefore $ a_{1}=11 $ and $ b_{1},b_{2}\in\{7,10\} $.
	
	\textbf{If} $ \boldsymbol{(a_{1},b_{1})=(11,7)} $ then $ b_{2}=10 $ and $ lk(4)=C_{9}([3,7,11],[11,a_{3},9],[9,8,5],[5, 6,\allowbreak0]) $ where $ a_{2}\in\{6,9\} $ and $ a_{3}\in\{2,10\} $. If $ a_{2}=6 $ then $ lk(6)=C_{9}([7,8,0],[0,4,5],[5,11,10],\allowbreak [10,3,2]) $ and then $ [1,9,11] $ will form a 3-gon, which implies $ lk(11) $ is not possible. If $ a_{2}=9 $ then $ lk(9)=C_{9}([1,0,8],[8,5,4],[4,11,10],[10,3,2]) $ which implies $ 9 $ will occur two times in $ lk(1) $.
	
	\textbf{If} $ \boldsymbol{(a_{1},b_{1})=(11,10)} $ then $ b_{2}=7 $ and $ a_{2}=9 $ (from $ lk(5) $ we see that $ a_{2}\neq 5 $). Now $ lk(9)=C_{9}([1,0,8],[8,5,4],[4,11,2],[2,3,7]) $ which implies $ [2,6,11] $ form a 3-gon and then $ lk(2)=C_{9}([11,4,9],[9,7,3],[3,0,1],[1,a_{3},6]) $. From $ lk(10) $ we see that $ a_{3}=5 $, therefore $ lk(5)=C_{9}([10,a_{4},1],[1,2,6],[6,0,4],[4,9,8]) $ which implies $ a_{4}\notin V $.
	
	\hspace{-0.6cm}\underline{\textbf{If} $ \boldsymbol{(a,b,c,d)=(11,5,10,3)\approx(10,5,11,3)} $} then $ lk(3)=C_{9}([4,a_{1},b_{1}],[b_{1},8,b_{2}],[b_{2},a_{2},\allowbreak2],[2,1,0]) $ and $ lk(5)=C_{9}([10,a_{3},b_{3}],[b_{3},0,b_{4}],[b_{4},a_{4},11],[11,9,8]) $ where $ a_{1}\in\{9,11\} $ and $ b_{1}\in\{7,10\} $.
	
	\textbf{If} $ \boldsymbol{(a_{1},b_{1})=(9,7)} $ then $ b_{2}=10 $ and either $ lk(9)=C_{9}([1,0,8],[8,5,11],[11,a_{5},4],\allowbreak [4,3,7]) $ or $ lk(9)=C_{9}([11,5,87],[8,0,1],[1,a_{5},4],[4,3,7]) $. If $ lk(9)=C_{9}([1,0,8],[8,5,11],\allowbreak [11,a_{5},4],[4,3,7]) $ then $ [2,6,11] $ will form a 3-gon and $ lk(4)=C_{9}([3,7,9],[9,11,a_{5}],[a_{5},a_{6},\allowbreak5],\allowbreak [5,6,0]) $ which implies $ b_{3}=4,a_{6}=10 $ and $ a_{3}=2 $ with $ a_{3}=a_{5} $. Now $ lk(5)=C_{9}([10,2,4],[4,\allowbreak 0,6],[6,a_{4},11],[11,9,8]) $ which implies $ a_{4}=1 $ and then $ lk(1) $ will not possible. If $ lk(9)=C_{9}([11,5,87],[8,0,1],[1,a_{5},4],[4,3,7]) $ then $ [1,2,6] $ will form a 3-gon and $ lk(4)=C_{9}([0,6,5],\allowbreak [5,a_{6},a_{5}],[a_{5},1,9],[9,7,3]) $ and $ lk(1)=C_{9}([6,a_{7},a_{5}],[a_{5},4,9],[9,8,0],\allowbreak[0,3,2]) $ which implies $ a_{5},a_{6}\in\{10,11\} $ that means either $ 10 $ or $ 11 $ will occur two times in $ lk(5) $.
	
	\textbf{If} $ \boldsymbol{(a_{1},b_{1})=(9,10)} $ then $ b_{2}=7 $ and then $ a_{2}=11 $ as from $ lk(5) $ we see that $ a_{2}\neq 5 $. Now $ lk(4)=C_{9}([3,10,9],[9,a_{5},a_{6}],[a_{6},a_{7},5],[5,6,0]) $ which implies $ b_{3}=6,b_{4}=4 $ and then $ a_{7}=11 $. Now we have $ 10 $ can not occur in new 3-gon, therefore $ [9,a_{5}] $ will be an edge in a 3-gon and then $ lk(9) $ will not possible.
	
	\textbf{If} $ \boldsymbol{(a_{1},b_{1})=(11,7)} $ then $ b_{2}=10, b_{3}=4,b_{4}=6 $ and $ a_{2}\in\{6,9\}, a_{4}\in\{1,2,3\} $. Now $ lk(4)=C_{9}([3,7,11],[11,a_{5},a_{6}],[a_{6},a_{7},5],[5,6,0]) $ which implies $ a_{7}=10 $. Now from $ lk(4),lk(5) $ and $ lk(11) $ we see that either $ a_{4}=a_{5} $ or $ a_{4}=7 $. But $ a_{4}\neq a_{4} $, therefore $ a_{4}=a_{5} $. Now $ lk(11)=C_{9}([9,8,5],[5,6,a_{5}],[a_{5},a_{6},4],[43,7]) $ which implies $ [1,2,6] $ form a 3-gon. Then $ a_{2}=9 $ and then $ lk(9) $ is not possible.
	
	\textbf{If} $ \boldsymbol{(a_{1},b_{1})=(11,10)} $ then $ b_{2}=7 $ and $ a_{2}=9 $ as from $ lk(5) $ we see that $ a_{2}\neq 5 $. Now $ lk(4)=C_{9}([3,10,11],[11,a_{5},a_{6}],[a_{6},a_{7},5],[5,6,0]) $ which implies $ a_{7}\notin V $.
	
	\hspace{-0.6cm}\underline{\textbf{If} $ \boldsymbol{(a,b,c,d)=(2,11,10,3)\approx(2,10,11,3)} $} then $ lk(3)=C_{9}([4,a_{1},b_{1}],[b_{1},8,b_{2}],[b_{2},a_{2},\allowbreak2],[2,1,0]) $ where $ a_{1}\in\{9,11\}, b_{1}\in\{7,10\} $ and $ a_{2},b_{2}\in V $.
	
	\textbf{If} $ \boldsymbol{(a_{1},b_{1})=(9,7)} $ then $ b_{2}=10 $ and $ a_{2}\in\{5,6\} $. Now $ lk(4)=C_{9}([3,7,9],[9,a_{3},\allowbreak a_{4}],[a_{4},a_{5},5],[5,6,0]) $. From $ lk(4),lk(9) $ and $ lk(9) $ we get $ a_{3}\in\{1,2\} $. \textbf{If} $ \boldsymbol{a_{3}=1} $ then $ [1,5,6] $ and $ [2,7,9] $ form 3-gons. Now $ lk(7)=C_{9}([9,4,3],[3,\allowbreak10,8],[8,0,6],[6,a_{6},2]) $ which implies $ a_{6}\notin V $. \textbf{If} $ \boldsymbol{a_{3}=2} $ then $ [1,7,9] $ and $ [2,5,6] $ form 3-gons. Now $ lk(7)=C_{9}([9,4,3],[3,\allowbreak10,8],[8,0,6],[6,a_{6},1]) $ which implies $ a_{6}=11 $ as if $ a_{6}=2 $ then face sequence will not follow in $ lk(1) $. Therefore $ lk(1)=C_{9}([9,8,0],[0,3,2],[2,a_{7},11],[11,6,\allowbreak7]) $ which implies $ a_{7}=5,a_{2}=6 $ and then $ lk(6)=C_{9}([5,4,0],[0,8,7],[7,a_{8},10],[10,3,2]) $ which implies $ [7,10] $ form a non-edge in two 4-gon, which is not possible.
	
	\textbf{If} $ \boldsymbol{(a_{1},b_{1})=(9,10)} $ then $ b_{2}=7 $ and $ a_{2}\in\{5,11\} $. Now $ lk(4)=C_{9}([3,10,9],[9,\allowbreak a_{3},a_{4}],[a_{4},a_{5},5],[5,6,0]) $. From $ lk(4),lk(8) $ and $ lk(9) $ we see that $ a_{3}\in\{1,2\} $, but form any $ a_{3} $, $ 10 $ will occur in two 3-gon, which is a contradiction.
	
	\textbf{If} $ \boldsymbol{(a_{1},b_{1})=(11,7)} $ then $ b_{2}=10 $ and $ a_{2}\in\{5,6,9\} $. Now $ lk(4)=C_{9}([3,7,11],\allowbreak [11,a_{3},a_{4}],[a_{4},a_{5},5],[5,6,0]) $, $ lk(11)=C_{9}([10,a_{4},4],[4,3,7],[7,a_{6},2],[2,9,8]) $ and $ lk(10)=C_{9}([11,4,a_{4}],[a_{4},a_{7},a_{2}],[a_{2},2,3],[3,7,8]) $. From these we see that $ a_{3}=10,a_{4}=1 $ and form any $ a_{5} $, $ 10 $ will occur in two 3-gon, which is a contradiction.
	
	\textbf{If} $ \boldsymbol{(a_{1},b_{1})=(11,10)} $ then $ b_{2}=7 $ and $ a_{2}\in\{5,9\} $. Now $ lk(11)=C_{9}([8,9,2],[2,\allowbreak a_{3},a_{4}],[a_{4},a_{5},4],[4,3,10]) $ and $ lk(4)=C_{9}([3,10,11],[11,a_{4},a_{5}],[a_{5},a_{6},5],[5,6,0]) $. We see that if $ a_{2}=5 $ then $ a_{2}=a_{3}=5 $ and then $ [1,2,9] $ will form a 3-gon which is not possible as then $ 9 $ will occur two times in $ lk(1) $ and if $ a_{2}=9 $ then $ lk(2)=C_{9}([1,0,3],[3,7,9],[9,8,11],\allowbreak[11,a_{4},\allowbreak a_{3}]) $ which implies $ a_{4}\notin V $.
	
	\hspace{-0.6cm}\underline{\textbf{If} $ \boldsymbol{(a,b,c,d)=(4,11,10,3)\approx(4,10,11,3)} $} then $ lk(3)=C_{9}([4,a_{1},b_{1}],[b_{1},8,b_{2}],[b_{2},a_{2},\allowbreak2],[2,1,0]) $ and $ lk(4)=C_{9}([3,b_{1},a_{1}],[a_{1},8,b_{3}],[b_{3},a_{3},5],[5,6,0]) $ where $ a_1\in\{9,11\},b_{1}\in\{7,10\} $.
	
	\textbf{If} $ \boldsymbol{(a_{1},b_{1})=(9,7)} $ then either $ lk(7)=C_{9}([6,0,8],[8,10,3],[3,4,9],[9,a_{4},2]) $ or $ lk(7)=C_{9}([9,4,3],[3,10,8],[8,0,6],[6,a_{4},2]) $.
	\begin{case}
		If $ lk(7)=C_{9}([6,0,8],[8,10,3],[3,4,9],[9,a_{4},2]) $: as we have $ [1,9] $ can not be an edge and a non-edge of a 4-gon, therefore $ [1,5,9] $ will be a 3-gon and then $ lk(9)=C_{9}([5,2,7],[7,3,4],\allowbreak [4,11,8],[8,0,1]) $ which implies $ a_{4}=5$. Now $ lk(5)=C_{9}([9,7,2],[2,a_{7},\allowbreak a_{8}],[a_{8},0,a_{5}],[a_{5},a_{6},\allowbreak 1]), lk(6)=C_{9}([2,a_{9},a_{6}],[a_{6},1,5],[5,4,0],[0,8,7]) $ and $ lk(2)=C_{9}([6,\allowbreak a_{10},3],[3,0,1],[1,4,5],\allowbreak [5,9,7]) $ which implies $ a_{5}=6,a_{8}=4, a_{6}=a_{10}=11, a_{9}=3 $ and $ a_{7}=1 $ which is a contradiction.
	\end{case}
	\begin{case}
		If $ lk(7)=C_{9}([9,4,3],[3,10,8],[8,0,6],[6,a_{4},2]) $: then $ [1,5,6] $ will form a 3-gon which implies $ a_{4}=11,a_{3}=3 $. Now $ lk(11)=C_{9}([10,a_{5},6],[6,7,2],[2,5,4],[4,9,8]) $ and $ lk(6)=C_{9}([1,10,11],[11,2,7],[7,8,0],[0,4,5]) $ which implies $ a_{5}=1 $, then $ lk(1) $ will not possible.
	\end{case}
	\textbf{If} $ \boldsymbol{(a_{1},b_{1})=(9,10)} $ then $ b_{2}=7 $. Now $ lk(9)=C_{9}([1,0,8],[8,11,4],[4,3,10],[10,\allowbreak a_{5},a_{4}]) $. From here we see that $ a_{4}\in\{5,6\} $ and therefore $ [2,7] $ form an edge in a 3-gon. Therefore $ [2,7] $ form an edge and non-edge, which is a contradiction.
	
	\textbf{If} $ \boldsymbol{(a_{1},b_{1})=(11,7)} $ then $ b_{2}=10,b_{3}=9 $. Now $ lk(7)=C_{9}([6,0,8],[8,10,3],[3,4,\allowbreak 11],[11,a_{4},a_{5}]) $ which implies $ a_{5}\in\{1,2\} $ and then $ [5,9] $ form an edge in a 3-gon i.e. $ [5,9] $ form an edge and a non-edge which is a contradiction.
	
	\textbf{If} $ \boldsymbol{(a_{1},b_{1})=(11,10)} $ then $ 10 $ will occur two times in $ lk(11) $ which make contradiction.
	
	\hspace{-0.6cm}\underline{\textbf{If} $ \boldsymbol{(a,b,c,d)=(5,11,10,3)\approx(5,10,11,3)} $} then $ lk(3)=c_{9}([4,a_{1},b_{1}],[b_{1},8,b_{2}],[b_{2},a_{2},\allowbreak2],[2,1,0]) $ where $ a_{1}\in\{9,11\} $ and $ b_{1}\in\{7,10\} $.
	
	\textbf{If} $ \boldsymbol{(a_{1},b_{1})=(9,7)} $ then $ b_{2}=10 $ and either $ lk(7)=C_{9}([6,0,8],[8,10,3],[3,4,9],[9,a_{3},\allowbreak a_{4}]) $ or $ lk(7)=C_{9}([9,4,3],[3,10,8],[8,0,6],[6,a_{3},a_{4}]) $. If  $ lk(7)=C_{9}([6,0,8],[8,10,3], [3,4,\allowbreak9],\allowbreak[9, a_{3},a_{4}]) $ then $ a_{3}\in\{1,5\} $. We see that for any $ a_{3} $, $ [4,9] $ will form an edge in a 3-gon which implies $ 4 $ occur in two 3-gon which is not possible. If $ lk(7)=C_{9}([9,4,3],[3,10,8],[8,0,\allowbreak 6],[6,a_{3},a_{4}]) $ then $ lk(9)=C_{9}([a_{4},11,8],[8,0,1],[1,a_{5},4],[4,3,7]) $ which implies $ a_{4}=5 $ and then $ [5,6] $ form an edge and a non-edge which make contradiction.
	
	\textbf{If} $ \boldsymbol{(a_{1},b_{1})=(9,10)} $ then $ b_{2}=7 $ and then $ lk(9) $ is not possible.
	
	\textbf{If} $ \boldsymbol{(a_{1},b_{1})=(11,7)} $ then $ b_{2}=10 $. Now $ lk(7)=C_{9}([6,0,8],[8,10,3],[3,4,11],\allowbreak [11,a_{3},\allowbreak a_{4}]) $ and $ lk(11)=C_{9}([10,a_{5},4],[4,3,7],[7,1,5],[5,9,8]) $. From these we see that $ a_{3}=5 $ and we have $ [9,11] $ is a non-edge in a 4-gon, therefore $ a_{4}=1 $ and $ [2,5,9] $ form a 3-gon and then $ lk(5) $ is not possible.
	
	\textbf{If} $ \boldsymbol{(a_{1},b_{1})=(11,10)} $ then $ 11 $ will occur two times in $ lk(10) $, which is a contradiction.
	
	\hspace{-0.6cm}\underline{\textbf{If} $ \boldsymbol{(a,b,c,d)=(d,5,2,4)} $} then from $ lk(2) $ we see that $ [3,5] $ form an edge and non-edge in two 4-gon which make contradiction. Therefore for $ (a,b,c,d)=(3,5,2,4),(10,5,2,4) $, SEM will not complete.
	
	\hspace{-0.6cm}\underline{\textbf{If} $ \boldsymbol{(a,b,c,d)=(3,10,2,4)} $} then $ lk(2)=C_{9}([10,a_{1},3],[3,0,1],[1,a_{2},4],[4,7,8]) $. From $ lk(4) $ we see that $ a_{2}=5 $ and then $ a_{1}=11 $. Therefore $ lk(4)=C_{9}([3,11,7],[7,8,2],[2,1,5],[5,\allowbreak 6,0]) $ and then $ 11 $ will occur two times in $ lk(3) $, which is a contradiction.
	
	\hspace{-0.6cm}\underline{\textbf{If} $ \boldsymbol{(a,b,c,d)=(11,10,2,4)\approx(10,11,2,4)} $} then $ lk(2)=C_{9}([10,a_{1},3],[3,0,1],[1,a_{2},4],\allowbreak[4,7,8]) $. From $ lk(4) $ we see that $ a_{2}=5 $ and then $ a_{1}=6 $. Therefore $ lk(4)=C_{9}([3,11,7],[7,\allowbreak8,2],[2,1,5],[5,\allowbreak6,\allowbreak 0]) $ and $ lk(3)=C_{9}([4,7,11],[11,a_{3},6],[6,10,2],[2,1,0]) $. From here we see that $ [6,10] $ form an edge in a 3-gon i.e. $ 10 $ will occur in two 3-gon which is not possible.
	
	\hspace{-0.6cm}\underline{\textbf{If} $ \boldsymbol{(a,b,c,d)=(3,2,5,4)} $} then $ lk(2)=C_{9}([5,a_{1},a_{2}],[a_{2},a_{3},1],[1,0,3],[3,9,8]) $, $ lk(3)=C_{9}([4,a_{4},a_{5}],[a_{5},a_{6},9],[9,8,2],[2,1,0]) $, $ lk(5)=C_{9}([2,a_{2},a_{1}],[a_{1},a_{7},6],[6,0,4],[4,7,8]) $ and $ lk(4)=C_{9}([3,a_{5},a_{4}],[a_{4},a_{8},7],[7,8,5],[5,6,0]) $. From here we see that two of $ a_{1},a_{2},a_{3} $ are $ 10 $ and $ 11 $, two of $ a_{1},a_{2},a_{7} $ are $ 10 $ and $ 11 $, two of $ a_{4},a_{5},a_{6} $ are $ 10 $ and $ 11 $ and two of $ a_{4},a_{5},a_{8} $ are $ 10 $ and $ 11 $ which implies $ a_{1},a_{2}\in\{10,11\} $ and $ a_{4},a_{5}\in\{10,11\} $. With out loss of generality, let $ (a_{1},a_{2})=(10,11) $, then $ a_{3}\in\{4,6\} $. If $ a_{3}=4 $ then $ lk(4) $ is not possible, therefore $ a_{3}=6 $. Now $ a_{7}\in\{1,3,9\} $, therefore $ lk(6) $ is possible only if $ a_{7}=1 $ and hence $ lk(6)=C_{9}([11,2,1],[1,10,5],[5,4,0],[0,8,7]) $ which implies $ [1,9,10] $ form a 3-gon and $ a_{5}=11,a_{4}=10 $ and then $ a_{6}=7 $. Now $ lk(7)=C_{9}([11,3,9],[9,10,4],[4,5,8],[8,0,6]) $ which implies $ a_{8}=9 $. Therefore $ lk(1)=C_{9}([10,5,6],[6,11,2],[2,3,0],[0,8,9])\Rightarrow lk(9)=C_{9}([1,0,8],[8,2,3],[3,11,7],[7,4,10])\Rightarrow lk(10)=C_{9}([9,7,4],[4,3,11],[11,2,5],[5,6,1])\Rightarrow lk(11)=C_{9}([7,9,3],[3,4,10],[10,5,2],[2,1,6]) $. Let this be the map $ \boldsymbol{KNO_{2[(3,4^4)]}} $.
	
\begin{sloppypar}
\hspace{-0.6cm}\underline{\textbf{If} $ \boldsymbol{(a,b,c,d)=(10,2,5,4)} $} then $ lk(2)=C_{9}([5,a_{1},b_{1}],[b_{1},0,b_{2}],[b_{2},a_{2},10],[10,9,8]) $, $ lk(4)\allowbreak =C_{9}([3,a_{3},a_{4}],[a_{4},a_{5},7],[7,8,5],[5,6,0]) $ and $ lk(5)=C_{9}([2,b_{1},a_{1}],[a_{1},a_{6},6],[6,0,4],[4,7,\allowbreak8]) $ which implies $ a_{1}=11 $ and $ b_{1},b_{2}\in\{1,3\} $.
\end{sloppypar}	
	\begin{case}
		\textbf{If} $ \boldsymbol{(b_{1},b_{2})=(1,3)} $ then $ a_{2}\in\{6,7\} $. If $ a_{2}=6 $ then from $ lk(3) $ and $ lk(5) $ we see that face sequence is not followed in $ lk(6) $ and if $ a_{2}=7 $ then we get face sequence is not followed in $ lk(7) $ from $ lk(3) $ and $ lk(4) $.
	\end{case}
	\begin{case}
		\textbf{If} $ \boldsymbol{(b_{1},b_{2})=(3,1)} $ then $ a_{2}\in\{6,7\} $.
		\begin{subcase}
			\textbf{If} $ \boldsymbol{a_{2}=6} $ then $ a_{6}\in\{1,10\} $. \textbf{If} $ \boldsymbol{a_{6}=1} $ then $ lk(6)=C_{9}([10,2,1],[1,11,5],\allowbreak [5,4,0],[0,8,7]) $, $ lk(1)=C_{9}([9,8,0],[0,3,2],[2,10,6],[6,5,11]) $, $ lk(10)=C_{9}([7,a_{5},a_7],[a_{7},\allowbreak a_{8},\allowbreak 9],[9,8,2],[2,1,6]) $ and $ lk(7)=C_{9}([10,a_{7},a_{5}],[a_{5},a_{4},4],[4,5,8],[8,0,6]) $. Now we see that only one of $ a_{3},a_{4},a_{5} $ is $ 11 $, only one of $ a_{5},a_{7},a_{8} $ is $ 11 $ and only one of $ a_{4},a_{5},a_{7} $ is $ 11 $ which means either $ a_{5}=11 $ or $ a_{4}=11=a_{8} $ or $ a_{3}=11=a_{7} $. If $ a_{4}=11=a_{8} $ then from $ lk(11) $ we see that $ a_{5}=a_{7} $ which is not possible and if $ a_{3}=11=a_{7} $ then $ [9,11] $ form an edge and a non-edge in two 4-gon which is not possible, therefore $ a_{5}=11 $. Now $ lk(11)=C_{9}([9,4,7],[7,10,3],[3,2,5],[5,6,1]) $ which implies $ a_{7}=3,a_{4}=9 $ and then $ lk(7)=C_{9}([10,3,11],[11,9,4],[4,5,8],[8,0,6]) $. Now $ lk(4)=C_{9}([3,10,9],[9,11,7],[7,8,5],[5,6,0])\allowbreak\Rightarrow a_{3}=10,a_{4}=9 $ and $ lk(9)=C_{9}([11,7,4],[4,3,10],[10,2,8],[8,0,1])\Rightarrow a_{8}=4\Rightarrow lk(3)=C_{9}([4,9,10],[10,7,11],[11,5,2],[2,1,0]) $. Let this be the map $ \boldsymbol{K_{1}} $. $ \boldsymbol{K_1} $ is isomorphic to $ \boldsymbol{KNO_{1[(3,4^4)]}} $ under the map $ (0,1)(2,8)(3,9)(4,10,7,11) $. \textbf{If} $ \boldsymbol{a_{6}=10} $  then $ lk(6)=C_{9}([1,2,\allowbreak10],[10,11,5],[5,\allowbreak 4,0],[0,8,7]) $ and then $ [9,10,11] $ will be a 3-gon which implies $ lk(1)=C_{9}([7,11,9],[9,8,0],[0,\allowbreak 3,2],[2,10,6]) $ and $ lk(7)=C_{9}([1,9,11],[11,a_{7},4],[4,5,8],[8,\allowbreak 0,6]) $. We see that $ 3 $ is not a member of $ lk(6) $ and $ lk(8) $, therefore $ a_{7}=3 $ which implies $ 3 $ occur two times in $ lk(4) $ which is a contradiction.
		\end{subcase}
		\begin{subcase}
			\textbf{If} $ \boldsymbol{a_{2}=7} $ then $ lk(3)=C_{9}([4,a_{4},a_{3}],[a_{3},a_{7},11],[11,5,2],[2,1,0]) $ and $ lk(7)\allowbreak =C_{9}([a_{9},a_{8},a_{5}],[a_{5},a_{4},4],[4,5,8],[8,0,6]) $ which implies $ a_{8}=2;a_{5},a_{9}\in\{1,10\}\Rightarrow a_{4}=11 $ which implies $ 11 $ will occur two times in $ lk(3) $ which is a contradiction.
		\end{subcase}
	\end{case}
\begin{sloppypar}
	\hspace{-0.6cm}\underline{\textbf{If} $ \boldsymbol{(a,b,c,d)=(2,10,5,4)} $} then $ lk(4)=C_{9}([3,a_{1},a_{2}],[a_{2},a_{3},7],[7,8,5],[5,6,0]) $ and $ lk(5)\allowbreak =C_{9}([10,a_{4},a_{5}],[a_{5},a_{6},6],[6,0,4],[4,7,8]) $ where $ a_{1},a_{2}\in\{9,10,11\} $ and $ a_{4},a_{5}\in\{1,3,11\} $.
\end{sloppypar}
	\begin{case}
		\textbf{If} $ \boldsymbol{a_{2}=9} $ then $ [9,a_{1}] $ will form an edge in a 3-gon which implies $ a_{1}=11 $ and $ a_{3}\in\{1,2\} $. \textbf{If} $ \boldsymbol{a_{3}=1} $ then $ lk(9)=C_{9}([11,3,4],[4,7,1],[1,0,8],[8,10,2]) $ which implies $ [1,6,7] $ will be a 3-gon and then $ lk(1)=C_{9}([7,4,9],[9,8,0],[0,3,2],[2,11,6]) $ and $ lk(6)=C_{9}([7,8,0],[0,4,5],[5,10,11],[11,2,1]) $ (as $ 10 $ is not a member of $ lk(0) $ and $ lk(1) $) which implies $ a_{5}=10,a_{6}=11 $ and then $ 10 $ will occur two times in $ lk(5) $ which is a contradiction. \textbf{If} $ \boldsymbol{a_{3}=2} $ then $ lk(2)=C_{9}([1,0,3],[3,11,10],[10,8,9],[9,4,7]) $ (as $ 11 $ is not a member of $ lk(0) $ and $ lk(3) $) which implies $ 11 $ occur two times in $ lk(3) $ which is a contradiction.
	\end{case}
	\begin{case}
		\textbf{If} $ \boldsymbol{a_{2}=10} $ then $ lk(10)=C_{9}([5,a_{5},a_{4}],[a_{4},a_{7},a_{8}],[a_{8},a_{9},2],[2,9,8]) $ which implies $ a_{3}=2,a_{9}=7,a_{8}=4,a_{7}=3,a_{4}=1 $ which implies $ a_{1}=11=a_{4} $ and from $ lk(5) $ we see that $ a_{5}=1 $ which implies $ [1,a_{4}] $ form an edge in a 3-gon and then $ a_{4}=11 $ and $ a_{6}=9 $ as if $ a_{6}=2 $ then $ lk(2) $ will not possible. Therefore $ lk(1)=C_{9}([9,8,0],[0,3,2],[2,6,5],[5,10,11]) $ which implies $ [2,6,7] $ will form a 3-gon and then $ lk(2) $ is not possible.
	\end{case}
	\begin{case}
		\textbf{If} $ \boldsymbol{a_{2}=11} $ then one of $ a_{4},a_{5},a_{6} $ is $ 11 $. \textbf{If} $ \boldsymbol{a_{4}=11} $ then $ a_{5}=1 $. As $ [3,11] $ is a non-edge in a 4-gon then $ a_{6}\in\{2,9\} $. If $ a_{6}=2 $ then $ [1,9,11] $ and $ [2,6,7] $ form 3-gons and then $ lk(1)=C_{9}([9,8,0],[0,3,2],[2,6,5],[5,10,11]) $ which implies $ lk(2) $ is not possible. If $ a_{6}=9 $ then $ [1,2,11] $ and $ [6,7,9] $ form 3-gons and $ lk(1)=C_{9}([2,3,0],[0,8,9],[9,6,5],[5,10,\allowbreak11])\Rightarrow lk(9)=C_{9}([7,11,2],[2,10,8],[8,0,1],[1,5,6])\Rightarrow lk(2)=C_{9}([11,7,9],[9,8,10],[10,6,\allowbreak3],[3,\allowbreak0, 1])\Rightarrow lk(10)=C_{9}([5,1,11],[11,4,6],[6,3,2],[2,9,8]) $ which implies $ 5 $ occur two times in $ lk(6) $ which is a contradiction. \textbf{If} $ \boldsymbol{a_{5}=11} $ then $ a_{4}=1 $ as $ [3,11] $ is a non-edge in a 4-gon and then either $ lk(1)=C_{9}([2,3,0],[0,8,9],[9,a_{7},10],[10,5,11]) $ implies $ [9,10] $ form non-edge in two 4-gon or $ lk(1)=C_{9}([9,8,0],[0,3,2],[2,a_{7},10],[10,5,11]) $ implies $ [2,10] $ edge and non-edge in two 4-gon which is a contradiction. \textbf{If} $ \boldsymbol{a_{6}=11} $ then $ a_{4},a_{5}\in\{1,3\} $. We have $ [3,11] $ is a non-edge in a 4-gon, therefore $ a_{4}=3, a_{5}=1 $. Now $ lk(5)=C_{9}([10,3,1],[1,11,6],[6,0,4],[4,7,8]) $ which implies $ lk(1) $ is not possible.
	\end{case}
	\hspace{-0.6cm}\underline{\textbf{If} $ \boldsymbol{(a,b,c,d)=(3,10,5,4)} $} then $ lk(3)=C_{9}([4,a_{1},b_{1}],[b_{1},8,b_{2}],[b_{2},a_{2},2],[2,1,0]) $ and $ lk(5)\allowbreak =C_{9}([10,a_{3},a_{4}],[a_{4},a_{5},6],[6,0,4],[4,7,8]) $ where $ a_{1}\in\{7,11\} $ and $ b_{1}\in\{9,10\} $.
	
	\textbf{If} $ \boldsymbol{(a_{1},b_{1})=(7,9)} $ then $ a_{2}=11,b_{2}=10 $ and $ lk(10)=C_{9}([5,a_{4},a_{3}],[a_{3},a_{6},11], [11,2,\allowbreak3],\allowbreak[3,9,8]) $ which implies $ a_{5}=11 $ and then $ a_{3},a_{4}\in\{1,2\} $, therefore $ 2 $ occur two times in $ lk(10) $ which is not possible.
	
	\textbf{If} $ \boldsymbol{(a_{1},b_{1})=(7,10)} $ then $ a_{2}=11,b_{2}=9 $ and $ lk(10)=C_{9}([5,a_{4},a_{3}],[a_{3},a_{7},7],[7,\allowbreak 4,3],[3,9,8]) $, $ lk(5)=C_{9}([10,a_{3},a_{4}],[a_{4},a_{8},6],[6,0,4],[4,7,8]) $ and $ lk(6)=C_{9}([a_{7},a_{9},a_{8}],\allowbreak [a_{8},a_{4},5],[5,4,0],[0,8,7]) $. Now we see that one of $ a_{3},a_{7} $ is $ 11 $, one of $ a_{3},a_{8} $ is $ 11 $ and one of $ a_{7},a_{8},a_{9} $ is $ 11 $. From here we see that $ a_{3}=11 $ and $ a_{9}=11 $ and $ a_{4},a_{7}\in\{1,2\} $ which implies $ 6\notin V(lk(10)) $ and we have $ 6\notin V(lk(3)) $, therefore $ 6\in V(lk(a')) $ for each $ a'\in V\diagdown \{3,6\} $. Therefore $ a_{8}=9 $ and $ a_{7}=2 $ as $ [1,9] $ is an edge in a 4-gon. Now we see that $ a_{4}=1 $ which implies $ [1,9,11] $ form a 3-gon and then $ 11 $ will occur two times in $ lk(9) $ which is not possible.
	
	\textbf{If} $ \boldsymbol{(a_{1},b_{1})=(11,9)} $ then $ b_{2}=10 $ and $ a_{2}\in\{6,7\} $. If $ a_{2}=6 $ then $ lk(6) $ is not possible and if $ a_{2}=7 $ then $ lk(7)=C_{9}([6,0,8],[8,5,4],[4,11,10],[10,3,2]) $ and $ lk(4)=C_{9}([3,9,11],[11,10,7],[7,8,5],[5,6,0]) $ and then $ 11 $ occur two times in $ lk(9) $ which is a contradiction.
	
	\textbf{If} $ \boldsymbol{(a_{1},b_{1})=(11,10)} $ then $ b_{2}=9 $ and $ a_{2}\in\{5,6,7\} $. If $ a_{2}=5 $ then $ lk(5) $ will not possible. If $ a_{2}=6 $ then either $ lk(6)=C_{9}([9,3,2],[2,11,5],[5,4,0],[0,8,7]) $ then $ lk(9)=C_{9}([7,11,1],[1,0,8],[8,10,3],[3,2,6]) $ which implies $ 11 $ will occur two times in $ lk(1) $ OR $ lk(6)=C_{9}([2,3,9],[9,11,5],[5,4,0],[0,8,7]) $ and then $ lk(9)=C_{9}([1,0,8],[8,10,3],[3,2,6],\allowbreak [6,5,11]) $ and then after completing $ lk(5) $ we see that $ [10,11] $ form an edge and a non-edge in two 4-gon which is not possible. If $ a_{2}=7 $ then $ lk(7)=C_{9}([c_{1},3,c_{2}],[c_{2},11,4],[4,5,8],[8,0,\allowbreak6]) $ and $ lk(4)=C_{9}([3,c_{3},11],[11,c_{2},7],[7,8,5],[5,6,0]) $ which implies $ [3,11] $ form an edge and a non-edge in two 4-gon which is not possible.
	
	\hspace{-0.6cm}\underline{\textbf{If} $ \boldsymbol{(a,b,c,d)=(11,10,5,4)\approx(10,11,5,4)} $} then $ lk(5)=C_{9}([10,a_{1},a_{2}],[a_{2},a_{3},6],[6,0,\allowbreak4],[4,7,8]) $ and $ lk(4)=C_{9}([3,a_{4},a_{5}],[a_{5},a_{6},7],[7,8,5],[5,6,0]) $ where $ a_{1},a_{2}\in\{1,2,3\} $.
	
	\textbf{If} $ \boldsymbol{(a_{1},a_{2})=(1,2)} $ then $ lk(2)=C_{9}([a_{7},a_{8},3],[3,0,1],[1,10,5],[5,6,a_{3}]) $ and $ lk(1)=C_{9}([a_{9},a_{10},10],[10,5,2],[2,3,0],[0,8,9]) $ which implies $ a_{3}=11,a_{7}=7,a_{9}=6,a_{8}=9 $ and then $ lk(7) $ will not possible.
	
	\textbf{If} $ \boldsymbol{(a_{1},a_{2})=(2,1)} $ then $ lk(1)=C_{9}([9,8,0],[0,3,2],[2,10,5],[5,6,a_{3}]) $ which implies $ a_{3}=11 $ and then $ 11 $ will occur two times in $ lk(9) $ which is not possible.
	
	\textbf{If} $ \boldsymbol{(a_{1},a_{2})=(2,3)} $ then $ a_{3}\in\{9,11\} $ and $ lk(6)=C_{9}([a_{5},a_{6},7],[7,8,0],[0,4,5],\allowbreak [5,3,\allowbreak a_{3}]) $. If $ \boldsymbol{a_{3}=9} $ then $ lk(9)=C_{9}([1,0,8],[8,10,11],[11,4,3],[3,5,6]) $ which implies $ [2,7,11] $ form a 3-gon and $ lk(6)=C_{9}([1,11,7],[7,8,0],[0,4,5],[5,3,9]) $, $ lk(1)=C_{9}([9,8,0], [0,3,2],\allowbreak[2,4,11],\allowbreak[11,7,6]) $. From here we see that $ [2,11] $ form an edge and a non-edge in two different 4-gon which is not possible. If $ \boldsymbol{a_{3}=11} $ then $ a_{4}\in\{7,9\} $. If $ a_{4}=7 $ then either $ lk(11)=C_{9}([6,5,3],[3,4,7],[7,a_{5},10],[10,8,9]) $ which implies $ lk(7) $ is not possible OR $ lk(11)=C_{9}([7,\allowbreak4,3],[3,5,6],[6,a_{5},10],[10,8,9]) $ then from $ lk(6) $ we see that $ 7 $ will occur in two 3-gon, which is not possible. If $ a_{4}=9 $ then $ lk(11)=C_{9}([a_{5},a_{6},10],[10,8,9],[9,4,3],[3,5,\allowbreak 6]) $ and $ lk(9)=C_{9}([a_{7},a_{8},4],[4,3,11],[11,10,8],[8,0,1]) $. From here we see that $ a_{5}=7 $ and $ a_{7}=2 $ which implies $ 2 $ will occur two times in $ lk(1) $ which is a contradiction.
	
	\textbf{If} $ \boldsymbol{(a_{1},a_{2})=(3,2)} $ then either $ lk(6)=C_{9}([a_{4},a_{5},7],[7,8,0],[0,4,5],[5,2,a_{3}]) $ or $ lk(6)\allowbreak=C_{9}([7,8,0],[0,4,5],[5,2,a_{3}],[a_{3},a_{4},a_{5}]) $ and $ a_{3}\in\{9,11\} $. When $ lk(6)=C_{9}([a_{4},a_{5},7],  [7,\allowbreak8,0],\allowbreak[0,4,5],[5,2,a_{3}]) $ then $ lk(7)=C_{9}([a_{6},a_{7},4],[4,5,8],[8,0,6],[6,a_{4},a_{5}]) $. From here we see that $ a_{4},a_{5}\in\{1,9,11\} $ which implies one of $ a_{6},a_{7} $ is $ 3 $ as $ 3\notin V(lk(6)),V(lk(8)) $. But then $ 3 $ will occur two times in $ lk(4) $ which is a contradiction. When $ lk(6)=C_{9}([7,8,0],[0,4,\allowbreak5],[5,2, a_{3}],[a_{3},a_{4},a_{5}]) $ then $ lk(7)=C_{9}([a_{5},a_{6},a_{7}],[a_{7},a_{8},4],[4,5,8],[8,0,6]) $. We have $ [9,11] $ is an edge in a 4-gon, therefore $ a_{5}=1 $ and then $ [2,9,11] $ will form a 3-gon. Since $ [1,9] $ is an edge in a 4-gon, therefore $ a_{3}=11 $. Now $ lk(1)=C_{9}([7,a_{7},2],[2,3,0],[0,8,a_{4}],[a_{4},11,\allowbreak6]) $ which implies $ a_{4}=9 $ \& $a_{6}=2 $ and then $ a_{7}=10 $. From here we see that $ 10 $ will occur two times in $ lk(2) $ which is not possible.
	
	\hspace{-0.6cm}\underline{\textbf{If} $ \boldsymbol{(a,b,c,d)=(3,2,10,4)} $} then $ lk(4)=C{9}([3,a_{1},b_{1}],[b_{1},8,b_{2}],[b_{2},a_{2},5],[5,6,0]) $, $ lk(3)\allowbreak=C_{9}([4,b_{1},a_{1}],[a_{1},b_{3},9],[9,8,2],[2,1,0]) $ and $ lk(2)=C_{9}([10,a_{3},a_{4}],[a_{4},a_{5},1],[1,0,3],[3,9,\allowbreak8]) $. From this we see that $ a_{1}=11 $ \& $ b_{1}\in\{7,10\} $.
	
	\textbf{If} $ \boldsymbol{b_{1}=7} $ then $ b_{2}=10 $ and $ b_{3}\in\{5,6,10\} $. For $ \boldsymbol{b_{3}\in\{6,10\}} $, we see that $ a_{5}=5 $ \& $ a_{3},a_{4}\in\{6,11\} $. If $ \boldsymbol{(a_{3},a_{4})=(6,11)} $ then $ b_{3}=6 $ and after completing $ lk(11) $ we will see that $ 1 $ will occur two times in $ lk(9) $. If $ \boldsymbol{(a_{3},a_{4})=(11,6)} $ then from $ lk(6) $ we see that $ [6,7,11] $ will be a 3-gon and then $ 11 $ will occur two times in $ lk(7) $ which is a contradiction. If $ \boldsymbol{b_{3}=5} $ then $ 10\in V(lk(a')) $ for each $ a'\in V\diagdown\{0,3\} $. Now either $ lk(9)=C_{9}([c_{1},10,5],[5,11,3],[3,2,8],[8,0,1]) $ which implies $ [5,10] $ form an edge and a non-edge in two different 4-gon OR $ lk(9)=C_{9}([c_{1},10,1],[1,0,8],[8,2,3],[3,11,5]) $ which implies $ lk(10)=C_{9}([2,a_{4},a_{3}],[a_{3},a_{5},a_{2}],[a_{2},5,4],[4,7,8]) $. From here we see that $ a_{2}=1,a_{3}=c_{1} $. Then $ c_{1}=6 $ which implies $ [1,7,11] $ form a 3-gon \& $ a_{3}\in\{5,7\} $, then $ a_{4}=7 $ i.e. $ [7,10] $ form a non-edge in two 4-gon which is not possible.
	
	\textbf{If} $ \boldsymbol{b_{1}=10} $ then $ b_{2}=7 $ and $ lk(10)=C_{9}([2,a_{4},a_{3}],[a_{3},a_{6},11],[11,3,4],[4,7,8]) $ which implies $ a_{3},a_{4}\in\{5,6\} $. If $ \boldsymbol{(a_{3},a_{4})=(6,5)} $ then $ a_{6}=1 $ and $ [5,9,11] $ form a 3-gon which means $ [9,11] $ form an edge and a non-edge in two different 4-gon, which is not possible. If $ \boldsymbol{(a_{3},a_{4})=(5,6)} $ then $ lk(5)=C_{9}([a_{6},11,10],[10,2,6],[6,0,4],[4,7,a_{2}]) $. Then $ a_{6}=1, a_{2}=9 $ and $ [6,7,11] $ form a 3-gon. Now $ lk(9)=C_{9}([5,7,4],[4,11,3],[3,2,8],[8,0,1]) $ which implies $ b_{3}=4 $ which is not possible.
	
	\hspace{-0.6cm}\underline{\textbf{If} $ \boldsymbol{(a,b,c,d)=(5,2,10,4)} $} then $ lk(2)=C_{9}([10,a_{1},b_{1}],[b_{1},0,b_{2}],[b_{2},a_{2},5],[5,9,8]) $ and $ lk(4)=C_{9}([3,a_{3},b_{3}],[b_{3},8,b_{4}],[b_{4},a_{4},5],[5,6,0]) $. From here we see that $ a_{3}\in\{9,11\},b_{3}\in\{7,10\} $.
	
	$ \boldsymbol{(a_{3},b_{3})=(9,7)}\Rightarrow b_{4}=10,a_{4}=11 $. Now $ lk(10)=C_{9}([8,7,4],[4,5,11],[11,a_{5},a_{1}],\allowbreak [a_{1},b_{1},2]) $ which implies $ a_{1}=6 $ and then for any $ b_{1} $, $ lk(6) $ will not possible.
	
	$ \boldsymbol{(a_{3},b_{3})=(9,10)}\Rightarrow b_{4}=7,a_{4}=11 $ which implies $ lk(10)=C_{9}([2,b_{1},a_{1}],[a_{1},a_{5},9],\allowbreak [9,3,4],[4,7,8]) $. From here we see that $ b_{1}=1 $ \& $ a_{5}=5,a_{1}=1 $. Now $ lk(3)=C_{9}([4,10,9],\allowbreak[9,a_{6},a_{2}],[a_{2},5,2],[2,1,0]) $. Therefore $ a_{2} $ can not be $ 11 $, hence $ a_{6}=11 $ and then $ [9,11] $ form an edge and a non-edge in two different 4-gon which is a contradiction.
	
	$ \boldsymbol{(a_{3},b_{3})=(11,7)}\Rightarrow b_{4}=10 $. Then $ lk(10)=C_{9}([2,b_{1},a_{1}],[a_{1},a_{5},a_{4}],[a_{4},5,4],[4,7,\allowbreak8]) $ which implies $ a_{1}=6 $. To exist $ lk(6) $, $ a_{5}=7 $ \& $ b_{1}=3 $. Now $ lk(6)=C_{9}([1,2,10],[10,a_{4},\allowbreak 7],[7,8,0],[0,4,5]) $ which implies $ [7,10] $ form non-edge in two 4-gon which is not possible.
	
	$ \boldsymbol{(a_{3},a_{4})=(11,10)}\Rightarrow b_{4}=7 $. Then $ lk(10)=C_{9}([2,b_{1},a_{1}],[a_{1},a_{5},11],[11,3,4],[4,7,\allowbreak 8]) $ which implies $ a_{1}=6\Rightarrow a_{5}=5,b_{1}=1 $.Therefore $ lk(6)=C_{9}([7,8,0],[0,4,5],[5,11,\allowbreak 10], [10,2,1]) $ and then $ [5,9,11] $ form a 3-gon which implies $ lk(5)=C_{9}([11,10,6],[6,0,\allowbreak 4],[4,7,2],[2,8,9])\Rightarrow a_{4}=2 $ and then $ lk(2) $ will not possible.
	
	\hspace{-0.6cm}\underline{\textbf{If} $ \boldsymbol{(a,b,c,d)=(3,5,10,4)} $} then $ lk(4)=C_{9}([3,a_{1},b_{1}],[b_{1},8,b_{1}],[b_{1},a_{2},5],[5,6,0]) $ and $ lk(5)\allowbreak =C_{9}([10,b_{2},4],[4,0,6],[6,a_{4},3],[3,9,8]) $. We see that $ [10,b_{2}] $ form an edge and a non-edge in two 4-gon which is not possible.
	
	\hspace{-0.6cm}\underline{\textbf{If} $ \boldsymbol{(a,b,c,d)=(2,11,10,4)\approx(2,10,11,4)} $} then $ lk(4)=C_{9}([3,a_{1},b_{1}],[b_{1},8,b_{2}],[b_{2},a_{2},\allowbreak5],[5,6,0]) $ where $ a_{1}\in\{9,11\},b_{1}\in\{7,10\} $.
	
	$ \boldsymbol{(a_{1},b_{1})=(9,7)}\Rightarrow b_{2}=10,a_{2}\in\{1,2\} $ and either $ lk(7)=C_{9}([a_{3},a_{4},9],[9,3,4],\allowbreak [4,10,8],[8,0,6]) $ which implies one of $ a_{3},a_{4} $ will be $ 11 $ as $ 11\notin V(lk(0)),V(lk(4)) $; but for any one, it is not possible as $ [9,11] $ is a non-edge in a 4-gon OR $ lk(7)=C_{9}([a_{3},a_{4},6],[6,0,8],\allowbreak [8,10,4],[4,3,9]) $ which implies $ a_{4}=11 $ and $ a_{3}\in\{1,2\} $.  For $ a_{3}=1 $, $ lk(9)=C_{9}([1,0,8],[8,\allowbreak11,2],[2,5,3],[3,4,7])  $ i.e. $ [2,3] $ form an edge and a non-edge in two 4-gon which is not possible. For $ a_{3}=2 $, $ lk(9)=C_{9}([7,4,3],[3,5,1],[1,0,8],[8,11,2])\Rightarrow lk(1)=C_{9}([6,11,2],[2,3,\allowbreak0],[0,8,9],[9,3,5]) $. We see that $ [2,6] $ form a non-edge in two 4-gon which is not possible.
	
	$ \boldsymbol{(a_{1},b_{1})=(9,10)} \Rightarrow b_{2}=7 $ and $ lk(9)=C_{9}([2,11,8],[8,0,1],[1,a_{3},3],[3,4,10]) $ which implies $ [5,6,7] $ form a 3-gon which is a contradiction.
	
	$ \boldsymbol{(a_{1},b_{1})=(11,7)}\Rightarrow b_{2}=10,a_{2}\in\{1,2,9\} $. Now $ lk(7)=C_{9}([a_{3},a_{4},11],[11,3,4],\allowbreak [4,10,8],[8,0,6]) $. Therefore $ a_{3}=1 $ as $ [2,11] $ is an edge in a 4-gon, so $ [2,5,9] $ form a 3-gon. From $ lk(1) $ we see that $ a_{4}=2 $, therefore $ lk(1)=C_{9}([6,10,9],[9,8,0],[0,3,2],[2,\allowbreak 11,7]) $, $ lk(6)=C_{9}([1,9,10],[10,a_{5},5],[5,4,0],[0,8,7]) $ which implies $ [5,10] $ form non-edge in two 4-gon which is not possible.
	
	$ \boldsymbol{(a_{1},b_{1})=(11,10)}$ then $ 11 $ will occur two times in $ lk(10) $ which is a contradiction.
	
	\hspace{-0.6cm}\underline{\textbf{If} $ \boldsymbol{(a,b,c,d)=(3,11,10,4)\approx(3,10,11,4)} $} then $ lk(3)=C_{9}([4,a_{1},b_{1}],[b_{1},8,b_{2}],[b_{2},a_{2},\allowbreak2],[2,1,0]) $ where $ a_{1}\in\{7,10\},b_{1}\in\{9,11\} $.
	
	$ \boldsymbol{(a_{1},b_{1})=(7,9)}\Rightarrow b_{2}=11 $ and $ a_{2}\in\{5,6\} $ that means $ 10\in V(lk(a')) $ for each $ a'\in V\setminus \{0,3\} $. Now either $ lk(9)=C_{9}([a_{3},a_{4},1],[1,0,8],[8,11,3],[3,\allowbreak 4,7]) $ or $ lk(9)=C_{9}([a_{3},a_{4},7],\allowbreak[7,\allowbreak4,3],[3,11,8],[8,0,1]) $. For $ lk(9)=C_{9}([a_{3},a_{4},1],[1,0,8],[8,\allowbreak 11,3],[3,4,7]) $, $ lk(7)=C_{9}(\allowbreak[a_{3},\allowbreak a_{5},6],[6,0,8],[8,10,4],[4,3,9]) $ and then $ a_{3}=5 $ which implies $ [5,6] $ form an edge and a non-edge in two 4-gon which is not possible. For $ lk(9)=C_{9}([a_{3},a_{4},7],[7,4,3],\allowbreak [3,11,8],[8,0,1]) $, $ a_{3}=5 $ which means $ a_{4}=10 $ and then $ [7,10] $ form an edge and non-edge in two 4-gon which is not possible.
	
	$ \boldsymbol{(a_{1},b_{1})=(7,11)}\Rightarrow b_{2}=11 $ and $ lk(7)=C_{9}([a_{3},a_{4},11],[11,3,4],[4,10,8],[8,0,\allowbreak 6]) $ and $ lk(11)=C_{9}([10,a_{5},a_{4}],[a_{4},a_{3},7],[7,4,3],[3,9,8]) $. Then $ a_{3}\in\{1,2\} $ and $ a_{4}\in\{5,12\} $. For $ a_{3}=1 $, from $ lk(1) $ we see that $ a_{4}=2 $. Therefore after complete $ lk(1) $ we see that $ 5 $ will occur two times in $ lk(6) $. For $ a_{3}=2 $ then after completing $ lk(2) $ we see that $ lk(5) $ is not possible.
	
	$ \boldsymbol{(a_{1},b_{1})=(10,9)}\Rightarrow b_{2}=11 $. Now $ lk(9)=C_{9}([a_{3},a_{4},10],[10,4,3],[3,11,8],[8,0,\allowbreak 1]) $ which implies $ a_{3}\in\{5,6\} $ and $ a_{4}\in\{2,5,6\} $. When $ a_{3}=5 $, $ [2,6,7] $ form a 3-gon. Then $ lk(10)=C_{9}([11,a_{5},a_{4}],[a_{4},5,9],[9,3,4],[4,7,8]) $ which implies $ a_{4}=6 $. Now $ lk(6)=C_{9}([2,11,10],[10,9,5],[5,4,8],[8,0,7]) $ which implies $ a_{2}=5 $ and then $ [2,11] $ form an edge and a non-edge in two 4-gon which is not possible. For $ a_{3}=6 $, $ [2,5,6] $ form a 3-gon and then $ lk(10)=C_{9}([11,a_{5},a_{4}],[a_{4},6,9],[9,3,4],[4,7,8]) $ which implies $ a_{4}=5 $. Therefore $ lk(5)=C_{9}([2,a_{6},4],[4,0,6],[6,9,10],[10,11,7])\Rightarrow a_{5}=7 $ that means $ 7 $ occur two times in $ lk(10) $ which is a contradiction.
	
	$ \boldsymbol{(a_{1},b_{1})=(10,11)}|\Rightarrow $ $ 10 $ will occur two times in $ lk(11) $ which is a contradiction.
	
	\hspace{-0.6cm}\underline{\textbf{If} $ \boldsymbol{(a,b,c,d)=(5,11,10,4)\approx(5,10,11,4)} $} then $ lk(4)=C_{9}([3,a_{1},b_{1}],[b_{1},8,b_{2}],[b_{2},a_{2},\allowbreak5],[5,6,0]) $ which implies $ a_{1}\in \{9,11\},b_{1}\in\{7,10\} $.
	
	$ \boldsymbol{(a_{1},b_{1})=(9,7)}\Rightarrow b_{2}=10 $ and $ a_{2}\in\{1,2\} $. For $ a_{2}=1 $, $ lk(10)=C_{9}([11,a_{3},a_{4}],[a_{4},\allowbreak a_{5},1],[1,5,4],[4,7,8]) $ which implies $ a_{4}=6 $ and then to complete $ lk(6) $, one of $ a_{3},a_{5} $ has to be $ 5 $ or $ 7 $ which is not possible. For $ a_{2}=2 $, $ lk(10)=C_{9}([11,a_{3},a_{4}],[a_{4},a_{5},2],[2,5,4],\allowbreak [4,7,8]) $ which $ a_{4}=6 $ and then to complete $ lk(6) $, one of $ a_{3},a_{5} $ has to be $ 5 $ or $ 7 $ which is not possible.
	
	$ \boldsymbol{(a_{1},b_{1})=(9,10)}\Rightarrow b_{2}=7 $. We have $ [3,9] $ \& $ [9,10] $ can not be an edge of a 3-gon, therefore $ lk(9) $ is not possible.
	
	$ \boldsymbol{(a_{1},b_{1})=(11,7)}\Rightarrow b_{2}=10,a_{2}\in\{1,2,9\} $. Now $ lk)7 =C_{9}([a_{3},a_{4},11],[11,3,4],[4,10,\allowbreak 8],[8,0,6])$ and $ lk(11)=C_{9}([10,a_{3},3],[3,4,7],[7,a_{5},5],[5,9,8]) $ which implies $ a_{10}=10 $ which is a contradiction.
	
	$ \boldsymbol{(a_{1},b_{1})=(11,10)}\Rightarrow 11 $ will occur two times in $ lk(10) $ which is a contradiction.
	
	\hspace{-0.6cm}\underline{\textbf{If} $ \boldsymbol{(a,b,c,d)=(4,10,2,5)} $} then $ lk(2)=C_{9}([10,a_{1},b_{1}],[b_{1},0,b_{2}],[b_{2},a_{2},5],[5,7,8]) $ which implies $ a_{1}\in\{6,11\},b_{1}\in\{1,3\} $.
	
	For $ \boldsymbol{(a_{1},b_{1})=(6,1)} $, to exist $ lk(5) $ and $ lk(6) $, $ [1,5,6] $ will have to be a 3-gon. But then $ lk(5) $ will not complete which is a contradiction.
	
	$ \boldsymbol{(a_{1},b_{1})=(6,3)}\Rightarrow a_{2}=11,b_{2}=1 $ and then $ 3 $ will occur two in two 3-gon which is not possible.
	
	$ \boldsymbol{(a_{1},b_{1})=(11,1)}\Rightarrow b_{2}=3 $ and either $ lk(1)=C_{9}([a_{3},a_{4},11],[11,10,2], [2,3,0],[0,8,\allowbreak 9]) $ or $ lk(1)=C_{9}([a_{3},a_{4},9],[9,8,0],[0,3,2],[2,10,11]) $. $ lk(1)=C_{9}([a_{3},a_{4},11],[11,10,2],\allowbreak [2,3,0],[0,8,9])\Rightarrow lk(9)=C_{9}([a_{3},a_{5},a_{6}],[a_{6},a_{7},4],[4,10,8],[8,0,1]) $. As $ [11,a_{3}] $ form an edge in a 4-gon and $ 11\notin V(lk(0)),V(lk(8)) $, therefore $ a_{7}=11 $ and then $ lk(4) $ is not possible. $ lk(1)=C_{9}([a_{3},a_{4},9],[9,8,0],[0,3,2],[2,10,11])\Rightarrow lk(9)=C_{9}([a_{5},a_{6},4],[4,10,8],[8,0,\allowbreak 1],[1,a_{3},a_{4}]), lk(4)=C_{9}([3,a_{7},a_{9}],[a_{9},8,a_{10}],[a_{10},a_{8},5],[5,6,0]) $ which implies $ a_{7},a_{8}\in\{7,\allowbreak 11\} $ and one of $ a_{7},a_{8} $ is $ a_{5} $. Therefore $ a_{5}=7; a_{6}\in\{3,5\};a_{3},a_{4}\in\{5,6\};a_{8}\in \{7,a_{2}\};a_{2}\in\{2,9\} $ and $ [5,6] $ can not be an edge in a 3-gon, therefore $ a_{8}=a_{2} $ which is a contradiction.
	
	$ \boldsymbol{(a_{1},b_{1})=(11,3)}\Rightarrow b_{2}=1,a_{2}\in\{4,6\} $ and either $ lk(1)=C_{9}([a_{3},a_{4},a_{2}],[a_{2},5,2],[2,\allowbreak 3,0],[0,8,9]) $ or $ lk(1)=C_{9}([a_{3},a_{4},9],[9,8,0],[0,3,2],[2,5,a_{2}]) $. $ lk(1)=C_{9}([a_{3},a_{4},a_{2}],\allowbreak [a_{2},5,\allowbreak2],[2,3,0],[0,8,9]) $ implies one of $ a_{3},a_{4} $ is $ 11 $ and $ a_{3}\in\{6,7,11\} $. If $ a_{2}=6 $ then $ a_{3}=11,a_{4}=10 $ which implies $ 10 $ occur in two 3-gon which is not possible. If $ a_{2}=4 $ then $ lk(4) $ is not possible.  $ lk(1)=C_{9}([a_{3},a_{4},9],[9,8,0],[0,3,2],[2,5,a_{2}]) $ implies $ lk(9)=C_{9}([a_{4},a_{3},1],[1,0,8],[8,\allowbreak10,4],[4,a_{5},a_{6}]) $ and one of $ a_{3},a_{4} $ is $ 11; a_{3},a_{4}\in\{6,7,11\};a_{2}=6 $. Therefore $ a_{3}=11,a_{4}=7 $ and $ a_{6}=5 $ which implies $ [4,5] $ form an edge and a non-edge in two 4-gon which is a contradiction.
	
	\hspace{-0.6cm}\underline{\textbf{If} $ \boldsymbol{(a,b,c,d)=(3,2,10,5)} $} then $ lk(2)=C_{9}([10,a_{1},a_{2}],[a_{2},a_{3},1],[1,0,3],[3,9,8]) $ which implies $ a_{1},a_{2}\in\{4,6,11\} $.
	
	$ \boldsymbol{a_{2}=4}\Rightarrow a_{1}=11 $ and then $ lk(4)=C_{9}([3,a_{4},11],[11,10,2],[2,1,5],[5,6,0]) $ which implies $ a_{3}=5,a_{4}\in\{7,9\} $. Now $ lk(10)=C_{9}([2,4,11],[11,a_{5},a_{6}],[a_{6},a_{7},5],[5,7,8]) \Rightarrow a_{7}\in\{1,6\}$. But $ a_{7}=1\Rightarrow [5,6,7]$ make a 3-gon, therefore $ a_{7}=1 $ and then $ lk(5)=C_{9}([7,8,10],[10,a_{6},6],[6,0,4],[4,2,1]) $. From here we see that $ a_{6}=11 $ and then $ 11 $ occur two times in $ lk(10) $ which is not possible.
	
	$ \boldsymbol{a_{2}=6}\Rightarrow a_{1}=11,a_{3}\in\{5,7\} $. For $ a_{3}=5 $, $ lk(6)=C_{9}([7,8,0],[0,4,5],[5,1,2],[2,10,\allowbreak 11]) $ and then $ [1,5,9] $ form a 3-gon which implies $ 9 $ occur two times in $ lk(1) $ which is a contradiction. For $ a_{3}=7 $, $ lk(6)=C_{9}([5,4,0],[0,8,7],[7,1,2],[2,10,11]) $ and then $ [1,7,9] $ form a 3-gon which implies $ 9 $ occur two times in $ lk(1) $ which is a contradiction.
	
	$ \boldsymbol{a_{2}=11}\Rightarrow a_{1}\in\{4,6\} $. For $ a_{1}=4 $, $ lk(4)=C_{9}([3,9,10],[10,2,11],[11,7,5],[5,6,0]) $ and $ lk(10)=C_{9}([8,7,5],[5,1,9],[9,3,4],[4,11,2]) $ and after that $ lk(9) $ will not possible. For $ a_{1}=6 $, $ lk(6)=C_{9}([5,4,0],[0,8,7],[7,a_{4},10],[10,2,11]) $ which implies $ [1,7,9] $ form a 3-gon. Then $ lk(7)=C_{9}([1,11,5],[5,10,8],[8,0,6],[6,10,9]) $ and $ lk(9)=C_{9}([7,6,10],[10,11,3],[3,\allowbreak2,8],\allowbreak[8,0,1]) $ which implies $ 11 $ occur two times in $ lk(10) $ which is a contradiction.
	
	\hspace{-0.6cm}\underline{\textbf{If} $ \boldsymbol{(a,b,c,d)=(11,2,10,5)\approx(10,2,11,5)} $} then $ lk(2)=C_{9}([10,a_{1},b_{1}],[b_{1},0,b_{2}],[b_{2},a_{2},\allowbreak11],[11,9,8]) $ which implies $ a_{1}\in\{4,6\},b_{1}\in\{1,3\} $.
	
	$ \boldsymbol{(a_{1},b_{1})=(4,1)} \Rightarrow b_{2}=3 $ and $ lk(4)=C_{9}([3,a_{3},10],[10,2,1],[1,a_{4},5],[5,6,0]) $ which implies $ a_{4}=11 $. Now $ lk(1)=C_{9}([9,8,0],[0,3,2],[2,10,4],[4,5,11]) $ and then $ [5,6,7] $ will form a 3-gon which is a contradiction.
	
	$ \boldsymbol{(a_{1},b_{1})=(4,3)}\Rightarrow $ $ 4 $ will occur two times in $ lk(3) $ which is a contradiction.
	
	$ \boldsymbol{(a_{1},b_{1})=(6,1)}\Rightarrow a_{2}=4,b_{2}=3 $ and $ lk(6)=C_{9}([5,4,0],[0,8,7],[7,11,10],[10,2,1]) $. From here we see that $ [7,10] $ form non-edge in two 4-gon which is a contradiction.
	
	$ \boldsymbol{(a_{1},b_{1})=(6,3)} $ then $ lk(6) $ is not possible as for that either $ [6,3] $ or $ [6,10] $ have to form an edge in a 3-gon which is not possible.
	
	\hspace{-0.6cm}\underline{\textbf{If} $ \boldsymbol{(a,b,c,d)=(2,11,10,5)\approx(2,10,11,5)} $} then $ (0,8)(1,9)(3,11)(4,10)(6,7):(2,11,\allowbreak10,5)\approx (2,3,4,\allowbreak5) $ and $ (a,b,c,d)=(2,3,4,5) $ is not possible.
	
	\hspace{-0.6cm}\underline{\textbf{If} $ \boldsymbol{(a,b,c,d)=(3,11,10,5)\approx(3,10,11,5)} $} then $ lk(3)=C_{9}([4,a_{1},b_{1}],[b_{1},8,b_{2}],[b_{2},a_{2},\allowbreak2],[2,1,0]) $ which implies $ a_{1}\in\{7,10\},b_{1}\in\{9,11\} $.
	
	$ \boldsymbol{(a_{1},b_{1})=(7,9)} $ then $ lk(7) $ is not possible.
	
	$ \boldsymbol{(a_{1},b_{1})=(7,11)} $, similar to above.
	
	$ \boldsymbol{(a_{1},b_{1})=(10,9)}\Rightarrow b_{2}=11 $ and $ lk(9)=C_{9}([a_{3},a_{4},10],[10,4,3],[3,11,8],[8,0,1]) $. Now $ lk(10)=C_{9}([11,a_{5},4],[4,3,9],[9,a_{3},5],[5,7,8]) $ which implies $ a_{4}=5,a_{3}=6,a_{5}\in\{1,2\} $ and $ [2,5,7] $ form a 3-gon. Therefore $ lk(5)=C_{9}([2,a_{6},4],[4,0,6],[6,9,10],[10,8,\allowbreak 7]) $. We see that $ [4,11] $ is an edge in a 4-gon, therefore $ a_{6}\neq 11 $ i.e. $ a_{6}=1 $ and then from $ lk(1) $ we see that $ [4,6] $ form non-edge in two 4-gon which is a contradiction.
	
	$ \boldsymbol{(a_{1},b_{1})=(10,11)} $ implies that $ 10 $ will occur two times in $ lk(11) $ which is not possible.
	
	\hspace{-0.6cm}\underline{\textbf{If} $ \boldsymbol{(a,b,c,d)=(4,11,10,5)\approx(4,10,11,5)} $} then $ lk(4)=C_{9}([3,a_{1},b_{1}],[b_{1},8,b_{2}],[b_{2},a_{2},\allowbreak5],[5,6,0]) $ which implies $ a_{1}\in\{7,10\},b_{1}\in\{9,11\} $.
	
	$ \boldsymbol{(a_{1},b_{1})=(7,9)}\Rightarrow b_{2}=11 $. Now $ lk(9)=C_{9}([a_{3},a_{4},1],[1,0,8],[8,11,4],[4,3,7]) $. As $ a_{4}\neq 10 $, therefore $ a_{4}=10 $ and then $ a_{3}=6 $. Therefore $ [1,2,5] $ form a 3-gon. Now from $ lk(1) $ we see that $ [5,10] $ from a non-edge and we have $ [5,10] $ is an edge, which is a contradiction.
	
	$ \boldsymbol{(a_{1},b_{1})=(7,11)}\Rightarrow b_{2}=9 $. We see that to complete $ lk(7) $, either $ [7,11] $ or $ [7,3] $ has to be an edge in a 3-gon which is not possible.
	
	$ \boldsymbol{(a_{1},b_{1})=(10,9)}\Rightarrow b_{2}=11 $ and $ lk(9)=C_{9}([a_{3},a_{4},10],[10,3,4],[4,11,8],[8,0,1]) $ which implies $ a_{3}=6 $ and then $ [2,5,7] $ form a 3-gon. Now $ lk(6)=C_{9}([1,a_{5},7],[7,8,0],\allowbreak [0,4,5],[5,10,9]) $ which implies $ a_{4}=5,a_{5}=11 $. Now $ lk(7)=C_{9}([2,a_{6},11],[11,1,6],[6,\allowbreak 0,8],[8,10,5]) $ and $ lk(5)=C_{9}([2,11,4],[4,0,6],[6,9,10],[10,8,7]) $ which implies $ a_{2}=2 $ and then we see that $ [2,11] $ form an edge and a non-edge in two 4-gon which is a contradiction.
	
	$ \boldsymbol{(a_{1},b_{1})=(10,11)} $ implies $ 10 $ occur two times in $ lk(11) $ which is not possible.
	
	\hspace{-0.6cm}\underline{\textbf{If} $ \boldsymbol{(a,b,c,d)=(3,5,2,10)} $} then $ lk(3)=C_{9}([4,a_{1},9],[9,8,5],[5,a_{2},2],[2,1,0]) $. We see that $ [2,5] $ form an edge and a non-edge in two 4-gon which is not possible.
	
	\hspace{-0.6cm}\underline{\textbf{If} $ \boldsymbol{(a,b,c,d)=(4,5,2,10)} $} then $ lk(2)=C_{9}([5,a_{1},b_{1}],[b_{1},0,b_{2}],[b_{2},a_{2},10],[10,7,8]) $ and $ lk(5)=C_{9}([2,b_{1},a_{1}],[a_{1},a_{3},6],[6,0,4],[4,9,8]) $ which implies $ a_{1}=11,b_{1}\in\{1,3\} $.
	
	$ \boldsymbol{b_{1}=1} $ then $ b_{2}=3 $ and $ a_{2}\in\{6,9\} $. For $ a_{2}=6 $, $ lk(6)=C_{9}([7,8,0],[0,4,5],[5,11,\allowbreak 3],[3,2,10]) $ and then $ [1,9,11] $ form a 3-gon which implies $ 9 $ will occur two times in $ lk(1) $ which is not possible. If $ a_{2}=9 $ then $ lk(9)=C_{9}([1,0,8],[8,5,4],[4,11,3],[3,2,10]) $ and then $ [3,4] $ form an edge and a non-edge in two 4-gon which is a contradiction.
	
	$ \boldsymbol{b_{1}=3} $ then $ b_{2}=1 $ and $ a_{2}\in\{4,6\} $. If $ a_{2}=4 $ then $ lk(4) $ is not possible and if $ a_{3}=6 $ then $ a_{3}\in\{1,10\} $. If $ a_{3}=1 $ then $ lk(6)=C_{9}([7,8,0],[0,4,5],[5,11,1],[1,2,10]) $ which implies $ 2 $ occur two times in $ lk(10) $. If $ a_{3}=10 $ then $ lk(6)=C_{9}([7,8,0],[0,4,5],[5,11,10],[10,2,\allowbreak 1]) $ and then $ [9,10,11] $ form a 3-gon. Now $ lk(10)=C_{9}([9,a_{4},7],[7,8,2],[2,1,6],[6,5,11]) $ implies $ a_{4}=3 $ an $ 3\notin V(lk(8)),V(lk(6)) $. Now after completing $ lk(7) $ and $ lk(11) $ we see that $ [1,9] $ form non-edge in a 4-gon and we have $ [1,9] $ is an edge in a 4-gon which is a contradiction.
	\begin{sloppypar}
	\hspace{-0.6cm}\underline{\textbf{If} $ \boldsymbol{(a,b,c,d)=(3,2,5,10)} $} then $ lk(5)=C_{9}([2,11,b_{1}],[b_{1},0,b_{2}],[b_{2},a_{1},10],[10,7,8]) $ which implies $ b_{1}\in\{4,6\} $.
	\end{sloppypar}
	$ \boldsymbol{b_{1}=4} $ then $ lk(2)=C_{9}([8,9,3],[3,0,1],[1,a_{2},11],[11,4,5]) $. Now either $ lk(1)=C_{9}([a_{3},\allowbreak a_{4},a_{2}],[a_{2},11,2],[2,3,0],[0,8,1]) $ or $ lk(1)=C_{8}([a_{3},a_{4},9],[9,8,0],[0,3,2],[2,11,a_{2}]) $. For  $ lk(1)\allowbreak=C_{9}([a_{3},a_{4},a_{2}],[a_{2},11,2],[2,3,0],[0,8,1]) $, $ a_{2}=10 $ as for $ a_{2}\in\{6,7\} $, $ [6,7] $ will form a non-edge which will not possible. Therefore $ a_{4}\in\{7,a_{1}\} $. But $ a_{4}\neq a_{1} $, so $ a_{4}=7,a_{3}=6 $ and then $ [7,10,11] $ form a 3-gon which is a contradiction. For $ lk(1)=C_{8}([a_{3},a_{4},9],[9,8,0],[0,3,2],\allowbreak[2,11,a_{2}]) $, $ [9,11] $ form a edge in a 3-gon and then $ lk(9)=C_{9}([11,a_{5},3],[3,2,8],[8,0,1],[1,a_{3},\allowbreak a_{4}]) $. From here we see that $ a_{5}=4 $ and then $ 2 $ occur two times in $ lk(3) $ which is a contradiction.
	
	$ \boldsymbol{b_{1}=6}\Rightarrow b_{2}=4,a_{1}\in\{1,9\} $ and $ lk(2)=C_{9}([8,9,3],[3,0,1],[1,a_{2},11],[11,6,5]) $ which implies $ a_{2}\in\{4,7,10\} $. For $ a_{2}=4 $, $ lk(4)=C_{9}([3,7,11],[11,2,a_{1}],[a_{1},10,5],[5,6,0]) $ which implies $ a_{1}=1 $. Now $ lk(1)=C_{9}([10,5,4],[4,11,2],[2,3,0],[0,8,9]) $. Therefore $ [6,7,11] $ form a 3-gon and $ lk(7)=C_{9}([6,0,8],[8,5,10],[10,9,3],[3,4,11]) $ and then $ 3 $ will occur two times in $ lk(4) $ which is not possible. For $ a_{4}\neq 4 $ then $ lk(4)=C_{9}([3,a_{3},a_{4}],[a_{4},a_{5},a_{1}],[a_{1},10,5],\allowbreak [5,6,0]) $. We see that $ a_{4}=11 $ is not possible as $ [1,11] $ is a non-edge in a 4-gon, therefore $ a_{3}=11 $ and then $ a_{4}=7 $. Hence $ a_{5}\in\{6,10\} $ which is not possible.
	
	\hspace{-0.6cm}\underline{\textbf{If} $ \boldsymbol{(a,b,c,d)=(4,2,5,10)} $} then $ lk(2)=C_{9}([5,a_{1},3],[3,0,1],[1,a_{2},4],[4,9,8]) $, $ lk(5)=C_{9}([2,3,a_{1}],[a_{1},0,4],[4,11,10],[10,7,8]) $ which implies $ a_{1}=6,a_{2}=11 $ and then $ lk(4) $ will not possible.
	
	\hspace{-0.6cm}\underline{\textbf{If} $ \boldsymbol{(a,b,c,d)=(2,11,5,10)\approx(2,10,5,11)} $}  then $ lk(5)=C_{9}([11,a_{1},b_{1}],[b_{1},0,b_{2}],[b_{2},\allowbreak a_{2},10],[10,7,8]) $ which implies $ a_{1}\in\{1,3\},b_{1}\in\{4,6\} $.
	
	$ \boldsymbol{(a_{1},b_{1})=(1,4)}\Rightarrow b_{2}=6 $. We see that to complete $ lk(1) $, either $ [1,11] $ or $ [1,4] $ have to be an edge in a 3-gon which is not possible.
	
	$ \boldsymbol{(a_{1},b_{1})=(1,6)}\Rightarrow b_{2}=4 $ and $ lk(1)=C_{9}([9,8,0],[0,3,2],[2,a_{3},11],[11,5,6]) $ which implies $ [2,11] $ form an edge and a non-edge in two 4-gon which is not possible.
	
	$ \boldsymbol{(a_{1},b_{1})=(3,4)} $ implies $ 3 $ will occur two times in $ lk(4) $ which is a contradiction.
	
	$ \boldsymbol{(a_{1},b_{1})=(3,6)}\Rightarrow b_{2}=4 $ and $ lk(6)=C_{9}([a_{3},a_{4},3],[3,11,5],[5,4,0],[0,8,7]) $. From here we see that $ a_{3}=9 $ and $ [1,2,10] $ form a 3-gon. Now $ lk(9)=C_{9}([7,a_{5},1],[1,0,8],[8,\allowbreak 11,a_{4}],[a_{4},3,6]) $ which implies $ a_{4}=2 $. So $ lk(2)=C_{9}([10,a_{6},11],[11,8,9],[9,6,3],[3,0,\allowbreak 1]) $ and $ lk(1)=C_{9}([10,a_{7},4],[4,7,9],[9,8,0],[0,3,2]) $ which implies $ lk(4) $ is not possible.

	\hspace{-0.6cm}\underline{\textbf{If} $ \boldsymbol{(a,b,c,d)=(4,2,11,10)\approx(4,2,10,11)} $} then $ lk(2)=C_{9}([11,6,3],[3,0,1],[1,5,4],\allowbreak[4,9,8]) $, $ lk(4) =C_{9}([3,10,9],[9,8,2],[2,1,5],[5,6,0]) $ and $ lk(3)=C_{9}([4,9,a_{1}],[a_{1},a_{2},6],[6,\allowbreak11,2],[2,1,0]) $.  From here we see that $ a_{1}=10 $ and $ a_{2}=7 $ as $ 7\notin lk(2),lk(4) $. Now $ lk(6)=C_{9}([11,2,3],[3,\allowbreak10,7],\allowbreak [7,8,0],[0,4,5]) $, this implies $ 10 $ will occur two times in $ lk(7) $ which is a contradiction.
	
	\hspace{-0.6cm}\underline{\textbf{If} $ \boldsymbol{(a,b,c,d)=(5,2,11,10)\approx(5,2,10,11)} $} then $ lk(5)=C_{9}([11,a_{1},b_{1}],[b_{1},0,b_{2}],[b_{2},\allowbreak a_{2},2],[2,9,8]) $ which implies $ a_{1}\in\{1,3\},b_{1}\in\{4,6\} $.
	
	$ \boldsymbol{(a_{1},b_{1})=(1,4)} $, to complete $ lk(1) $, either $ [1,11] $ or $ [1,4] $ have to be an edge in a 4-gon which is not possible.
	
	$ \boldsymbol{(a_{1},b_{1})=(1,6)}\Rightarrow b_{2}=6 $ and $ lk(6)=C_{9}([a_{3},a_{4},7],[7,8,0],[0,4,5],[5,11,1]) $ and $ lk(7)=C_{9}([a_{5},a_{6},a_{4}],[a_{4},a_{3},6],[6,0,8],[8,11,10]) $. As $ 3\notin lk(5),lk(8) $, therefore $ a_{4}=3 $. Now we see that $ a_{5}=9,a_{3}=2,a_{6}=5 $ which implies $ [5,9] $ form an edge and a non-edge in two 4-gon which is a contradiction.
	
	$ \boldsymbol{(a_{1},b_{1})=(3,4)} $ implies $ 3 $ occur two times in $ lk(4) $ which is a contradiction.
	
	$ \boldsymbol{(a_{1},b_{1})=(3,6)}\Rightarrow b_{2}=4 $, $ lk(3)=C_{9}([4,a_{3},11],[11,5,6],[6,a_{4},2],[2,1,0]) $, $ lk(6)=C_{9}([a_{4},2,3],[3,11,5],[5,4,0],[0,8,7]) $ which implies $ [1,2,10] $ form a 3-gon and $ a_{4}=9 $. Now after completing $ lk(2) $ and $ lk(9) $, from $ lk(1) $ we see that $ [10.11] $ form a non-edge and we have $ [10,11] $ is an edge in a 4-gon which is a contradiction.

\hspace{-0.6cm}\underline{\textbf{If} $ \boldsymbol{(a,b,c,d)=(3,5,10,11)} $} then $ lk(3)=C_{8}([4,a_1,9],[9,8,5],[5,a_2,2],[2,1,0]) $ and $ lk(5)\allowbreak=C_{8}([10,a_3,4],[4,0,6],[6,2,3],[3,9,8]) $ which implies $ a_2=6 $ and from $ lk(10) $ we get $ a_3=1 $, therefore $ 11\notin V(lk(0)),V(lk(3)) $ i.e. $ a_1=11 $ and therefore $ V(lk(7))=\{0,1,2,4,6,8,9,\allowbreak10,11\} $. Now we have faces $ [0,1,2,3], [0,1,9,8],[1,4,5,10] $, therefore $ [1,10] $ has to be an edge of a 3-gon, which is a contradiction.

\hspace{-0.6cm}\underline{\textbf{If} $ \boldsymbol{(a,b,c,d)=(3,10,5,11)} $} then $ lk(3)=C_{8}([4,a_1,b_1],[b_1,8,b_2],[b_2,a_2,2],[2,1,0]) $ where $ \{b_1,b_2\}=\{9,10\} $ and $ a_1\in \{7,11\} $. If $ a_1=7 $ then as we have faces $ [5,11,7,8], [0,6,7,8] $ and $ [3,4,7,b_1] $, therefore $ [7,b_1] $ will form an edge in a 3-gon. Therefore $ b_1=9,b_2=10 $ and then $ lk(7)=C_8([6,0,8],[8,5,11],[11,a_3,4],[4,3,9]) $ and $ lk(9)=C_{8}([6,a_4,1],[1,0,8],[8,10,3],[3,\allowbreak4,7]) $ which implies $ [1,2,11] $ is a face and then $ a_4\neq 11 $ i.e. $ 11\notin V(lk(0)),V(lk(4)) $, therefore $ a_2=11, a_4=5 $ which implies $ C(4,0,8,7,9,1,5)\in lk(6) $, which make contradiction. Therefore $ a_1=11 $ which implies $ a_2\in \{6,7\} $ and $ lk(5)=C_{8}([10,a_3,4],[4,0,6],[6,a_4,11],[11,\allowbreak7,8]) $. Now, from $ lk(4) $, we get $ b_1=9,b_2=10 $ and then $ lk(4)=C_{8}([3,9,11],[11,a_5,a_3],[a_3,\allowbreak10,5],[5,6,0]) $ which implies $ a_3=1 $ and then $ [1,10] $ has be an edge of a 3-gon, is a contradiction.
\begin{sloppypar}
\hspace{-0.6cm}\underline{\textbf{If} $ \boldsymbol{(a,b,c,d)=(4,5,10,2)} $} then $ lk(4)=C_{8}([3,a_1,a_2],[a_2,a_3,9],[9,8,5],[5,6,0]) $ and $ lk(5)\allowbreak=C_{8}([10,a_4,a_5],[a_5,a_6,6],[6,0,4],[4,9,8]) $ where $ a_1,a_2\in\{7,10,11\} $. If $ \boldsymbol{a_2=7} $ then from $ lk(7) $ we get $ a_3=2 $ and $ [6,7,a_1] $ is a face which implies $ a_1=11 $. Now $ lk(	7)=C_{8}([11,3,4],[4,9,2],[2,10,8],[8,0,6]) $ and $ lk(6)=C_{8}([11,a_7,a_6],[a_6,a_5,5],[5,4,0],[0,8,7]) $ which implies $ a_4=11,a_7=10 $ and then $ lk(10)=C_{8}([8,a_8,a_9],[a_9,a_10,a_6],[a_6,6,11],[11,\allowbreak a_5,5]) $ which implies $ a_5,a_6\in \{1,2,3\} $ i.e. either $ \{a_5,a_6\}=\{1,2\} $ or $ \{a_5,a_6\}=\{2,3\} $ as $ [1,3] $ is a non-edge of a 3-gon. Now if $ \{a_5,a_6\}=\{2,3\} $ then $ 1\notin V(lk(5)),V(lk(6)),V(lk(7)) $ and if $ \{a_5,a_6\}=\{1,2\} $ then $ 3\notin V(lk(5)),V(lk(6)),V(lk(8)) $, which are not possible. If $ \boldsymbol{a_2=10} $ then from $ lk(10) $, we get $ a_3=2,a_1=11 $. Therefore $ lk(10)=C_{8}([5,a_5,11],[11,3,\allowbreak4],[4,9,2],[2,7,8]) $ which implies $ a_5=1 $ and then $ [1,11] $ has to be an edge of a 3-gon. Now from $ lk(1) $, we get $ a_6=2 $. Therefore $ lk(1)=C_{8}([9,8,0],[0,3,2],[2,6,5],[5,10,11]) $ which implies $ [2,6,7] $ is a face and then $ C(2,10,8,0,6,7)\in lk(7) $, which is a contradiction. Therefore $ \boldsymbol{a_2=11} $. Now if $ a_1=7 $ then $ [7,11] $ has to be an edge of a 3-gon and either $ [6,7] $ or $ [2,7] $ will be an edge of a 3-gon. If $ [6,7] $ is an edge of a 3-gon, then $ [1,2,9] $ will be a face which implies $ C(2,3,0,8,9,1)\in lk(9) $, is a contradiction, therefore $ [2,7,11] $ and $ [1,6,9] $ form faces and then $ lk(7)=C_{8}([2,10,8],[8,0,6],[6,9,4],[4,3,11]) $ which implies $ a_3=6 $, which is not possible, therefore $ a_1=10 $. Now $ lk(10)=C_{8}([8,7,2],[2,a_7,11],[11,4,a_4],[a_4,a_5,5]) $ which implies $ a_4=3,a_6=11 $. Now $ V(lk(11))=\{2,3,4,5,6,9,10,a_5,a_7\} \Rightarrow \{a_5,a_7\}=\{1,7\} $, is a contradiction.
\end{sloppypar}
\hspace{-0.6cm}\underline{\textbf{If} $ \boldsymbol{(a,b,c,d)=(4,5,10,11)} $} then $ lk(4)=C_{8}([3,a_1,a_2],[a_2,a_3,9],[9,8,5],[5,6,0]) $ and $ lk(5)=C_{8}([10,a_4,a_5],[a_5,a_6,6],[6,0,4],[4,9,8]) $ where $ a_1,a_2\in\{7,10,11\} $. If $ \boldsymbol{a_2=7} $ then one of $ a_1,a_3 $ is $ 11 $ and $ [6,7] $ is an edge of a 3-gon. If $ a_3=11 $ then $ a_1\neq 10 $, is a contradiction, therefore $ a_1=11 $ and then $ [6,7,a_3] $ is a face which implies $ [2,6,7] $ and $ [1,9,11] $ form faces which implies $ a_3=2 $. Now $ lk(9)=C_{8}([11,10,2],[2,7,4],[4,5,8],[8,0,1]) $ and then $ lk(2) $ will not possible. If $ \boldsymbol{a_2=10} $ then from $ lk(10) $ we get one of $ a_1,a_3 $ is 11 i.e. $ \{a_1,a_3\}=\{11,a_4\} $. If $ a_3=11 $ then $ a_1=a_4=7 $ which is a contradiction as $ [7,10] $ form edge and non-edge in two 4-gon. Therefore $ a_1=11 $, $ a_3=a_4=2 $ and then $ lk(10)=C_{8}([8,7,11],[11,3,4],[4,9,2],[2,a_5,5]) $ which implies $ a_5=1 $ and then $ lk(1)=C_{8}([9,8,0],[0,3,2],[2,10,5],[5,6,a_6]) $. From here we get $ a_6=11 $ which implies $ [2,6,7] $ is a face and then $ 7\notin V(lk(1)),V(lk(4)),V(lk(5)) $, is a contradiction. Therefore $ \boldsymbol{a_2=11} $. Now if $ a_1=7 $ then $ [6,7] $ will be an edge of a 3-gon. Therefore $ lk(7)=C_{8}([6,0,8],[8,10,11],[11,4,3],[3,b_1,b_2]) $ which implies $ b_2=9 $ and therefore $ [1,2,11] $ form a face. Now from $ lk(9) $, we get $ b_1=1 $, which is not possible from $ lk(3) $. Therefore $ a_1=10 $ and then $ lk(10)=C_{8}([5,a_5,a_4],[a_4,a_7,3],[3,4,11],[11,7,8]) $ which implies $ a_4=9 $ and then $ [5,9] $ form non-edge in two 4-gon, a contradiction.

\hspace{-0.6cm}\underline{\textbf{If} $ \boldsymbol{(a,b,c,d)=(4,10,5,11)} $} then $ lk(4)=C_{8}([3,a_1,10],[10,8,9],[9,a_2,5],[5,6,0]) $ which implies $ lk(5)=C_{8}([10,a_3,6],[6,0,4],[4,9,11],[11,7,8]) $ i.e. $ a_2=11,a_1=7 $. From $ lk(7) $, we see that $ [7,10] $ is an edge of a 3-gon, a contradiction.

\hspace{-0.6cm}\underline{\textbf{If} $ \boldsymbol{(a,b,c,d)=(5,2,10,3)} $} then $ lk(3)=C_{8}([4,a_1,10],[10,8,7],[7,a_2,2],[2,1,0]) $ which implies $ lk(2)=C_{8}([10,a_3,1],[1,0,3],[3,7,5],[5,9,8]) $. From here we see that $ a_2=5,a_1=11=a_3 $ which implies $ C(1,2,8,7,3,4,11)\in lk(10) $, a contradiction.

\hspace{-0.6cm}\underline{\textbf{If} $ \boldsymbol{(a,b,c,d)=(5,10,2,4)} $} then $ lk(2)=C_{8}([10,a_1,3],[3,0,1],[1,a_2,4],[4,7,8]) $ and it implies $ lk(4)=C_{8}([3,a_3,7],[7,8,2],[2,1,5],[5,6,0]) $. Therefore $ a_2=5,a_1=11=a_3 $ which implies $ V(lk(10))=\{1,2,3,5,6,7,8,9,11\} $. Now from $ lk(5) $, we get either $ [1,5,9] $  or $ [5,6,9] $ is a face. But if $ [1,5,9] $ is a face then $ C(1,0,8,10,5)\in lk(9) $, a contradiction. Therefore $ [5,6,9] $ is a face and then $ lk(5)=C_{8}([9,8,10],[10,11,1],[1,2,4],[4,0,6]) $ which implies $ C(1,5,9,8,2,3,11)\in lk(10) $ a contradiction.
\begin{sloppypar}
\hspace{-0.6cm}\underline{\textbf{If} $ \boldsymbol{(a,b,c,d)=(10,2,11,3)} $} then $ lk(2)=C_{8}([11,a_1,1],[1,0,3],[3,a_2,10],[10,9,8]) $ which implies $ lk(3)=C_{8}([4,a_4,11],[11,8,7],[7,10,2],[2,1,0]) $. Therefore $ a_2=7,a_3=9 $ which implies one of $ [4,9] $ and $ [9,11] $ has to be an edge of a 3-gon, a contradiction.
\end{sloppypar}
\hspace{-0.6cm}\underline{\textbf{If} $ \boldsymbol{(a,b,c,d)=(10,2,11,4)} $} then $ lk(2)=C_{8}([11,a_1,b_1],[b_1,0,b_2],[b_2,a_2,10],[10,9,8]) $ and $ lk(4)=C_{8}([3,a_3,b_3],[b_3,8,b_4],[b_4,a_4,5],[5,6,0]) $ where $ \{b_1,b_2\}=\{1,3\} $, $ \{b_3,b_4\}=\{7,11\} $, $ a_1\in\{5,6\} $ and $ a_3\in \{9,10\} $. If $ \boldsymbol{(b_1,b_2)=(1,3)} $ then $ a_3=9 $ which implies $ [9,b_3] $ will be an edge of a 3-gon and then $ (b_3,b_4)=(7,11) $. Now $ lk(9)=C_{8}([10,2,8],[8,0,1],[1,\allowbreak a_5,3],[3,4,7]) $ and $ lk(7)=C_{8}([10,1,6],[6,0,8],[8,11,4],[4,3,9]) $. From here we get, $ [1,5,6] $ is a face which implies $ lk(1) $ is not possible. Therefore $ \boldsymbol{(b_1,b_2)=(3,1)} $ which implies $ (b_3,b_4)=(7,11) $ as $ [3,11] $ can not be non-edge in two 4-gon. Now if $ a_3=9 $ then from $ lk(9) $, we get $ [1,7,9] $ is a face, therefore $ [5,6,10] $ is a face. Now $ lk(7)=C_{8}([1,a_5,6],[6,0,8],[8,11,\allowbreak4],[4,3,9]) $ which implies $ a_5=2 $, which contradict to $ lk(2) $. Therefore $ a_3=10 $ and then $ a_4\in\{1,9\} $. If $ a_4=1 $ then from $ lk(1) $, we get $ [1,5,9] $ and $ [6,7,10] $ form faces. Which implies $ C(3,10,6,0,8,11,4)\in lk(7) $, a contradiction. Therefore $ a_4=9 $ and then from $ lk(9) $, we get either $ [1,5,9] $ or $ [5,9,10] $ is a face. If $ [1,5,9] $ is a face, then $ [6,7,10] $ will form a face which implies $ C(3,10,6,0,8,11,4)\in lk(7) $, a contradiction. Therefore $ [5,9,10] $ and $ [1,6,7] $ form faces. Now $ lk(9)=C_{8}([10,2,8],[8,0,1],[1,a_5,11],11,4,5) $, $ lk(1)=C_{8}([7,10,2],[2,3,0],[0,8,9],[9,11,6]) $ which implies $ a_2=7,a_5=6 $ and then $ lk(7)=C_{8}([1,2,10],[10,3,4],[4,11,8],[8,0,6]) $. From here we see that $ 5\notin V(lk(1)),V(lk(7)),\allowbreak V(lk(8)) $, a contradiction.

\hspace{-0.6cm}\underline{\textbf{If} $ \boldsymbol{(a,b,c,d)=(10,11,2,3)} $} then $ lk(2)=C_{8}([11,a_1,a_2],[a_2,a_3,1],[1,0,3],[3,7,8]) $ and $ lk(3)=C_{8}([4,a_4,a_5],[a_5,a_6,7],[7,8,2],[2,1,0]) $ where $ a_2\in\{4,5,6\} $ and $ a_5\in\{9,10,11\} $. If $ \boldsymbol{a_5=9} $ then from $ lk(9) $ we get one of $ a_4,a_6 $ must be 10 and $ [1,9] $ is an edge of a 3-gon. Therefore either $ [1,5,9] $ or $ [1,6,9] $ is a face. If $ [1,5,9] $ is a face then $ [6,7,10] $ will form a face and then $ lk(7)=C_{8}([10,a_7,a_7],[a_6,9,3],[3,2,8],[8,0,6]) $ which implies $ a_4=10 $. Therefore $ lk(9)=C_{8}([5,7,3],[3,4,10],[10,11,8],[8,0,1]) $ which implies $ a_6=5 $. Now $ lk(10)=C_{8}([7,5,4],[4,3,9],[9,8,11],[11,a_8,6]) $ which implies $ C(5,6,0,3,9,10,7)\in lk(4) $, a contradiction. If $ [1,6,9] $ is a face then $ [5,7,10] $ will form a face and then $ lk(7)=C_{8}([10,a_7,6],[6,0,8],[8,2,3],[3,9,5]) $ which implies $ a_6=5,a_4=10 $. Now $ lk(9)=C_{8}([1,0,\allowbreak8],[8,11,10],[10,4,3],[3,7,6]) $ and then $ lk(6) $ will not possible. If $ \boldsymbol{a_5=11} $ then from $ lk(11) $, we get $ \{a_4,a_6\}=\{10,a_1\} $. Now if $ a_1=a_4 $ then $ a_1\notin V $, therefore $ (a_4,a_6)=(10,a_1) $, which implies $ lk(11)=C_{8}([8,9,10],[10,4,3],[3,7,a_1=a_6],[a_1,a_2,2]) $and $ a_1\in\{5,6\} $. But if $ a_1=6 $ then $ C(0,8,2,3,11,6)\in lk(7) $, a contradiction. Therefore $ a_1=a_6=5 $ which implies $ [5,6] $ is an edge of a 3-gon and either $ [5,6,7] $ or $ [5,6,a_2] $ is a face. But $ [5,6,a_2] $ can not be a face as $ a_2\in\{4,5,6\} $ and $ [5,6,7] $ is also cannot be a face as then $ C(3,11,5,6,0,8,2)\in lk(3) $. Therefore $ \boldsymbol{a_5=10} $. Now from $ lk(3) $ and $ lk(4) $, we get $ a_4\in\{9,11\} $. If $ a_4=9 $ then from $ lk(9) $, we get $ [1,9] $ is an edge of a 3-gon. Therefore either $ [1,5,9] $ or $ [1,6,9] $ is a face. In both cases, $ [7,10] $ will be an edge of a 3-gon, which is a contradiction as $ [7,10] $ is a non-edge of a 4-gon. Therefore $ a_4=11 $ which implies $ a_6=5 $. Now $ lk(11)=C_{8}([8,9,10],[10,3,4],[4,a_7,a_1],[a_1,a_2,2]) $ which implies $ a_1\notin V $.

\begin{sloppypar}
\hspace{-0.6cm}\underline{\textbf{If} $ \boldsymbol{(a,b,c,d)=(10,11,2,5)} $} then $ lk(2)=C_{8}([11,a_1,b_1],[b_10,b_2],[b_2,a_2,5],[5,7,8]) $ where $ \{b_1,b_2\}=\{1,3\} $. If $ \boldsymbol{(b_1,b_2)=(1,3)} $ then $ a_1\in\{4,6\} $ and either $ a_2=6 $ or $ a_2\neq 6 $. If $ a_2=6 $ then $ a_1=4 $ and $ [5,7],[1,9] $ are edges of 3-gon which implies $ [5,7,10] $ and $ [1,6,9] $ are faces. From here we see that $ [4,6] $ is a non-edge of two 4-gon, a contradiction, therefore $ a_2\neq 6 $. Now either $ [5,6,7] $ or $ [5,6,a_2] $ is a face. But $ [5,6,7] $ can not be a face as then $ C(5,2,8,0,6)\in lk(7) $, therefore $ [5,6,a_2] $ is a face. Now $ lk(5)=C_{8}([a_2,3,2],[2,8,7],[7,a_3,4],[4,0,6]) $ which implies $ a_2\in\{9,10\} $. Now if $ a_2=9 $ then $ lk(9) $ will not possible. Therefore $ a_2=10 $ which implies $ [1,7,9] $ is a face and $ a_3\in\{1,9\} $. If $ a_3=1 $ then $ V(lk(9))=\{0,1,3,4,6,7,8,10,11\} $ which implies $ a_6=6 $ and $ lk(7)=C_{8}([9,a_4,6],[6,0,8],[8,2,5],[5,4,1]) $, $ lk(9)=C_{8}([1,0,8],[8,11,10],[10,a_5,a_4],[a_4,6,7]) $. From here we see that $ a_4=3,a_5=4 $ and then $ lk(3) $ will not possible. If $ a_3=9 $ then after completing $ lk(9) $, we see that $ [2,10] $ form edge and non-edge in two 4-gon, a contradiction. If $ \boldsymbol{(b_1,b_2)=(3,1)} $ then $ a_1=6 $ and then $ lk(6) $ will not possible.
\end{sloppypar}

\hspace{-0.6cm}\underline{\textbf{If} $ \boldsymbol{(a,b,c,d)=(3,10,5,2)} $} then $ lk(3)=C_{8}([4,a_1,b_1],[b_1,8,b_2],[b_2,a_2,2],[2,1,0]) $ and $ lk(5)=C_{8}([10,a_3,b_3],[b_3,0,b_4],[b_4,a_4,2],[2,7,8]) $ where $ \{b_1,b_2\}=\{9,10\} $, $ \{b_3,b_4\}=\{4,6\} $, $ a_1\in\{7,11\} $ and $ a_3\in\{1,11\} $. Now if $ a_1=7 $ then $ a_2=11 $ and from $ lk(7) $ we see that $ [6,7,b_1] $ form a face. Which implies $ (b_1,b_2)=(9,10) $ and then $ [1,2,11] $ form a face. Therefore $ C(1,0,3,b_2,11)\in lk(2) $, a contradiction, therefore $ a_1=11 $. If $ a_3=1 $ then $ a_4=11 $ and then from $ lk(1) $, we see that $ [1,9,b_3] $ from a face, which implies $ (b_3,b_4)=(6,4) $. Therefore $ [2,7,11] $ form a face and then $ C(7,8,5,b_4,11)\in lk(2) $, a contradiction, therefore $ a_3=11 $. Now we have faces $ [0,1,2,3] $, $ [2,3,b_2,a_2] $, $ [2,5,8,7] $ and $ [2,5,b_4,a_4] $. Therefore, by considering $ lk(2) $, we get either $ a_2\in\{7,a_4\} $ or $ a_4\in\{1,a_2\} $. But if $ a_2=a_4 $, then $ a_2\notin V $. Therefore either $ a_2=7 $ or $ a_4=1 $. If $ a_2=7 $ then $ {1,2,a_4} $ form a face, which implies $ a_4=9 $ and then $ C(2,3,0,8,9)\in lk(1) $, a contradiction. If $ a_4=1 $ then $ [2,7,a_2] $ form a face which implies $ a_2=6 $ and then $ C(2,5,8,0,6)\in lk(7) $, a contradiction.
\\
	\\

	If \underline{$ \boldsymbol{ lk(8)=C_{9}([a,b,c],[c,d,9],[9,1,0],[0,6,7]) } $} then $ a\in\{2,5,10\} $. If $ a=2 $ then remaining possible of two incomplete 3-gon are $ \{[1,5,6],[9,10,11]\},\{[1,5,9],[6,10,11]\},\{[1,5,\allowbreak10],[6,9,11]\},\{[1,6,9],[5,10,11]\},\{[1,6,10],[5,9,11]\},\{[1,9,10],[5,6,11]\},\{[1,10,11],[5,6,\allowbreak 9]\} $.
	\begin{case}
		When remaining two 3-gons are $ \{[1,5,6],[9,10,11]\} $ then either $ lk(1)=C_{9}([6,a_{2},\allowbreak 2],[2,3,0],[0,8,9],[9,a_{1},5]) $ or $ lk(1)=C_{9}([5,a_{2},2],[2,3,0],[0,8,9],[9,a_{1},6]) $.
		\begin{subcase}
			For $ lk(1)=C_{9}([6,a_{2},2],[2,3,0],[0,8,9],[9,a_{1},5]) $, with out loss of generality, assume $ a_{2}=10 $ and then $ lk(6)=C_{9}([1,2,10],[10,a_{3},7],[7,8,0],[0,4,5]) $ and $ lk(2)=C_{9}([a_{6},a_{7},3],[3,0,1],[1,6,10],[10,a_{4},a_{5}]) $. From here we see that $ a_{5},a_{6}\in\{7,8\} $ and from $ lk(7) $ we see that $ a_{5}=8,a_{6}=7 $. Now $ lk(8)=C_{9}([2,10,a_{4}],[a_{4},a_{8},9],[9,1,0],[0,6,7]) $ and $ lk(10)=C_{9}([a_{9},a_{10},a_{4}],[a_{4},8,2],[2,1,6],[6,7,a_{3}]) $. From here we see that $ a_{3},a_{9}\in\{9,11\} $  \& $ a_{4}\in\{4,5\} $ which implies $ a_{4}=4 $ (from $ lk(9) $). As $ a_{10}\neq 8 $ implies $ a_{9}=3,a_{3}=9 $ which implies $ [3,4] $ form an edge and a non-edge which is a contradiction.
		\end{subcase}
		\begin{subcase}
			For $ lk(1)=C_{9}([5,a_{2},2],[2,3,0],[0,8,9],[9,a_{1},6]) $ then $ lk(6)=C_{9}([1,9,10],\allowbreak [10,a_{3},7],[7,8,0],[0,4,5]) $ which implies $ a_{1}\in\{10,11\} $. With out loss of generality, let $ a_{1}=10 $, then $ a_{2}=11,a_{3}=3 $. Now $ lk(2)=C_{9}([8,a_{5},3],[3,0,1],[1,5,11],[11,a_{4},7]),lk(7)=C_{9}([2,11,a_{4}],[a_{4},a_{6},3],[3,10,6],[6,0,8]), lk(8)=C_{9}([2,3,a_{5}],[a_{5},a_{7},9],[9,1,0],[0,6,7]) $ and $ lk(3)=C_{9}([4,a_{4},7],[7,6,a_{5}],[a_{5},8,2],[2,1,0]) $. From here we see that $ a_{6}=4,a_{5}=10,a_{4}=9 $ which implies $ [9,10] $ form an edge and a non-edge which is a contradiction.
		\end{subcase}
	\end{case}
	\begin{case}
		When remaining two 3-gons are $ \{[1,5,9],[6,10,11]\} $ then either $ lk(6)=C_{9}([10,a_{2},\allowbreak 5],[5,4,0],[0,8,7],[7,a_{1},11]) $ or $ lk(6)=C_{9}([11,a_{2},5],[5,4,0],[0,8,7],[7,a_{1},10]) $.
		\begin{subcase}
			For $ lk(6)=C_{9}([10,a_{2},5],[5,4,0],[0,8,7],[7,a_{1},11]) $ implies $ lk(7)=C_{9}([2,\allowbreak a_{3},a_{4}],[a_{4},a_{5},a_{1}],[a_{1},11,6],[6,0,8]) $ and $ a_{1}\in\{1,3,9\} $. If $ \boldsymbol{a_{1}=1} $ then $ lk(1) $ is not possible. If $ \boldsymbol{a_{1}=9} $ then $ lk(9)=C_{9}([5,a_{4},7],[7,6,11],[11,b_{1},8],[8,0,1]) $ which implies $ a_{5}=5,a_{4}=\allowbreak4 $ and then $ lk(5)=C_{9}([1,a_{6},a_{2}],[a_{2},10,6],[6,0,4],[4,7,9]) $ which implies $ a_{2}\notin V $. If $ \boldsymbol{a_{1}=3} $ then $ lk(3)=C_{9}([4,a_{4},7],[7,6,11],[11,a_{6},2],[2,1,0]) $ which implies $ a_{5}=4 $ \& $ a_{4}\in\{9,10\} $. If $ a_{4}=9 $ then $ lk(9)=C_{9}([5,2,7],[7,3,4],[4,a_{7},8],[8,0,1]) $ which implies $ a_{3}=5,a_{7}=10 $ and $ lk(4)=C_{9}([3,7,9],[9,8,10],[10,a_{8},5],[5,6,0]) $. From here we see that $ a_{8}=11 $ as $ 11 $ is not a member of $ lk(0) $ \& $ lk(9) $, and then $ lk(5) $ will not possible. If $ a_{4}=10 $ then $ lk(7)=C_{9}([2,a_{5},10],[10,4,3],[3,11,6],[6,0,8]) $, $ lk(10)=C_{9}([11,a_{7},4],[4,3,7],[7,2,a_{2}],[a_{2},5,6]) $ which implies $ a_{2}=a_{3} $ and $ a_{2},a_{7}\in\{1,9\} $. Now $ lk(4)=C_{9}([3,7,10],[10,11,a_{7}],[a_{7},a_{8},5],[5,\allowbreak6,0]) $ and $ lk(5)=C_{9}([a_{9},a_{10},a_{8}],[a_{8},a_{7},4],[4,0,6],[6,10,a_{2}]) $ which implies $ a_{2},a_{9}\in\{1,9\} $ which is a contradiction.
		\end{subcase}
		\begin{subcase}
			For $ lk(6)=C_{9}([11,a_{2},5],[5,4,0],[0,8,7],[7,a_{1},10]) $, $ lk(7)=C_{9}([2,a_{3},a_{4}],\allowbreak [a_{4},a_{5},a_{1}],[a_{1},10,6],[6,0,8]) $ which implies $ a_{1}\in\{1,3,9\} $. For $ \boldsymbol{a_{1}=1} $, $ lk(1) $ will not possible. If $ \boldsymbol{a_{1}=3} $ then $ lk(3)=C_{9}([4,a_{4},7],[7,6,10],[10,a_{6},2],[2,1,0]) $ which implies $ a_{5}=4 $ and $ a_{4}\in\{9,11\} $. If $ a_{4}=9 $ then $ lk(9)=C_{9}([5,2,7],[7,3,4],[4,a_{7},8],[8,0,1]) $ which implies $ a_{3}=5 $ and $ a_{7}=11 $. Now $ lk(4)=C_{9}([3,7,9],[9,8,11],[11,a_{8},5],[5,6,0]) $ which implies $ a_{8}=10 $ and then $ lk(5) $ will not possible. Therefore $ a_{4}=9 $ is not possible. Now if $ a_{4}=11 $ then $ lk(2)=C_{9}([8,a_{7},3],[3,0,1],[1,a_{8},a_{3}],[a_{3},11,7]) $ which implies $ a_{3}=4 $ and $ lk(1)=C_{9}([5,a_{9},a_{8}],[a_{8},4,2],[2,3,0],[0,8,9]) $ implies $ a_{8}=10 $ and then $ lk(4) $ will not possible. Therefore $ a_{4}=11 $ is not possible and then $ a_{1}=3 $ is not possible. If $ \boldsymbol{a_{1}=9} $ then $ lk(9)=C_{9}([1,0,8],[8,a_{6},10],[10,6,7],[7,a_{4},5]) $ which implies $ a_{5}=5,a_{4}=4 $ ($ a_{4}\neq 11 $ as $ [5,11] $ is a non-edge in a 4-gon) and $ a_{6}=3 $. Now $ lk(4)=C_{9}([3,a_{7},a_{3}],[a_{3},2,7],[7,9,5],[5,6,0]) $ which implies $ a_{3}=11 $ as $ [1,3] $ can not be non-edge in two 4-hon. We see that $ a_{7}\neq 1 $ and $ 1\notin lk(7) $, therefore $ a_{2}=1 $ and then $ lk(11) $ will not possible. Therefore $ a_{1}=9 $ is not possible.
		\end{subcase}
	\end{case}
	\begin{case}
		When remaining two 3-gons are $ \{[1,5,10],[6,9,11]\} $ then either $ lk(6)=C_{9}([9,a_{2},\allowbreak5], [5,4,0],[0,8,7],[7,a_{1},11]) $ or $ lk(6)=C_{9}([11,a_{2},5],[5,4,0],[0,8,7],[7,a_{1},9]) $.
		\begin{subcase}
			For $ lk(6)=C_{9}([9,a_{2},5],[5,4,0],[0,8,7],[7,a_{1},11]) $, $ a_{1}\in\{1,3,10\} $. If $ \boldsymbol{a_{1}=1} $, then $ lk(1) $ will not possible. If $ a_{1}=3 $ then $ lk(7)=C_{9}([2,a_{3},a_{4}],[a_{4},a_{5},3],[3,11,6],\allowbreak[6,0,8]) $ and $ lk(3)=C_{9}([4,a_{4},7],[7,6,11],[11,a_{6},2],[2,1,0]) $ which implies $ a_{5}=4 $ and $ a_{4}\in\{9,10\} $. When $ a_{4}=9 $, then to complete $ lk(9) $, edges of two consecutive 4-gons in $ lk(9) $ are $ [9,1],[9,4],\allowbreak [9,8] $ which implies $ a_{2}=1=a_{3} $ which is a contradiction. Now $ a_{4}=9 $ implies $ a_{3}\in\{1,5,9\} $ and $ lk(4)=C_{9}([3,7,10],[10,a_{7},a_{8}],[a_{8},a_{9},5],[5,6,0]) $. Now if $ a_{1}=1 $ then $ lk(2)=C_{9}([8,a_{7},5],[5,11,3],[3,0,1],[1,10,7]) $ which implies $ lk(5) $ will not possible. If $ a_{1}\in\{5,9\} $ then $ lk(2)=C_{9}([8,a_{7},3],[3,0,1],[1,a_{8},a_{3}],[a_{3},10,7]) $ which implies $ a_{6}=8,a_{7}=11 $. Now $ lk(1)=C_{9}([a_{8},a_{3},2],[2,3,0],[0,8,9],[9,a_{9},a_{10}]) $ which implies $ a_{8}=5,a_{10}=10,a_{3}=9 $ and then $ [5,9] $ form an edge and a non-edge in two 4-gon which is a contradiction. For $ \boldsymbol{a_{1}=10} $, $ a_{2}\in\{1,2,3\} $. If $ a_{1}=1 $ then $ lk(1)=C_{9}([10,a_{3},2],[2,3,0],[0,8,9],[9,6,5]) $ and either $ lk(2)=C_{9}([7,c_{1},3],[3,0,1],[1,10,a_{3}],[a_{3},a_{4},8]) $ which implies $ c_{1}\notin V $. Or $ lk(2)=C_{9}([8,a_{4},3],[3,0,1],[1,10,a_{3}],[a_{3},c_{1},7]) $ which implies $ a_{4}=11,a_{3}=4 $. Now $ lk(4)=C_{9}([3,a_{5},\allowbreak 10],[10,1,2],[2,7,5],[5,6,0]) $ implies $ c_{1}=5 $ and $ lk(3)=C_{9}([4,10,a_{5}],[a_{5},a_{6},11],\allowbreak[11,8,2],[2,\allowbreak 1,0]) $. We see that $ a_{5}\neq 11 $ i.e. $ 11\notin V(lk(0)),V(lk(1)),V(lk(4)) $ which is a contradiction. If $ a_{1}\in\{2,3\} $ then $ lk(5)=C_{9}([1,a_{4},4],[4,0,6],[6,9,a_{2}],[a_{2},a_{3},10])\Rightarrow lk(9)=C_{9}([11,a_{5},1],[1,0,8],\allowbreak [8,a_{6},a_{2}],[a_{2},5,6])\Rightarrow lk(1)=C_{9}([a_{7},a_{8},2],[2,3,0],[0,8,9],[9,11,a_{5}]) $. Therefore we see that $ a_{5}=10,a_{7}=5\Rightarrow a_{8}=4,a_{4}=2 $ and then $ a_{2}=3,a_{3}=7 $ which implies $ lk(7) $ will not possible.
		\end{subcase}
		\begin{subcase}
			For $ lk(6)=C_{9}([11,a_{2},5],[5,4,0],[0,8,7],[7,a_{1},9]) $, $ a_{1}\in\{1,3,10\} $. If $ \boldsymbol{a_{1}=1} $ then $ lk(1) $ will not possible. If $ \boldsymbol{a_{1}=3} $ then $ lk(3)=C_{9}([4,a_{3},7],[7,6,9],[9,a_{4},2],[2,\allowbreak1,0]) $. We see that $ a_{4}\in\{5,8,10\} $, but for $ a_{4}\in\{5,10\} $, $ lk(9) $ will not possible, therefore $ a_{4}=8 $. Now $ lk(2)=C_{9}([7,a_{6},a_{7}],[a_{7},a_{8},1],[1,0,3],[3,9,8]) \Rightarrow a_{2}=10$. Therefore $ lk(5)=C_{9}([1,a_{9},a_{10}],[a_{10},a_{11},4],[4,0,6],[6,11,10])\Rightarrow lk(1)=C_{9}([a_{8},a_{7},2],[2,3,0],[0,8,9],[9,a_{12},\allowbreak a_{13}])\Rightarrow lk(9)=C_{9}([11,a_{13},1],[1,0,8],[8,2,3],[3,7,6]) $ which implies $ a_{12}=11,a_{13}\neq10,5 $ which is a contradiction. If $ \boldsymbol{a_{1}=10} $ then $ lk(9)=C_{9}([11,a_{3},1],[1,0,8],[8,a_{4},10],[10,7,\allowbreak6]) $ and $ lk(1)=C_{9}([10,a_{5},2],[2,3,0],[0,8,9],[9,11,5]) $ which implies $ a_{3}=5 $ and then $ [5,11] $ form an edge and a non-edge which is a contradiction.
		\end{subcase}
	\end{case}
	\begin{case}
		When remaining two 3-gons are $ \{[1,6,9],[5,10,11]\} $ then either $ lk(6)=C_{9}([1,a_{2},\allowbreak5],\allowbreak [5,4,0],[0,8,7],[7,a_{1},9]) $ or $ lk(6)=C_{9}([9,a_{2},5],[5,4,0],[0,8,7],[7,a_{1},1]) $.
		\begin{subcase}
			$ lk(6)=C_{9}([1,a_{2},5],[5,4,0],[0,8,7],[7,a_{1},9])\Rightarrow lk(5)=C_{9}([a_{3},a_{4},a_{5}],[a_{5},\allowbreak a_{6},4],[4,0,6],[6,1,a_{2}]),lk(1)=C_{9}([6,5,a_{2}],[a_{2},a_{7},2],[2,3,0],[0,8,9]) $ which implies $ a_{2}\in\{10,11\} $. with out loss of generality, let $ a_{2}=10 $, then $ a_{1}=11,a_{7}=4 $. Now $ lk(4)=C_{9}([3,a_{8},10],[10,1,2],[2,a_{9},5],[5,6,0]) $ which implies $ a_{9}=11 $ and then $ [2,11] $ will form an edge in a 3-gon which is not possible.
		\end{subcase}
		\begin{subcase}
			$ lk(6)=C_{9}([9,a_{2},5],[5,4,0],[0,8,7],[7,a_{1},1]) $ implies $ lk(1)=C_{9}([6,7,a_{1}],\allowbreak[a_{1}, a_{3},2],[2,3,0],[0,8,9]) $. W.L.G., assume $ a_{1}=10 $, then $ a_{2}\in\{3,11\} $ and $ a_{3}\in\{4,5,11\} $. If $ \boldsymbol{a_{2}=3} $ then $ a_{3}=11 $ and $ lk(3)=C_{9}([4,a_{4},9],[9,6,5],[5,a_{5},2],[2,1,0]),lk(2)=C_{9}([a_{5},5,\allowbreak3],\allowbreak [3,0,1],[1,10,11],[11,a_{7},a_{6}]),lk(9)=C_{9}([6,5,3],[3,4,a_{4}],[a_{4},a_{8},8],[8,0,1]) $. From here we see that $ a_{5}\in\{7,8\}\Rightarrow a_{4}=11 $ and then $ a_{8}=10 $ and then $ [10,11] $ form an edge and a non-edge in two 4-gon which is a contradiction. $ \boldsymbol{a_{2}=11}\Rightarrow a_{3}\in\{4,5,11\} $. If $ a_{3}=4 $ then $ lk(4)=C_{9}([3,a_{4},10],[10,1,2],[2,a_{5},5],[5,6,0]), lk(5)=C_{9}([10,a_{6},a_{5}],[a_{5},2,4],[4,0,6],[6,\allowbreak9,\allowbreak 11]) $ which implies $ a_{5}=8,a_{6}=9 $ which is not possible. If $ a_{3}=5 $ then $ lk(2)=C_{9}([a_{4},4,5],[5,\allowbreak 10,1],[1,0,3],[3,a_{6},a_{5}]) $ which implies $ a_{6}=11 $. Now $ lk(4)=C_{9}([3,a_{7},a_{8}],\allowbreak[a_{8},a_{9},a_{4}],[a_{4},2,\allowbreak 5],[5,6,0]) $ and $ lk(3)=C_{9}([4,a_{8},a_{7}],[a_{7},a_{10},11],[11,a_{5},2],[2,1,0]) $ which implies $ a_{9}=11 $. Now from $ lk(8) $, we see that $ a_{4}\neq 8 $, therefore $ a_{4}=7,a_{5}=8 $ and then $ lk(7) $ will not possible. If $ a_{3}=11 $ then $ lk(5)=C_{9}([10,b{1},b_{2}],[b_{2},b_{3},4],[4,0,6],[6,9,11]) $, $ lk(10)=C_{9}([5,b_{2},b_{1}],[b_{1},\allowbreak b_{4},7],[7,6,1],[1,2,11]) $ which implies $ b_{2}=8 $ and $ b_{1},b_{3}\in\{2,9\} $ which is not possible.
		\end{subcase}
	\end{case}
	\begin{case}
		When remaining two 3-gons are $ \{[1,6,10],[5,9,11]\} $ then either $ lk(6)=C_{9}([1,a_{2},\allowbreak 5],[5,4,0],[0,8,7],[7,a_{1},10]) $ or $ lk(6)=C_{9}([10,a_{2},5],[5,4,0],[0,8,7],[7,a_{1},1]) $.
		\begin{subcase}
			\begin{sloppypar}
			For $ lk(6)=C_{9}([1,a_{2},5],[5,4,0],[0,8,7],[7,a_{1},10]) $, $ a_{2}\in\{2,9\} $. If $ \boldsymbol{a_{2}=2} $ then $ lk(1)=C_{9}([6,5,2],[2,3,0],[0,8,9],[9,a_{3},10]),lk(2)=C_{9}([b_{3},b_{4},3],[3,0,1],[1,6,5],[5,\allowbreak b_{1},\allowbreak b_{2}]),lk(5)=C_{9}([b_{5},b_{6},4],[4,0,6],[6,1,2],[2,b_{2},b_{1}]) $ which implies $ b_{1}\in\{9,11\} $ and then $ b_{2}=7,b_{3}=8 $. Now $ lk(7)=C_{9}([2,5,b_{1}],[b_{1},b_{7},a_{1}],[a_{1},10,6],[6,0,8]) $ which implies $ a_{1}=3 $ and then $ lk(3)=C_{9}([4,b_{1},7],[7,6,10],[10,8,2],[2,1,0]) $. Therefore $ b_{7}=4 $ and then $ b_{1}=11 $ which implies $ 9\notin V(lk(2)),V(lk(7)),V(lk(3)) $. If $ \boldsymbol{a_{2}=9} $ then $ lk(1)=C_{9}([6,5,9],[9,8,0],[0,3,2],[2,\allowbreak a_{3},10]),\allowbreak  lk(2)=C_{9}([b_{3},b_{4},3],[3,0,1],[1,10,a_{3}],[a_{3},b_{1},b_{2}]), lk(8)=C_{9}([2,c_{1},c_{2}],[c_{2},c_{3},9],\allowbreak[9,1,0],[0,\allowbreak 6,7]) $ and $ lk(7)=C_{9}([2,c_{4},c_{5}],[c_{5},c_{6},a_{1}],[a_{1},10,6],[6,0,8]) $. If $ a_{3}\neq 11 $ then $ a_{1}=11 $ and one of $ b_{1},b_{4} $ i.e. one of $ c_{2},c_{5} $ is $ 11 $. But $ c_{5}\neq 11 $ as then $ c_{2}=11 $ which make contradiction. Therefore $ a_{3}=11\Rightarrow a_{1}=3 $ and then $ b_{2}=7,b_{3}=8,c_{1}=11,b_{1}=c_{2},c_{4}=3,c_{5}=b_{4}\Rightarrow c_{6}=11 $ which implies $ lk(3) $ will not possible.
		\end{sloppypar}
		\end{subcase}
		\begin{subcase}
			\begin{sloppypar}
			For $ lk(6)=C_{9}([10,a_{2},5],[5,4,0],[0,8,7],[7,a_{1},1]) $, $ lk(9)=C_{9}([a_{5},a_{6},7],[7,\allowbreak 6,1],[1,0,8],[8,a_{3},a_{4}]) $ and from $ lk(1) $ we get $ a_{1}=9 $ and $ a_{2}\in\{2,3,11\} $. But $ a_{2}\neq 2 $ as then two 4-gon at $ 2 $ has no common vertex except $ 2 $ and no edge of these two 4-gon does not form an edge in a 3-gon, therefore $ a_{2}\in\{3,11\} $. If $ \boldsymbol{a_{2}=3} $ then $ lk(5)=C_{9}([b_{3},b_{4},3],[3,10,6],[6,0,4],[4,b_{1},b_{2}]), lk(7)=C_{9}([2,c_{1},a_{6}],[a_{6},a_{5},9],[9,1,6],[6,0,8])$, $\allowbreak lk(8) =C_{9}([2,c_{2},a_{3}],[a_{3},a_{4},9],[9,1,0],[0,6,7]),lk(3)=C_{9}([4,a_{7},10],[10,6,5],[5,b_{3},2],[2,\allowbreak1,0]) $ which implies $ b_{4}=2 $ and then $ lk(2)=C_{9}([8,a_{3},1],[1,0,3],[3,5,b_{3}],[b_{3},a_{6},7]) $. From here we see that $ c_{1}=b_{3},a_{5}=11 $ and $ b_{3}=11,b_{2}=9 $ which implies $ a_{4}=11,a_{5}=5,a_{6}=4\Rightarrow b_{4}=2 $ which is a contradiction. If $ \boldsymbol{a_{2}=11} $ then $ lk(7)=C_{9}([2,b_{1},b_{2}],[b_{2},b_{3},9],[9,1,\allowbreak6],[6,\allowbreak0,8]),\allowbreak lk(8)=C_{9}([2,b_{4},b_{5}],[b_{5},b_{6},9],[9,1,0],[0,6,7]),lk(9)=C_{9}([b_{6},b_{5},8],[8,0,1],[1,\allowbreak6,7,],\allowbreak[7,b_{2},\allowbreak b_{3}]) $ and $ lk(5)=C_{9}([9,c_{1},c_{2}],[c_{2},c_{3},4],[4,0,6],[6,10,11]) $ which implies $ c_{2}=b_{2}\text{ or } b_{5} $, that implies $ c_{2}=2 $ which is a contradiction.
		\end{sloppypar}
		\end{subcase}
	\end{case}
	\begin{case}
		When remaining two 3-gons are $ \{[1,9,10],[5,6,11]\} $ then $ lk(6)=C_9([11,a_1,a_2],\allowbreak[a_2,a_3,7],[7,8,0],[0,4,5]) $ and $ lk(1)=C_9([10,a_4,a_5],[a_5,a_6,2],[2,3,0],[0,8,9]) $ which implies $ a_2\in\{1,3,9,10\} $. If $ \boldsymbol{a_2=1} $, then by $ lk(1) $, $ \{a_1,a_3\}=2,10 $, but $ a_3\neq 2 $, therefore $ (a_1,a_3)=(2,10) $ i.e. $ a_5=6, a_6=11,a_4=7 $. Now considering $ lk(2) $, $ b\in \{3,11\} $, i.e. $ 10\notin V(lk(0)), V(lk(8)) $, therefore $ [2,7,10,d_1] $ will be a face for some $ d_1\in\{3,11\} $ which implies $ C(1,6,0,8,2,10)\in lk(11) $, a contradiction. If $ \boldsymbol{a_2=3} $ then considering $ lk(3) $, one of $ a_1,a_3 $ is 2, i.e. $ a_1=2 $. Now $ lk(3)=C_9([0,1,2],[2,11,6],[6,7,a_3],[a_3,b_1,4]) $ which implies $ a_3\in\{9,10\} $ and $ lk(2)=C_9([7,a_5,1],[1,0,3],[3,6,11],[11,c,8]) $ i.e. $ a_6=7,b=11 $, therefore $ a_5,c,d\neq 10 $ i.e. $ |v_{lk(10)}|\leq 8 $, a contradiction. If $ \boldsymbol{a_2=9} $ then considering $ lk(9) $, $ \{a_1,a_3\}=\{d,a_3\} $, which implies $ d=3 $, therefore $ [3,6,9,b_1] $ will be a face for some $ b_1\in\{7,11\} $ and then $ lk(3) $ will not possible. Therefore $ \boldsymbol{a_2=10} $. Now if $ b=10 $ then one of $ [10,a_1],[10,a_3] $ will be an adjacent edge of a 3-gon and a 4-gon and the other one will be one of $ [2,10],[10,c] $. If the other one is $ [10,c] $ then $ c\notin V $ and if it is $ [2,10] $ then considering $ lk(10) $, $ \{a_1,a_3\}=\{2,9\} $ i.e. $ (a_1,a_3)=(2,9) $ and then after completing $ lk(10) $, $ lk(2) $ will not possible. Therefore $ b\neq 10 $ i.e. $ 10\notin V(lk(0)), V(lk(8)) $ i.e. $ V(lk(10))=\{1,2,3,4,5,6,7,11\} $, therefore $ \{a_1,a_3\}=\{2,3\} $ i.e. $ (a_1,a_3)=(2,3) $ and then $ lk(3) $ will not possible.
	\end{case}
	\begin{case}
		When remaining two 3-gons are $ \{[1,10,11],[5,6,9]\} $ then $ lk(1)=C_{9}([10,a_{2},2],[2,\allowbreak 3,0],[0,8,9],[9,a_{1},11]) $ and $ lk(8)=C_{9}([2,b_{1},b_{2}],[b_{2},b_{3},9],[9,1,0],[0,6,7]) $ which implies $ b_{2}=4 $ and $ b_{1}\in\{5,11\} $. If $ \boldsymbol{b_{1}=5} $ then $ lk(4)=C_{9}([3,a_{3},b_{3}],[b_{3},9,8],[8,2,5],[5,6,0]) $ and from here we see that $ [9,b_{3}] $ will not form an edge in a 3-gon, therefore $ b_{3}=10 $ and then $ lk(9)=C_{9}([a_{1},11,1],[1,0,8],[8,4,10],[10,a_{4},a_{5}]) $ which implies $ a_{4}=7 $. Now $ lk(10)=C_{9}([a_{3},3,4],[4,8,9],[9,a_{5},7],[7,a_{6},a_{7}])\Rightarrow a_{3},a_{7}\in\{1,11\}\Rightarrow a_{3}=11,a_{7}=1\Rightarrow a_{2}=7,a_{6}=2\Rightarrow a_{5}=5,a_{1}=6 $ and then after completing $ lk(7) $ we see that $ [5,6] $ will norm a non-edge in a 4-gon which is a contradiction. $ \boldsymbol{b_{1}=11} $ implies $ lk(4)=C_{9}([3,a_{3},a_{4}],[a_{4},2,a_{5}],[a_{5},a_{6},5],\allowbreak [5,6,0])\Rightarrow b_{3}\in\{3,5\} $. But $ b_{3}\neq 5 $ as then face sequence will not follow in $ lk(5) $, therefore $ b_{3}=3 $ and then $ a_{3}=9,a_{4}=8,a_{5}=11,a_{6}=10 $. Now $ lk(9)=C_{9}([a_{1},11,1],[1,0,8],[8,4,3],\allowbreak [3,a_{7},a_{8}]) $. From $ lk(9) $ we see that $ a_{1}\in\{5,6\} $ and from $ lk(1)\text{ \& }lk(11) $ we see that $ a_{1}\in\{6,7\} $, therefore $ a_{1}=6 $ which implies $ a_{8}=5,a_{7}=10 $. Now $ lk(6)=C_{9}([9,1,11],[11,10,7],[7,8,0],\allowbreak [0,4,5]) $ and then $ lk(10) $ will not possible.
	\end{case}
	
	If $ a=5 $ then remaining possible of two incomplete 3-gon are $ \{[1,2,6],[9,10,11]\},\{[1,2,9],\allowbreak [6,10,11]\},\{[1,2,10],[6,9,11]\},\{[1,6,9],[2,10,11]\},\{[1,6,10],[2,9,11]\},\{[1,9,10],[2,6,11]\},\allowbreak \{[1,10,11],[2,6,9]\} $.
	\begin{case}
		When remaining two 3-gons are $ \{[1,2,6],[9,10,11]\} $ then either $ lk(6)=C_{9}([2,a_{2},\allowbreak 5],[5,4,0],[0,8,7],[7,a_{1},1]) $ or $ lk(6)=C_{9}([1,a_{2},5],[5,4,0],[0,8,7],[7,a_{1},2]) $. For $ lk(6)=C_{9}([2,a_{2},5],[5,4,0],[0,8,7],[7,a_{1},1]) $, $ a_{2}\in\{9,10\} $. If $ a_{2}=9 $ then $ lk(5)=C_{9}([a_{5},a_{6},4],[4,0,\allowbreak 6],[6,2,9],[9,a_{3},a_{4}]) $. As two 4-gons at $ 9 $ are $ [1,0,8,9],[2,6,5,9] $ and they have no common vertex except $ 9 $ and no edges of these 4-gon form an edge in a 3-gon which implies $ lk(9) $ will not possible. If $ a_{2}=10 $ then $ a_{1}=11 $ and then after completing $ lk(1) $ we see that $ [9,11] $ is a non-edge in a 4-gon which is a contradiction. If $ lk(6)=C_{9}([1,a_{2},5],[5,4,0],[0,8,7],[7,a_{1},\allowbreak 2]) $ then $ a_{2}=10 $ and then after completing $ lk(1) $ we see that $ [9,10] $ form a non-edge which is not possible.
	\end{case}
	\begin{case}
		When remaining two 3-gons are $ \{[1,2,9],[6,10,11]\} $ then $ 9 $ will occur two times in $ lk(1) $ which is a contradiction.
	\end{case}
	\begin{case}
		When remaining two 3-gons are $ \{[1,2,10],[6,9,11]\} $ then $ lk(1)=C_9([2,3,0],[0,8,\allowbreak9],[9,a_3.a_2],[a_2.a_1,10]) $ and $ lk(6)=C_9([c_1,d_1,7],[7,8,0],[0,4,5],[5,d_2,c_2]) $ which implies $ a_2\in\{4,5,7\} $. Now if $ \boldsymbol{a_2=4} $ then considering $ lk(4) $, one of $ a_1,a_3 $ is 5. If $ a_1=5 $ then by $ lk(5) $, $ d_2\in\{7,8\} $, a contradiction. Similarly $ a_3\neq 5 $ and therefore $ a_2\neq 4 $. If $ \boldsymbol{a_2=5} $ then considering $ lk(5) $, one of $ [5,a_1],[5,a_3] $ have to be adjacent with the face [0,4,5,6] and other one will be the edge [5,7], therefore $ b=4 $ i.e. one of $ a_1,a_6 $ is 6 and then $ (c_1,c_2)=(11,9) $ which implies $ d_2=1,a_1=7 $. Now $ lk(5)=C_9([8,c,4],[4,0,6],[6,9,1],[1,10,7]) $ which implies $ c=2 $ i.e. $ 11\notin V(lk(0)),V(lk(5)) $, therefore $ d=11 $ and then $ lk(2) $ will not possible. If $ \boldsymbol{a_2=7} $ then one of $ a_1,a_3 $ is one of 5, 6. If one of $ a_1,a_3 $ is 6 then $ a_3=6,d_1=1,c_1=9,c_2=11 $. Now $ lk(7)=C_9([5,b_1,a_1],[a_1,10,1],[1,9,6],[6,0,8]) $ which implies $ a_1=11 $ and then $ [5,11] $ form non-edge in two 4-gon, a contradiction. If one of $ a_1,a_3 $ is 5 then considering $ lk(5) $, $ b=4 $, therefore $ lk(5)=C_9([8,c,4],[4,0,6],[6,c_2,d_2],[d_2,1,7]) $ which implies $ d_2\in\{9,10\} $, therefore $ d_2=10,a_1=5 $ and then $ lk(7)=C_9([8,0,6],[6,c_1,d_1],[D_1,9,1],[1,10,5]) $ which implies $ (c_1,c_2)=(11,9),a_3=d_1 $, implies $ d_1\notin V $.
	\end{case}
	\begin{case}
		When remaining two 3-gons are $ \{[1,6,9],[2,10,11]\} $ then after completing $ lk(6) $ and $ lk(1) $ we see that $ [2,\alpha],\alpha\in\{10,11\} $, form a non-edge which is not possible.
	\end{case}
	\begin{case}
		When remaining two 3-gons are $ \{[1,6,10],[2,9,11]\} $ then either $ lk(6)=C_{9}([1,a_{2},\allowbreak 5],[5,4,0],[0,8,7],[7,a_{1},10]) $ or $ lk(6)=C_{9}([1,a_{2},5],[5,4,0],[0,8,7],[7,a_{1},10]) $.
		\begin{subcase}
			If $ lk(6)=C_{9}([1,a_{2}, 5],[5,4,0],[0,8,7],[7,a_{1},10]) $ then $ lk(5)=C_{9}([b_{3},b_{4},4],\allowbreak [4,0,6],[6,1,a_{2}],[a_{2},b_{1},b_{2}]) $ where $ b_{2},b_{3}\in\{7,8\} $. If $ \boldsymbol{(b_{2},b_{3})=(7,8)} $ then $ lk(8)=C_{9}([5,4,\allowbreak b_{4}],[b_{4},b_{5},9],[9,1,0],[0,6,7]) $ which implies $ b_{4}=10 $. Therefore $ lk(1)=C_{9}([10,a_{4},2],[2,3,0],\allowbreak [0,8,9],[9,a_{3},6]) $ and $ lk(10)=C_{9}([1,2,a_{4}],[a_{4},5,8],[8,9,a_{1}],[a_{1},7,6]) $ which implies $ a_{4}=4\text{ \& }b_{5}=a_{1}\in\{3,11\} $. $ a_{1}=3\Rightarrow 11\notin lk(0),lk(8),lk(10) $ which is a contradiction and $ a_{1}=11\Rightarrow a_{3}=9 $ which implies $ lk(9) $ is not possible which is a contradiction. If $ \boldsymbol{(b_{2},b_{3})=(8,7)} $ then $ lk(8)=C_{9}([5,a_{2},b_{1}],[b_{1},b_{5},9],[9,1,0],[0,6,7]) $ which implies $ a_{2}=2,b_{1}\in\{3,10\} $ and then from $ lk(2) $ we see that $ [2,b_{1}] $ is an edge in a 3-gon which is a contradiction.
		\end{subcase}
		\begin{subcase}
			If $ lk(6)=C_{9}([1,a_{2},5],[5,4,0],[0,8,7],[7,a_{1},10]) $ then $ lk(5)=C_{9}([c_{2},a_{4},4],\allowbreak [4,0,6],[6,1,a_{2}],[a_{2},a_{3},c_{1}]) $ where $ c_{1},c_{2}\in\{7,8\} $. If $ \boldsymbol{(c_{1},c_{2})=(7,8)} $ then $ lk(8)=C_{9}([5,4,\allowbreak a_{4}],[a_{4},a_{5},9],[9,1,0],[0,6,7]) $ which implies $ a_{4}=10 $. Now $ lk(1)=C_{9}([10,a_{6},2],[2,3,0],[0,\allowbreak8,\allowbreak 9],[9,5,6])\Rightarrow a_{2}=9 $ and then $ lk(10)=C_{9}([6,7,a_{1}],[a_{1},9,8],[8,5,a_{6}],[a_{6},2,1]) $ which implies $ a_{6}=4,a_{1}=a_{5} $. Now $ lk(9)=C_{9}([a_{1},10,8],[8,0,1],[1,6,5],[5,7,a_{3}]) $. From here we see that $ a_{1}=11,a_{3}=2 $ and then $ 3\notin V(lk(8)),V(lk(9)),V(lk(10)) $, which is a contradiction. If $ \boldsymbol{(c_{1},c_{2})=(8,7)} $ then $ lk(8)=C_{9}([5,a_{2},a_{3}],[a_{3},a_{5},9],[9,1,0],[0,6,7]) $ and $ lk(1)=C_{9}([b_{3},b_{4},\allowbreak2],[2,3,0],[0,8,9],[9,b_{1},b_{2}]) $ and from this we have $ a_{2}=2 $ and then $ lk(2)=C_{9}([a_{3},8,5],[5,6,\allowbreak 1],[1,0,3],[3,a_{6},a_{7}]) $ which implies $ a_{3}\notin V $.
		\end{subcase}
	\end{case}
	\begin{case}
		When remaining two 3-gons are $ \{[1,9,10],[2,6,11]\} $ then either $ lk(6)=C_{9}([2,a_{2},\allowbreak5], [5,4,0],[0,8,7],[7,a_{1},11]) $ or $ lk(6)=C_{9}([11,a_{2},5],[5,4,0],[0,8,7],[7,a_{1},2]) $.
		\begin{subcase}
			\begin{sloppypar}
			If $ lk(6)=C_{9}([2,a_{2},5],[5,4,0],[0,8,7],[7,a_{1},11]) $ then $ lk(5)=C_{9}([c_{2},a_4,4],\allowbreak [4,0,6],[6,2,a_{2}],[a_2,a_3,c_1]) $ where $ c_1,c_2\in\{7,8\} $. If $ \boldsymbol{(c_1,c_2)=\allowbreak (7,8)} $ then $ lk(8)=C_9([5,4,\allowbreak a_4],[a_4,a_5,9],[9,1,0],[0,6,7]) $ which implies $ a_4=11,a_2=10 $. Now $ lk(2)=C_9([11,b_{1},1],[1,\allowbreak0,\allowbreak 3],[3,b_2,10],[10,5,6]) $ which implies $ a_5=3 $ and then $ lk(3)=C_{9}([4,b_3,9],[9,8,11],[11,b_4,\allowbreak2],\allowbreak [2,1,0]) $ and from this we have $ a_1=9 $ which implies $ lk(9) $ will not possible. If $ \boldsymbol{(c_1,c_2)\allowbreak=\allowbreak(8,7)} $ then $ lk(8)=C_9([5,a_2,a_3],[a_3,a_5,9],[9,1,0],[0,6,7]) $ which implies $ a_2=10 $ and then after completing $ lk(10) $ and $ lk(2) $ we see that $ 1,9,10 $ occur in a 4-gon which is not possible as $ [1,9,10] $ is a 3-gon.
		\end{sloppypar}
		\end{subcase}
		\begin{subcase}
			If $ lk(6)=C_{9}([11,a_{2},5],[5,4,0],[0,8,7],[7,a_{1},2]) $ then $ lk(5)=C_9([c_2,a_4,4],\allowbreak [4,0,6],[6,11,a_2],[a_2,a_3,c_1]) $ where $ c_1,c_2\in\{7,8\} $. If $ \boldsymbol{(c_1,c_2)=(7,8)} $ then $ lk(8)=C_9([4,\allowbreak5,\allowbreak a_4],[a_4,a_5,9],[9,1,0],[0,6,7]) $. From here we get $ a_4=2 $. Therefore we have two 4-gons $ [0,1,2,3] $ and $ [2,4,5,8] $ which have no common vertex except $ 2 $ and no edges og these two 4-gon form an edge in a 3-gon which is a contradiction. If $ \boldsymbol{(c_1,c_2)=(8,7)} $ then $ lk(8)=C_9([5,a_2,a_3],[a_3,a_5,9],[9,1,0],[0,6,7]) $. From here we see that $ a_3\in\{2,3\} $. But if $ a_3=2 $ then to complete $ lk(2) $, $ a_2 $ have to be either $ 6 $ or $ 11 $ which is not possible, therefore $ a_3=3 $ and then after completing $ lk(10) $ and $ lk(1) $ we see that $ 3 $ occurs two times in $ lk(9) $ which is a contradiction.
		\end{subcase}
	\end{case}
	\begin{case}
		When remaining two 4-gons are $ \{[1,10,11],[2,6,9]\} $ then $ lk(1)=C_9([10,a_2,2],[2,\allowbreak 3,0],[0,8,9],[9,a_1,11]) $ where $ a_2\in\{4,5,6,7\} $. If $ \boldsymbol{a_2\neq 6} $ then $ lk(2)=C_9([b_3,b_4,3],[3,0,1],\allowbreak [1,10,a_2],[a_2,b_1,b_2]) $ which implies $ b_2=9,b_3=6 $. If $ a_2=4 $ then after completing $ lk(4) $ we have $ b_1=5 $ and then from $ lk(9) $ we see that $ a_1\in\{5,6\} $. But for any $ a_1 $, $ lk(a_1) $ will not possible. And similarly for $ a_2=5,6,7 $, $ lk(a_2) $ will not possible.
	\end{case}
	
	If $ a=10 $ then remaining possible of two incomplete 3-gon are $ \{[1,2,5],[6,9,11]\},\{[1,2,6],\allowbreak [5,9,11]\},\{[1,2,9],[5,6,11]\},\{[1,2,11],[5,6,9]\},\{[1,5,6],[2,9,11]\},\{[1,5,9],[2,6,11]\},\{[1,\allowbreak5,\allowbreak 11],[2,6,9]\},\{[1,6,9],[2,5,11]\},\{[1,6,11],[2,5,9]\},\{[1,9,11],[2,5,6]\} $.
	\begin{case}
		When remaining two 3-gons are $ \{[1,2,5],[6,9,11]\} $ then $ lk(1)=C_{9}([5,a_1,a_2],[a_2,\allowbreak a_3,9],[9,8,0],[0,3,2]) $ and $ lk(6)=C_9([c_2,b_2,5],[5,4,0],[0,8,7],[7,b_1,c_1]) $ where $ a_2\in\{7,10\} $ \& $ c_1,c_2\in\{9,11\} $. If $ \boldsymbol{(c_1,c_2)=(9,11)} $ then for $ a_2=7 $, $ lk(7)=C_9([10,b_3,a_1],[a_1,5,1],[1,9,\allowbreak 6],[6,0,8]) $ which implies $ a_1=4,b_3=11 $ and then considering $ lk(4) $ and $ lk(11) $ we see that $ [8,9] $ form an edge and a non-edge in two 4-gon which is not possible, and for $ a_2=10 $, after considering $ lk(7) $ and $ lk(5) $ we have $ a_1=4 $ which implies $ b_1\notin V $. If $ \boldsymbol{(c_1,c_2)=(11,9)} $ then $ a_2=10 $ and $ b_1,b_2\in\{2,3\} $ and considering $ lk(5) $, we have $ a_1\in\{4,b_2\} $. But $ a_{1}\neq b_2 $, therefore $ a_1=4 $ and then $ [2,b_2] $ form an edge and a non-edge which is not possible.
	\end{case}
	\begin{case}
		When remaining two 3-gons are $ \{[1,2,6],[5,9,11]\} $ then $ lk(1)=C_9([6,a_1,a_2],[a_2,\allowbreak a_3, 9],[9,8,0],[0,3,2]) $ which implies $ a_2=10 $ and considering $ lk(6) $ \& $ lk(7) $ we get $ a_1=5 $ and then $ lk(6)=C_9([2,11,7],[7,8,0],[0,4,5],[5,10,1]) $. Now after completing $ lk(5) $ \& $ lk(9) $ we see that $ [10,11] $ form an edge and a non-edge which is not possible.
	\end{case}
	\begin{case}
		When remaining two 3-gons are $ \{[1,2,9],[5,6,11]\} $ then $ C(9,8,0,3,2)\in lk(1) $ which is a contradiction.
	\end{case}
	\begin{case}
		When remaining two 3-gons are $ \{[1,2,11],[5,6,9]\} $ then $ lk(6)=C_9([9,a_1,a_2],[a_2,\allowbreak a_3, 7],[7,8,0],[0,4,5]), lk(1)=C_9([11,a_4,a_5],[a_5,a_6,9],[9,8,0],[0,3,2]),lk(8)=C_9([10,a_7,\allowbreak a_8], [a_8,a_9,9],[9,1,0],[0,6,7]) $ and $ lk(9)=C_9([c_2,a_5,1],[1,0,8],[8,a_8,a_9],[a_9,a_2,c_1]) $ where $ c_1,c_2\in\{5,6\} $. When $ \boldsymbol{(c_1,c_2)=(6,5)} $ then $ a_5\in\{7,10\} $. If $ a_5=7 $ then after completing $ lk(7) $ we see that $ a_1=9=a_9 $ which implies $ [8,10] $ form an edge and a non-edge in two 4-gon which is not possible. If $ a_5=10 $ then $ a_4\in\{4,7\} $. If $ a_4=4 $ then considering $ lk(4),lk(10) $ \& $ lk(5) $ we see that $ lk(2) $ will not possible, therefore $ a_4=7 $ and then after completing $ lk(7) $ we get $ [10,11] $ form an edge and a non-edge in two 4-gon which is not possible. When $ \boldsymbol{(c_1,c_2)=(5,6)} $ then $ a_4\in\{5,7\} $. If $ a_4=7 $ then $ lk(4) $ is not possible and if $ a_4=5 $ then after completing $ lk(4) $, we see that $ lk(9) $ will not possible.
	\end{case}
	\begin{case}
		When remaining two 3-gons are $ \{[1,5,6],[2,9,11]\} $ then $ lk(6)=C_9([1,a_1,a_2],[a_2,\allowbreak a_3,7],[7,8,0],[0,4,5]) $ and $ lk(1)=C_9([c_2,a_4,2],[2,3,0],[0,8,9],[9,a_2,c_1]) $ where $ c_1,c_2\in\{5,\allowbreak 6\} $. If $ (c_1,c_2)=(6,5) $ then after completing $ lk(3) $ we see that either $ [7,8] $ or $ [2,5] $ form an edge and a non-edge in two 4-gon, which is a contradiction and if $ (c_1,c_2)=(5,6) $ then considering $ lk(3),lk(7) $ \& $ lk(9) $ we see that $ [5,10] $ form an edge and a non-edge in two 4-gon which is not possible.
	\end{case}
	\begin{case}
		When remaining two 3-gons are $ \{[1,5,9],[2,6,11]\} $ then $ lk(1)=C_9([5,a_1,a_2],[a_2,\allowbreak a_3,2],[2,3,0],[0,8,9]) $ and $ lk(6)=C_9([c_2,b_2,5],[5,4,0],[0,8,7],[7,b_1,c_1]) $ where $ c_1,c_2\in\{2,\allowbreak 11\} $. If $ \boldsymbol{(c_1,c_2)=(11,2)} $ then $ lk(7)=C_9([10,b_3,b_4],[b_4,b_5,b_1],[b_1,11,6],[6,0,8]) $ which implies $ b_1\in\{3,9\} $. If $ b_1=3 $ then after completing $ lk(3),lk(2) $ \& $ lk(9) $ we see that $ [6,8] $ form an edge and a non-edge in two 4-gon which is not possible. If $ b_1=9 $ then after completing $ lk(5) $ and $ lk(9) $ we have $ [2,11] $ form an edge and a non-edge in two 4-gon which is a contradiction. If $ \boldsymbol{(c_1,c_2)=(2,11)} $ then $ b_1\in\{1,3,9\} $. If $ b_1=1 $ then $lk(7)=C_9([10,a_4,4],[4,5,1],[1,2,6],[6,0,8]) $. Therefore $ a_1=4,a_2=7,a_3=6 $ and $ b_2=10 $. Now $ lk(5)=C_9([9,b_3,10],[10,11,6],[6,0,4],[4,7,1]) $. Considering $ lk(10) $ we see that $ a_4 $ is either $ 11 $ or $ b_3 $, but if $ a_4=b_3 $ then $ a_4\notin V $, therefore $ a_4=11 $ and then $ b_3=3 $ as $ 3\notin lk(6),lk(7) $. Therefore $ lk(10)=C_9([2,8,11],[11,6,5],[5,9,3],[3,4,7]) $, $ lk(2)=C_9([11,10,8],[8,9,3],[3,0,1],[1,7,6]) $, $ lk(8)=C_9([10,3,2],[2,11,9],[9,1,0],[0,6,7]) $, $ lk(9)=C_9([5,10,3],[3,4,11],[11,2,8],[8,0,1]) $, $ lk(3)=C_9([0,1,2],[2,8,10],[10,5,9],[9,11,4]) $, $ lk(4)\allowbreak =C_9([3,9,11],[11,10,7],[7,1,5],[5,6,0]) $ and $ lk(11)=C_9([6,5,10],[10,7,4],[4,3,9],[9,8,2]) $. Let this be the map $ \boldsymbol{KO_{[(3,4^4)]}} $. If $ b_1=3 $ then considering $ lk(2) $ we see that $ a_1=7,a_2=10,a_3=4 $ and then completing $ lk(7) $ we see that $ lk(5) $ will not possible. If $ b_1=9 $ then after completing $ lk(2) $ we see that $ b_2=9 $ which implies $ [6,9] $ form non-edge in two 4-gon which is not possible.
	\end{case}
	\begin{case}
		When remaining two 3-gons are $ \{[1,5,11],[2,6,9]\} $ then either $ lk(1)=C_9([5,a_2,2],\allowbreak [2,3,0],[0,8,9],[9,a_1,11]) $ or $ lk(1)=C_9([11,a_2,2],[2,3,0],[0,8,9],[9,a_1,5]) $.
		\begin{subcase}
			If $ lk(1)=C_9([5,a_2,2],[2,3,0],[0,8,9],[9,a_1,11]) $ then $ a_2\in\{4,6,7,10\} $. If $ a_2=6 $ then after completing $ lk(2) $  \& $ lk(9) $ we see that $ [6,8] $ form an edge and a non-edge in two 4-gon which is not possible. If $ a_2\neq 6 $ then $ lk(2)=C_9([c_2,a_4,3],[3,0,1],[1,5,a_2],[a_2,a_3,\allowbreak c_1]) $ where $ c_1,c_2\in\{6,9\} $. For $ (c_1,c_2)=(6,9) $ we get $ a_2=10 $ and then after completing $ lk(6) $ \& $ lk(9) $ we get $ a_1\notin V $, which is a contradiction. For $ (c_1,c_2)=(9,6) $ then considering $ lk(2) $ \& $ lk(6) $ we see that $ 11\notin V(lk(0)),V(lk(2)),V(lk(6)) $ which is a contradiction.
		\end{subcase}
		\begin{subcase}
			If $ lk(1)=C_9([11,a_2,2],[2,3,0],[0,8,9],[9,a_1,5]) $ then $ lk(2)=C_9([c_2,a_4,3],\allowbreak [3,0,1],[1,11,a_2],[a_2,a_3,c_1]) $ where $ c_1,c_2\in\{6,9\} $. If $ (c_1,c_2)=(6,9) $ then after completing $ lk(6), lk(8) $ \& $ lk(9) $ we see that $ [1,11] $ form an edge and a non-edge in two 4-gon which is not possible. If $ (c_1,c_2)=(9,6) $ then either $ lk(6)=C_9([2,3,5],[5,4,0],[0,8,7],[7,a_5,9]) $ and then after completing $ lk(5),lk(9) $ we get $ C(0,8,9,6)\in lk(7) $ which is not possible, OR $ lk(6)=C_9([9,a_5,5],[5,4,0],[0,8,7],[7,3,2]) $ and then after completing $ lk(8),lk(9) $ we see that $ [10,11] $ form an edge and a non-edge in two 4-gon which is not possible.
		\end{subcase}
	\end{case}
	\begin{case}
		When remaining two 3-gons are $ \{[1,6,9],[2,5,11]\} $ then $ lk(1)=C_9([6,a_1,a_2],[a_2,\allowbreak a_3,2],[2,3,0],[0,8,9]) $ and either $ lk(6)=C_9([1,a_2,5],[5,4,0],[0,8,7],[7,a_4,9]) $ or $ lk(6)=C_9([9,a_4,5],[5,4,0],[0,8,7],[7,a_2,1]) $. If $ lk(6)=C_9([1,a_2,5],[5,4,0],[0,8,7],[7,a_4,9]) $ then after completing $ lk(5),lk(2),lk(7) $ we see that $ [4,6] $ form non-edge in two 4-gon which is not possible and if $ lk(6)=C_9([9,a_4,5],[5,4,0],[0,8,7],[7,a_2,1]) $ then $ a_2=10 $ which implies $ C(8,0,6,10) \in lk(7)$ which is a contradiction.
	\end{case}
	\begin{case}
		When remaining 3-gons are $ \{[1,6,11],[2,5,9]\} $ then either $ lk(6)=C_9([c_1,a_2,5],[5,\allowbreak 4,0],[0,8,7],[7,a_1,c_2]) $ where $ c_1,c_2\in\{1,11\} $.
		\begin{subcase}
			If $ \boldsymbol{(c_1,c_2)=(1,11)} $ then either $ lk(1)=C_9([6,5,2],[2,3,0],[0,8,9],[9,a_3,\allowbreak 11]) $ or $ lk(1)=C_9([11,a_3,2],[2,3,0],[0,8,9],[9,5,6]) $. If $ lk(1)=C_9([6,5,2],[2,3,0],[0,8,9],\allowbreak [9,a_3,11]) $ then after completing $ lk(5) $ and $ lk(7) $ we see that $ C(6,0,3,7,5)\in lk(4) $ which is not possible and if $ lk(1)=C_9([11,a_3,2],[2,3,0],[0,8,9],[9,5,6]) $ then $ a_1=3,a_3=10 $. Now $ lk(7)=C_9([10,b_1,b_2],[b_2,b_3,3],[3,11,6],[6,0,8]) $ which implies $ b_2=5 $ and then $ lk(5)=C_9([2,3,7],[7,10,4],[4,0,6],[6,1,9]) $. From here we see that $ b_1=4,b_3=2 $ and then after completing $ lk(2),lk(10) $ we see that $ 3 $ occur two times in $ lk(11) $ which is not possible.
		\end{subcase}
		\begin{subcase}
			If $ \boldsymbol{(c_1,c_2)=(11,1)} $ then either $ lk(1)=C_9([6,7,2],[2,3,0],[0,8,9],[9,a_3,\allowbreak 11]) $ or $ lk(1)=C_9([11,a_3,2],[2,3,0],[0,8,9],[9,7,6]) $. If $ lk(1)=C_9([6,7,2],[2,3,0],[0,8,9],\allowbreak [9,a_3,11]) $ then considering $ lk(2),lk(9) $ and completing $ lk(2),lk(5) $ we see that $ C(3,8,0,1,\allowbreak11,\allowbreak 10)\in lk(9) $ which is not possible. If $ lk(1)=C_9([11,a_3,2],[2,3,0],[0,8,9],[9,7,6]) $ then $ a_2\in\{3,10\} $. If $ a_2=3 $ then $ a_3=10 $ and $ lk(3)=C_9([4,a_4,11],[11,6,5],[5,10,2],[2,1,0]) $, $ lk(11)=C_9([6,5,3],[3,4,a_4],[a_4,a_5,10],[10,2,1]) $ which implies $ a_4 =9$ and then $ lk(9) $ will not possible. If $ a_2=10 $ then $ lk(2)=C_9([5,a_5,3],[3,0,1],[1,11,a_3],[a_3,a_4,9]) $ and $ lk(5)=C_9([9,a_6,4],[4,0,6],[6,11,10],[10,3,2]) $ which implies $ a_5=10,a_3=4 $ and then $ [4,9] $ form non-edge in two 4-gon which is not possible.
		\end{subcase}
	\end{case}
	\begin{case}
		When remaining 3-gons are $ \{[2,5,6],[1,9,11]\} $ then $ lk(1)=C_9([11,a_1,a_2],[a_2,a_3,\allowbreak 2],[2,3,0],[0,8,9]) $, $ lk(6)=C_9([2,a_4,a_5],[a_5,a_6,7],[7,8,0],[0,4,5]) $ and $ lk(8)=C_9([10,b_1,\allowbreak b_2],[b_2,b_3,9],[9,1,0],[0,6,7]) $ where $ a_2\in\{4,7,10\} $ and $ a_5\in\{9,11\} $. If $ \boldsymbol{a_5=9} $ then considering $ lk(9) $, we see that $ b_3 $ is either $ a_4 $ or $ a_6 $, therefore $ a_3=3 $ and then $ a_4,a_6\in\{3,11\} $. So $ 10\notin V(lk(0)), V(lk(6)) $ which means $ 10\in V(lk(9)) $ and therefore $ b_2=10 $ which implies $ C(b_1,10)\in lk(8) $ which is a contradiction. Therefore $ \boldsymbol{a_5=11} $. Considering $ lk(11) $ we see that either $ a_1 $ is equal to one of $ a_4,a_6 $ OR $ a_1\neq a_4,a_6 $. If $ a_1 $ is one of $ a_4,a_6 $ then $ a_1=10 $ and as $ a_6\neq10 $, therefore $ a_1=a_4=10,a_6=3 $. Therefore $ [3,7,6,11] $ and $ [11,1,a_2,10] $ form faces which implies $ a_2=4 $ and then after completing $ lk(4) $ we see that $ C(2,1,4,0,6)\in lk(5) $ which is a contradiction. If $ a_1\neq a_4,a_6 $ then either $ a_4 $ or $ a_6 $ is $ 9 $. If $ a_4=9 $ then completing $ lk(2) $, we see that $ b_3\notin V $. If $ a_6=9 $ then considering $ lk(2) $, we get $ a_4=10,a_2=4 $. Now $ lk(11)=C_9([9,7,6],[6,2,10],[10,d_1,a_1],[a_1,4,1]) $ which implies $ a_1=5 $ and $ lk(2)=C_9([5,d_2,3],[3,0,1],[1,4,10],[10,11,6]) $ and from here we get $ a_3=10 $. As $ [4,11] $ is an adjacent edge of two 4-gon, therefore $ [10,d_1] $ will be an adjacent edge of a 3-gon and a 4-gon which implies $ d_1=8 $ and therefore $ b_1=11,b_2=5 $ and then the number of faces at $ 5 $ is more that $ 5 $ which is a contradiction.
	\end{case}
	
	When $ \boldsymbol{\underline{ lk(8)=C_{9}([a,b,c],[c,d,7],[7,6,0],[0,1,9]) }} $ then $ a\in\{2,5,10\} $. \\
	For $ \boldsymbol{a=2} $, remaining possible of two incomplete 3-gons are $ \{[1,5,6],[7,10,11]\} $, $ \{[1,5,\allowbreak 7],[6,10,11]\} $, $ \{[1,5,10],[6,7,11]\} $, $ \{[1,6,7],[5,10,11]\} $, $ \{[1,6,10],[5,7,11]\} $, $ \{[1,7,10],[5,6,\allowbreak 11]\} $.
	\begin{case}
		When remaining two 3-gons are $ \{[1,5,6],[7,10,11]\} $ then after completing $ lk(1) $ and $ lk(6) $ we see that $ [7,10] $ form an edge and a non-edge in two 4-gon, which is a contradiction.
	\end{case}
	\begin{case}
		When remaining two 3-gons are $ \{[1,5,7],[6,10,11]\} $ then $ lk)(1)=C_9([c_1,a_2,2],[2,\allowbreak 3,0],[0,8,9],[9,a_1,c_2]) $ and $ lk(6)=C_9([10,a_4,5],[5,4,0],[0,8,7],[7,a_3,11]) $ where $ c_1,c_2 \in\{5,7\}$. If $ \boldsymbol{(c_1,c_2)=(5,7)} $ then $ a_2\in \{4,a_4\} $. If $ a_2=4 $ then considering $ lk(5) $ we see that $ a_4\in\{3,9\} $. If $ a_4=9 $ then considering $ lk(9) $ we see that $ [7,9] $ form  non-edge in two 4-gon which is a contradiction and if $ a_4=3 $ then after competing $ lk(3) $ we see that number of faces at $ 7 $ is more that 5 which is a contradiction. If $ a_2=a_4 $ then $ a_2\notin V $. If $ \boldsymbol{(c_1,c_2)=(7,5)} $ then $ a_1\in\{4,a_4\} $. If $ a_1=4 $ then after completing $ lk(3) $ and $ lk(5) $ we see that $ 11\notin V(lk(0)),V(lk(3)),V(lk(5)) $ which is a contradiction and if $ a_1 =a_4$ then completing $ lk(5) $ we get $ a_2=11 $ which implies $ [7,11] $ form an edge and a non-edge in two 4-gon which is a contradiction.
	\end{case}
	\begin{case}
		When remaining two 3-gons are $ \{[1,5,10],[6,7,11]\} $ then either $ lk(1)=C_9([5,a_2,\allowbreak 2],[2,3,0],[0,8,9],[9,a_1,10]) $ or $ lk(1)=C_9([10,a_2,2],[2,3,0],[0,8,9],[9,a_1,5]) $.
		\begin{subcase}
			When $ lk(1)=C_9([5,a_2,2],[2,3,0],[0,8,9],[9,a_1,10]) $ then $ lk(2)=C_9([c_1,\allowbreak a_4,\allowbreak 3],[3,0,1],[1,5,a_2],[a_2,a_3,c_2]) $ where $ c_1,c_2\in\{8,9\} $. If $ (c_1,c_2)=(8,9) $ then after completing $ lk(8) $ we see that $ 11 $ occur two times in $ lk(2) $ which is a contradiction. If $ (c_1,c_2)=(9,8) $ then $ lk(8)=C_9([2,a_2,a_3],[a_3,a_5,7],[7,6,0],[0,1,9]) $ and $ lk(9)=C_9([2,3,a_4],[a_4,a_6,\allowbreak a_1],[a_1,\allowbreak 10,1],[1,0,8]) $ which implies $ a_2\in\{4,11\} $. Now if $ a_2=4 $ then $ a_3=10 $ and $ a_1=11=a_4 $ i.e. $ C(11,a_6)\in lk(9) $, which is not possible, therefore $ a_2=11 $ and then after completing $ lk(11) $ and $ lk(5) $ we see that $ [3,6,5,10] $ and $ [6,11,10,*] $ form faces which is not possible.
		\end{subcase}
		\begin{subcase}
			When $ lk(1)=C_9([10,a_2,2],[2,3,0],[0,8,9],[9,a_1,5]) $ then $ lk(2)=C_9([c_1,\allowbreak a_4,\allowbreak 3],[3,0,1],[1,10,a_2],[a_2,a_3,c_2]) $ where $ c_1,c_2\in\{8,9\} $. If $ (c_1,c_2)=(8,9) $ then after completing $ lk(8) $ and $ lk(5) $ we see that $ [1,0,3,2] $ and $ [1,5,3,*] $ form faces which is not possible. If $ (c_1,c_2)=(9,8) $ then $ lk(8)=C_9([2,a_2,a_3],[a_3,a_5,7],[7,6,0],[0,1,9]) $ which implies $ a_2\in\{4,11\},a_3\in\{4,5\} $. If $ a_2=4 $ then $ a_3=5 $ and $ a_1=11=a_4 $ which is not possible, therefore $ a_2=11 $. Now if $ a_3=5 $ then number of faces at $ 5 $ is more than $ 5 $, therefore $ a_3=4 $ and then after completing $ lk(9) $ we see that $ [4,9,7,*] $ and $ [4,8,7,*] $ are faces, which is a contradiction.
		\end{subcase}
	\end{case}
	\begin{case}
		When remaining two 3-gons are $ \{[1,6,7],[5,10,11]\} $ then $ lk(6)=C_9([1,b_1,b_2],[b_2,\allowbreak b_3,5],[5,4,0],[0,8,7]) $ which implies $ b_2\notin V $.
	\end{case}
	\begin{case}
		When remaining two 3-gons are $ \{[1,6,10],[5,7,11]\} $ then $ lk(6)=C_9([d_1,a_2,5],[5,\allowbreak 4,0],[0,8,7],[7,a_1,d_2]) $ where $ d_1,d_2\in\{1,10\} $.
		\begin{subcase}
			When $ \boldsymbol{(d_1,d_2)=(1,10)} $ then $ a_2\in\{2,9\} $. If $ \boldsymbol{a_2=2} $ then $ lk(1)=C_9([6,5,\allowbreak 2],[2,3,0],[0,8,9],[9,a_3,10]), lk(5)=C_9([c_3,c_4,4],[4,0,6],[6,1,2],[2,c_1,c_2]) $ and $ lk(2)=C_9\allowbreak ([c_1,c_2,5],[5,6,1],[1,0,3],[3,c_5,c_6]) $. From here we see that $ c_1\in\{8,9\} $. Now $ c_1=8\Rightarrow c_6=9,c_2=11,c_5=7,c_3=7 $. Therefore we have $ [2,3,7,9],[0,6,7,8] $ are faces, but among these two faces has no common edge and no edges is an adjacent edge of a 3-gon and a 4-gon which is a contradiction. $ c_1=9\Rightarrow c_6=9 $ and then after completing $ lk(9) $ and $ lk(8) $ we see that $ C(7,a_1)\in lk(10) $ which is not possible. If $ \boldsymbol{a_2=9} $ then $ lk(1)=C_9([10,a_3,2],[2,3,0],[0,8,9],[9,5,6]) $ and $ a_1\in\{3,11\},a_3\in\{4,11\} $. Now $ a_3=4\Rightarrow a_1=11 $ and then $ lk(4)=C_9([3,a_4,10],[10,1,2],[2,a_5,5],[5,6,0]) $ which implies $ a_5=11 $ and considering $ lk(4) $ we get $ a_4=9 $. Now $ lk(10)=C_9([1,2,4],[4,3,9],[9,a_6,11],[11,7,6]) $ and then for any $ a_6 $, $ lk(9) $ will not possible. $ a_3=11\Rightarrow a_1=3 $ and then similarly (for $ a_3=4 $) after completing $ lk(3) $ and $ lk(10) $ we see that $ lk(9) $ will not possible.
		\end{subcase}
		\begin{subcase}
			When $ \boldsymbol{(d_1,d_2)=(10,1)} $ then $ a_1\in\{2,9\} $. If $ \boldsymbol{a_1=2} $ then $ lk(1)=C_9([10,\allowbreak a_3,9],[9,8,0],[0,3,2],[2,7,6]) $ where $ a_3\in\{4,11\} $. $ a_3=4\Rightarrow a_2=11 $ and $ lk(4)=C_9([3,a_4,\allowbreak 10],[10,1,9],[9,a_5,5],[5,6,0]) $ which implies $ a_5=11 $ and considering $ lk(8) $ we get $ a_4=7 $ and then $ lk(7) $ will not possible. $ a_3=11\Rightarrow a_2=3 $ and then similarly after completing $ lk(3) $ we see that $ lk(7) $ will not possible. Is $ \boldsymbol{a_1=9} $ then $ lk(1)=C_9([10,a_3,2],[2,3,0],[0,8,9],[9,\allowbreak7,6]) $ which implies $ a_3\in\{4,11\} $. Now $ a_3=4\Rightarrow a_2=11 $ and $ lk(4)=C_9([3,a_4,10],[10,1,2],\allowbreak[2,11,\allowbreak 5],[5,6,0]) $, but for any $ a_4 $, $ lk(a_4) $ will not possible. Similarly $ a_3=11 $ is also not possible.
		\end{subcase}
	\end{case}
	\begin{case}
		When remaining two 3-gons are $ \{[1,7,10],[5,6,11]\} $ then $ lk(6)=C_9([11,b_1,b_2],\allowbreak[b_2,b_3,7],[7,8,0],[0,4,5]) $ and $ lk(1)=C_9([c_1,a_1,9],[9,8,0],[0,3,2],[2,a_2,c_2]) $ which implies $ b_2\in\{2,3,9\} $ and $ a_1\in\{4,5,6,11\} $. Now considering $ lk(2) $, we get $ b\in\{3,a_2\} $. If $ b=3 $ then $ lk(2)=C_9([8,c,3],[3,0,1],[1,c_2,a_2],[a_2,a_3,9]) $ and then $ b_2=2 $ implies $ a_2=6,c_2=7,a_3=11,b_1=9,b_3=1 $ which implies $ |v_{lk(11)}|\leq 8 $ and for $ b_2=3,9 $, one of $ b_1,b_3 $ will be $ 4,a_1 $, respectively, is a contradiction. Similarly if $ b=a_2 $ then $ lk(2)=C_9([8,c,a_2],[a_2,c_2,1],[1,0,3],\allowbreak[3,a_3,9]) $ which implies for $ b_2=2,3,9 $, one of $ b_1,b_3 $ will be $ 8,4,a_1 $ respectively, is a contradiction.
	\end{case}
	
	For $ \boldsymbol{a=5} $, remaining possible of two incomplete 3-gons are $ \{[1,2,6],[7,10,11]\} $, $ \{[1,2,7],\allowbreak [6,10,11]\} $, $ \{[1,2,10],[6,7,11]\} $, $ \{[1,6,7],[2,10,11]\} $, $ \{[1,6,10],[2,7,11]\} $, $ \{[1,7,10],[2,6,11]\} $, $ \{[1,10,11],[2,6,7]\} $.
	\begin{case}
		When remaining two 3-gons are $ \{[1,2,6],[7,10,11]\} $ then $ lk(1)=C_9([6,b_1,b_2],[b_2,\allowbreak b_3,9],[9,8,0],[0,3,2]) $ and $ lk(6)=C_9([c_1,10,5],[5,4,0],[0,8,7],[7,a_1,2]) $, this implies $ b_2=10 $ and $ c_1,c_2\in\{1,2\} $.
		\begin{subcase}
			When $ \boldsymbol{(c_1,c_2)=(1,2)} $ then $ lk(5)=C_9([9,a_3,4],[4,0,6],[6,1,10],[10,a_2,\allowbreak 8]) $, $ lk(8)=C_9([5,10,a_2],[a_2,a_4,7],[7,6,0],[0,1,9]) $, $ lk(10)=C_9([b_3,9,1],[1,6,5],[5,8,a_2],\allowbreak [a_2,b_4,b_5]) $ and $ lk(9)=C_9([5,4,a_3],[a_3,a_5,b_3],[b_3,10,1],[1,0,8]) $. We see that $ a_2\neq 11 $, therefore at least one of $ a_3,b_3 $ is $ 11 $. But if $ a_3=11 $ then $ b_3\neq 7,11\Rightarrow lk(10) $ is not possible, therefore $ b_3=11,b_5=7,a_1=11,a_4 $. Now $ a_2\in\{2,3\} $ and we have $ a_3\neq 3 $, therefore if $ a_2\neq 3 $ implies $ 3\notin V(lk(5)),V(lk(6)),V(lk(8)) $ which is a contradiction, hence $ a_2=3 $ and then $ lk(3) $ will not possible.
		\end{subcase}
		\begin{subcase}
			When $ \boldsymbol{(c_1,c_2)=(2,1)} $ then $ a_1=11,b_3=4 $ and then $ lk(4)=C_9([3,a_2,9],\allowbreak [9,1,10],[10,a_3,5],[5,6,0]) $ which implies $ a_2=11 $. Now $ lk(5)=C_9([a_3,10,4],[4,0,6],[6,2,\allowbreak 11],[11,a_4,a_5]) $. From here we see that $ a_3=8,a_5=9 $, therefore $ [3,4,9,11],[5,9,a_4,11] $ are faces which is a contradiction.
		\end{subcase}
	\end{case}
	\begin{case}
		When remaining two 3-gons are $ \{[1,2,7],[6,10,11]\} $ then $ lk(1)=C_9([7,b_1,b_2],[b_2,\allowbreak b_3,9],[9,8,0],[0,3,2]) $ and $ lk(6)=C_9([10,a_2,5],[5,4,0],[0,8,7],[7,a_1,11]) $ where $ b_2\in\{4,\allowbreak10\} $. Now If $ b_2=4 $ then considering $ lk(4) $ and $ lk(7) $, we get $ b_3=5 $ which implies $ C(5,4,1,0,8)\in lk(9) $ which is a contradiction. If $ b_2=10 $ and if $ b_3=11 $ then after completing $ lk(10) $ and $ lk(7) $, we see that $ a_2\notin V $ and $ b_3\neq 11 $ then also considering $ lk(10) $, we get $ a_2\notin V $.
	\end{case}
	\begin{case}
		When remaining two 3-gons are $ \{[1,2,10],[6,7,11]\} $ then $ lk(1)=C_9([10,a_1,a_2],\allowbreak [a_2, a_3,9],[9,8,0],[0,3,2]) $ and $ lk(6)=C_9([11,b_1,b_2],[b_2,b_3,5],[5,4,0],[0,8,7]) $ where $ b_2\in \{4,6,\allowbreak
		7,11\} $.\\
		\underline{\textbf{If} $ \boldsymbol{a_2=4} $} then $ lk(4)=C_9([0,6,5],[5,10,1],[1,9,a_3],[a_3,a_4,3]) $ which implies $ a_1=5, a_3\in \{7,11\} $ and $ lk(5)=C_9([8,a_5,10],[10,1,4],[4,0,6],[6,b_2,9]) $ i.e. $ b_3=9 $. Now if $ a_3=7 $ then $ a_4=11=a_5 $ and then $ lk(11) $ will not possible, therefore $ a_3=11 $ and then $ a_4=7 $. Now $ lk(11)=C_9([6,b_2,b_1],[b_1,b_4,9],[9,1,4],[4,3,7]) $ which implies $ b_2=2 $ and then after completing $ lk(2) $, we see that $ b_1=3 $ i.e. $ C(3,2,6,7)\in lk(11) $ which is not possible.\\
		\underline{\textbf{If} $ \boldsymbol{a_2=6} $} then $ a_1,a_3\in\{5,11\} $. If $ (a_1,a_3)=(11,5) $ then $ C(5,6,1,0,8)\in lk(9) $ which is not possible, therefore $ (a_1,a_3)=(5,11) $ and then $ lk(5)=C_9([c_1,a_5,4],[4,0,6],[6,1,10],[10,a_4,\allowbreak c_2]) $ where $ c_1,c_2\in \{8,9\} $. When $ (c_1,c_2)=(8,9) $ then $ lk(8)=C_9([5,4,a_5],[a_5,a_6,7],[7,6,0],\allowbreak [0,1,9]) $ which implies $ a_5=2 $ and $ a_6\neq 11 $ i.e. $ a_4=11 $. Now we have $ 3\notin V(lk(5)),V(lk(6)) $, therefore $ a_6=3 $ and then after completing $ lk(2) $ we see that $ C(2,*)\in lk(10) $ which is not possible. When $ (c_1,c_2)=(9,8) $ then $ lk(8)=C_9([5,10,a_4],[a_4,a_6,7],[7,6,0],[0,1,9]) $ which implies $ a_4=3,a_5=11,a_6=2 $ and then after completing $ lk(3) $ we see that $ C(2,*)\in lk(4) $ which is not possible.\\
		\underline{\textbf{If} $ \boldsymbol{a_2=7} $} then either $ lk(7)=C_9([11,10,1],[1,9,a_3],[a_3,a_4,8],[8,0,6]) $ or $ lk(7)=C_9([11,9,\allowbreak 1],[1,10,a_1],[a_1,a_4,8],[8,0,6]) $. When $ lk(7)=C_9([11,10,1],[1,9,a_3],[a_3,a_4,8],[8,0,6]) $ then $ a_3=4,a_1=11 $ and then $ lk(8)=C_9([5,a_5,a_4],[a_4,4,7],[7,6,0],[0,1,9]) $ and $ lk(4)=C_9([3,8,\allowbreak 7],[7,1,9],[9,a_6,5],[5,6,0]) $ which implies $ [5,8,9] $ and $ [4,5,a_6,9] $ are faces which is a contradiction. When $ lk(7)=C_9([11,9,1],[1,10,a_1],[a_1,a_4,8],[8,0,6]) $ then $ a_1=4,a_3=11 $ and then $ lk(8)=C_9([5,a_5,a_4],[a_4,4,7],[7,6,0],[0,1,9]) $ and $ lk(4)=C_9([3,8,7],\allowbreak[7,1,10],[10,a_6,\allowbreak 5],[5,6,0]) $ which implies $ a_4=3 $ and then $ lk(3)=C_9([4,7,8],[8,5,a_5],[a_5,\allowbreak a_7,2],[2,1,0]) $. From here we see that $ a_5=11,a_7=10 $ and then $ lk(11) $ will not possible.\\
		\underline{\textbf{If} $ \boldsymbol{a_2=11} $} then $ b_1 $ is either $ a_1 $ or $ a_3 $ and then $ b_1\notin V $.
	\end{case}
	\begin{case}
		When remaining two 3-gons are $ \{[1,6,7],[2,10,11]\} $ then $ lk(6)=C_9([1,b_1,b_2],[b_2,\allowbreak b_3,5],[5,4,0],[0,8,7]) $ and either $ lk(1)=C_9([6,b_2,2],[2,3,0],[0,8,9],[9,a_1,7]) $ or $ lk(1)=C_9([7,a_1,2],[2,3,0],[0,8,9],[9,b_2,6]) $.
		\begin{subcase}
			When $ lk(1)=C_9([6,b_2,2],[2,3,0],[0,8,9],[9,a_1,7]) $ then $ b_1=2,b_2=10 $ and $ a_1=11 $. Now $ lk(8)=C_9([5,c_1,c_2],[c_2,c_3,7],[7,6,0],[0,1,9]) $ which implies $ c_1=11 $. Now we have $ [1,9,7,11] $ and $ [5,8,c_2,11] $ are faces and none of $ [1,11],[7,11],[5,11] $ are not an adjacent edge of a 3-gon and a 4-gon, therefore to complete $ lk(11) $, $ [11,c_2] $ must be an adjacent edge of a 3-gon and a 4-gon i.e. $ c_2\in \{2,10\} $. Now $ lk(7)=C_9([1,9,11],[11,c_4,c_3],[c_3,c_2,8],\allowbreak[8,0,6]) $ and $ lk(5)=C_9([9,10,6],[6,0,4],[4,b_4,11],[11,c_2,\allowbreak8]) $ which implies $ b_3=9,c_2=2,b_4=7,c_3=3 $ and then number of faces at $ 7 $ is more that $ 5 $ which is a contradiction.
		\end{subcase}
		\begin{subcase}
			\begin{sloppypar}
			When $ lk(1)=C_9([7,a_1,2],[2,3,0],[0,8,9],[9,b_2,6]) $ then $ b_2=10 $ and $ lk(5)=C_9([d_1,a_3,4],[4,0,6],[6,10,b_3],[b_3,a_2,d_2]) $ where $ d_1,d_2\in \{8,9\} $. If $ (d_1,d_2)=(8,9) $ then $ lk(8)=C_9([5,4,a_3],[a_3,a_4,7],[7,6,0],\allowbreak[0,1,9]) $ and then $ a_3=11 $ (as if $ a_3=2 $ then $ [2,4] $ will form an edge in a 3-gon which is not possible) and then $ 11\notin V(lk(0)),V(lk(1)),V(lk(6)) $ which is a contradiction. If $ (d_1,d_2)=(9,8) $ then $ lk(8)=C_9([5,b_3,a_2],[a_2,a_4,7],[7,6,0],\allowbreak[0,1,9]) $ and $ lk(7)=C_9([1,2,a_1],[a_1,a_5,a_4],[a_4,a_2,8],[8,0,6]) $ where $ a_2\in \{2,3,11\} $. If $ a_2=2 $ then $ [2,8,7,a_4] $ and $ [2,1,7,a_1] $ are two faces of the polyhedron, which is not possible. If $ a_2=3 $ then after completing $ lk(2) $ we see that $ 10\notin V(lk(0)),V(lk(7)),V(lk(8)) $ which is a contradiction and $ a_2=11 $ implies $ b_3=11 $ and then $ C(8,9,a_3,4,0,6,10,11)\in lk(5) $ which is not possible.
		\end{sloppypar}
		\end{subcase}
	\end{case}
	\begin{case}
		When remaining two 3-gons are $ \{[1,6,10],[2,7,11]\} $ then either $ lk(1)=C_9([6,a_2,\allowbreak 2],[2,3,0],[0,8,9],[9,a_1,10]) $ or $ lk(1)=C_9([10,a_2,2],[2,3,0],[0,8,9],[9,a_1,6]) $.
		\begin{subcase}
			When $ lk(1)=C_9([6,a_2,2],[2,3,0],[0,8,9],[9,a_1,10]) $ then $ a_2\in\{5,7\} $.\\
			\underline{If $ \boldsymbol{a_2=5} $} then $ lk(6)=C_9([1,2,5],[5,4,0],[0,8,7],[7,a_3,10]) $ and $ lk(2)=C_9([c_1,a_5,3],[3,0,\allowbreak 1],[1,6,5],[5,a_4,c_2]) $ where $ c_1,c_2\in\ {7,11} $. If $ (c_1,c_2)=(7,11) $ then $ a_3 $ is either $ a_5 $ or $ 8 $ or $ 11 $. If $ a_3=11 $ then $ a_5=10 $ and then $ [2,3,10,7],[10,6,7,11] $ form faces which is not possible. If $ a_3=a_5 $ then $ a_3=9 $ and $ a_4=8 $ and then after complete $ lk(8) $ we see that $ [7,8,11,*] $ and $ [2,7,11] $ are faces which is not possible. If $ a_3=8 $ then after complete $ lk(7) $ we see that $ [6,7,9,10] $ and $ [1,9,a_1,10] $ form faces which is not possible. Therefore $ (c_1,c_2)=(11,7) $ and then $ a_4\in\{8,9\} $. If $ a_4=8 $ then after completing $ lk(7) $ we see that $ C(5,9,1,0,6,7,2)\in lk(8) $ which is a contradiction. If $ a_4=9 $ then after completing $ lk(7) $ we see that $ [6,7,9,10] $ and $ [1,9,a_1,10] $ are form faces which is not possible.\\
			\underline{If $ \boldsymbol{a_2=7} $} then $ lk(6)=C_9([10,a_3,5],[5,4,0],[0,8,7],[7,2,1]) $ where $ a_3\in\{3,11\} $. If $ a_3=3 $ then $ lk(3)=C_9([4,a_4,10],[10,6,5],[5,a_5,2],[2,1,0]) $. We see that $ a_1,a_4\in lk(10) $ and $ a_1\neq a_4 $, so if $ a_4=11 $ then $ a_1\neq 11 $ and then $ 11\notin V(lk(0)),V(lk(1)),V(lk(6)) $ which is a contradiction, therefore $ a_1=11 $ and it implies $ a_5=11,a_4\in\{7,8\} $. Now considering $ lk(8) $, we see that $ a_4=7 $ and then $ lk(7) $ will not possible. If $ a_3=11 $ then $ a_1=4 $ and $ lk(5)=C_9([8,a_5,4],[4,0,6],[6,10,11],[11,a_4,9]),lk(8)=C_9([5,4,a_5],[a_5,a_6,7],[7,6,0],[0,1,9]) $ and then $ a_5\notin V $.
		\end{subcase}
		\begin{subcase}
			When $ lk(1)=C_9([10,a_2,2],[2,3,0],[0,8,9],[9,a_1,6]) $ then $ lk(6)=C_9([10,\allowbreak a_3, 5],[5,4,0],[0,8,7],[7,9,1]) $ which implies $ a_1=7,a_2\in\{4,11\},a_3\in\{3,11\} $. Now $ a_2=4\Rightarrow a_3=11 $ and after completing $ lk(4) $ we get $ C(2,4,0,6,10,11)\in lk(5) $ and $ a_2=11\Rightarrow a_3=3 $ and after completing $ lk(3) $ we get $ C(5,3,0,1,10,11)\in  lk(2)  $ which are not possible.
		\end{subcase}
	\end{case}
	\begin{case}
		When remaining two 3-gons are $ \{[1,7,10],[2,6,11]\} $ then $ c\in\{2,3,11\} $ and $ lk(1)=C_9([c_1,a_1,9],[9,8,0],[0,3,2],[2,a_2,c_2]) $ and $ lk(6)=C_9([c_3,a_3,7],[7,8,0],[0,4,5],[5,\allowbreak a_4,c_4]) $ where $ \{c_1,c_2\}=\{7,10\} $ and $ \{c_3,c_4\}=\{2,11\} $. Now $ c=2 $ implies one $ b,d $ is 3 by considering $ lk(2) $. $ b=3 $ implies $ \{d,a_2\}=\{6,11\} $, by considering $ lk(2) $, but $ d\neq 6 $, therefore $ (d,a_2)=(11,6) $ i.e. $ 10\notin V(lk(0)), V(lk(8)) $, therefore one of $ a_3,a_4 $ is 10. As [2,8,7,11] and [2,3,5,8] are faces, therefore neither $ c_3 $ nor $ c_4 $ is 2, is a contradiction. Similarly $ d\neq 3 $. If $ c=11 $ then $ c_3=11 $ implies $ d=6 $ and $ c_4=11 $ implies $ b=8 $, a contradiction, therefore $ c=3 $. Now considering $ lk(5) $, either $ a_4=9 $ or $ b\in\{4,a_4\} $. $ b=a_4 $ implies $ b=10 $ and then $ a_3\in\{1,9\} $. Now $ a_3=9 $ implies $ c_1=10 $ and then $ lk(9)=C_9([5,a_5,7],[7,6,11],[11,10,1],[1,0,8]) $ i.e. $ a_1=11=c_3 $ and then considering $ lk(7) $, we get $ \{d,a_5\}=\{1,10\} $, is a contradiction. $ a_3=1 $ implies $ c_3=2,c_2=7,a_2=6,c_1=10,c_4=11 $ and then $ lk(7)=C_9([10,a_5,d],[d,3,8],[8,0,6],[6,2,1]) $ and $ lk(3)=C_9([4,7,8],[8,5,10],[10,a_6,2],[2,1,0]) $, therefore $ d=4 $ i.e. $ 11\notin V(lk(0)), V(lk(8)) $, therefore $ a_1=11,a_5=11,a_6=11 $ and then [4,7,10,11], [2,3,10,11] and [1,10,11,9] are faces, is a contradiction. $ b=4 $ implies $ lk(5)=C_9([8,3,4],[4,0,6],[6,c_4,a_4],[a_4,a_5,9]) $ and $ lk(3)=C_9([0,1,2],[2,a_6,d],[d,7,8],[8,5,4]) $ which implies $ d\neq 11 $ i.e. $ 11\notin V(lk(0)),V(lk(8)) $ i.e. $ a_6=11 $. Now $ lk(2)=C_9([6,a_7,a_2],[a_2,c_2,1],[1,0,3],[3,d,11]) $ which implies $ a_1=11,d=10 $ and then $ c_2=7,c_1=10 $ and therefore $ a_2\notin V $. $ a_4=9 $ implies $ lk(5)=C_9([9,c_4,6],[6,0,4],[4,a_5,b],[b,3,8]) $ and $ lk(9)=C_9([5,6,c_4],[c_4,a_6,a_1],[a_1,c_1,1],[1,0,8]) $. Now $ c_4=2 $ implies $ lk(2)=C_9([6,5,9],[9,a_1,3],[3,0,1],[1,c_2,11]) $ i.e. $ a_6=3,a_2=11 $ and then $ a_1\notin V $, therefore $ c_4=11,c_3=2 $. Now for $ b=2 $, considering $ lk(2) $, $ \{a_2,a_5\}=\{6,11\} $, a contradiction, therefore $ b=10 $ and $ a_1=4 $ and then considering $ lk(4) $, $ c_1=a_5=7, a_6=3 $ and then $ lk(3) $ will not possible.
	\end{case}
	\begin{case}
		When remaining two 4-gons are $ \{[1,10,11],[2,6,7]\} $ then $ lk(1)=C_9([10,a_2,2],[2,\allowbreak 3,0],[0,8,9],[9,a_1,11]) $ and $ lk(6)=C_9([2,b_1,b_2],[b_2,b_3,5],[5,4,0],[0,8,7]) $ where $ a_2\in\{4,5,\allowbreak 6,7\} $.\\
		\underline{If $ \boldsymbol{a_2=4} $} then $ lk(2)=C_9([6,b_2,3],[3,0,1],[1,10,4],[4,b_4,7]) $ and $ lk(4)=C_9([3,b_5,10],[10,\allowbreak1,\allowbreak 2],[2,7,5],[5,6,0]) $ which implies $ b_1=3, b_4=5,b_2=11 $ and then after completing $ lk(5) $ we see that $ [5,7,*,8] $ and $ [0,6,7,8] $ are faces of polygon, which is not possible.\\
		\underline{If $ \boldsymbol{a_2=5} $} then after completing $ lk(2) $, $ lk(5) $ will not possible.\\
		\underline{If $ \boldsymbol{a_2=6} $} then $ b_1=1,b_2=10 $ and $ b_3\in\{3,9\} $. When $ b_3=3 $ then after completing $ lk(3) $, $ lk(5) $ will not possible and if $ b_3=9 $ then $ lk(9)=C_9([5,6,10],[10,b_4,a_1],[a_1,11,1],[1,0,8]) $ which implies $ a_1\in\{4,7\} $. But if $ a_1=4 $ then after completing $ lk(4) $ and $ lk(5) $ we see that $ [2,11] $ is an adjacent edge of a 3-gon and a 4-gon which is a contradiction and if $ a_1=7 $ then after completing $ lk(7) $, we get $ [1,2,6,10] $ and $ [2,7,9,10] $ are faces of polyhedron, which is a contradiction.\\
		\underline{If $ \boldsymbol{a_2=7} $} then $ a_1=4 $ and $ lk(4)=C_9([3,a_3,9],[9,1,11],[11,a_4,5],[5,6,0]),lk(9)=C_9([5,a_5,\allowbreak a_3],[a_3,3,4],[4,11,1],[1,0,8]) $ which implies $ a_3=10,a_5=2 $ and then $ [2,5,9,10],[1,2,7,10] $ are two faces of polyhedron which is a contradiction.
	\end{case}
	
	For $ \boldsymbol{a=10} $, remaining possible of two incomplete 3-gons are $ \{[1,2,5],[6,7,11]\} $, $ \{[1,2,6],\allowbreak [5,7,11]\} $, $ \{[1,2,7],[5,6,11]\} $, $ \{[1,5,6],[2,7,11]\} $, $ \{[1,5,7],[2,6,11]\} $, $ \{[1,5,11],[2,6,7]\} $, $ \{[1,\allowbreak6,7],[2,5,11]\} $, $ \{[1,6,11],[2,5,7]\} $, $ \{[1,7,11],[2,5,6]\} $.
	\begin{case}
		When remaining two 3-gons are $ \{[1,2,5],[6,7,11]\} $ then $ lk(1)=C_9([5,a_1,a_2],[a_2,\allowbreak a_3,9],[9,8,0],[0,3,2]) $ and $ lk(6)=C_9([11,b_1,b_2],[b_2,b_3,5],[5,4,0],[0,8,7]) $ where $ a_2\in\{7,11\} $.\\
		\underline{If $ \boldsymbol{a_2=7} $} then one of $ a_1,a_3 $ must e $ 11 $. If $ a_1=11 $ then after completing $ lk(7) $ and $ lk(4) $ we get $ C(8,0,1,7,4,5,10)\in lk(9) $ and if $ a_2=11 $ then similarly $ C(10,4,7,6,0,1,9)\in lk(8) $ which are not possible.\\
		\underline{If $ \boldsymbol{a_2=11} $} then $ a_3\in 4,7 $. If $ a_3=4 $ then completing $ lk(4) $, we get $ [4,5,*,11],[5,1,11,a_1] $ are faces of polyhedron which is a contradiction. If $ a_3=7 $ then either $ lk(5)=C_9([2,b_2,6],\allowbreak[6,0,\allowbreak 4],[4,b_4,a_1],[a_1,11,1]) $ or $ lk(5)=C_9([2,b_4,b_3],[b_3,b_2,6],[6,0,4],\allowbreak[4,11,1]) $. If $ lk(5)=C_9([2,\allowbreak b_2,6],[6,0,4],[4,b_4,a_1],[a_1,11,1]) $ then $ b_2=9 $ which implies $ [1,9,7,11],[6,9,b_1,11] $ are two faces of polyhedron which is not possible. If $ lk(5)=C_9([2,b_4,b_3],[b_3,b_2,6],[6,0,4],\allowbreak[4,11,1]) $ then $ b_4=7 $ and to complete $ lk(9) $, $ b_3=9,b_2=10 $ and then after completing $ lk(7),lk(11) $, we see that $ lk(2) $ will not possible.
	\end{case}
	\begin{case}
		When remaining two 3-gons are $ \{[1,2,6],[5,7,11]\} $ then $ lk(1)=C_9([6,b_1,b_2],[b_2,\allowbreak b_3, 9],[9,8,0],[0,3,2]) $ and $ lk(6)=C_9([a_3,a_4,5],[5,4,0],[0,8,7],[7,a_1,a_2]) $ which implies $ b_1\in \{5,7\} $ and $ b_2=11 $. \\
		\underline{If $ \boldsymbol{b_1=5} $} then $ b_3=4 $ and then $ lk(4)=C_9([3,c_1,c_2],[c_2,1,c_3],[c_3,c_4,5],[5,6,0]) $. We have $ [0,4,5,6] $ and $ [1,9,4,11] $ are faces of polyhedron, therefore $ 4\notin V(lk(7)),V(lk(8)) $, i.e. $ c_2,c_3\in \{9,11\} $ and $ c_1,c_4\in\{2,10\} $. But $ c_1\neq 2 $, therefore $ c_1=10,c_4=2 $. Now $ lk(9)=C_9([10,c_5,c_6],[c_6,c_7,4],[4,11,1],[1,0,8]) $ which implies $ c_2=11,c_2=9,c_6=2,c_7=5 $ and then $ lk(2)=C_9([6,10,9],[9,4,5],[5,c_8,3],[3,0,1]) $ which implies $ c_5=6 $ and then after completing $ lk(5) $ we see that $ c_8=7 $ and then $ lk(7) $ will not possible.\\
		\underline{If $ \boldsymbol{b_1=7} $} then $ b_3=4 $ and $ lk(4)=C_9([3,c_1,c_2],[c_2,1,c_3],[c_3,c_4,5],[5,6,0]) $ and $ c_2,c_3\in\{9,11\} $. If $ (c_2,c_3)=(9,11) $ then $ (c_1,c_4)=(7,10) $ which implies $ lk(7) $ is not possible. If $ (c_2,c_3)=(11,9) $ then $ c_1=10 $ and $ c_4\in \{2,7\} $. But if $ c_4=7 $ then considering $ lk(7) $ we see that $ [0,1,9,8] $ and $ [7,8,*,9] $ are two face of the polyhedron which is not possible, therefore $ c_4=2 $. Now $ lk(9)=C_9([10,c_5,2],[2,5,4],[4,11,1],[1,0,8]) $ and $ lk(2)=C_9([6,10,9],[9,4,5],[5,c_6,3],[3,0,1]) $ which implies $ c_5=6 $ and then $ lk(5)=C_9([c_6,3,2],[2,9,\allowbreak 4],[4,0,6],[6,c_7,c_8]) $. From here we see that $ c_6=7,c_8=11,c_7=a_3=1 $ and then after completing $ lk(7) $ we get $ c(3,7,6,0,1,9,10)\in lk(8) $ which is a contradiction.
	\end{case}
	\begin{case}
		When remaining two 3-gons are $ \{[1,2,7],[5,6,11]\} $ then $ lk(1)=C_9([7,a_1,a_2],[a_2,\allowbreak a_3,9],[9,8,0],[0,3,2]) $ and $ lk(6)=C_9([11,b_1,b_2],[b_2,b_3,7],[7,8,0],[0,4,5]) $ which implies $ a_2\in \{4,5,11\} $. \\
		\underline{If $ \boldsymbol{a_2=4} $} then $ a_1\in\{5,10,11\} $ and considering $ lk(4) $ we see that one of $ a_1,a_3 $ is $ 5 $. Now if $ a_1=5 $ then $ lk(4)=C_9([3,10,11],[11,9,1],[1,7,5],[5,6,0]) $ and $ lk(7)=C_9([2,b_2,6],[6,0,8],\allowbreak [8,a_6,5],[5,4,1]) $. Now considering $ lk(8) $, we see that $ b_2=10 $ and then $ [3,4,11,10],[10,6,11,\allowbreak *] $ are faces of the polyhedron, which is not possible. If $ a_1\neq 5 $ then $ lk(4)=C_9([3,a_4,a_1],[a_1,\allowbreak 7,1],[1,9,5],[5,6,0]) $ and $ lk(7)=C_9([2,a_5,8],[8,0,6],[6,a_6,a_1],[a_1,4,1]) $ which implies $ a_1=b_3,a_6=b_2 $ and then $ a_1=10,a_6\in \{3,9\} $, but considering $ lk(9) $ we get $ a_6=3 $ and then considering $ lk(3) $ we see that $ [3,4,*,10],[3,4,10,a_4] $ are faces of the polyhedron, which is not possible.\\
		\underline{If $ \boldsymbol{a_2=5} $} then $ a_1\in\{4,10,11\} $. But for $ a_1=4,11 $, we see that $ a_3\notin V $. Now if $ a_1=10 $ then either $ lk(5)=C_9([11,9,1],[1,7,10],[10,a_4,4],[4,0,6]) $ or $ lk(5)=C_9([11,a_4,10],[10,7,1],[1,\allowbreak 9,4],[4,0,6]) $. When $ lk(5)=C_9([11,9,1],[1,7,10],[10,a_4,4],[4,0,6]) $ then $ a_3=11,a_4\in\{2,8\} $ and $ lk(4)=C_9([3,c_1,c_2],[c_2,c_3,a_4],[a_4,10,5],[5,6,0]) $ and then considering $ lk(2)$, $lk(8) $ we see that for any $ a_4, c_3=7 $ and then considering $ lk(9) $, we get $ c_1=9,c_2=11 $. Now $ lk(9)=C_9([10,c_4,3],[3,4,11],[11,5,1],[1,0,8]) $. Now if $ a_4=2 $ then $ C(2,1,0,4,11,9,10)\in lk(3) $ which is not possible and if $ a_4=8 $ then after completing $ lk(8) $ we get $ c_4=6 $ and then $ lk(6) $ will not possible. When $ lk(5)=C_9([11,a_4,10],[10,7,1],[1, 9,4],[4,0,6]) $ then $ a_3=4 $ and $ lk(9)=C_9([10,c_1,c_2],[c_2,11,4],[4,5,1],[1,0,8]) $ which implies $ c_2=2 $ and then after completing $ lk(2) $ and $ lk(8) $ we get $ [4,11,6,*],[4,0,6,5] $ are two face of the polyhedron which is not possible.\\
		\underline{If $ \boldsymbol{a_2=11} $} then $ a_3\in\{4,5,6\} $. If $ \boldsymbol{a_3=4} $ then $ lk(4)=C_9([3,a_4,11],[11,1,9],[9,a_5,5],[5,6,\allowbreak 0]) $ which implies $ a_5\in\{2,7\} $ and from $ lk(9) $ we see that $ [10,a_5] $ will form a non-edge in a 4-gon, therefore one of $ a_1,a_4 $ must be $ 10 $. Now $ lk(9)=C_9([10,a_6,a_5],[a_5,5,4],[4,11,1],[1,0,\allowbreak8]) $ and considering $ lk(7) $ we get $ a_5=2 $ and then $ lk(2)=C_9([7,a_7,5],[5,4,9],[9,10,3],[3,0,\allowbreak1]) $. From here we see that $ a_6=3 $ and then $ a_1=10,a_7\in\{6,8\} $. Now $ lk(5)=C_9([11,a_8,a_7],[a_7,7,\allowbreak 2],[2,9,4],[4,0,6]) $ which implies $ a_7=8,a_8=10 $ and then after completing $ lk(7) $ we see that $ [3,4,*,6] $ and $ [4,0,6,5] $ are faces of the polyhedron which is not possible. If $ \boldsymbol{a_3=5} $ then $ a_1\in\{4,10\} $. When $ a_1=10 $ then $ lk(7)=C_9([2,a_4,8],[8,0,6],[6,\allowbreak b_2,10],[10,11,1]) $ which implies $ b_2=3,b_3=10 $ as if $ b_2=9 $ then $ C(9,6,5)\in lk(11) $ which is not possible. Now $ lk(2) $. Therefore $ b_1=2 $ and then completing $ lk(2),lk(4) $ we see that $ C(4,2,3,6,5)\in lk(11) $ which is a contradiction. If $ a_1=4 $ then $ lk(4)=C_9([3,a_4,11],[11,1,7],[7,a_5,5],[5,6,0]) $ which implies $ a_4=10 $ and $ lk(11)=C_9([6,b_2,10],\allowbreak[10,3,4],[4,7,1],[1,9,5]) $. From here we see that $ b_1=10,b_2=2 $ and then after completing $ lk(5) $, we se that $ lk(2) $ will not possible. If $ \boldsymbol{a_3=6} $ then $ b_1=1,b_2=9,a_1=10 $ and then after completing $ lk(7) $, we see that $ [8,9,10],[7,8,*,10] $ are two faces of the polyhedron, which is a contradiction.
	\end{case}
	\begin{case}
		When remaining two 3-gons are $ \{[1,5,6],[2,7,11]\} $ then $ lk(6)=C_9([1,b_1,b_2],[b_2,\allowbreak b_3, 7],[7,8,0],[0,4,5]) $ and $ lk(1)=C_9([a_3,a_4,2],[2,3,0],[0,8,9],[9,a_1,a_2]) $ where $ a_2,a_3\in\{5,6\},\allowbreak  b_1\in\{2,9\} $ and $ b_2 $ is either $ a_1 $ or $ a_4 $. In both cases $ b_2=10 $, therefore $ b_2=10 $. Now when $ \boldsymbol{b_1=2} $ then $ a_1\in\{7,11\} $, but $ a_1 $ can not be $ 7 $ as then $ [0,6,7,8], [1,5,7,9] $ form two faces of the polyhedron which has not common vertex except $ 7 $ and on edges of these two faces is an adjacent edge with a 3-gon and that implies $ lk(7) $ will not complete, therefore $ a_1=11 $ and then after completing $ lk(2) $ and $ lk(7) $ we see that $ lk(9) $ will not possible. When $ \boldsymbol{b_1=9} $ then $ (a_1,a_2,a_3)=(10,6,5) $ which implies $ C(8,0,1,6,10)\in lk(9) $ which is not possible.
	\end{case}
	\begin{case}
		When remaining two 3-gons are $ \{[1,5,7],[2,6,11]\} $ then either $ lk(1)=C_9([5,a_2,\allowbreak2],\allowbreak [2,3,0],[0,8,9],[9,a_1,7]) $ or $ lk(1)=C_9([7,a_2,2],[2,3,0],[0,8,9],[9,a_1,5]) $. When $ lk(1)=C_9([5,a_2,2],[2,3,0],[0,8,9],[9,a_1,7]) $ then $ lk(6) $ has two possibilities. In both cases, after completing $ lk(7) $ we see that either $ a_1=10 $ or $ a_1\notin V $, which is a contradiction. When $ lk(1)=C_9([7,a_2,2],[2,3,0],[0,8,9],[9,a_1,5]) $ then $ lk(6)=C_9([11,a_3,5],[5,4,0],[0,8,7],[7,\allowbreak1,\allowbreak 2]) $ which implies $ a_2=6,a_3=10 $ and then $ lk(5)=C_9([1,9,4],[4,0,6],[6,11,10],[10,a_4,\allowbreak7]) $. From there we see that $ a_1=4 $ and then after completing $ lk(7) $ we get $ a_4=3 $ which implies $ lk(3) $ is not possible.
	\end{case}
	\begin{case}
		When remaining two 3-gons are $ \{[1,5,11],[2,6,7]\} $ then $ lk(6)=C_9([2,b_1,b_2],[b_2,\allowbreak b_3, 5],[5,4,0],[0,8,7]) $ and $ lk(1)=C_9([c_1,a_2,2],[2,3,0],[0,8,9],[9,a_1,c_2]) $ where $ c_1,c_2\in\{5,\allowbreak 11\} $.
		\begin{subcase}
			When $ \boldsymbol{(c_1,c_2)=(5,11)} $ then $ a_1\in\{4,7\} $. If $ a_1=4 $ then $ lk(4)=C_9([3,a_3,\allowbreak 11],[11,1,9],[9,a_4,5],[5,6,0]) $ and $ lk(5)=C_9([11,a_4,6],[6,0,4],[4,9,7],[7,2,1]) $ which implies $ a_4\notin V $. Therefore $ a_1=7 $ and then $ lk(7)=C_9([2,a_3,9],[9,1,11],[11,a_4,8],[8,0,6])$, $  lk(8)=C_9([10,a_5,a_4],[a_4,11,7],[7,6,0],[0,1,9]) $ and $ lk(9)=C_9([10,a_6,a_3],[a_3,2,7],[7,11,\allowbreak 1],[1,0,8]) $ which implies $ a_3\in\{3,4,5\} $. Now if $ a_3=3 $ then we will not get 9 vertices to complete $ lk(3) $, if $ a_3=4 $ then $ a_4=3 $ and then after complete $ lk(4) $, we see that $ [3,4,9,10],[3,8,10,a_5] $ form faces of the polyhedron which is not possible and if $ a_3=5 $ then $ lk(5) $ will not possible.
		\end{subcase}
		\begin{subcase}
			\begin{sloppypar}
			When $ (c_1,c_2)=(11,5) $ then $ a_1\in\{4,6,7\} $. Now if \underline{$ \boldsymbol{a_1=4} $} then $ lk(4)=C_9([3,a_3,a_4],[a_4,a_5,9],[9,1,5],[5,6,0]) $, $ lk(5)=C_9([11,a_6,a_6],[a_7,a_8,6],[6,0,4],[4,9,1]) $ and $ lk(9)=C_9([10,a_9,a_5],[a_5,a_4,4],[4,5,1],[1,0,8]) $ which implies $ a_7\in\{3,10\} $. Now if $ a_7=3 $ then after completing $ lk(3) $ we get $ [2,3,5,11],[2,1,11,a_2] $ form faces of the polyhedron which is not possible. If $ a_7=10 $ then $ b_3=11,a_4=11,a_5=7 $ and then after completing $ lk(7),lk(11) $ we get $ a_2=10 $ which is a contradiction. \underline{If $ \boldsymbol{a_1=6} $} then  $ lk(6)=C_9([2,b_1,9],[9,1,5],[5,4,0],[0,8,7]) $ and from here we see that $ b_1=3 $. Now $ lk(9)=C_9([10,b_4,\allowbreak 3],[3,2,6],[6,5,1],[1,0,8]) $ which implies $ b_4=11,a_2=10 $ and then $ lk(3)=C_9([4,b_5,11],[11,\allowbreak 10,9],[9,6,2],[2,1,0]) $, $ lk(1)=C_9([5,6,9],[9,8,0],[0,3,2],[2,10,11]) $ and $ lk(11)=C_9([5,b_6,\allowbreak b_5],[b_5,4,3],[3,9,10],[10,2,1]) $ which implies $ b_5=7 $. Now $ lk(2)=C_9([7,b_7,10],[10,11,1],[1,\allowbreak 0,3],[3,9,6]) $, $ lk(7)=C_9([2,10,4],[4,3,11],[11,5,8],[8,0,6]) $. From here we see that $ b_6=8,b_7=4 $. Therefore $ lk(2)=C_9([7,4,10],[10,11,1],[1,0,3],[3,9,6]) $, $ lk(10)=C_9([8,b_8,4],[4,\allowbreak 7,2],[2,1,11],[11,3,9]) $, $ lk(8)=C_9([10,4,5],[5,11,7],[7,6,0],[0,1,\allowbreak9]) $ which implies $ b_8=5 $ and then $ lk(5)=C_9([1,9,6],[6,0,4],[4,10,8],[8,7,11]) $, $ lk(9)=C_9([10,11,3],[3,2,6],[6,5,1],\allowbreak [1,0,8]) $, $ lk(4)=C_9([3,11,7],[7,2,10],[10,8,5],[5,6,0]) $. Let this be the map $ \boldsymbol{K_2} $. $ \boldsymbol{K_2} $ is isomorphic to $ \boldsymbol{KO_{[(3,4^4)]}} $ under the map $ (1,6)(2,5)(3,4)(7,9) $. \underline{If $ \boldsymbol{a_1=7} $} then $ lk(7)=C_9([2,b_4,9],[9,\allowbreak 1,5],[5,b_5,8],[8,0,6]) $ and then after completing $ lk(2) $, we see that $ b_2\notin V $.
		\end{sloppypar}
		\end{subcase}
	\end{case}
	\begin{case}
		When remaining two 3-gons are $ \{[1,6,7],[2,5,11]\} $ then $ lk(6)=C_9([1,b_1,b_2],[b_2,\allowbreak b_3, 5],[5,4,0],[0,8,7]) $ and $ lk(1)=C_9([c_1,b_2,2],[2,3,0],[0,8,9],[9,a_1,c_2]) $ where $ c_1,c_2\in\{6,7\} $. When $ (c_1,c_2)=(6,7) $ then $ b_2=10 $ and after completing $ lk(2) $ we see that $ b_3\notin V $ and if $ (c_1,c_2)=(7,6) $ then $ b_1=9,b_2=11 $ and then $ [2,5,11],[5,6,11,b_3] $ forms faces of the polyhedron which is not possible.
	\end{case}
	\begin{case}
		When remaining two 3-gons are $ \{[1,6,11],[2,5,7]\} $ then $ lk(6)=C_9([d_1,a_2,5],[5,4,\allowbreak 0],[0,8,7],[7,a_1,d_2]) $ where $ d_1,d_2\in\{1,11\} $.
		\begin{subcase}
			When $ (d_1,d_2)=(1,11) $ then $ lk(1)=C_9([a_5,a_6,2],[2,3,0],[0,8,9],[9,a_3,\allowbreak a_4]) $ where $ a_4,a_5\in\{6,11\} $ and $ a_2\in\{2,9\} $. If $ \boldsymbol{a_2=2} $ then $ (a_1,a_4,a_5,a_6)=(10,11,6,5) $ and $ lk(2)=C_9([7,c_1,c_2],[c_2,c_3,3],[3,0,1],[1,6,5]) $ which implies $ c_2=9 $ and then $ lk(7)=C_9([5,c_4,8],[8,0,6],[6,11,10],[10,9,2]) $. From here we see that $ c_1=10,c_4=3 $ and after completing $ lk(8),lk(9) $ we see that $ c_3\notin V $. If $ \boldsymbol{a_2=9} $ then $ (a_3,a_4,a_5)=(5,6,11) $ and then $ lk(5)=C_9([e_1,a_6,4],[4,0,6],[6,1,9],[9,a_5,e_2]) $ where $ e_1,e_2\in\{2,7\} $. If $ (e_1,e_2)=(2,7) $ then $ lk(7)=C_9([2,a_7,8],[8,0,6],[6,11,a_1],[a_1,9,5]) $ which implies $ a_1=a_5=3, a_6=10=a_7 $ and $ a_4=10 $ and then $ lk(6) $ will not possible. If $ (e_1,e_2)=(7,2) $ then after completing $ lk(7) $ we see that $ C(0,1,5,2,10,8)\in lk(9) $ which is not possible.
		\end{subcase}
		\begin{subcase}
			When $ (d_1,d_2)=(11,1) $ then $ a_1\in\{2,9\} $. If $ \boldsymbol{a_1=2} $ then $ a_2=10,a_3=4 $ and then $ lk(4)=C_9([3,10,11],[11,1,9],[9,a_4,5],[5,6,0]) $ and considering $ lk(5) $ we see that $ a_4\notin V $. If $ \boldsymbol{a_1=9} $ then $ a_3\in\{4,10\} $. But if $ a_3=4 $ then $ [2,5,7] $ \& $ [2,4,5,*] $ form faces of the polyhedron which is not possible and similarly $ a_3\neq 10 $.
		\end{subcase}
	\end{case}
	\begin{case}
		When remaining two 3-gons are $ \{[1,7,11],[2,5,6]\} $ then $ lk(6)=C_9([2,b_1,b_2],[b_2,\allowbreak b_3, 7],[7,8,0],[0,4,5]) $ and $ lk(1)=C_9([d_1,a_2,2],[2,3,0],[0,8,9],[9,a_1,d_2]) $ where $ d_1,d_2\in\{7,11\} $.
		\begin{subcase}
			When $ (d_1,d_2)=(7,11) $ then $ lk(7)=C_9([11,b_2,6],[6,0,8],[8,a_3,a_2],[a_2,2,\allowbreak 1]) $ which implies $ b_2=10,b_3=11 $. We see that $ a_1,a_2\in\{4,5\} $, therefore $ b_1\in\{3,9\} $. Now $ lk(8)=C_9([10,a_4,a_3],[a_3,a_2,7],[7,6,0],[0,1,9]) $ and $ lk(2)=C_9([5,c_1,c_2],[c_2,c_3,c_4],[c_4,c_5,\allowbreak b_1],[b_1,10,\allowbreak 6]) $. From here we see that either $ c_2=3 $ or $ c_4=3 $ or $ b_1=3 $. If $ c_2=3 $ then $ b_1=a_2 $ which is a contradiction, if $ c_4=3 $ then $ a_2=5,b_1=9,c_5=a_1=4, a_3=3 $ which implies $ C(0,1,2,9,4)\in lk(3) $ which is a contradiction and if $ b_1=3 $ then $ c_2=a_2=4,a_1=5 $ which implies $ lk(4) $ is not possible.
		\end{subcase}
		\begin{subcase}
			When $ (d_1,d_2)=(11,7) $ then $ lk(7)=C_9([11,c_1,c_2],[c_2,c_3,c_4],[c_4,c_5,a_1],[a_1,\allowbreak 9,1]) $ which implies either $ c_2=8 $ or $ c_4=8 $. But if $ c_2=8 $ then $ a_1=b_3\notin V $, therefore $ c_4=8 $ and then $ c_1=10,c_2=6,c_3=0 $. Now considering $ lk(2) $, we get $ a_2\in\{5,b_1\} $, but if $ a_2=b_1 $ then $ b_1\notin V $, therefore $ a_2=5 $ and then $ lk(2)=C_9([5,11,1],[1,0,3],[3,b_4,b_1],[b_1,10,6]) $. From here we get $ b_1=9 $ and then considering $ lk(9) $ we see that $ b_4\notin V $.
		\end{subcase}
	\end{case}
\end{proof}
%\end{proof}

\begin{proof}[\textbf{Proof of theorem \ref{thm5}}]
	Let $K$ be a map of type $(3^4,4^2)$ on the surface of $\chi =-2$. Let $V(K)=\{0,1,2,3,4,5,6,\allowbreak 7,8,9,10,11\}$. We express $lk(v)$ by the notation $lk(v)=C_{8}(v_{1},v_{2},v_{3},\allowbreak [v_{4},v_{5},v_{6}],[v_{6},v_{7},v_{8}])$ where $[v,v_{8},v_{1}],[v,v_{1},v_{2}],[v,v_{2},v_{3}],[v,v_{3},v_{4}]$ form 3-gon face and $[v,v_{4},v_{5},v_{5}],[v,v_{6},v_{7},v_{8}]$ form 4-gon face.
	With out loss of generality let $ lk(0)=C_{8}(2,3,4,[5,6,7],[7,8,1]) $. Let $ lk(7)=C_{8}(a,b,c,[8,1,0],[0,5,6]) $, where $ a,b,c\in \lbrace 2,3,4,9,\allowbreak10,11\rbrace$. For any $ a\in \lbrace 2,3,4\rbrace $, $ [a,6] $ form an edge in a 4-gon which implies face sequence is not followed in $ lk(6) $, therefore $ a\neq 2,3,4 $. Similarly $ c\neq 2,3,4 $. If $ b=2 $ then either $ a=3 $ or $ c=3 $. If $ a=3 $ then $ [3,6] $ form an edge in a 4-gon, which implies face sequence is not followed in $ lk(6) $, similarly if $ c=3 $ then face sequence will not follow in $ lk(8) $. Similarly $ b\neq 4 $. If $ b=3 $ one of $ a,c $ is either $ 2 $ or $ 4 $. For any case, face sequence will not follow in either $ lk(6) $ or $ lk(8) $. Therefore $ (a,b,c)=(9,10,11) $. Then $ lk(5)=C_{8}(4,a_{1},a_{2},[a_{3},a_{4},6],[6,7,0]) $ where $ a_{1}\in \lbrace 1,2,3,8,9,10,11\rbrace $. \\
	\underline{\textbf{If} $ \boldsymbol{a_{1}=1} $} then either $ lk(1)=C_{8}(2,5,4,[b_{1},b_{2},8],[8,7,0]) $ or $ lk(1)=C_{8}(2,b_{1},5,[b_{2},b_{3},\allowbreak 8],[8,7,0]) $. If $ lk(1)=C_{8}(2,5,4,[b_{1},b_{2},8],[8,7,0]) $ then from $ lk(4) $ we see that $ [4,b_{1}] $ form an edge in a 4-gon, which implies face sequence is not followed in $ lk(b_{1}) $. Therefore $ lk(1)=C_{8}(2,b_{1},5,[b_{2},b_{3},8],[8,7,0]) $ and then either $ b_{1}=4 $ or $ b_{3}=4 $.
	\begin{case}
		If $ b_{1}=4 $ then $ lk(4)=C_{8}(0,5,1,[2,c_{1},c_{2}],[c_{2},c_{3},3]) $ where $ c_{2}\in \lbrace 9,10,11\rbrace $.
		
		\textbf{If} $ \boldsymbol{c_{2}=9} $ then four incomplete 4-gons are $ [2,4,9,c_{1}],[3,4,9,c_{3}],[8,1,b_{2},b_{3}]  $ and $ [6,5,\allowbreak c_{4},c_{5}] $. We see that one of $ c_{4},c_{5} $ is $ 2 $ and $ b_{2},b_{3}\neq 11 $, therefore one of $ c_{1},c_{4},c_{5} $ is $ 11 $. For any case, $ [2,11] $ form an edge in a 4-gon, therefore $ c_{1}=11 $ and one of $ c_{4},c_{5} $ is $ 11 $ and hence $ c_{3}=10 $ and $ b_{2},b_{3}\in\lbrace 3,10\rbrace $. Therefore $ lk(9)=C_{8}(7,6,8,[11,2,4],[4,3,10]) $ and $ lk(8)=C_{8}(11,9,6,[b_{3},b_{2},1],[1,0,7])\Rightarrow lk(6)=C_{8}(9,8,b_{3},[c_{5},c_{4},5],[5,0,7]) $.
		\begin{subcase}
			\textbf{If} $ \boldsymbol{(b_{3},b_{2})=(3,10)} $ then $ lk(3)=C_{8}(0,2,6,[8,1,10],[10,9,4])\Rightarrow lk(10)=C_{8}(7,11,5,[1,8,3],[3,4,9]) $ which implies $ c_{4}=11,c_{5}=2 $. Therefore $ lk(6)=C_{8}(9,8,3,[2,11,\allowbreak 5],[5,0,7])\Rightarrow lk(11)=C_{8}(8,7,10,[5,6,2],[2,4,9])\Rightarrow lk(1)=C_{8}(2,\allowbreak 4,5,[10,3,8],[8,7,10])\allowbreak\Rightarrow lk(2)=C_{8}(1,0,3,[6,5,11],[11,9,4])\Rightarrow lk(5)=C_{8}(4,1,\allowbreak 10,[11,2,6],[6,7,0]) $. Let this be the map $ \boldsymbol{KNO_{[(3^4,4^2)]}} $.
		\end{subcase}
		\begin{subcase}
			\textbf{If} $ \boldsymbol{(b_{3},b_{2})=(10,3)} $ then $ lk(10)=C_{8}(7,11,6,[8,1,3],[3,4,9])\Rightarrow lk(6)=C_{8}(9,8,10,[11,2,5],[5,0,7])\Rightarrow lk(11)=C_{8}(8,7,10,[6,5,2],[2,4,9])\Rightarrow lk(1)\allowbreak =C_{8}(2,4,5,[3,\allowbreak 10,8],[8,7,0])\Rightarrow lk(5)=C_{8}(4,1,3,[2,11,6],[6,7,0])\Rightarrow lk(3)=C_{8}(5,2,0,[4,9,10],[10,8,1])\allowbreak\Rightarrow lk(2)=C_{8}(1,0,3,[5,6,11],[11,9,4]) $. Let this be the map $ \boldsymbol{KO_{1[(3^4,4^2)]}} $.
		\end{subcase}
		\textbf{If} $ \boldsymbol{c_{2}=10} $ then four incomplete 4-gons are $ [2,4,10,c_{1}],[3,4,10,c_{3}], [1,8,b_{3},b_{2}] $ and $ [5,6,\allowbreak c_{5},c_{4}] $. We see that one of $ c_{4},c_{5} $ is $ 2 $ and $ b_{2},b_{3}\neq 11 $ which implies $ [2,11] $ form an edge in a 4-gon. We have $ [2,4] $ is not an adjacent edge of two 4-gon, therefore $ [2,11] $ will form an adjacent edge of two 4-gon which implies $ c_{1}=11 $ and one of $ c_{4},c_{5} $ is $ 11 $ and then face sequence will not follow in $ lk(10) $.
		
		\textbf{If} $\boldsymbol{ c_{2}=11} $ then four incomplete 4-gons are $ [2,4,11,c_{1}],[3,4,11,c_{3}],[1,8,b_{3},b_{2}] $ and $ [5,6,\allowbreak c_{5},c_{4}] $. We see that one of $ c_{4},c_{5} $ is $ 2 $ and $ c_{4},c_{5}\neq 9 $. Therefore $ [9,11] $ form an edge in a 4-gon. If $ c_{1}=9 $ then $ [2,9] $ will form an adjacent edge of two 4-gon which implies one of $ c_{4},c_{5} $ is $ 9 $ which is a contradiction. If $ c_{3}=9 $ then $ [3,9] $ will form an adjacent edge of two 4-gon which implies $ b_{2},b_{3}\in \lbrace 3,9\rbrace $ and $ c_{1}=10 $, one of $ c_{4},c_{5} $ is 10. Then this case is isomorphic to the case $ (a,b,c)=(11,9,10) $ under the map $ (9,11,10) $ and again this case is isomorphic to the case $ (a,b,c)=(9,10,11) $ under the map $ (9,10,11) $. Therefore the case $ \boldsymbol{c_{2}=11} $ is isomorphic to the case $ \boldsymbol{c_{2}=9} $.
	\end{case}
	\begin{case}
		If $ b_{3}=4 $ then $ b_{1}\in \lbrace 9,10,11\rbrace $. But for any $ b_{1} $, fase sequence will not follow in $ lk(b_{1}) $.
	\end{case}
	\hspace{-0.5cm}\underline{\textbf{If} $\boldsymbol{a_{1}=2}  $} then $ a_{2}\in \lbrace 1,3\rbrace $.
	\begin{case}
		If $ a_{2}=1 $ then $ lk(2)=C_{8}(5,1,0,[3,b_{2},b_{3}],[b_{3},b_{1},4]) $ where $ b_{1},b_{2},b_{3}\in \lbrace 9,10,11\rbrace $.
		\begin{subcase}
			When $ b_{3}=9 $ then either $ (b_{1},b_{2})=(10,11) $ or $ (b_{1},b_{2})=(11,10) $.
			
			If $\boldsymbol{(b_{1},b_{2})=(10,11)}  $ then two incomplete 4-gons are $ [5,6,a_{4},a_{3}] $ and $ [1,8,c_{1},c_{2}] $. As we have $ [3,11] $ is an edge of two adjacent 4-gon and $ c_{1},c_{2}\neq 11 $ therefore $ a_{3},a_{4}\in \lbrace 3,11\rbrace $ and $ c_{1},c_{2}\in \lbrace 4,10\rbrace $. Now we see that is $ a_{3}=11 $ then $ lk(11) $ will not possible, therefore $ a_{3}=3 $ and $ a_{4}=11 $. Now $ lk(3)=C_{8}(0,4,1,[5,6,11],[11,9,2]) $ which implies $ c_{1}=10 $ and $ c_{2}=4 $. Now $ lk(9)=C_{8}(7,6,d_{1},[11,3,2],[2,4,10]) $ and $ lk(11)=C_{8}(d_{1},7,10,[6,5,3],[3,2,9]) $ which implies $ d_{1}=8 $. Now $ lk(6)=C_{8}(9,8,10,[1,3,5],[5,0,7])\Rightarrow lk(8)=C_{8}(11,9,6,[10,4,1],[1,\allowbreak0,7])\allowbreak\Rightarrow lk(10)=C_{8}(7,\allowbreak 11,6,[8,1,4],[4,2,9])\Rightarrow lk(3)=C_{8}(0,4,1,[5,6,11],[11,9,2])\Rightarrow lk(4)=C_{8}(5,0,3,\allowbreak [1,8,10],[10,9,2])\Rightarrow lk(1)=C_{8}(2,5,3,[4,10,8],[8,7,0]) $. Let this be the map $ \boldsymbol{KO_{2[(3^4,4^2)]}} $.
			
			If $ \boldsymbol{(b_{1},b_{2})=(11,10)} $ then $ 8 $ and $ 11 $ occur in same 4-gon, which is a contradiction.
		\end{subcase}
		\begin{subcase}
			If $ b_{3}=10 $ then for any $ (b_{1},b_{2}) $ face sequence is not followed in $ lk(10) $.
		\end{subcase}
		\begin{subcase}
			If $ b_{3}=11 $ then either $ (b_{1},b_{2})=(9,10) $ or $ (b_{1},b_{2})=(10,9) $.\\ If $\boldsymbol{(b_{1},b_{2})=(9,10)}  $ then $ lk(11)=C_{8}(6,8,7,[10,3,2],[2,4,9]) $. Now remaining two incomplete 4-gons are $ [5,6,c_{1},c_{2}] $ and $ [1,8,c_{3},c_{4}] $. $ c_{1},c_{2} $ can not be $ 4 $, therefore one of $ c_{3},c_{4} $ is 4 and one of $ c_{1},c_{2} $ is 3 and therefore one of $ c_{3},c_{4} $ is 9 and then one of $ c_{1},c_{2} $ is $ 10 $. \textbf{If} $ \boldsymbol{c_{1}=3}\Rightarrow c_{2}=10 $. Therefore $ lk(3)=C_{8}(8,4,0,[2,11,10],[10,5,6])\Rightarrow lk(8)=C_{8}(11,6,3,[4,9,1],[1,0,7]) $ which implies $ c_{3}=4, c_{4}=9 $. Now $ lk(4)=C_{8}(3,0,5,[2,11,9],[9,\allowbreak1,8])\Rightarrow lk(5)=C_{8}(1,2,4,\allowbreak [0,7,6],[6,3,10])\Rightarrow lk(1)=C_{8}(10,\allowbreak 5,2,[0,7,8],[8,4,9])\Rightarrow lk(10)=C_{8}(1,9,7,[11,2,3],[3,6,\allowbreak 5])\Rightarrow lk(9)=C_{8}(10,7,6,\allowbreak [11,2,4],[4,8,1]) $. Let this be the map $ \boldsymbol{K_{1}} $. We see that $ (1, 5)(2, 4)\allowbreak(6, 8)(9, 11):\boldsymbol{KNO_{[(3^4,4^2)]}}\cong \boldsymbol{K_{1}} $. \textbf{If} $ \boldsymbol{c_{1}=10}\Rightarrow c_{2}=3 $. Now $ lk(5)=C_{8}(4,2,1,[3,10,6],[6,\allowbreak 7,0])\Rightarrow lk(1)=C_{8}(2,5,3,[4,9,8],[8,7,0]) $ which implies $ c_{3}=9,c_{4}=4 $. Now $ lk(3)=C_{8}(0,4,1,[5,6,10],[10,11,2])\Rightarrow lk(2)=C_{8}(0,1,5,[4,9,11],\allowbreak [11,10,3])\Rightarrow lk(4)=C_{8}(5,0,3,\allowbreak [1,8,9],[9,11,2])\Rightarrow lk(8)=C_{8}(11,6,\allowbreak10,[9,4,1],[1,\allowbreak 0,7])\Rightarrow lk(9)=C_{8}(6,7,10,[8,1,4],[4,\allowbreak 2,11])\Rightarrow lk(10)=C_{8}(7,9,8,[6,5,3],\allowbreak[3,2,11]) $. Let this be the map $ \boldsymbol{K_{2}} $ and then $ (0, 2, 5, 3, 1,\allowbreak 4)(6, 10, 8, 9, 7, 11):\boldsymbol{KO_{1[(3^4,4^2)]}}\cong \boldsymbol{K_{2}} $.\\
			If $ \boldsymbol{(b_{1},b_{2})=(10,9)} $ then $ lk(11)=C_{8}(7,8,5,[9,3,2],[2,4,10]) $ and then $ lk(5) $ will not be possible.
			
		\end{subcase}
	\end{case}
	
	\hspace{-0.5cm}\underline{\textbf{If} $ \boldsymbol{a_{1}=3} $} then $ 3 $ will occur two times in $ lk(3) $ which make contradiction.
	
	\hspace{-0.5cm}\underline{\textbf{If} $ \boldsymbol{a_{1}=8} $} then $ a_{2}\in \lbrace 2,3,9,10\rbrace $.
	\begin{case}
		If $ a_{2}=2 $ then $ a_{3}\in \lbrace 1,3\rbrace $. If $ a_{3}=1 $ then face sequence is not followed in $ lk(1) $ and id $ a_{3}=3 $ then $ 1 $ occur two times in $ lk(2) $. Therefore this case is not possible.
	\end{case}
	\begin{case}
		If $ a_{2}=3 $ then $ lk(8)=C_{8}(11,4,5,[3,b_{1},1],[1,0,7]) $ and $ lk(4)=C_{8}(0,5,8, [11,b_{4},\allowbreak b_{2}],\allowbreak b_{2},b_{3},3) $. Now four incomplete 4-gons are $ [1,8,3,b_{1}],[3,4,b_{2},b_{3}],[4,11,b_{4},\allowbreak b_{2}] $ and $ [2,5,\allowbreak6,a_{4}] $. We see that $ [3,4] $ and $ [3,8] $ are adjacent edges of a 3-gon and a 4-gon, therefore $ [3,b_{1}] $ will form an adjacent edge of two 4-gon which implies $ b_{1}=b_{3} $ and $ b_{4}=2 $. As $ [4,b_{2}] $ is an adjacent edge of two 4-gon, therefore $ [2,11] $ will form an adjacent edge of two 4-gon hich implies $ a_{4}=11 $. We see that $ b_{2}\neq 10 $, therefore $ b_{1}=10=b_{3} $ and $ b_{2}=9 $. Now $ lk(2)=C_{8}(3,0,1,[9,4,11],[11,6,5])\Rightarrow lk(11)=C_{8}(8,7,10,[6,5,2],[2,9,4])\Rightarrow lk(5)=C_{8}(4,8,3,[2,11,6],[6,7,0])\Rightarrow lk(9)=C_{8}(7,6,\allowbreak 1,[2,11,4],[4,3,10])\Rightarrow lk(6)=C_{8}(9,1,10,[11,2,5],[5,0,7])\Rightarrow lk(1)=C_{8}(6,9,2,[0,\allowbreak 7,8],[8,3,10])\Rightarrow lk(3)=C_{8}(0,2,5,[8,\allowbreak1,\allowbreak 10],[10,9,4])\Rightarrow lk(10)=C_{8}(7,11,6,[1,8,\allowbreak 3],[3,4,9]) $. Let this be the map $ \boldsymbol{K_{3}} $. Then $ (0, 3, 2)(1, 4, 5, 8, 9, 6)(7, 10, 11):\boldsymbol{KNO_{[(3^4,4^2)]}} \cong \boldsymbol{K_{3}} $.
	\end{case}
	\begin{case}
		If $ a_{2}=9 $ then face sequence is not followed in $ lk(6) $.
	\end{case}
	\begin{case}
		If $ a_{2}=10 $ then $ lk(8)=C_{8}(11,4,5,[10,b_{1},1],[1,0,7]), lk(10)=C_{8}(5,11,7,[9,\allowbreak b_{2},b_{1}],[b_{1},1,8]) $ and $ lk(11)=C_{8}(10,7,8,[4,b_{3},a_{4}],[a_{4},6,5]) $. Now four incomplete 4-gons are $ [5,6,a_{4},11],[1,8,10,b_{1}],[9,10,b_{1},b_{2}] $ and $ [4,11,a_{4},b_{3}] $. Therefore $ a_{4}=2,b_{1}=3\Rightarrow b_{2}=4,b_{3}=9 $. Now $ lk(4)=C_{8}(0,5,8,[11,2,9],[9,10,3])\Rightarrow lk(9)=C_{8}(7,6,1,[2,11,4],[4,3,10])\allowbreak\Rightarrow lk(1)=C_{8}(2,9,6,[3,10,8],[8,7,0])\Rightarrow lk(3)=C_{8}(0,\allowbreak 2,6,[1,8,10],[10,9,4])\Rightarrow lk(6)=C_{8}(9,1,3,[2,11,5],[5,0,7])\Rightarrow lk(2)=C_{8}(1,0,3,\allowbreak [6,5,11],[11,4,9]) $. Let this be the map $ \boldsymbol{K_{4}} $ and then $ (0, 1, 3, 9, 11, 5)(2, 6, 7, 8, 10, 4):\boldsymbol{KNO_{[(3^4,4^2)]}}\cong \boldsymbol{K_{4}} $
	\end{case}
	\hspace{-0.5cm}\underline{\textbf{If} $ \boldsymbol{a_{1}=9} $} then $ a_{3}\in \lbrace 1,2,3,8\rbrace $ and $ [10,a_{3}] $ form an edge in a 4-gon which implies face sequence will not follow in $ lk(a_{3}) $.
	
	\hspace{-0.5cm}\underline{\textbf{If} $ \boldsymbol{a_{1}=10} $} then $ a_{2}\in \lbrace 9,11\rbrace $. If \textbf{$ a_{2}=9 $} then $ [6,9] $ will form an edge in a 4-gon and if \textbf{$ a_{2}=11 $} then $ [8,11] $ will form an edge in a 4-gon which are not possible.
	
	\hspace{-0.5cm}\underline{\textbf{If} $ \boldsymbol{a_{1}=11} $} then $ a_{2}\in \lbrace 8,10\rbrace $. If $ a_{2}=10 $ then $ [10,a_{3}] $ will form an edge in a 4-gon, which is not possible. Therefore $ a_{2}=8 $ and then $ lk(11)=C_{8}(5,8,7,[10,b_{1},2],[2,b_{2},\allowbreak 4]) $. Now four incomplete 4-gons are $ [5,6,a_{4},a_{3}], [1,8,c_{1},c_{2}], [2,11,4,b_{2}] $ and $ [2,11,10,\allowbreak b_{1}] $. We see that $ b_{2}=9 $ and $ b_{1}=3 $. Now as $ [4,5] $ and $ [6,9] $ are adjacent edges of two 4-gons, therefore $ a_{3},a_{4}\in \lbrace 3,10\rbrace $ and $ c_{1},c_{2}\in \lbrace 4,9\rbrace $.
	\begin{case}
		If $ (a_{3},a_{4})=(3,10) $ then $ lk(3)=C_{8}(8,4,0,[2,11,10],[10,6,5])\Rightarrow lk(8)=C_{8}(3,5,11,\allowbreak [7,0,1],[1,9,4]) $. Therefore $ c_{1}=4 $ and $ c_{2}=9 $. Now $ lk(2)=C_{8}(0,1,6,[9,\allowbreak 4,11],[11,10,3])\Rightarrow lk(6)=C_{8}(9,2,1,[10,3,5],[5,0,7])\Rightarrow lk(1)=C_{8}(10, 6,2,[0,7,\allowbreak 8],[8,4,\allowbreak9])\Rightarrow lk(10)=C_{8}(1,9,7,[11,2,3],[3,5,6])\Rightarrow lk(9)=C_{8}(6,7, 10,[1,8,4],[4,\allowbreak 11,2]) $. Let this be the map $ \boldsymbol{K_{5}} $. Then $ (0,5)(1, 3, 11)(2, 8, 10)(6, 7):\boldsymbol{KNO_{[(3^4,4^2)]}}\cong \boldsymbol{K_{5}} $
	\end{case}
	\begin{case}
		If $ (a_{3},a_{4})=(10,3) $ then $ lk(8)=C_{8}(11,5,10,[9,4,1],[1,0,7]) $. Therefore $ c_{1}=9 $ and $ c_{2}=4 $. Now $ lk(4)=C_{8}(5,0,3,[1,8,9],[9,2,11])\Rightarrow lk(10)=C_{8}(7,9,8,[5,6,3],[3,2,\allowbreak11])\Rightarrow lk(3)=C_{8}(0,4,1,[6,5,10],[10,11,2])\Rightarrow lk(9)=C_{8}(6,\allowbreak 7,10,[8,1,4],[4,11,2])\Rightarrow lk(6)=C_{8}(9,2,1,[3,10,5],[5,0,7])\Rightarrow lk(1)=C_{8}(2,6,3,\allowbreak [4,9,8],[8,7,0])\Rightarrow lk(2)=C_{8}(6,\allowbreak1,0,[3,10,\allowbreak 11],[11,4,9]) $. Let this be the map $ \boldsymbol{K_{6}} $. We see that $ (0, 3, 4, 1, 2)(5, 6)(7, 10, 9, 8,\allowbreak 11):\boldsymbol{KNO_{[(3^4,4^2)]}}\cong \boldsymbol{K_{6}} $
	\end{case}
\end{proof}
\begin{proof}[\textbf{Proof of Theorem \ref{thm6}}]
	Let $K$ be a map of type $(3^3,4,3,4)$ on the surface of $\chi =-2$. Let $V(K)=\{0,1,2,3,4,5,6,\allowbreak 7,8,9,10,11\}$. We express $lk(v)$ by the notation $lk(v)=C_{8}(v_2,v_3,[v_4,v_5,v_6],\allowbreak [v_7,v_8,v_1])$ where $[v,v_1,v_2]$, $[v,v_2,v_3]$, $[v,v_3,v_4]$, $[v,v_6,v_7]$ form 3-gon faces and $[v,v_4,v_5,v_6],[v,\allowbreak v_1,v_8,v_7]$ form 4-gon faces in $ v $.
	With out loss of generality, let us consider $lk(0)=C_{8}(2,3,[4,\allowbreak 5,6],[7,8,1])$. Then $lk(1)=C_{8}(2,a,[b,c,d],[8,7,0])$ where $a\in \lbrace 4,5,6,9,10,11\rbrace$ and $b,c,d\in \lbrace 3,4,5,6,9,10,11\rbrace$ or $lk(1)=C_{8}(c,d,[8,7,0],[2,a,b])$, where $a,b,c,d\in \lbrace 4,5,6,9,10,$ $11\rbrace$.
	
	Let $lk(1)=C_{8}(2,a,[b,c,d],[8,7,0])$.
	By isomorphic, we get (4,3,6,5), (4,3,6,9)$\approx\lbrace$ (4,3,10,5), (4,3,11,5), (4,3,9,5),\allowbreak  (4,3,6,10), (4,3,6,11) $\rbrace$,  $ (4,3,9,6) \approx \lbrace$ (4,3,10,6), (4,3,11,6) $\rbrace$, $ (4,3,5,9) \approx\lbrace (4,3,5,10), (4,\allowbreak 3,5,11) \rbrace$,  $ (4,3,9,10) \approx\lbrace$ (4,3,9,11), (4,3,11,10), (4,3,10,9), (4,3,11,9),  (4,3,10,11) $\rbrace$,  $ (4,5,\allowbreak 6,9) \approx\lbrace$ (9,3,4,5), (10,3,4,5), (11,3,4,5), (4,5,6,10), (4,5,6,11) $\rbrace$,  $ (4,5,6,3) \approx (6,3,4,5),\allowbreak (4,5,9,3) \approx\lbrace$(6,3,4,9), (6,3,4,10), (6,3,4,11), (4,5,10,3), (4,5,11,3) $\rbrace$,  $ (4,5,9,6) \approx\lbrace(9,3,\allowbreak 4,6), (10,3,4,6), (11,3,4,6), (4,5,10,6), (4,5,11,6) \rbrace$, $ (4,5,9,10) \approx\lbrace$ (9,3,4,10), (9,3,4,11),\allowbreak (10,3,4,9), (10,3,4,11), (11,3,4,9), (11,3,4,10), (4,5,9,11), (4,5,10,9), (4,5, 10,11), (4,5,11,9), (4,5,11,10)$ \rbrace$, $ (4,5,3,6) \approx  (5,3,4,6), (4,5,3,9) \approx\lbrace$ (5,3,4,11), (5,3,4, 10), (4,5,3,9), (4,5,3,10), (4,5,3,11)$\rbrace$, $ (4,6,3,5) \approx (5,3,6,4)$, $ (4,6,3,9) \approx \lbrace$ (5,3,9,4), (5,3,10,4), (5,3,11,4), (4,6,3,10), (4,6,3,11)$\rbrace$, $ (4,6,5,3) \approx (6,3,5,4)$, $ ( 4,6,5,9)  \approx\lbrace$ (9,3,5,4), (10,3,5,4), (11,3,5,4), (4,6,5,10), (4,6,5,11)$ \rbrace$, $ (4,6,9,3) \approx\lbrace$ (6,3,9,4), (6,3,10, 4), (6,3,11,4), (4,6,10,3), (4,6,11,3) $\rbrace$, $ (4,6,9,5)\allowbreak \approx\lbrace$ (9,3,6,4), (10,3,6, 4), (11,3,6,4), (4,6,10,5), (4,6,11,5)$\rbrace$, (4,9,5,3) $\approx\lbrace$ (6,3,5,9), (6,3,5,10), (6,3,5,11), (4,10,5,3), (4,11, 5,3) $\rbrace$, (4,9,5,6) $\approx\lbrace$ (9,3,5,6), (10,3,5,6), (11,3,5,6), (4,10,5,6), (4,11,5,6) $\rbrace$,  $ (4,9,5,10) \approx\lbrace$ (9,3,5,10), (9,3,5,11), (10,3,5,9), (10,3,5,11), (11,3,5,9), (11,3,5, 10), (4,9,5,11), (4,10,5,9), (4,10,5,11), (4,11,5,9), (4,11,5, 10)$ \rbrace$,  (4,9,6,3) $\approx\lbrace$ (6,3,9,5), (6,3,10,5), (6,3,11,5), (4,10,6,3), (4,11,6,3) $\rbrace$, $ (4,9,6,5) \approx\lbrace$ (9,3,6,5), (10,3,6,5), (11,3, 6,5), (4,10,6,5), (4,11,6,5) $\rbrace$, $ (4,9,6,10) \approx\lbrace$ (9,3,10,5), (9,3,11,5), (10,3,9,5), (10,3,11, 5), (11,3,9,5), (11,3,10,5), (4,9,6,11), (4,10,6,9), (4,10,6,11), (4,11,6,9), (4,11,6,10) $\rbrace$,(4,9,10,3) $\approx\lbrace$ (6,3,9,10), (6,3,9,11), (6,3,10,9), (6,3,10,11), (6,3,11,9), (6,3,11,\allowbreak10), (4,9,11,3), (4,10,9,3), (4,10,11,3), (4,11,9,3), (4,11,10,3) $\rbrace$, $ (4,9,10,6) \approx$ $\lbrace$ (9,3,10,6), (9,3,11,6), (10,3,9,6), (10,3,11,\allowbreak 6), (11,3,9,6), (11,3,10,6), (4,9,11,6), (4,10,9,6), (4,10, 11,6), (4,11,9,6), (4,11,10,6)$\rbrace$, (4,6,9,\allowbreak 10) $\approx$ $\lbrace$ (9,3,10,4), (9,3,11,4), (10,3,9,4), (10,3, 11,4), (11,3,9,\allowbreak4), (11,3,10,4), (4,6,9,11), (4,6,10,9), (4,6,10,11), (4,6,11,9), (4,6,11,10) $\rbrace$, (4,9,3,5) $\approx$ $\lbrace$ (5,3,6,9), (5,3,6,10), (5,3,6,11), (4,10,3,5), (4,11,3,5) $\rbrace$, (4,9,3,6) $\approx$ $\lbrace$ (5,3,9,6), (5,3,10,6), (5,3,11,6), (4,10,3,6), (4,11,3,6)$\rbrace$, (4,9,3,10) $\approx$ $\lbrace$ (5,3,9,10), (5,3,9,11), (5,3,10,9), (5,3,10,11), (5,3,11,9), (5,3,11,10), (4,9,3,11), (4,10,3,9), (4,10, 3,11), (4,11,3,9), (4,11,3,10) $\rbrace$, (4,9,5,3) $\approx$ $\lbrace$ (6,3,5,9), (6,3,5,10), (6,3,5,11), (4,10,5,3), (4,11,5,3) $\rbrace$, (4,9,10,11) $\approx$ $\lbrace$ (9,3,10,11), (9,3,11,10), (10,3,9,11), (10,3,11,9), (11,3,\allowbreak 9,10), (11,3,10,9), (4,9,11,10), (4,10,9,11), (4,10,\allowbreak11,9), (4,11,9,10), (4,11,10,9) $\rbrace$, (4,9, 10,5) $\approx$ $\lbrace$ (9,3,6,10), (9,3,6,11), (10,3,6,9), (10,3,6,11), (11,3,6,9), (11,3,6,10), (4,9,11, 5), (4,10,9,5), (4,10,11,5), (4,11,9,5), (4,11,10,5) $\rbrace$, (5,4,3,6), (5,4,3,9) $\approx$ $\lbrace$ (5,4,3,10), (5,4,3,11) $\rbrace$, (5,4,6,3) $\approx$  (6,4,3,5), (5,4,6,9) $\approx$ $\lbrace$ (9,4,3,5), (10,4,3,5), (11,4,3,5), (5,4,6, 10), (5,4,6,11) $\rbrace$, (5,4,9,3) $\approx$ $\lbrace$ (6,4,3,9), (6,4,3,10), (6,4,3,11), (5,4,10,3), (5,4,11,3) $\rbrace$, (5,4,9,6) $\approx$ $\lbrace$ (9,4,3,6), (10,4,3,6), (11,4,3,6), (5,4,10,6), (5,4,11,6) $\rbrace$, (5,4,10,9) $\approx$ $\lbrace$ (9,4,3,10), (9,4,3,11), (10,4,3,9), (10,4,3,11), (11,4,3,9), (11,4,3,9), (5,4,10,11), (5,4,9,10), (5,4,9,11), (5,4,11,9), (5,4,11,10) $\rbrace$, (5,6,3,4), (5,6,3,9) $\approx$ $\lbrace$ (5,9,3,4), (5,10, 3,4), (5,11,3,4), (5,6,3,10), (5,6,3,11) $\rbrace$, (5,6,4,3) $\approx$ (6,5,3,4), (5,6,4,9 )$\approx$ $\lbrace$ (9,5,3,4), (10,5,3,4), (11,5,3,4), (5,6,4,10), (5,6,4,11) $\rbrace$, (5,6,9,3) $\approx$ $ \lbrace$(6,9,3,4), (6,10,3,4), (6,11, 3,4), (5,6,10,3), (5,6,11,3) $\rbrace$, (5,6,9,4) $\approx$ $\lbrace$ (9,6,3,4), (10,6,3,4), (11,6,3,4), (5,6,10,4), (5,6,11,4) $\rbrace$, (5,6,9,10) $\approx$ $\lbrace$ (9,10,3,4), (9,11,3,4), (10,9,3,4), (10,11,3,4), (11,9,3,4), (11,10,3,4), (5,6,9,11), (5,6,10,9), (5,6,10,11), (5,6,11,9), (5,6,11,10) $\rbrace$, (5,9,3,6) $\approx$ $\lbrace$ (5,10,3,6), (5,11,3,6) $\rbrace$, (5,9,3,10) $\approx$ $\lbrace$(5,9,3,11), (5,10,3,9), (5,10,3,11), (5,11,3,9), (5,11,3,10) $\rbrace$, (5,9,4,3) $\approx$ $ \lbrace$ (5,10,4,3), (5,11,4,\allowbreak3), (6,5,3,9), (6,5,3,10), (6,5,3,11) $\rbrace$, (5,9,4,6) $\approx$ $\lbrace$ (9,5,3,6), (10,5,3,6), (11,5,3,6), (5,10,4,6), (5,11,4,6) $\rbrace$, (5,9,4,10) $\approx$ $\lbrace$ (9,5,3,10), (9,5,3,11), (10,5,3,9), (10,5,3,11), (11,5,3,9), (11,5,3,\allowbreak10), (5,9,3,11), (5,10,3,9), (5,10,3,11), (5,11,3,9), (5,11,3,10) $\rbrace$, (5,9,6,3) $\approx$ $\lbrace$ (6,9,3,5), (6,10,3,5), (6,11,3,5), (5,10,6,3), (5,11,6,3) $\rbrace$, (5,9,6,4) $\approx$ $\lbrace$ (9,6,3,5), (10,6,3,5), (11,6,3,5), (5,10, 6,4), (5,11,6,4) $\rbrace$, (5,9,6,10) $\approx$ $\lbrace$ (9,10,3,5), (9,11,3,5), (10,9,3,5), (10,11,3,5), (11,9,3, 5), (11,10,3,5), (5,9,6,11), (5,10,6,9), (5,10,6,11), (5,11,6,9), (5,11,6,10) $\rbrace$, (5,9,10,3) $\approx$ $ \lbrace$ (6,9,3,10), (6,9,3,11), (6,10,3,9), (6,10,3,11), (6,11,3,9), (6,11,3,10), (5,9,11,3), (5,10,9,3), (5,10,11,3), (5,11,9,3), (5,11,10,3) $\rbrace$, (5,9,10,4) $\approx$ $ \lbrace$ (9,6,3,10), (9,6,3,11), (10,6,3,9), (10,6,3,\allowbreak 11), (11,6,3,9), (11,6,3,9), (5,9,11,4), (5,10,9,4), (5,10,11,4), (5,11, 9,4), (5,11,10,4) $\rbrace$, (5,9,10,\allowbreak 6) $\approx$ $\lbrace$ (9,10,3,6), (9,11,3,6), (10,9,3,6), (10,11,3,6), (11,9, 3,6), (11,10,3,6), (5,9,11,6), (5,10,\allowbreak 9,6), (5,10,11,6), (5,11,9,6), (5,11,10,6) $\rbrace$, (5,9,10, 11) $\approx$ $ \lbrace$ (9,10,3,11), (9,11,3,10), (10,9,3,11), (10,11,3,9), (11,9,3,10), (11,10,3,9), (5,9,11,10), (5,10,9,11), (5,10,11,9), (5,11,9,\allowbreak10), (5,11,10,\allowbreak9) $\rbrace$, (6,4,5,3) $\approx$ (6,5,4,3), (6,4,5,9) $\approx$ $ \lbrace$ (9,4,5,3), (10,4,5,3), (11,4,5,3), (6,4,5,\allowbreak10), (6,4,5,11) $\rbrace$, (6,4,9,3) $\approx$ $\lbrace$ (6,4,10,3), (6,4,11,3) $\rbrace$, (6,4,9,5) $\approx$ $\lbrace$ (9,4,6,3), (10,4,6,3), (11,4,6,3), (6,4,10,5), (6,4,11,5) $\rbrace$, (6,4,9,10) $\approx$ $\lbrace$ (10,4,9,3), (10,4,11,3), (9,4,10,3), (9,4,11,3), (11,4,9,3), (11,4,10,3), (6,4,9,11), (6,4,10,9), (6,4,10,11), (6,4,11,9), (6,4,11,10) $\rbrace$, (6,5,4,9) $\approx$ $\lbrace$ (9,5,4,3), (10,5,4,3), (11,5,4,3), (6,5,4,10), (6,5,4,11) $\rbrace$, (6,5,9,3) $\approx$ $\lbrace$ (6,9,4,3), (6,10,4,3), (6,11,4,3), (6,5,10,3), (6,5,11,3), (9,5,6,3), (10,5,6,3), (11,5,6,3) $\rbrace$, (6,5,9,4) $\approx$ $\lbrace$ (9,6,4,3), (10,6,4,3), (11,6,4,3), (6,5,10,4), (6,5,11,4) $\rbrace$, (6,5,9,10) $\approx$ $\lbrace$ (9,10,4,3), (9,11,4,3), (10,9,4,3), (10,11,4,3), (11,9,4,3), (11,10,4,3), (6,5,9,11), (6,5,10,9), (6,5, 10,11), (6,5,11,9), (6,5,11,10) $\rbrace$, (6,9,4,5) $\approx$ $ \lbrace$ (9,5,6,3), (10,5,6,3), (11,5,6,3), (6,10,4, 5), (6,11,4,5) $\rbrace$, (6,9,4,10) $\approx$ $\lbrace$ (9,5,10,3), (9,5,11,3), (10,5,9,3), (10,5,11,3), (11,5,9,3), (11,5,10,3), (6,9,4,11), (6,10,4,9), (6,10,4,11), (6,11,4,9), (6,11,4,10) $\rbrace$, (6,9,5,3) $\approx$ $\lbrace$ (6,10,5,3), (6,11,5,3) $\rbrace$, (6,9,5,4) $\approx$ $ \lbrace$ (9,6,5,3), (10,6,5,3), (11,6,5,3), (6,10,5,4), (6,11,5,4) $\rbrace$, (6,9,5,10) $\approx$ $\lbrace$ (9,10,5,3), (9,11,5,3), (10,9,5,3), (10,11,5,3), (11,9,5,3), (11,10,5,3), (6,9,5,11), (6,10,5,9), (6,10,5,11), (6,11,5,9), (6,11,5,10)$\rbrace$, (6,9,10,3) $\approx$ $\lbrace$ (6,9,11,3), (6,10,9,3), (6,10,11,3), (6,11,9,3), (6,11,10,3) $\rbrace$, (6,9,10,4) $\approx$ $ \lbrace$ (9,6,10,3), (9,6,11,3), (10,6,9,3), (10,6,11,\allowbreak3), (11,6,9,3), (11,6,10,3), (6,9,11,4), (6,10,9,4), (6,10, 11,4), (6,11,9,4), (6,11,10,4)$\rbrace$, (6,9,10,\allowbreak5) $\approx$ $\lbrace$ (9,10,6,3), (9,11,6,3), (10,9,6,\allowbreak3), (10,11, 6,3), (11,9,6,3), (11,10,6,3), (6,9,11,5), (6,10,9,\allowbreak5), (6,10,11,5), (6,11,9,5), (6,11,\allowbreak10,5) $\rbrace$, (6,9,10,11) $\approx$ $\lbrace$ (9,10,11,3), (9,11,10,3), (10,9,11,3), (10,11,9,3), (11,9,10,3), (11,10, 9,3) $\rbrace$, (9,4,5,6) $\approx$ $ \lbrace$ (10,4,5,6), (11,4,5,6) $\rbrace$, (9,4,5,10) $\approx$ $\lbrace$ (9,4,5,11), (10,4,5,9), (10,4,5,11), (11,4,5,9), (11,4,5,10) $\rbrace$, (9,4,6,5) $\approx$ $\lbrace$ (10,4,6,5), (11,4,6,5) $\rbrace$, (9,4,6,10) $\approx$ $\lbrace$ (9,4,10,5), (9,4,11,5), (10,4,9,5), (10,4,11,5), (11,4,9,5), (11,4,10,5), (9,4,6,\allowbreak11), (10,4,6,9), (10,4,6,11), (11,4,6,9), (11,4,6,10) $\rbrace$, (9,4,10,6) $\approx$ $\lbrace$ (9,4,11,6), (10,4,9,6), (10,4,11,6), (11,4,9,6), (11,4,\allowbreak10,6) $\rbrace$, (9,4,10,11)$ \approx\lbrace $ (9,4,11,10), (10,4,9,11), (10,4,11,9), (11,4,\allowbreak9,10), (11,9,10,9) $\rbrace$, (9,5,4,\allowbreak6) $\approx$ $\lbrace$ (10,5,4,6), (11,5,4,6) $\rbrace$, (9,5,4,10) $\approx$ $\lbrace$ (9,5,4, 11), (10,5,4,9), (10,5,4,11), (11,5,4,9), (11,5,4,10) $\rbrace$, (9,5,6,4) $\approx$ $\lbrace$ (9,6,4,5), (10,6,4,5), (11,6,4,5), (10,5,6,4), (11,5,6,4) $\rbrace$, (9,5,6,10)\allowbreak$ \approx\lbrace $ (9,5,6,11), (10,5,6,9), (10,5,6,11), (11,5, 6,9), (11,5,6,10), (9,10,4,5), (9,11,4,5), (11,9,4,\allowbreak5), (11,10,4,5), (10,9,4,5), (10,11,4,5) $\rbrace$, (9,5,10,4) $\approx$ $\lbrace$ (9,6,4,10), (9,6,4,11), (10,6,4,9), (10,6,4,11), (11,6,4,9), (11,6,4,10), (9,5,11,4), (10,5,9,4), (10,5,11,4), (11,5,9,4), (11,5,10,\allowbreak4) $\rbrace$, (9,5,10,6) $\approx$ $\lbrace$ (9,10,4,6), (9,11,4,6), (10,9,4,6), (10,11,4,6), (11,9,4,6), (11,10,4,\allowbreak6), (9,5,11,6), (10,5,9,6), (10,5, 11,6), (11,5,9,6), (11,5,10,6)$\rbrace$, (9,5,10,11) $\approx$ $\lbrace$ (9,10,4,11), (9,11,4,10), (10,9,4,11), (10,11,4,9), (11,9,4,10), (11,10,4,9), (9,5,11,10), (10,5,9,11), (10,5,\allowbreak11,\allowbreak9), (11,5,9,10), (11,5,10,9) $\rbrace$, (9,6,5,4) $\approx$ $\lbrace$ (10,6,5,4), (11,6,5,4) $\rbrace$, (9,6,5,10) $\approx$ $\lbrace$ (9,10,5,4), (9,11,5,4), (10,9,5,4), (10,11,5,4), (11,9,5,4), (11,10,5,4), (9,5,6,11), (10,5,6,9), (10,5,6,11), (11,5, 6,9), (11,5,6,10) $\rbrace$, (9,6,10,4) $\approx$ $\lbrace$ (9,6,11,4), (10,6,9,4), (10,6,11,4), (11,6,9,4), (11,6, 10,4) $\rbrace$, (9,6,10,5) $\approx$ $\lbrace$ (9,10,6,4), (9,11,6,4), (10,9,6,4), (10,11,6,4), (11,9,6,4), (11,10, 6,4), (9,6,11,5), (10,6,9,5), (10,6,11,5), (11,6,9,5), (11,6,10,5) $\rbrace$, (9,6,10,11) $\approx$ $\lbrace$ (9,6, 11,10), (10,6,9,11), (10,6,11,9), (11,6,9,10), (11,6,10,9), (9,10, 11,4), (9,11,10,4), (10,9, 11,4), (10,11,\allowbreak9,4), (11,9,10,4), (11,10,9,4) $\rbrace$, (9,10,5,6) $\approx$ $\lbrace$ (9,11,5,6), (10,9,5,6), (10,11,5,6), (11,9,5,6), (11,10,5,6)$\rbrace$, (9,10,5,11) $\approx$ $\lbrace$ (9,11,5,10), (10,9,5,11), (10,11,5, 9), (11,9,5,10), (11,10,5,9) $\rbrace$, (9,10,6,11) $\approx$ $\lbrace$ (9,10,11,5), (9,11,10,5), (10,9,11,5), (10,11,9,5), (11,9,10,5), (11,10,9,5), (9,11,6,10), (10,9,6,11), (10,11,6,9), (11,9,6,10), (11,10,6,9) $\rbrace$, (9,10,11,6) $\approx$ $\lbrace$ (9,11,10,6), (10,9,11,6), (10,11,9,6), (11,9,10,6), (11,10, 9,6) $\rbrace$, (9,10,6,5) $\approx$ $\lbrace$ (9,11,6,5), (10,9,6,5), (10,11,6,5), (11,9,6,5), (11,10,6,5)$\rbrace$.\\

	Now, $(a,b,c,d)$ can not be a combination of 3,4,5,6, so for any $(a,b,c,d)\in \lbrace$ (4,3,5,6), (4,3,6,5), (4,5,3,6),  (4,5,6,3), (4,6,3,5), (4,6,5,3), (5,4,3,6), (5,4,6,3), (5,6,3,4), (5,6,4,3), (6,4,\allowbreak5,3) $\rbrace$ the SEM will not complete.
	
	Since [2,5], [5,6] are edges of 4-gon, therefore, for $(a,b,c,d)=(5,6.c.d)$, face sequence will not follow in $lk(5)$. Therefore (5,6,3,9), (5,6,9,3), (5,6,9,10) are not possible.
	
	If $b=4$, then $c\notin \lbrace 5,6\rbrace$ and $d\notin \lbrace 3,5,6\rbrace$, otherwise it will not follow the face sequence, so for any $(a,b,c,d)\in \lbrace$(5,4,6,9), (5,4,9,3), (5,4,9,6), (6,4,5,9), (6,4,9,3), (6,4,9,5), (9,4,5,6), (9,4,5,10), (9,4,6,5), (9,4,6,10), (9,4,10,6)$\rbrace$ the SEM will not complete. Similarly, if $b=5$, then $c,d\notin \lbrace 4,6\rbrace$, which implies for any $(a,b,c,d)\in \lbrace$(4,5,6,9), (4,5,9,6), (6,5,4,9), (6,5,9,4), (9,5,4,6), (9,5,4,10), (9,5,6,4), (9,5,6,10), (9,5,10,4), (9,5,10,6)$\rbrace$ the SEM will not complete. If $b=6$, then $c,d\notin \lbrace 4,5\rbrace$, which implies for any $(a,b,c,d)\in \lbrace$(4,6,5,9), (4,6,9,5), (5,6,4,9), (5,6,9,4), (9,6,5,4), (9,6,5,10), (9,6,10,4), (9,6,10,5)$\rbrace$ the SEM will not complete.If $d=4$, then $b\notin \lbrace 3,5,6\rbrace$ and $d\notin \{5,6\}$, which implies for any $(a,b,c,d)\in \lbrace$(5,9,6,4), (6,9,5,4)$\rbrace$ the SEM will not complete. If $d=5$, then $b,c\notin \lbrace 4,6\rbrace$, which implies for any $(a,b,c,d)\in \lbrace$(4,9,6,5), (6,9,4,5), (9,10,6,5)$\rbrace$ the SEM will not complete. If $d=6$, then $b,c \notin \lbrace 4,5\rbrace$, which implies for any $(a,b,c,d)\in \lbrace$(4,9,5,6), (5,9,4,6), (9,10,5,6)$\rbrace$ the SEM will not complete.
	
	For any $(a,b,c,d)\in \lbrace$(4,3,6,9), (4,3,9,6), (4,3,5,9), (4,3,9,10), (5,4,3,9), (5,4,10,9), (5,9,4,3), (6,5,9,3), (6,5,9,10), (9,4,10,11)$\rbrace$, the SEM will not orientable.
	
	\begin{sloppypar}
	For any $(4,b,c,d)$ with $b\neq 3$, then $lk(4)$ will not complete. So for any $(a,b,c,d)\in \lbrace$(4,5,3,9), (4,5,9,3), (4,5,9,10), (4,6,3,9), (4,6,9,3), (4,9,5,3), (4,9,5,10), (4,9,6,3), (4,9,6,5), (4,9,6,10), (4,9,10,3), (4,9,10,6), (4,6,9,10), (4,9,3,5), (4,9,3,6), (4,9,3,10), (4,9,5,3), (4,9,10,5), (4,9,10,\allowbreak11)$\rbrace$, SEM will not complete.
\end{sloppypar}
	If $d=6$, then $lk(6)=C_{8}(7,8,[1,b,c],[5,4,0])$, which implies 8 will appear two times in $lk(7)$, therefore $(a,b,c,d)\notin \lbrace$(5,9,3,6), (5,9,10,6), (9,10,11,6)$\rbrace$.
	\begin{case}
		When $(a,b,c,d)=(5,9,3,10)$, then either $lk(3)=C_{8}(0,4,[10,1,9],[a_{1},b_{1},2])$ or $lk(3)=C_{8}(0,4,[9,1,10],[a_{1},b_{1},2])$ where $a_{1}, b_{1}\in V$. If $lk(3)=C_{8}(0,4,[10,1,9],[a_{1},$ $ a_{2},2])$, then $lk(4)=C_{8}(3,10,[a_{2},b_{2},9],[5,6,0])$ which implies face sequence is not follow in $lk(9)$, where $a_{2},b_{2}\in V$. If $lk(3)=C_{8}(0,4,[9,1,10],[a_{1},b_{1},2])$,then $lk(4)=C_{8}(3,9,[a_{2},b_{2},c_{2}],[5,6,\allowbreak0])$, which implies face sequence is not follow in $lk(9)$, where $a_{2},b_{2},c_{2}\in V$.
	\end{case}
	\begin{case}
		When $(a,b,c,d)=(5,9,4,10)$, then $lk(4)=C_{8}(3,a_{1},[9,1,10],[5,6,0])$, $a_{1}\in V$, which implies in $lk(3)$ has four consecutive 3-gon, which is not possible.
	\end{case}
	\begin{case}
		When $(a,b,c,d)=(5,9,6,3)$, the $lk(3)=C_{8}(0,4,[6,9,1],[8,a_{1},2])$, $a_{1}\in V$, which implies 6 will appear two times in $lk(4)$, which is not possible.
	\end{case}
	\begin{case}
		When $(a,b,c,d)=(5,9,6,10)$, we have total six 4-gon, and these are $[0,4,5,6]$, $[0,1,8,7]$, $[1,9,6,10]$, $[2,3,a_{1},b_{1}]$, $[2,5,a_{2},b_{2}]$, $[3,a_{3},b_{3},c_{3}]$, where $a_{1}$, $a_{2}$, $a_{3}$, $b_{1}$, $b_{2}$, $b_{3}$, $c_{3}\in V$. From these six 4-gon we see that 4 will appear in any one of last three 4-gon. But 4 can not be appear in $4^{\text{th}}$ as [3,4] is an edge, so $b_{1}$ can not be 4 and in $lk(3)$ [3,4] is adjacent with a 3-gon, so $a_{1}$ can not be 4. 4 will not appear in $5^{\text{th}}$ 4-gon as 4,5 appear in $1^{\text{st}}$ 4-gon, so it can not be appear in another 4-gon. 4 will not appear in last 4-gon also as [3,4] is an adjacent edge of 3-gon. So in $lk(4)$, there is only one 4-gon, which is not possible.
	\end{case}
	\begin{case}
		When $(a,b,c,d)=(5,9,10,3)$, then $lk(3)=C_{8}(0,4,[10,9,1],[8,a_{1},2])$, where $p_{1}\in \lbrace 6,8\rbrace$.
		
		If $p_{1}=6$, then six 4-gon are $[0,4,5,6]$, $[0,1,8,7]$, $[1,3,10,9]$, $[2,3,8,6]$, $[2,5,a_{1},b_{1}]$, $[4,a_{2},b_{2},c_{2}]$, where $a_{1}$, $a_{2}$, $b_{1}$, $b_{2}$, $c_{2}\in V$. 11 will appear in last two 4-gon. 9 will not appear in $5^{\text{th}}$ 4-gon as [5,9] is an adjacent edge of a 3-gon in $lk(5)$, so 9 will appear in last 4-gon. If 7 will appear in $5^{\text{th}}$ 4-gon, then 9,10 will appear in last 4-gon, which is not possible. So 10 will appear in $5^{\text{th}}$ 4-gon and 7 will appear in last 4-gon. Therefore $lk(2)$ is either $lk(2)=C_{8}(1,0,[3,8,6],[10,11,5])$ or $lk(5)=C_{8}(1,0,[3,8,6],[11,10,5])$. If $lk(2)=C_{8}(1,0,[3,8,6],[10,11,5])$ then $lk(6)=C_{8}(7,10,[2,3,8],[5,4,$ $0])$, $lk(10)=C_{8}(6,7,[3,1,9],\allowbreak[11,5,2])$, which implies [3,7] is an edge of a 4-gon, which is not possible. If $lk(5)=C_{8}(1,0,[3,8,6],[11,10,5])$, then $lk(6)=C_{8}(7,11,[2,3,8],[5,4,0])$ after that $lk(5)$ will not complete.
		
		If $p_{1}=11$ then six 4-gon are $[0,4,5,6]$, $[0,1,8,7]$, $[1,3,10,9]$, $[2,3,8,11]$, $[2,5,a_{3},$ $b_{3}]$, $[4,a_{4},b_{4},c_{4}]$, where $a_{3}$, $a_{4}$, $b_{3}$, $b_{4}$, $c_{4}\in V$. We can see that 6 will not appear $5^{\text{th}}$ 4-gon as 5,6 appear in $1^{\text{st}}$ 4-gon, so it can not appear in another 4-gon and also not appear in last 4-gon as 4,6 appear in $1^{\text{st}}$ 4-gon, so it can not appear in another 4-gon.
	\end{case}
	\begin{case}
		When $(a,b,c,d)=(5,9,10,4)$, then after completing $lk(4)$ and $lk(3)$, $lk(8)$ is not possible.\\
		When $(a,b,c,d)=(6,4,9,10)$, then face sequence will not follow in $lk(6)$.
	\end{case}
	\begin{case}
		When $(a,b,c,d)=(6,9,4,10)$, then $lk(6)=C_{8}(1,9,[5,4,0],[7,a_{1},2])$. Six 4-gon are $[0,4,5,6]$, $[0,1,8,7]$, $[1,9,4,10]$, $[2,6,7,b_{1}]$, $[2,3,b_{2},b_{3}]$, $[4,b_{4},b_{5},_{6}]$, where $b_{1}$, $b_{2}$, $b_{3}$, $b_{4}$, $b_{5}$, $b_{6}\in V$. Now $b_{1}\neq 3$ as if $b_{1}=3$ then 6 will appear two times in $lk(3)$ and 3 will not appear in last 4-gon as [3,4] is an edge of two 3-gon, which implies at 3 has one 4-gon which is not possible.
	\end{case}
	\begin{case}
		When $(a,b,c,d)=(6,9,5,3)$, then $lk(6)=C_{8}(1,9,[5,4,0],[7,10,2])$, $lk(3)=C_{8}(0,4,[5,9,$ $1],[8,a_{1},2])$, where $a_{1}\in V$ which implies 4 will appear two times in $lk(5)$.
	\end{case}
	\begin{case}
		When $(a,b,c,d)=(6,9,5,10)$, then $lk(6)=C_{8}(1,9,[5,4,0],[7,a)_{1},2])$, where $a_{1}\in V$. Six 4-gon are $[0,4,5,6]$, $[0,1,8,7]$, $[1,9,5,10]$, $[2,6,7,a_{1}]$, $[2,3,a_{2},a_{3}]$, $[4,a_{4},a_{5},a_{6}]$, where $a_{2}$, $a_{3}$, $a_{4}$, $a_{5}$, $a_{6}\in V$. 3 will not appear will not appear in $4^{\text{th}}$ 4-gon as 2,3 will not appear in two 4-gon an 3 will not appear in last 4-gon as [3,4] is an adjacent edge of two 3-gon.
	\end{case}
	\begin{case}
		When $(a,b,c,d)=(6,9,10,3)$, then $lk(6)=C_{8}(1,9,[5,4,0],[7,a_{1},2])$ and $lk(3)=C_{8}(0,4,[10,9,1],[8,a_{2},2])$, where $a_{1}$, $a_{2}\in V$. Now six 4-gon are $[0,4,5,6]$, $[0,1,8,7]$, $[1,3,10,\allowbreak9]$, $[2,6,7,a_{1}]$, $[2,3,8,a_{2}]$, $[4,a_{3},a_{4},a_{5}]$, where $a_{3}$, $a_{4}$, $a_{5}\in V$. Now 5 will not appear in $4^{\text{th}}$ 4-gon as 5,6 will not appear in two 4-gon and 5 will also not appear in last 4-gon as 4,5 will not appear in two 4-gon, so 5 will appear in $5^{\text{th}}$ 4-gon i.e. $a_{2}=5$. Therefore 11 will appear in $4^{\text{th}}$ and last 4-gon, therefore $a_{1}=11$ and 9,10 will appear in last 4-gon, which contradict as 9,10 can not appear in two 4-gon.
	\end{case}
	\begin{case}
		When $(a,b,c,d)=(6,9,10,5)$, then $lk(6)=C_{8}(1,9,[5,4,0],[7,11,2])$, $lk(4)=C_{8}(8,3,[0,$ $6,5],[10,9,1])$, which implies $lk(8)=C_{8}(3,4,[a_{1},a_{2},a_{3}],[7,0,1])$, $lk(3)=C_{8}(0,4,\allowbreak[8,b_{1},b_{2}],[b_{3},b_{4},3])$, where $a_{1}$, $a_{2}$, $a_{3}$, $b_{1}$, $b_{2}$, $b_{3}$, $b_{4}\in V$. In $lk(8)$ we see that [3,8] in adjacent edge of two 3-gon and in $lk(3)$ we see that [3,8] is an adjacent edge of a 3-gon an a 4-gon, which is a contradiction.
	\end{case}
	\begin{case}
		When $(a,b,c,d)=(6,9,10,4)$, then $lk(6)=C_{8}(1,9,[5,4,0],[7,a_{1},2])$ where $a_{1}\in V$ and $lk(5)=C_{8}(8,9,[6,0,4],[10,9,1])$, which implies $lk(9)=C_{8}(6,5,[8,$  $a_{2},a_{3}],[10,5,1])$ and $lk(8)=C_{8}(5,9,[b_{1},b_{2},b_{3}],[7,0,4])$ where $a_{2}$, $a_{3}$, $b_{1}$, $b_{2}$, $b_{3}\in V$. From $lk(9)$ we see that [8,9] is an adjacent edge of a 3-gon and a 4-gon and from $lk(8)$ we see that [8,9] is an adjacent edge of two 3-gon which make contradiction.
	\end{case}
	\begin{case}
		When $(a,b,c,d)=(6,9,10,11)$, then $lk(6)=C_{8}(1,9,[5,4,0],[7,a_{1},2])$, where $a_{1}\in V$. Now all six 4-gon are $[0,4,5,6]$, $[0,1,8,7]$, $[1,9,10,11]$, $[2,6,7,a_{1}]$, $[2,3,a_{2},a_{3}]$, $[4,a_{4},a_{5},a_{6}]$, where $a_{2}$, $a_{3}$, $a_{4}$, $a_{5}$, $a_{6}\in V$. Now 3 will not appear in $4^{\text{th}}$ 4-gon as 2,3 can not be appear in two 4-gon and 3 will also not appear in last 4-gon as [3,4] is an adjacent edge of two 3-gon in $lk(3)$, so in $lk(3)$ has only one 4-gon, which is not possible.
	\end{case}
	\begin{case}
		When $(a,b,c,d)=(9,5,10,11)$, then either $lk(5)=C_{8}(9,a_{1},[4,0,6],[10,$  $11,1])$ where $a_{1}\in V$ or $lk(5)=C_{8}(9,8,[6,0,4],[10,11,1])$. If $lk(5)=C_{8}(9,8,[6,0,4],[10,11,1])$ then $lk(8)=C_{8}(5,6,[7,0,1],[11,b_{1},9])$, where $b_{1}\in V$ which implies 8 will appear two times in $lk(7)$, which is a contradiction. If $lk(5)=C_{8}(9,a_{1},[4,0,6],[10,11,1])$, form $lk(0)$ and $lk(3)$, we see that [3,4] is an adjacent edge of two 3-gon, therefore [4,5] is an adjacent edge of a 3-gon and a 4-gon, therefore $[a_{1},4]$, $[a_{1},9]$ are adjacent edge of a 3-gon and a 4-gon, which implies $lk(a_{1})$ is not following the face sequence.
	\end{case}
	\begin{case}
		When $(a,b,c,d)=(9,6,10,11)$, then $lk(6)=C_{8}(9,7,[0,4,5],[10,11,1])$. From $lk(6)$ and $lk(9)$ we see that [0,7] and [7,9] respectively are adjacent edge of a 3-gon and a 4-gon, which implies face sequence is not following in $lk(7)$.
	\end{case}
	\begin{case}
		When $(a,b,c,d)=(9,10,6,11)$, then either $lk(9)=C_{8}(7,9,[10,1,11],[5,4,$  $0])$ or $lk(6)=C_{8}(7,9,[11,1,10],[5,4,0])$. In both cases $lk(9)$ will not complete.
	\end{case}
	\begin{case}
		When $(a,b,c,d)=(9,10,5,11)$ then $lk(2)=C_{8}(1,0,[3,c_{1},c_{2}],[d_{2},d_{1},9])$, $lk(9)=C_{8}(10,$ $1,[2,d_{2},d_{1}],[a_{8},a_{7},a_{1}]),lk(10) =C_{8}(9,a_{1},[a_{2},b_{2},b_{1}],[5,11,1])$, $lk(a_{2}) =C_{8}(a_{1},a_{3},\allowbreak[a_{4},\allowbreak a_{5},a_{6}],[b_{2},$ $b_{1},10])$, where $a_{1}$, $a_{2}$, $a_{3}$, $a_{4}$, $a_{5}$, $a_{6}$, $a_{7}$, $a_{8}$, $b_{1}$, $b_{2}$, $c_{1}$, $c_{2}$, $d_{1}$, $d_{2}\in V$. Now all six 4-gon are $[0,4,5,6]$, $[0,1,8,7]$, $[1,10,5,11]$, $[2,3,c_{1},c_{2}]$, $[2,9,d_{1},d_{2}]$, $[9,a_{1},a_{7},a_{8}]$. Now we see that 9 is in $lk(10)$, so 10 will not appear in last two 4-gon, therefore 10 will appear in $4^{\text{th}}$ 4-gon. Now in $4^{\text{th}}$ 4-gon, if $c_{2}=10$ and $d_{2}=5$ which implies 5 appear in three 4-gon, which is a contradiction. therefore $b_{1}=3$, $b_{2}=2$. 6 will not appear in $4^{\text{th}}$ 4-gon as if 6 appear in $4^{\text{th}}$ 4-gon then $a_{3}=7$, $a_{4}=0$, $a_{5}=4$, $a_{6}=5$, which implies [2,5,6] is a 3-gon i.e. $c_{2}=6$, $d_{2}=5$ which implies 5 appear in three 4-gon, which make contradiction. Now remaining possible value of $a_{2}$ are 7,8. We see that 3 will appear in last 4-gon and [3,4] is an adjacent edge of two 3-gon, therefore 4 will appear in $5^{\text{th}}$ 4-gon, which implies 6 will appear in last 4-gon as 4,6 will appear in only one 4-gon. Now if $a_{2}=7$, then after completing $lk(7)$ we see that $a_{3}=6$ and we have 6 appear in last 4-gon, which is a contradiction. If $a_{2}=8$, then $a_{3}=11$ and 11 appear in last 4-gon, which is not possible.
	\end{case}
	\begin{case}
		When $(a,b,c,d)=(5,9,10,11)$, then all six 4-gon are $[0,4,5,6]$, $[0,1,8,7]$, $[1,9,10,\allowbreak11]$, $[2,3,a_{1},a_{2}]$, $[2,5,a_{3},a_{4}]$, $[3,a_{5},a_{6},a_{7}]$ where $a_{1}$, $a_{2}$, $a_{3}$, $a_{4}$, $a_{5}$, $a_{6}$,$a_{7}\in V$. From $lk(3)$ we see that [3,4] is an adjacent edge of two 3-gon, therefore 4 will not appear in $4^{\text{th}}$ and last 4-gon. 4,5 will appear only in one 4-gon, so 4 will also not appear in $5^{\text{th}}$ 4-gon.
	\end{case}
	\begin{case}
		When $(a,b,c,d)=(9,4,10,11)$, then $lk(4)=C_{8}(9,3,[0,6,5],[10,11,1])$ and $lk(3)=C_{8}(0,4,[9,a_{1},a_{2}],[a_{3},a_{4},2])$, $lk(9)=C_{8}(1,4,[3,a_{2},a_{1}],[a_{5},a_{6},2])$, where $a_{1},a_{2},a_{3},\allowbreak a_{4},\allowbreak a_{5},a_{6}\in P$, where $P=\lbrace 5,6,7,8,10,11\rbrace$. All six 4-gon are [0,4,5,6], [0,1,8,7], [1,4,10,11], $[2,3,a_{3},a_{4}]$, $[2,9,a_{5},a_{6}]$, $[3,9,a_{1},a_{2}]$. If $a_{1}=5$, then $a_{6}\neq 6$ as $(5,6)$ can not appear in two 4-gon. Similarly $(a_{1},a_{2})\neq $ (6,5), (7,8), (8,7), (10,11), (11,10), and $(a_{3},a_{4})\neq$ (5,6), (6,5), (7,8), (8,7), (10,11), (11,10).\\
		%		\justify
		Let $A=\lbrace$(5,7,8,10), (5,7,8,11), (5,7,10,6), (5,7,10,8), (5,7,11,6), (5,7,11,8), (5,8,6,10), (5,8,6,\allowbreak11), (5,8,10,6), (5,8,10,7), (5,8,11,6), (5,8,11,7), (5,11,6,7), (5,11,6,8), (5,11,7, 10), (5,11,10,7), (5,11,10,8), (6,7,5,11), (6,7,8,10), (6,7,8,11), (6,7,10,5), (6,7,10,8), (6,7,11,5), (6,8,5,11), (6,8,\allowbreak10,5), (6,8,10,7), (6,8,11,5), (6,8,11,7), (6,10,5,7), (6,10,5, 8), (6,10,7,5), (6,10,7,11), (6,10,8,5), (6,10,8,11), (6,11,5,7), (6,11,5,8), (6,11,7,5), (6,11,7,10), (6,11,10,7), (6,11,10,8), (7,5,6,10), (7,5,6,11), (7,5,8,10), (7,5,8,11), (7,5, 11,6), (7,5,11,8), (7,10,5,8), (7,10,5,11), (7,10,6,8), (7,10,6,11), (7,10,8,5), (7,11,5,8), (7,11,5,10), (7,11,6,8), (7,11,6,10), (7,11,10,5), (7,11,10,6), (8,5,6,10), (8,5,6,11), (8,5, 7,10), (8,5,7,11), (8,5,11,6), (8,5,11,7), (8,6,7,5), (8,6,7,10), (8,6,7,\allowbreak11), (8,6,10,5), (8,6,10,7), (8,6,11,5), (8,6,11,7), (8,10,5,7), (8,10,\allowbreak5,11), (8,10,6,7), (8,10,6,11),  (8,11, 5,7),  (8,11,5,10),  (8,11,6,7),  (8,11,6,10),  (8,11,7,5), (8,11,10,5), (8,10,7,5), (8,11,10,6), (10,5,6,7), (10,5,6,8), (10,5,7,11), (10,5,8,6), (10,5,8,11), (10,5,11,7), (10,5,11,8), (10, 6,7,5), (10,6,7,11), (10,6,8,5), (10,6,8,11), (10,6,11,7), (10,7,5,\allowbreak8), (10,7,5,11), (10,7,8, 5), (10,7,8,6), (10,7,11,5), (10,7,11,6), (10,8,5,7), (10,8,5,11), (10,8,\allowbreak6,7), (10,8,6,11), (10,8,7,5), (10,8,11,5), (10,8,11,6), (10,8,11,7), (11,5,6,7), (11,5,6,8), (11,5,\allowbreak7,10), (11, 5,8,6), (11,5,8,10), (11,6,7,5), (11,6,7,10), (11,6,8,5), (11,6,8,10), (11,6,10,7), (11,6,10, 8), (11,7,5,8), (11,7,6,8), (11,7,6,10), (11,7,8,5), (11,7,8,6), (11,7,10,5), (11,7,10,6) $\rbrace$,\\
		$B = \{$ (5,7,6,8), (5,7,6,10), (5,7,6,11). (5,8,7,6), (5,8,7,10), (5,8,7,11), (5,10,6,7), (5,10,6,8), (5,10,6,11), (5,10,7,6), (5,10,7,11), (5,10,8,6), (5,10,8,11), (5,10,11,6), (5, 10,11,7), (5,10,11,\allowbreak8), (5,11,7,6), (5,11,8,6), (5,11,8,10), (6,7,5,10), (6,7,11,8), (6,8,5, 10), (6,8,7,5), (6,8,7,10), (6,8,7,\allowbreak11), (6,10,11,5), (6,10,11,7), (6,10,11,8), (6,11,5,10), (6,11,8,5), (6,11,8,10), (7,5,10,6), (7,5,\allowbreak10,8), (7,6,5,8), (7,6,5,10), (7,6,5,11), (7,6,8,5), (7,6,8,10), (7,6,8,11), (7,6,10,5), (7,6,\allowbreak10,8), (7,6,11,5), (7,6,11,8), (7,10,11,5), (7,10,11, 6), (7,10,11,8), (7,11,8,5), (7,11,8,6), (7,11,8,10), (8,5,7,6), (8,5,10,6), (8,5,10,7), (8,6, 5,7), (8,6,5,10), (8,6,5,11), (8,10,11,5), (8,10,11,6), (8,10,\allowbreak11,7), (8,11,7,6), (8,10,7,6), (10,5,7,6), (10,6,5,7), (10,6,5,8), (10,6,5,11), (10,6,11,8), (10,7,6,\allowbreak8), (10,7,6,11), (10,7, 11,8), (10,8,7,6), (10,8,7,11), (11,5,7,6), (11,5,10,6), (11,5,10,7), (11,5,\allowbreak10,8), (11,6,5,7), (11,6,5,8), (11,6,5,10), (11,7,5,10), (11,8,5,7), (11,8,5,10), (11,8,6,7), (11,8,6,\allowbreak10), (11,8, 7,5), (11,8,7,6), (11,8,7,10), (11,8,10,5), (11,8,10,6), (11,8,10,7) $\} $.\\
		$C=\lbrace$(5,7,8,6), (5,8,6,7), (5,11,6,10), (5,11,10,6), (6,7,5,8), (6,7,8,5), (6,8,5,7),(6,10,5, 11), (6,11,10,5), (7,5,6,8), (7,5,8,6), (7,10,8,11), (7,11,10,8), (8,5,6,7), (8,11,7,10), (8,11,10,7), (8,10,7,11), (10,5,6,11), (10,5,11,6), (10,6,11,5), (10,7,8,11), (11,5,6,10), (11,6,10,5), (11,7,8,\allowbreak10), (11,7,10,8)$\rbrace$. Therefore $(a_{1},a_{2},a_{3},a_{4})\in A\cup B\cup C$.
		
		Now for any $(a_{1},a_{2},a_{3},a_{4})\in B$, SEM will not orientable.
		
		For any $(a_{1},a_{2},a_{3},a_{4})$, $a_{5},a_{6}\in P\setminus\lbrace a_{1},a_{2},a_{3},a_{4}\rbrace$ and $\lbrace a_{5},a_{6}\rbrace\neq\lbrace 5,6\rbrace$ or $\lbrace 7,8\rbrace$ or $\lbrace 10,11\rbrace$. Therefore, for any $(a_{1},a_{2},a_{3},a_{4})\in C$, SEM will not complete.
		
		If $(a_{1},a_{2},a_{3},a_{4})=(5,7,8,10)$, them $lk(5)=C_{8}(10,b,[7,3,9],[6,0,4])$, $lk(7)=C_{8}(6,b,[5,\allowbreak9,3],[8,1,0])$, $lk(10)=C_{8}(5,b,[8,3,2],[11,1,4])$. From $lk(5)$, $b\in \{8,11\}$, from $lk(7)$, $b\in\{10,11\}$ and from $lk(10)$, $b\in\{6,7,9\}$, which implies $b$ does not exists.
		
		If $(a_{1},a_{2},a_{3},a_{4})=(5,7,8,11)$, then $lk(5)=C_{9}(10,b,[7,3,9],[6,0,4])$, $lk(7)=C_{8}(6,b,[5,\allowbreak9,3],[8,1,0])$. From $lk(5)$, $b\in\{8,11\}$ and from $lk(7)$, $b\in\{10,11\}$, which implies $b=11$, this implies $lk(11)=C_{8}(5,7,[6,c,8],[1,4,10])$, $c\in V$, which implies $[6,11]$ form an edge in a 4-gon, which is a contradiction.
		
		If $(a_{1},a_{2},a_{3},a_{4})=(5,7,10,a_{4})$, then $lk(5)=C_{8}(10,8,[7,3,9],[6,0,4])$ and $lk(7)=C_{8}(6,10,[3,9,$ $5],[8,1,0])$, which implies face sequence is not followed in $lk(10)$. Therefore for (5,7,10,6), (5,7,10,8), SEM is not possible.
		
		If $(a_{1},a_{2},a_{3},a_{4})=(5,7,11,6)$, then $lk(5)=C_{8}(10,8,[7,3,9],[6,0,4])$ and $lk(7)=C_{8}(6,\allowbreak11,[3,9,5],[8,1,0])$, which implies $lk(8)=C_{8}(5,10,[b_{1},b_{2},11],[1,0,$ $7])$, this implies $[8,11]$ occur in a 4-gon, which is not possible by above six 4-gon.
		
		If $(a_{1},a_{2},a_{3},a_{4})=(5,7,11,8)$, then $lk(5)=C_{8}(10,8,[7,3,9],[6,0,4])\Rightarrow lk(7)=C_{8}(6,11,\allowbreak[3,9,5],[8,1,0])\Rightarrow lk(11)=C_{8}(6,7,[3,2,8],[1,4,10])\Rightarrow lk(8)=C_{8}(5,10,\allowbreak [2,3,\allowbreak11],[1,0,7])\allowbreak\Rightarrow lk(6)=C_{8}(7,11,[10,2,9],[5,4,0])\Rightarrow lk(2)=C_{8}(0,1,[9,6,10],\allowbreak [8,11,3])\Rightarrow lk(10)=C_{8}(5,8,[2,9,6],[11,1,$ $4])\Rightarrow lk(9)=C_{8}(1,4,[3,7,5],[6,10,2])$. This map is same as $ \boldsymbol{KNO_{1[(3^3,4,3,4)]}} $, given in figure \ref{33434non}.
		
		If $(a_{1},a_{2},a_{3},a_{4})=(5,8,6,b)$, where $b\in\{10,11\}$, then $lk(8)=C_{8}(11,6,[3,9,5],[7,\allowbreak 0,1])\Rightarrow lk(6)=C_{8}(7,8,[3,2,b],[5,4,0])$. We see that $[6,8,11]$ is in $lk(8)$, but is not in $lk(6)$, which is a contradiction. Therefore for $(5,8,6,10),(5,8,6,11)$, SEM will not complete.
		
		If $(a_{1},a_{2},a_{3},a_{4})=(5,8,10,b)$, where $b\in \{6,7\}$ then $lk(10)=C_{8}(5,8,[3,2,b],[11,\allowbreak 1,4])$. But we have $[4,5]$ and $[5,8]$ form an edge in two 4-gon, which implies face sequence is not followed in $lk(5)$. Therefore for $(5,8,10,6),(5,8,10,7)$, SEM will not complete.
		
		For $(a_{1},a_{2},a_{3},a_{4})=(a_{1},8,11,a_{4})$, $[8,1]$ and $[8,3]$ form an edge in two 4-gon, which implies face sequence will not followed in $lk(8)$. Therefore, for (5,8,11,6), (5,8,11,7), (6,8,11,5), (6,8,11,7), (10,8,11,5), (10,8,11,7), SEM will not complete.
		
		If $(a_{1},a_{2},a_{3},a_{4})=(5,11,6,7)$, then $lk(11)=C_{8}(8,6,[3,9,5],[10,4,1])\Rightarrow lk(6)=C_{8}(8,\allowbreak11,[3,2,7],[0,4,5])\Rightarrow lk(5)=C_{8}(8,10,[11,3,9],[4,0,6])$, which implies 4 will occur two times in $lk(9)$.
		
		If $(a_{1},a_{2},a_{3},a_{4})=(5,11,6,8)$, then $lk(8)=C_{8}(11,5,[6,3,2],[7,0,1])$ and we have $[1,11]$ and $[5,11]$ form an edge of two 4-gon, which implies face sequence will not followed in $lk(11)$.
		
		If $(a_{1},a_{2},a_{3},a_{4})=(5,11,7,10)$, then $lk(11)=C_{8}(8,7,[3,9,5],[10,4,1])$ and we have $[1,8]$ and $[7,8]$ form an edge of two 4-gon, which implies face sequence will not followed in $lk(8)$.
		
		If $(a_{1},a_{2},a_{3},a_{4})=(5,11,10,7)$, then $lk(11)=C_{8}(8,b,[5,9,3],[10,4,1])\Rightarrow lk(10)\allowbreak =C_{8}(5,b,[7,2,$ $3],[11,1,4])$. From $lk(11)$ we have $b\in\{6,7\}$ and from $lk(10)$ we have $b\in\lbrace 6,8,9\rbrace$, which implies $b=6$ and we have $[6,8]$ and $[1,8]$ form an edge of two 4-gon, which implies face sequence will not followed in $lk(8)$.
		
		If $(a_{1},a_{2},a_{3},a_{4})=(5,11,10,8)$, then $lk(11)=C_{8}(8,b_{1},[5,9,3],[10,4,1])$. Now $[5,b_{1}]$ form an edge of a 4-gon and $b_{1}$ can not be 9, therefore $b_{1}=6$. And $lk(10)=C_{8}(5,b_{2},[8,2,3],\allowbreak[11,1,\allowbreak4])$, where $[8,b_{2}]$ form an edge of a 4-gon and $b_{2}$ can not be 2, therefore $b_{2}=7$. These implies $lk(8)=C_{8}(11,6,[2,3,10],[7,0,1])\Rightarrow lk(2)=C_{8}(0,1,[9,7,6],[8,10,3])\Rightarrow lk(6)=C_{8}(11,8,[2,9,7],[0,4,5])$ $\Rightarrow lk(7)=C_{8}(10,5,[9,\allowbreak 2,6],[0,1,8])\Rightarrow lk(9)=C_{8}(1,4,[3,11,5],[7,\allowbreak6,2])\Rightarrow lk(5)=C_{8}(10,7,[9,3,11],[6,0,\allowbreak 4])$. This is the SEM $ \boldsymbol{KO_{1[(3^3,4,3,4)]}} $, given in figure \ref{33434or}
		
		If $(a_{1},a_{2},a_{3},a_{4})=(6,7,5,11)$, then $lk(5)=C_{8}(10,7,[3,2,11],[6,0,4])\Rightarrow lk(11)\allowbreak =C_{8}(8,6,[5,3,$ $2],[10,4,1])\Rightarrow lk(6)=C_{8}(11,8,[9,3,7],[0,4,5])\Rightarrow lk(10)=C_{8}(5,7,\allowbreak [8,9,2],\allowbreak[11,1,4])\Rightarrow lk(7)=C_{8}(10,5,[3,9,6],[0,1,8])\Rightarrow lk(8)=C_{8}(11,6,[9,2,10],\allowbreak [7,0,1])\Rightarrow lk(9)=C_{8}(1,4,[3,7,6],[8,10,2])$ $\Rightarrow lk(2)=C_{8}(0,1,[9,8,10],[11,5,3])$. This SEM isomorphic to $ \boldsymbol{KO_{1[(3^3,4,3,4)]}} $, given in figure \ref{33434or}, under the map $ (1,4)(2,3)(5,8)(6,7)(10,11) $.
		
		For $(a_{1},a_{2},a_{3},a_{4})=(6,7,8,a_{4}), a_{4}\in\{10,11\}$, 6 will appear two times in $lk(7)$. Therefore, for $(6,7,8,10),(6,7,8,11)$, SEM is not possible.
		
		If $(a_{1},a_{2},a_{3},a_{4})=(6,7,10,5)$, then $lk(10)=C_{8}(3,b_{1},[11,1,4],[5,2,3])$ and $[11,b_{1}]$ form an edge in a 4-gon, therefore, $b_{1}=8$ and we have $[7,8]$ and $[3,7]$ are two edges of two 4-gon, which implies face sequence is not followed in $lk(7)$.
		
		If $(a_{1},a_{2},a_{3},a_{4})=(6,7,10,8)$, then $lk(10)=C_{8}(5,7,[3,2,8],[11,1,4])$, which implies 8 will appear two times in $lk(11)$.
		
		If $(a_{1},a_{2},a_{3},a_{4})=(6,7,11,5)$, then $lk(11)=C_{8}(8,7,[3,2,5],[10,4,1])$, which implies 5 will appear two times in $lk(10)$.
		
		For $(a_{1},a_{2},a_{3},a_{4})=(6,8,a_{3},a_{4}),a_{3}\neq 7$, considering $lk(8)$, we see that $[6,7,8]$ is a face, which is not possible as then 6 will appear two times in $lk(7)$. Therefore for (6,8,5,11), (6,8,10,5), (6,8,10,7), SEM will not complete.
		
		For $(a_{1},a_{2},a_{3},a_{4})=(a_{1},10,5,a_{4})$, face sequence will not followed in $lk(10)$. Therefore, for $(6,10,5,7)$, $(6,10,5,8)$, $(7,10,5,8)$, $(7,10,5,11)$, $(8,10,5,7)$, $(8,10,5,\allowbreak 11)$, SEM will not complete.
		
		For $(a_{1},a_{2},a_{3},a_{4})=(6,10,7,a_{4})$, $lk(10)=C_{8}(5,7,[3,9,6],[11,1,4])$, and we have [5,7], [3,7] are two edges of two 4-gon, which implies face sequence is not followed in $lk(7)$. Therefore, for $(6,10,7,5),(6,10,7,11)$, SEM will not complete.
		
		If $(a_{1},a_{2},a_{3},a_{4})=(6,10,8,5)$, then $lk(10)=C_{8}(5,8,[3,9,6],[11,1,4])$ and we have $[5,8],[5,4]$ are two edges of two 4-gon, which implies face sequence will not followed in $lk(5)$.
		
		If $(a_{1},a_{2},a_{3},a_{4})=(6,10,8,11)$, then $lk(10)=C_{8}(5,8,[3,9,6],[11,1,4])\Rightarrow lk(8)=\allowbreak C_{8}(10,5,[7,0,1],[11,2,3])\Rightarrow lk(6)=C_{8}(11,7,[0,4,5],[9,3,10])\Rightarrow lk(11)=C_{8}(6,7,[2,3,8],\allowbreak[1,4,10])\Rightarrow lk(5)=C_{8}(10,8,[7,2,9],[6,0,4])\Rightarrow lk(7)=C_{8}(11,6,[0,1,8],[5,\allowbreak 9,2])\Rightarrow lk(9)=C_{8}(1,4,[3,10,6],[5,7,2])\Rightarrow lk(2)=C_{8}(0,1,[9,5,7],[11,8,3])$. This map isomorphic to $ \boldsymbol{KO_{1[(3^3,4,3,4)]}} $, given in figure \ref{33434or}, under the map $ (0,9)(1,3)(2,4)(5,6,7)(8,11,10) $.
		
		For $(a_{1},a_{2},a_{3},a_{4})=(6,11,5,a_{4})$, $lk(11)=C_{8}(8,5,[3,9,6],[10,4,1])\Rightarrow lk(10)=C_{8}(6,\allowbreak b_{1},\allowbreak[b_{2},b_{3},5],[4,1,11])$, which implies $[5,10]$ form an edge of a 4-gon, which is not possible by above six 4-gon. Therefore for $(6,11,5,7), (6,11,5,8)$, SEM will not complete.
		
		If $(a_{1},a_{2},a_{3},a_{4})=(6,11,7,5)$, then $lk(7)=C_{8}(6,11,[3,2,5],[8,1,0])$, and we have $[0,6], [6,11]$ form two edges of two 4-gon, which implies face sequence will not followed in $lk(6)$.
		
		If $(a_{1},a_{2},a_{3},a_{4})=(6,11,7,10)$, then 8 will occur two times in $lk(11)$.
		
		If $(a_{1},a_{2},a_{3},a_{4})=(6,11,10,7)$, then $lk(10)=C_{8}(5,6,[7,2,3],[11,1,4])$, which implies 7 will appear two times in $lk(6)$.
		
		If $(a_{1},a_{2},a_{3},a_{4})=(6,11,10,8)$, then $lk(11)=C_{8}(b_{1},8,[1,4,10],[3,9,6])$. We have $[1,8]$ is an edge in a 4-gon, therefore $[6,b_{1}]$ form an edge in a 4-gon, which implies $b_{1}=5$. Therefore $lk(5)=C_{8}(11,8,[b_{2},b_{3},10],[4,0,6])$, which implies $[5,10]$ form an edge in a 4-gon, which is not possible.
		
		If $(a_{1},a_{2},a_{3},a_{4})=(7,5,6,10)$, then $lk(5)=C_{8}(10,b_{1},[7,9,3],[6,0,4])$. We have [4,10] is an edge in a 4-gon, therefore $[7,b_{1}]$ form an edge in a 4-gon, therefore $b_{1}=8$. Now, $lk(6)=C_{8}(7,b_{2},[10,2,3],[5,4,0])$. We have $[0,7]$ form an edge in a 4-gon, therefore $[10,b_{2}]$ form an edge in a 4-gon, therefore $b_{2}=11$. Which implies $lk(7)=C_{8}(6,11,[9,3,5],[8,1,0])\Rightarrow lk(10)=C_{8}(5,8,[2,3,6],[11,1,4])\Rightarrow lk(8)=C_{8}(5,10,[2,9,11],[1,0,7])\Rightarrow lk(11)=C_{8}(6,\allowbreak7,\allowbreak[9,2,8],[1,4,10])\Rightarrow lk(2)=C_{8}(0,1,[9,11,8],[10,6,3])\Rightarrow lk(9)=C_{8}(1,4,[3.5.7],[11,8,\allowbreak2])$. This is the map $ \boldsymbol{KO_{2[(3^3,4,3,4)]}} $, given in figure \ref{33434or}.
		
		If $(a_{1},a_{2},a_{3},a_{4})=(7,5,6,11)$, then $lk(6)=C_{8}(7,b_{1},[11,2,3],[5,4,0])$. [0,7] form an edge in a 4-gon, therefore $[11,b_{1}]$ will form an edge in a 4-gon, therefore $b_{1}=10$, which implies $lk(10)=C_{8}(6,7,[b_{2},b_{3},5],[4,1,11])$, which implies [5,10] form an edge in a 4-gon, which is not possible.
		
		For $(a_{1},a_{2},a_{3},a_{4})=(7,5,a_{3},a_{4})$, if $a_{3}\neq 6$, then $lk(5)=C_{8}(10,a_{3},[3,9,7],[6,0,\allowbreak 4])$, which implies 7 will occur two times in $lk(6)$. Therefore (7,5,8,10), (7,5,8,11), (7,5,11,6), (7,5,11,8) are not possible.
		
		For $(a_{1},a_{2},a_{3},a_{4})=(7,10,6,a_{4})$, $lk(10)=C_{8}(5,6,[3,9,7],[11,1,4])$, which implies 6 will appear two times in $lk(5)$. Therefore $(7,10,6,8),(7,10,6,11)$ is not possible.
		
		If $(a_{1},a_{2},a_{3},a_{4})=(7,10,8,5)$, then $lk(8)=C_{8}(11,10,[3,2,5],[7,0,1])$, and we have [1,11],[10,11] form an edge in same 4-gon, therefore face sequence in $lk(11)$ is not followed.
		
		If $(a_{1},a_{2},a_{3},a_{4})=(7,11,5,8)$, then $lk(11)=C_{8}(8,5,[3,9,7],[10,4,1])$. We have [1,8], [5,8] form an edge in 4-gon, which implies face sequence is not followed in $lk(11)$.
		
		If $(a_{1},a_{2},a_{3},a_{4})=(7,11,5,10)$, then $lk(11)=C_{8}(8,5,[3,9,7],[10,4,1])\Rightarrow lk(5)\allowbreak =C_{8}(11,8,[b_{1},0,b_{2}],[10,2,3])$, where $b_{1},b_{2}\in \lbrace 4,6\rbrace$. We have $[3,11]$ is an edge in a 4-gon, therefore $[8,b_{1}]$ will form an edge in a 4-gon, therefore $b_{1}=6\Rightarrow b_{2}=4\Rightarrow lk(7)=C_{8}(6,10,[11,3,9],[8,1,0])\Rightarrow lk(10)=C_{8}(7,6,[2,3,5],[4,1,11])\Rightarrow lk(6)=C_{8}(7,10,[2,9,8],\allowbreak[5,4,0]),\Rightarrow lk(8)=C_{8}(5,11,[1,0,7],$ $[9,2,6]),\Rightarrow lk(2)=C_{8}(0,1,[9,8,6],[10,5,3]),\Rightarrow lk(9)\allowbreak=C_{8}(1,4,[3,11,7],[8,1,0])$. This map isomorphic to $ \boldsymbol{KNO_{1[(3^3,4,3,4)]}} $, given in figure \ref{33434non}, under the map $ (1,4)(2,3)(5,8)(6,7)(10,11) $.
		
		If $(a_{1},a_{2},a_{3},a_{4})=(7,11,6,8)$, then $lk(11)=C_{8}(8,6,[3,9,7],[10,4,1])\Rightarrow lk(6)=C_{8}(7,\allowbreak11,[3,2,8],[5,4,0])$. We have [0,7] and [7,11] form edges in 4-gon, therefore face sequence is not followed in $lk(7)$.
		
		If $(a_{1},a_{2},a_{3},a_{4})=(7,11,6,10)$, then $lk(11)=C_{8}(8,6,[3,9,7],[10,4,1])$, which implies $lk(10)=C_{8}(7,b_{1},[b_{2},b_{3},5],[4,1,11])$, which implies [5,10] form an edge in a 4-gon, which is a contradiction.
		
		If $(a_{1},a_{2},a_{3},a_{4})=(7,11,10,5)$, then 5 will occur two times in $lk(10)$.
		
		If $(a_{1},a_{2},a_{3},a_{4})=(7,11,10,6)$, then $lk(10)=C_{8}(b,5,[4,1,11],[3,2,6])$. We have [4,5] is an edge in a 4-gon, which implies $[6,b]$ form an edge in a 4-gon, which implies $b\in \lbrace 0,5\rbrace$, which is not possible.
		
		For $(a_{1},a_{2},a_{3},a_{4})=(8,5,6,a_{4})$, then $lk(5)=C_{8}(10,b,[8,9,3],[6,0,4])$. We have [4,10] form an edge in a 4-gon, therefore $[8,b]$ will form an edge in a 4-gon, which implies $b\in \lbrace 7,1\rbrace$, $b\neq 1$ as 10,1 appear in a same 4-gon, therefore $b=7$, which implies $lk(7)=C_{8}(5,10,[b_{1},b_{2},6],[0,1,8])$, which implies [6,7] form an edge in a 4-gon, which is a contradiction. Therefore $(8,5,6,10),(8,5,6,11)$ are not possible.
		
		For $(a_{1},a_{2},a_{3},a_{4})=(8,5,7,10)$, $lk(5)=C_{8}(10,7,[3,9,8],[6,0,4])$. We have [4,10], [7,10] form edges in 4-gon, which implies face sequence is not followed in $lk(10)$.
		
		If $(a_{1},a_{2},a_{3},a_{4})=(8,5,7,11)$, then $lk(5)=C_{8}(10,7,[3,9,8],[6,0,4])$, then $lk(7)$ will not complete.
		
		If $(a_{1},a_{2},a_{3},a_{4})=(8,5,11,a_{4})$, $lk(5)=C_{8}(10,11,[3,9,8],[6,0,4])$. We have $[4,10],[10,\allowbreak11]$ form edges in 4-gon, which implies face sequence is not followed in $lk(10)$. Therefore $(8,5,11,6),(8,5,11,7)$ are not possible.
		
		For $(a_{1},a_{2},a_{3},a_{4})=(8,6,7,a_{4})$, face sequence is not followed in $lk(6)$. Therefore (8,6,7,5), (8,6,7,10), (8,6,7,11) are not possible.
		
		If $(a_{1},a_{2},a_{3},a_{4})=(8,6,10,5)$, then $lk(6)=C_{8}(7,10,[3,9,8],[5,4,0]),\Rightarrow lk(10)=C_{8}(6,7,\allowbreak [11,1,4],[5,2,3])\Rightarrow lk(8)=C_{8}(11,5,[6,3,9],[7,0,1])\Rightarrow
		lk(5)=C_{8}(8,11,[2, 3,10],\allowbreak[4,0,6])\allowbreak \Rightarrow lk(11)=C_{8}(8,5,[2,9,7],[10,4,1])\Rightarrow
		lk(2)=C_{8}(0,1,[9,7,11],[5,\allowbreak 10,3])\Rightarrow
		lk(9)=C_{9}(1,4,[3,6,8],[7,11,2])$ $\Rightarrow lk(7)=C_{8}(6,10,[11,2,9],[8,1,0])$. This map is isomorphic to $ \boldsymbol{KO_{1[(3^3,4,3,4)]}} $, given in figure \ref{33434or}, under the map $ (0,1,4)(2,9,3)(5,7,11)(6,8,10) $.
		
		If $(a_{1},a_{2},a_{3},a_{4})=(8,6,10,7)$, then $lk(6)=C_{8}(7,10,[3,9,8],[5,4,0])$. We have [0,7], [7,10] are edges in 4-gon, therefore face sequence is not followed in $lk(7)$.
		
		If $(a_{1},a_{2},a_{3},a_{4})=(8,6,11,a_{4})$, then $lk(11)=C_{8}(8,6,[3,2,a_{4}],[10,4,1])$. We have [1,8], [6,8] are edges in 4-gon, which implies face sequence is not followed in $lk(8)$. Therefore (8,6,11,5), (8,6,11,7) are not possible.
		
		If $(a_{1},a_{2},a_{3},a_{4})=(8,10,6,7)$, then $lk(6)=C_{8}(10,b_{1},[5,4,0],[7,2,3])$. Since [3,10] is an edge in a 4-gon, therefore $b_{1}\in \lbrace 2,9,11\rbrace$. $b_{1}\neq 2$. Since (9,10) occur in same 4-gon and does not form an edge, therefore $b_{1}\neq 9$. $b_{1}\neq 11$ as [10,11] is an edge in a 4-gon.
		
		If $(a_{1},a_{2},a_{3},a_{4})=(8,10,6,11)$, then $lk(10)=C_{8}(5,6,[3,9,8],[11,1,4])$, which implies 6 will appear two times in $lk(5)$.
		
		For $(a_{1},a_{2},a_{3},a_{4})=(8,11,5,a_{4})$, $lk(5)=C_{8}(10,11,[3,2,a_{4}],[6,0,4])$. We have 4,10,11 occur in same 4-gon, therefore 11 will occur two times in $lk(10)$. Therefore (8,11,5,7), (8,11,5,10) are not possible.
		
		If $(a_{1},a_{2},a_{3},a_{4})=(8,11,6,10)$, then $lk(6)=C_{8}(7,11,[3,2,10],[5,4,0])$, which implies 10 will occur two times in $lk(5)$. $(a_{1},a_{2},a_{3},a_{4})=(8,11,6,7)$ then $ lk(6) $ will not possible.
		
		If $(a_{1},a_{2},a_{3},a_{4})=(8,11,7,5)$, then $lk(7)=C_{8}(6,11,[3,2,5],[8,1,0])\Rightarrow  lk(11)=C_{8}(7,6,\allowbreak[10,4,1],[8,9,3])\Rightarrow
		lk(5)=C_{8}(8,10,[4,0,6],[2,3,7])\Rightarrow
		lk(8)=C_{8}(5,10,[9,3,11],\allowbreak[1,0,7])\allowbreak \Rightarrow lk(10)=C_{8}(8,5,[4,1,11],[6,2,9])\Rightarrow
		lk(2)=C_{8}(0,1,[9,10,6],[5,7,\allowbreak 3])\Rightarrow
		lk(9)=C_{8}(1,4,[3,11,8],[10,6,2])$ $\Rightarrow
		lk(6)=C_{8}(11,7,[0,4,5],[2,9,10])$. This map isomorphic to $ \boldsymbol{KO_{1[(3^3,4,3,4)]}} $, given in figure \ref{33434or}, under the map $ (0,4,1)(2,3,9)(5,11,7)(6,10,8) $.
		
		If $(a_{1},a_{2},a_{3},a_{4})=(8,11,10,a_{4})$, then 8 will appear two times in $lk(11)$. Therefore (8,11,10,5), (8,11,10,6) are not possible.
		
		If $(a_{1},a_{2},a_{3},a_{4})=(8,10,7,5)$, then $lk(10)=C_{8}(5,7,[3,9,8],[11,1,4])$. We have [4,5], [5,7] are edges of 4-gon, therefore face sequence is not followed in $lk(5)$.
		
		If $(a_{1},a_{2},a_{3},a_{4})=(10,5,6,7)$ then $ C(0,4,5,3,2,7)\in lk(6) $ which is a contradiction.
		
		If $(a_{1},a_{2},a_{3},a_{4})=(10,5,6,8)$, then $lk(6)=C_{8}(7,b_{1},[8,2,3],[5,4,0])$. Since [0,7] is an edge in a 4-gon, therefore $[8,b_{1}]$ will form an edge in a 4-gon, which implies $b_{1}=1$, which is a contradiction as 1,7 appear in same 4-gon.
		
		If $(a_{1},a_{2},a_{3},a_{4})=(10,5,7,11)$, then $lk(7)=C_{8}(6,5,[3,2,11],[8,1,0])$. which implies 5 will appear two times in $lk(6)$.
		
		If $(a_{1},a_{2},a_{3},a_{4})=(10,5,8,6)$, then $lk(8)=C_{8}(11,5,[3,2,6],[7,0,1])$, which implies 6 will appear two times in $lk(7)$.
		
		If $(a_{1},a_{2},a_{3},a_{4})=(10,5,8,11)$, then $lk(5)=C_{8}(b_{1},8,[3,9,10],[4,0,6])$. We have [3,8] is an edge of a 4-gon, therefore $[6,b_{1}]$ will form an edge in a 4-gon, which implies face sequence is not followed in $lk(6)$.
		
		If $(a_{1},a_{2},a_{3},a_{4})=(10,5,11,7)$, then $lk(11)=C_{8}(8,5,[3,2,7],[10,4,1])\Rightarrow
		lk(7)\allowbreak =C_{8}(6,10,[11,3,2],[8,1,0])\Rightarrow
		lk(10)=C_{8}(7,6,[9,3,5],[4,1,11])\Rightarrow
		lk(5)=C_{8}(11,8,[6,0,4],\allowbreak[10,9,3])\Rightarrow
		lk(8)=C_{8}(11,5,[6,9,2],[7,0,1])\Rightarrow
		lk(2)=C_{8}(0,1,[9,6,8],\allowbreak [7,11,3])\Rightarrow
		lk(9)=C_{8}(1,4,[3,5,10],[6,8,2])\Rightarrow
		lk(6)=C_{8}(10,7,[0,4,5],[8,2,9])$. This map isomorphic to $ \boldsymbol{KO_{2[(3^3,4,3,4)]}} $, given in figure \ref{33434or}, under the map $ (0,9)(1,2)(3,4)(6,7,11,10) $.
		
		If $(a_{1},a_{2},a_{3},a_{4})=(10,5,11,8)$, then $lk(11)=C_{8}(b_{1},5,[3,2,8],[1,4,10])$. We have [3,5] is an edge in a 4-gon, therefore $[10,b_{1}]$ will form an edge in a 4-gon, which implies $b_{1}\in \lbrace 5,9\rbrace$ and any $b_{1}$ is not possible as 5,9 appear in same 4-gon and does not form an edge.
		
		For $(a_{1},a_{2},a_{3},a_{4})=(10,6,7,a_{4})$, face sequence is not followed in $lk(6)$. Therefore (10,6,7,5), (10,6,7,11) are not possible.
		
		For $(a_{1},a_{2},a_{3},a_{4})=(10,6,8,a_{4})$, $lk(6)=C_{8}(7,8,[3,9,10],[5,4,0])$, which implies 8 will occur two times in $lk(7)$. Therefore (10,6,8,5), (10,6,8,11) are not possible.
		
		If $(a_{1},a_{2},a_{3},a_{4})=(10,6,11,7)$, then $lk(11)=C_{8}(8,6,[3,2,7],[10,4,1])$, which implies $lk(6)$ is not possible.
		
		If $(a_{1},a_{2},a_{3},a_{4})=(10,7,5,a_{4})$, then $lk(5)=C_{8}(10,7,[3,2,a_{4}],[6,0,4])$. We have [4,10], [7,10] are edges in 4-gon, which implies face sequence is not followed in $lk(10)$. Therefore (10,7,5,8), (10,7,5,11) are not possible.
		
		If $(a_{1},a_{2},a_{3},a_{4})=(10,7,8,5)$, then $lk(8)=C_{8}(11,b,[5,2,3],[7,0,1])\Rightarrow lk(7)=C_{8}(b_{1},6,\allowbreak[0,1,8],[3,9,10])$. We have [0,6] is an edge in a 4-gon, therefore $[10,b_{1}]$ will form an edge in a 4-gon, which implies $b_{1}\in \lbrace 4,11\rbrace$. $b_{1}\neq 4$ as 4,6 occur in same 4-gon and does not form an edge, therefore $b_{1}=11$, which implies $lk(10)=C_{8}(5,b_{2},[9,3,7],[11,1,4])$. Since [4,5] form an edge in a 4-gon, therefore $[9,b_{2}]$ will form an edge in a 4-gon, which implies $b_{2}\in \lbrace 2,6\rbrace$. If $b_{2}=2$, then face sequence is not followed in $lk(2)$ and if $b_{2}=6$, then 6 will appear two times in $lk(5)$.
		
		If $(a_{1},a_{2},a_{3},a_{4})=(10,7,8,6)$, then $lk(7)=C_{8}(6,b_{1},[10,9,3],[8,1,0])$. We have [0,6] is an edge in a 4-gon, therefore $[10,b_{1}]$ will form an edge in a 4-gon, which implies $b_{1}\in \lbrace 4,11\rbrace$. $b_{1}\neq 4$ as 4,6 appear in a same 4-gon and it does not form an edge, therefore $b_{1}=11$. Therefore $lk(11)=C_{8}(7,6,[b_{2},b_{3},8],[1,4,10])$, from there we see that [8,11] is an edge in a 4-gon, which is not possible.
		
		If $(a_{1},a_{2},a_{3},a_{4})=(10,7,11,a_{4})$, then $lk(7)=C_{8}(6,11,[3,9,10],[8,1,0])$, which implies $lk(11)$ is not possible. Therefore (10,7,11,5), (10,7,11,6) are not possible.
		
		If $(a_{1},a_{2},a_{3},a_{4})=(10,8,5,7)$, then $lk(8)=C_{8}(11,5,[3,9,10],[7,0,1])$, which implies $lk(5)$ is not possible.
		
		For $(a_{1},a_{2},a_{3},a_{4})=(10,8,a_{3},11)$, if $a_{3}\neq 7$, then $lk(8)=C_{8}(11,a_{3},[3,9,10],[7,0,\allowbreak 1])$. We have [1,11], [$a_{3}$,11] are edges in 4-gon, therefore face sequence is not followed in $lk(11)$. Therefore (10,8,5,11), (10,8,6,11) are not possible.
		
		If $(a_{1},a_{2},a_{3},a_{4})=(10,8,7,5)$, then $lk(8)=C_{8}(b_{1},11,[1,0,7],[3,9,10])\Rightarrow lk(7)=C_{8}(6,b_{2},[5,2,$ $3],[8,1,0])$. Since [0,6] is an edge in a 4-gon, therefore $[5,b_{2}]$ will form an edge in a 4-gon, which implies $b_{2}\in\lbrace 4,6\rbrace$, which is not possible.
		
		If $(a_{1},a_{2},a_{3},a_{4})=(10,8,11,6)$ then face sequence will not follow in $ lk(8) $.		
		
		If $(a_{1},a_{2},a_{3},a_{4})=(11,5,6,7)$, then 0 will appear two times in $lk(6)$.
		
		If $(a_{1},a_{2},a_{3},a_{4})=(11,5,6,8)$, then $lk(5)=C_{8}(10,b_{1},[11,9,3],[6,0,4])$. Since [4,10] is an edge in a 4-gon, therefore $[11,b_{1}]$ will form an edge in a 4-gon, which implies $b_{1}\in \lbrace 1,10\rbrace$. $b_{1}=10$ is not possible and $b_{1}=1$ is not possible as then 4 will occur two times in $lk(10)$.
		
		If $(a_{1},a_{2},a_{3},a_{4})=(11,5,7,10)$, then $lk(5)=C_{8}(10,7,[3,9,11],[6,0,4])$, which implies 3 will occur two times in $lk(7)$.
		
		For $(a_{1},a_{2},a_{3},a_{4})=(11,5,8,a_{4})$, $lk(5)=C_{8}(10,8,[3,9,11],[6,0,4])$ and then $lk(8)$ will not possible. Therefore (11,5,8,6), (11,5,8,10) are not possible.
		
		For $(a_{1},a_{2},a_{3},a_{4})=(11,6,7,a_{4})$, then face sequence in not followed in $lk(6)$. Therefore (11,6,7,5), (11,6,7,10) are not possible.
		
		For $(a_{1},a_{2},a_{3},a_{4})=(11,6,8,a_{4})$, then $lk(6)=C_{8}(7,8,[3,9,11],[5,4,0])$, which implies 8 will occur two times in $lk(7)$. Therefore (11,6,8,5), (11,6,8,10) are not possible.
		
		If $(a_{1},a_{2},a_{3},a_{4})=(11,6,10,7)$, then $lk(6)=C_{8}(7,10,[3,9,11],[5,4,0])$. Since [0,7], [7,10] are edges of 4-gon, therefore face sequence is not followed in $lk(7)$.
		
		If $(a_{1},a_{2},a_{3},a_{4})=(11,6,10,8)$, then $lk(6)=C_{8}(7,10,[3,9,11],[5,4,0])$, which implies $lk(10)$ is not possible.
		
		If $(a_{1},a_{2},a_{3},a_{4})=(11,7,5,8)$, then $lk(7)=C_{8}(6,5,[3,9,11],[8,1,0])$, which implies 5 will appear two times in $lk(6)$.
		
		For $(a_{1},a_{2},a_{3},a_{4})=(11,7,6,a_{4})$, face sequence will not followed in $lk(6)$. Therefore (11,7,6,8), (11,7,6,10) are not possible.
		
		If $(a_{1},a_{2},a_{3},a_{4})=(11,7,8,5)$, then $lk(7)=C_{8}(6,b_{1},[11,9,3],[8,1,0])$. Since [0,6] is an edge in a 4-gon, therefore $[11,b_{1}]$ will form an edge in a 4-gon, which implies $b_{1}\in \lbrace 4,10\rbrace$. We see that for each $b_{1}$, 6, $b_{1}$ will appear in same 4-gon, which is a contradiction.
		
		If $(a_{1},a_{2},a_{3},a_{4})=(11,7,8,6)$, then $lk(7)=C_{8}(6,b_{1},[11,9,3],[8,1,0])$. Since [0,6] is an edge in a 4-gon, therefore $[11,b_{1}]$ will form an edge in a 4-gon, which implies $b_{1}=10$. Then $lk(8)=C_{8}(11,5,[6,2,3],[7,0,1])\Rightarrow
		lk(6)=C_{8}(7,10,[2,3,8],[5,4,0])\Rightarrow
		lk(10)=C_{8}(6,7,[11,1,4],[5,9,2])$ $\Rightarrow
		lk(5)=C_{8}(8,11,[9,2,10],[4,0,6])\Rightarrow
		lk(11)=C_{8}(8,5,[9,3,7],\allowbreak [10,4,1])\Rightarrow
		lk(9)=C_{8}(1,4,[3,7,11],[5,10,2])\Rightarrow
		lk(2)=C_{8}(0,1,\allowbreak [9,5,10],[6,8,3])$. This map is isomorphic to $ \boldsymbol{KO_{1[(3^3,4,3,4)]}} $, given in figure \ref{33434or}, under the map $ (0,2,1,9,4,3)(5,10,11)\allowbreak(6,8,7) $.
		
		For $(a_{1},a_{2},a_{3},a_{4})=(11,7,10,a_{4})$, $lk(6,10,[3,9,11],[8,1,0])\Rightarrow lk(10)=C_{8}(7,6, [11,1,\allowbreak4],[a_{4},2,3])$. Therefore $a_{4}=6$ is not possible and if $a_{4}=5$, then from $lk(6)$ we will see that [6,11] will form an edge in a 4-gon, which is not possible. Therefore (11,7,10,5), (11,7,10,6) are not possible.
		
		If $(a_{1},a_{2},a_{3},a_{4})=(10,8,6,7)$, then $lk(8)=C_{8}(11,6,[3,9,10],[7,0,1])\Rightarrow lk(6)=C_{8}(8,\allowbreak11,[5,4,0],[7,2,3])\Rightarrow lk(7)=C_{8}(b_{1},10,[8,1,0],[6,3,2])$. Since [8,10] is an edge in a 4-gon, therefore $[2,b_{1}]$ eill form an edge in a 4-gon, which implies $b_{1}\in \lbrace 5,9,11\rbrace$. $b_{1}$ can not be 2 and $b_{1}\neq 11$ as [10,11] is an edge in a 4-gon. Therefore $b_{1}=5$, which implies $lk(5)=C_{8}(7,10,[4,0,6],[11,9,2])\Rightarrow lk(10)=C_{8}(5,7,[8,3,9],[11,1,4])\Rightarrow
		lk(11)=C_{8}(8,6,[5,2,9],[10,4,1])\Rightarrow
		lk(2)=C_{8}(0,1,[9,11,5],[7,6,3])$ $\Rightarrow
		lk(9)=C_{8}(1,4,[3,8,10],\allowbreak[11,5,2])\Rightarrow
		lk(7)=C_{8}(10,5,[2,3,6],[0,1,8])$. This map isomorphic to $ \boldsymbol{KO_{2[(3^3,4,3,4)]}} $, given in figure \ref{33434or}, under the map $ (0,3,4,9,1,2)(5,7,6)(8,10,11) $.
		
		\textbf{Non-orientable cases:}
		
		We see that for any $ (a_1,a_2,a_3,a_4) $, $ a_5,a_6 $ will appear in same 4-gon, therefore the following cases are not possible: (5,7,6.8), (5,8,7,6), (5,10,11,6), (5,10,6,11), (6,8,7,5), (6,10,11,5), (6,11,5,10), (7,6,5,8), (7,6,8,5), (7,10,11,8), (7,11,8,10), (8,6,5,7), (8,10,11,7), (8,5,7,6), (10,6,\allowbreak 5,11), (10,8,7,11), (10,7,11,8), (11,6,5,10), (11,8,7,10), (11,8,10,7), (11,5,10,\allowbreak6).
		
		If $ (a_{1},a_{2},a_{3},a_{4})=(5,7,6,p) $ then after completing $ lk(9) $, $ [5,6,7] $ form a 3-gon in $ lk(5) $ which is not possible. Therefore for $ (5,7,6,10), (5,7,6,11) $ SEM will not exist.
		
		If $ (a_{1},a_{2},a_{3},a_{4})=(5,8,7,p) $ then $ lk(9)=C_{8}(1,4,[3,8,5],[b_{1},b_{2},2]) $ where $ b_{1}\in\lbrace 6,11\rbrace $. If $ b_{1}=6 $ then $ lk(5)=C_{8}(10,b_{3},[8,3,9],[6,0,4]) $. From here we see that $ [8,b_{3}] $ will form an edge in a 4-gon, which implies $ b_{3}=7 $, which implies $ [6,7] $ form an edge in a 4-gon, which is not possible. If $ b_{1}=11 $ then $ lk(5)=C_{8}(10,11,[9,3,8],[6,0,4]) $ which implies $ lk(11) $ is not possible. Therefore for $ (5,8,7,10),(5,8,7,11) $ SEM will not exist.
		
		If $ (a_{1},a_{2},a_{3},a_{4})=(5,10,7,6) $ then $ lk(7)=C_{8}(10,b_{1},[8,1,0],[6,2,3]) $. Since $ [8,b_{1}] $ form an edge in a 4-gon, therefore $ b_{1}=9 $ which implies face sequence is not followed in $ lk(9) $.
		
		If $ (a_{1},a_{2},a_{3},a_{4})=(5,10,11,b_{1}) $ then 5 will occur two times in $ lk(10) $. Therefore for $ (5,10,11,7),(5,10,11,8) $ SEM is not possible.
		
		If $ (a_{1},a_{2},a_{3},a_{4})=(5,11,7,6) $ then $ lk(11)=C_{8}(8,7,[3,9,5],[10,4,1]) $ which implies face sequence is not followed in $ lk(8) $.
		
		For $ (a_{1},a_{2},a_{3},a_{4})=(5,11,8,b_{1}) $, face sequence is not followed in $ lk(11) $. Therefore for $ (5,11,8,6), (5,11,8,10) $, SEM will not possible.
		
		If $ (a_{1},a_{2},a_{3},a_{4})=(5,10,6,7) $ then $ lk(10)=C_{8}(6,b_{1}),[11,1,4],[5,9,3] $. From here we see that $ [11,b_{1}] $ is an adjacent edge of two 3-gon which implies $ b_{1}\neq V $. Therefore in this case SEM will not exist.
		
		If $ (a_{1},a_{2},a_{3},a_{4})=(5,10,6,8) $ then $ lk(10)=C_{8}(b_{1},6,[3,9,5],[4,1,11]) $. We see that $ [11,b_{1}] $ form an edge in a 4-gon, therefore $ b_{1}=7 $. Now $ lk(8)=C_{8}(11,5,\allowbreak [6,3,2],[7,0,1])\Rightarrow lk(7)=C_{8}(6,10,[11,9,2],[8,1,0])\Rightarrow lk(5)=C_{8}(8,11,[9,3,10],\allowbreak [4,0,6])\Rightarrow lk(11)=C_{8}(5,8,\allowbreak[1,4,10],[7,2,9])\Rightarrow lk(2)=C_{8}(0,1,[9,11,7],[8,6,3])\Rightarrow lk(9)=C_{8}(1,4,[3,0,5],[11,7,2]) $. This map is same as the map $ \boldsymbol{KNO_{2[(3^3,4,3,4)]}} $, given in figure \ref{33434non}.
		
		If $ (a_{1},a_{2},a_{3},a_{4})=(5,10,7,11) $ then $ lk(10)=C_{8}(7,b_{1},[11,1,4],[5,9,3]) $. We see that $ [11,b_[1]] $ form an edge in a 4-gon, therefore $ b_{1}\in \lbrace 2,7\rbrace $ which is not possible.
		
		If $ (a_{1},a_{2},a_{3},a_{4})=(5,10,6,p) $ then $ lk(10)=C_{8}(8,b_{1},[11,1,4],[5,9,3]) $. We see that $ [11,b_{1}] $ form an edge in a 4-gon which implies $ b_{1}\in\lbrace 2,7,9\rbrace $. But $ b_{1}\neq 2,9 $ and for $ b_{1}=7 $ face sequence will not follow in $ lk(8) $. Therefore for $ (5,10,8,6),(5,10,8,11) $, SEM is not possible.
		
		If $ (a_{1},a_{2},a_{3},a_{4})=(6,8,7,b_{1}) $ then $ lk(8)=C_{8}(11,b_{2},[6,9,3],[7,b_{1},2]) $. As we have $ [1,11] $ is an edge in a 4-gon, therefore $ [6,b_{2}] $ will form an edge in a 4-gon, which implies $ b_{2}\in\lbrace 0,5\rbrace $. But $ b_{2}\neq 0 $ and for $ b_{2}=5 $ face sequence will not followed in $ lk(5) $. Therefore $ (a_1,a_2,a_3,a_4)\neq$ (6,8,7,10), (6,8,7,11).
		
		If $ (a_{1},a_{2},a_{3},a_{4})=(6,10,11,b_{1}) $ then $ lk(10)=C_{8}(b_{2},5,[4,1,11],[3,9,6]) $. As $ [4,5] $ is an edge in a 4-gon, therefore $ [6,b_{2}] $ form an edge in a 4-gon which implies $ b_{2}\in\lbrace 0,5\rbrace $, which is not possible. Therefore $ (6,10,11,7) $ and $ (6,10,11,8) $ are not possible.
		
		If $ (a_{1},a_{2},a_{3},a_{4})=(6,7,5,10) $ then $ lk(7)=C_{8}(5,b_{1},[8,1,0],[6,9,3]) $. As $ [3,5] $ is an edge in a 4-gon, therefore $ [8,b_{1}] $ will form an edge in a 4-gon, which implies $ b_{1}=11 $. Therefore $ lk(5)=C_{8}(7,11,[6,0,4],[10,2,3]) $, this implies $ [6,11] $ form an edge in a 4-gon,which is not possible.
		
		If$ (a_{1},a_{2},a_{3},a_{4})=(6,8,5,10) $ then $ [6,7,8] $ form a 3-gon in $ lk(8) $ which is not possible.
		
		If $ (a_{1},a_{2},a_{3},a_{4})=(6,7,11,8) $ then $ lk(7)=C_{8}(11,b_{1},[8,1,0],[6,9,3]) $. As $ [11,3] $form an edge in a 4-gon, therefore $ [8,b_{1}] $ will form an edge in a 4-gon, which implies $ b_{}=2 $, which implies 11 occur two times in $ lk(2) $ which is a contradiction.
		
		For $ (a_{1},a_{2},a_{3},a_{4})=(6,11,8,a_{4}) $, face sequence will not follow in $ lk(11) $. Therefore $ (6,11,8,5) $ and $ (6,11,8,10) $ is not possible.
		
		For $ (a_{1},a_{2},a_{3},a_{4})=(7,6,5,a_{4}) $, 7 will occur two times in  $ lk(6) $ which is not possible. Therefore $ (7,6,5,10) $ and $ (7,6,5,11) $ are not possible.
		
		If $ (a_{1},a_{2},a_{3},a_{4})=(7,6,8,10) $ then $ lk(8)=C_{8}(11,6,[3,2,10],[7,0,1])\Rightarrow lk(6)=C_{8}(8,11,\allowbreak[5,4,0],[7,9,3])\Rightarrow lk(7)=C_{8}(10,b_{1},[9,3,6],[0,1,8]) $. We have $ [8,10] $ is an edge in a 4-gon, therefore $ [9,b_{1}] $ will form an edge in a 4-gon which implies $ b_{1}=5 $. Now $ lk(5)=C_{8}(7,10,[4,0,6],[11,2,9])\Rightarrow lk(10)=C_{8}(7,5,[4,1,11],[2,3,8])\Rightarrow lk(11)=C_{8}(6,8,[1,4,\allowbreak10],[2,9,5])\Rightarrow lk(9)=C_{8}(1,4,[3,6,7],[5,11,2])\Rightarrow lk(2)=C_{8}(0,1,[9,5,11],[10,8,3]) $. This map is isomorphic to the map $ \boldsymbol{KNO_{1[(3^3,4,3,4)]}} $, given in figure \ref{33434non}, under the map $ (0,1,4)(2,9,3)\allowbreak(5,7,11)(6,8,10) $.
		
		If $ (a_{1},a_{2},a_{3},a_{4})=(7,6,8,11) $ then $ lk(8)=C_{8}(6,b_{1},[7,0,1],[11,2,3]) $. As $ [3,6] $ is an edge in a 4-gon, therefore $ [7,b_{1}] $ form an edge in a 4-gon, therefore $ b_{1}\in \lbrace 6,9\rbrace $,which is not possible. Similarly $ (7,6,10,5) $ and $ (7,6,10,8) $ are not possible.
		
		If $ (a_{1},a_{2},a_{3},a_{4})=(7,6,11,5) $ then $ lk(11)=C_{8}(8,6,[3,2,5],[10,4,1])\Rightarrow lk(6)=C_{8}(11,\allowbreak8,[5,4,0],[7,9,3]) $ which implies $ [5,8] $ form an edge in a 4-gon, which is not possible.
		
		If $ (a_{1},a_{2},a_{3},a_{4})=(7,6,11,8) $ then $ lk(6)=C_{8}(11,b_{1},[5,4,0],[7,9,3]) $, which implies $ b_{1}\notin V $.
		
		For $ (a_{1},a_{2},a_{3},a_{4})=(7,5,10,a_{4}) $ then face sequence is not followed in $ lk(10) $. Therefore $ (7,5,10,6) $ and $ (7,5,10,8) $ are not possible.
		
		If $ (a_{1},a_{2},a_{3},a_{4})=(7,10,11,5) $ then $ lk(11)=C_{8}(8,b_{1},[5,2,3],[10,4,1]) $. As $ [1,8] $ form an edge in a 4-gon, therefore $ [5,b_{1}] $ form an edge in a 4-gon which implies $ b_{1}\in \lbrace 2,11\rbrace $ which is not possible.
		
		If $ (a_{1},a_{2},a_{3},a_{4})=(7,10,11,6) $ then $ lk(11)=C_{8}(8,b_{1},[6,2,3],[10,4,1]) $. As $ [1,8] $ form an edge in a 4-gon, therefore $ [6,b_{1}] $ form an edge in a 4-gon, which implies $ b_{1}=5 $ and then $ [5,10] $ form an edge in a 4-gon, which is not possible.
		
		%		For $ (a_{1},a_{2},a_{3},a_{4})=(7,10,11,8) $, 8 will occur two times in $ lk(11) $.
		
		For $ (a_{1},a_{2},a_{3},a_{4})=(7,11,8,a_4) $, face sequence will not follow in $ lk(11) $, therefore $ (a_{1},a_{2},\allowbreak a_{3},a_{4})\neq (7,11,8,5),(7,11,8,6) $.
		
		For $ (a_{1},a_{2},a_{3},a_{4})=(8,6,5,a_{4}) $ then $ lk(6)=C_{8}(7,b_{1},[8,9,3],[5,4,0]) $. As $ [0,7] $ is an edge in a 4-gon, therefore $ [8,b_{1}] $ form an edge in a 4-gon, which implies $ b_{1}\in \lbrace 1,7\rbrace $ which is not possible. Therefore $ (8,6,5,10) $ and $ (8,6,5,11) $ are not possible.
		
		For $ (a_{1},a_{2},a_{3},a_{4})=(8,5,10,a_{4}) $, face sequence is not followed in $ lk(5) $. Therefore $ (8,5,10,6) $ and $ (8,5,10,7) $ are not possible.
		
		If $ (a_{1},a_{2},a_{3},a_{4})=(8,10,11,5) $ then $ lk(8)=C_{8}(b_{1},5,[4,1,11],[3,9,8]) $. As $ [4,5] $ is an edge in a 4-gon, therefore $ [8,b_{1}] $ will form an edge in a 4-gon, which implies $ b_{1}=7 $. Now $ lk(8)=C_{8}(11,b_{2},[9,3,10],[7,0,1]) $. As $ [1,11] $ is an edge in a 4-gon, therefore $ [9,b_{2}] $ will form an edge in a 4-gon which implies $ b_{2}=6 $. Now $ lk(9)=C_{8}(1,4,[3,10,8],[6,7,2])\Rightarrow lk(6)=C_{8}(8,11,[5,4,0],[7,2,9])\Rightarrow lk(5)=C_{8}(10,7,[2,3,11],[6,0,4])\Rightarrow lk(7)=C_{8}(10,5,\allowbreak[2,9,6],\allowbreak[0,1,8])\Rightarrow lk(11)=C_{8}(8,6,\allowbreak [5,2,3],[10,4,1])\Rightarrow lk(2)=C_{8}(0,1,[9,6,7],[5,11,3]) $. This map isomorphic to the map $ \boldsymbol{KNO_{1[(3^3,4,3,4)]}} $, given in figure \ref{33434non}, under the map $ (0,1)(3,\allowbreak9)(5,10)\allowbreak(6,11)(7,8) $.
		
		If $ (a_{1},a_{2},a_{3},a_{4})=(8,10,11,6) $ then $ lk(10)=C_{8}(5,b_{1},[8,9,3],[11,1,4]) $. As $ [4,5] $ is an edge in a 4-gon, therefore $ [8,b_{1}] $ will form an edge in a 4-gon, which implies $ b_{1}=7 $ and then $ [6,7] $ will form an edge in a 4-gon, which is not possible.
		
		If $ (a_{1},a_{2},a_{3},a_{4})=(8,11,7,6) $ then $ lk(7)=C_8(b_1,11,[3,2,6],[0,1,8]) $. Therefore $ [8,b_1] $ form an edge in a 4-gon, which implies $ b_1=9 $ and then $ C(7,8,11)\in lk(9) $ which is a contradiction.
		
		If $ (a_{1},a_{2},a_{3},a_{4})=(8,10,7,6) $ then after completing $ lk(10) $ we see that $ C(1,4,10,8)\in lk(11) $ which is a contradiction.
		
		If $ (a_{1},a_{2},a_{3},a_{4})=(10,6,5,7) $ then $ lk(6)=C_8(7,b_1,[10,9,3],[5,4,0]) $. We see that $ [10,b_1] $ form an edge in a 4-gon, therefore $ b_1=11 $. Now $ lk(10)=C_8(5,b_2,[9,3,6],[11,1,4]) $ and $ lk(9)=C_8(1,4,[3,6,10],[b_1,11,2]) $, then $ [9,b_2] $ form an edge in a 4-gon, therefore $ b_2=8 $. Now $ lk(8)=C_8(10,5,[7,0,1],[11,2,9]) $, $ lk(7)=C_8(6,11,[2,3,5],[8,1,0]) $, $ lk(11)=C_8(6,7,[2,9,8],[1,4,10]) $, $ lk(2)=C_8(0,1,[9,8,11],[7,5,3]) $, $ lk(5)=C_8(10,8,[7,2,3],[6,0,\allowbreak4]) $. This map isomorphic to the map $ \boldsymbol{KNO_{2[(3^3,4,3,4)]}} $, given in figure \ref{33434non}, under the map $ (1,4)(2,3)\allowbreak(5,8)(6,7)(10,11) $.
		
		If $ (a_{1},a_{2},a_{3},a_{4})=(10,6,5,8) $ then $ lk(6)=C_8(7,b_1,[10,9,3],[5,4,0]) $ and $ lk(5)=C_8(10,\allowbreak b_2,[8,2,3],[6,0,4]) $. Form here we see that $ [10,b_1] $ and $ [8,b_2] $ are adjacent edges of a 3-gon and a 4-gon, therefore $ b_1=11,b_2=7 $ and then face sequence will not follow in $ lk(7) $.
		
		If $ (a_{1},a_{2},a_{3},a_{4})=(10,5,7,6) $ then $ lk(5)=C_8(7,b_1,[6,0,4],[10,9,3]) $ which implies $ [6,b_1] $ is an adjacent edge of a 3-gon and a 4-gon, therefore $ b_1=2 $ and then $ C(5,6,7)\in lk(2) $ which is a contradiction.
		
		If $ (a_{1},a_{2},a_{3},a_{4})=(10,7,6,8), (10,7,6,11) $ then face sequence will not follow in $ lk(7) $.
		
		If $ (a_{1},a_{2},a_{3},a_{4})=(10,8,7,6) $ then $ C(2,3,8,1,0,6)\in lk(7) $ which is not possible.
		
		If $ (a_{1},a_{2},a_{3},a_{4})=(10,6,11,8) $ then $ lk(6)=C_8(7,11,[3,9,10],[5,4,0]) $, $ lk(11)=C_8(6,7,\allowbreak [10,4,1],[8,2,3]) $ which implies $ [7,10] $ is an adjacent edge if a 3-gon and a 4-gon which is not possible.
		
		If $ (a_{1},a_{2},a_{3},a_{4})=(11,5,7,6) $ then $ lk(5)=C_8(10,7,[3,9,11],[6,0,4]) \Rightarrow lk(7)=C_8(5,\allowbreak 10,[8,1,0],[6,2,3])\Rightarrow lk(11)=C_8(8,6,[5,3,9],[10,4,1])\Rightarrow lk(9)=C_8(1,4,[3,5,11],[10,8,\allowbreak 2])\Rightarrow lk(10)=C_8(7,5,[4,1,11],[9,2,8])\Rightarrow lk(6)=C_8(11,8,[2,3,7],[0,4,5])\Rightarrow lk(8)=C_8(6,11,[1,0,7],[10,9,2])\Rightarrow lk(2)=C_8(0,1,[9,10,8],[6,7,3])$. This map is isomorphic to the map $ \boldsymbol{KNO_{2[(3^3,4,3,4)]}} $, given in figure \ref{33434non}, under the map $ (0,4,1)(2,3,9)(5,11,7)(6,10,8) $.
		
		If $ (a_{1},a_{2},a_{3},a_{4})=(11,6,5,7) $ then $ lk(5)=C_8(10,b_1,[7,2,3],[6,0,4]) $ and $ lk(6)=C_8(7,\allowbreak b_2,[11,9,3],[5,4,0]) $. From here we see that $ [7,b_1], [11,b_2] $ are adjacent edges of a 3-gon and a 4-gon, therefore $ b_1=8,b_2=10 $ which implies face sequence will not follow in $ lk(10) $. Similarly $ (a_{1},a_{2},a_{3},a_{4})=(11,6,5,8) $ is also not possible.
		
		If $ (a_{1},a_{2},a_{3},a_{4})=(11,8,5,7) $ then after completing $ lk(5) $ we see that $ C(0,4,5,7)\in lk(6) $ which is a contradiction.
		
		If $ (a_{1},a_{2},a_{3},a_{4})=(11,8,5,10) $ then $ lk(5)=C_8(8,b_1,[6,0,4],[10,2,3]) $ which implies $ [6,b_1] $ is an adjacent edge of a 3-gon and therefore $ b_1=7 $. This implies $ C(0,1,8,5,6)\in lk(7) $ which is a contradiction.
		
		If $ (a_{1},a_{2},a_{3},a_{4})=(11,8,6,7) $ then $ lk(8)=C_8(6,b_1,[7,0,1],[11,9,3]) $ which implies $ [7,b_1] $ is an adjacent edge of a 3-gon and a 4-gon, therefore $ b_1=2 $ and then face sequence will not follow in $ lk(2) $.
		
		If $ (a_{1},a_{2},a_{3},a_{4})=(11,8,6,10) $ then $ lk(8)=C_8(b_1,6,[3,9,11],[1,0,7]) $ which implies $ [7,b_1] $ is an adjacent edge of a 3-gon and a 4-gon, therefore $ b_1\in\{2,5\} $. But for $ b_1=2 $, face sequence will not follow in $ lk(2) $ and for $ b_1=5 $ then face sequence will not follow in $ lk(6) $. Similarly $ (a_{1},a_{2},a_{3},a_{4})=(11,8,10,5) $ is not possible.
		
		If $ (a_{1},a_{2},a_{3},a_{4})=(11,8,7,a_4) $ then face sequence will not follow in $ lk(8) $, Therefore $ (a_{1},a_{2},a_{3},a_{4})=(11,8,7,5), (11,8,7,6) $ are not possible.
		
		If $ (a_{1},a_{2},a_{3},a_{4})=(11,8,10,6) $ then $ lk(8)=C_8(b_1,10,[3,9,11],[1,0,7]) $ which implies $ [7,b_1] $ is an adjacent edge of a 3-gon and a 4-gon, therefore $ b_1=5 $. Now $ lk(10)=C_8(8,5,[4,1,\allowbreak 11],[6,2,3]) $ and $ lk(11)=C_8(6,b_2,[9,3,8],[1,4,10]) $. From here we see that $ [9,b_2] $ is an adjacent edge of a 3-gon and a 4-gon, therefore $ b_2=7 $. Now $ lk(6)=C_8(11,7,[0,4,5],[2,3,10])\Rightarrow lk(9)=C_8(1,4,[3,8,11],[7,5,2])\Rightarrow lk(7)=C_8(6,11,[9,2,5],\allowbreak[8,1,0])\Rightarrow lk(5)=C_8(8,10,\allowbreak [4,0,6],[2,9,7])\Rightarrow lk(2)=C_8(0,1,[9,7,5],[6,10,3]) $. This map isomorphic to the map $ \boldsymbol{KNO_{1[(3^3,4,3,4)]}} $, given in figure \ref{33434non}, under the map $ (0,2,1,9,4,\allowbreak3)\allowbreak(5,11)(6,8)(7,10) $.
		
		If $ (a_{1},a_{2},a_{3},a_{4})=(11,5,10,a_4) $ then face sequence will not follow in $ lk(5) $. Therefore $ (a_{1},a_{2},a_{3},a_{4})=(11,5,10,8), (11,5,10,7)  $ are not possible.
		
		If $ (a_{1},a_{2},a_{3},a_{4})=(11,7,5,10) $ then after completing $ lk(7)$ we see that $ C(3,2,10,4,0,6,\allowbreak 7) \in lk(5) $ which is a contradiction.
	\end{case}
	\begin{case}
		When $ (a,b,c,d)=(4,3,6,9) $ then $ lk(3)=C_8(0,4,[1,9,6],[a_1,b_1,2]) $. Now if $ a_1,b_1\notin\{10,11\} $ then vertices in $ lk(10) $ i.e. $ V(lk(10))=\{2,4,5,6,7,8,9,u\}=V(lk(11)) $, where $ u=11 $ for $ V(lk(10)) $ and $ u=10 $ for $ V(lk(11)) $, then $ V(lk(6))=\{0,1,3,4,5,7,9,10,\allowbreak 11,a_1\} $. If $ a_1\in V\setminus\{2,8\} $ then number of vertex in $ lk(6) $ i.e. $ |v_{lk(6)}| =9$ and otherwise $ |v_{lk(6)}|=10 $, both cases are not possible. Therefore at least one of $ a_1,b_1 $ is 10. Now if $ a_1\neq 10 $ then $ b_1=10 $ and as $ 4\in V(lk(10)),V(lk(11)) $ i.e. $ 10,11\in V(lk(4)) $, therefore either $ [2,10] $ is an adjacent edge of two 4-gon or 10 occur two times in $ lk(4) $ which are not possible. Therefore $ a_1=10 $ and then $ 10\notin V(lk(4)) $. Now if $ b_1\neq 11 $ then $ V(lk(10))=\{2,3,5,6,7,8,9,11\} $ and $ V(lk(11))=\{2,4,5,6,7,8,9,10\} $, therefore $ V(lk(6))=\{0,1,3,4,5,7,9,10,11\} $ i.e. $ |v_{lk(6)}|=9 $ which make contradiction. Therefore $ b_1=11 $ and then $ 10,11\notin V(lk(4)) $ i.e. $ V(lk(10))=\{2,3,5,6,7,8,9,u\}=V(lk(11)) $ where $ u=11 $ for $ V(lk(10)) $ and $ u=10 $ for $ V(lk(11)) $ and then $ V(lk(6))=\{0,1,3,4,5,7,9,10,11\} $ i.e. $ |v_{lk(6)}|=9 $ which make contradiction. Therefore $ (a,b,c,d)=(4,3,6,9) $ is not possible.
	\end{case}	
	If $ (a,b,c,d)=(4,3,9,6) $ then after completing $ lk(6) $, we see that $ C(6,0,1,8)\in lk(7) $, which is a contradiction.
	\begin{case}
		When $ (a,b,c,d)=(4,3,5,9) $ then $ lk(3)=C_8(0,4,[1,9,5],[a_1,b_1,2]) $, $ lk(4)=C_8(1,3,[0,6,5],[a_2,b_2,2]) $ and $ lk(5)=C_8(a_1,a_2,[4,0,6],[9,1,3]) $. Now if $ a_1,b_1\notin\{10,11\} $ then $ 10,11\in V(lk(5)) $ i.e. $ |v_{lk(5)}|=9 $ which make contradiction. Therefore at least one of $ a_1,b_1 $ is 10, say $ a_1=10 $, then $ 10\notin V(lk(4)) $ as $ a_2\neq 10 $ or form an adjacent edge of two 4-gon. Now if $ b_1= 11 $ then $ 11\notin V(lk(4)) $ as then either $ [2,11] $ is an adjacent edge of two 4-gon or 11 occur two times in $ lk(2) $, therefore $ 11\in V(lk(5)) $ i.e. $ V(lk(5))=\{0,1,3,4,6,9,10,11,a_2\} $ i.e. $ |v_{lk(5)}|=9 $ which make contradiction. If $ b_1\neq 11 $ then $ 11\in V(lk(4)) $ i.e. $ a_2=11 $ and $ [4,11],[10,11] $ are adjacent edges of two 4-gon, which implies face sequence will not follow in $ lk(11) $. Therefore $ (a,b,c,d)=(4,3,5,9) $ is not possible.
	\end{case}
	\begin{case}
		When $ (a,b,c,d)=(4,3,9,10) $ then $ lk(2)=C_8(0,1,[4,a_1,a_2],[a_3,a_4,3]) $, $ lk(3)=C_8(0,4,[1,10,9],[a_4,a_3,2]) $ and $ lk(4)=C_8(3,1,[2,a_2,a_1],[5,6,0]) $. We see that exactly one of $ a_1,a_2,a_3,a_4 $ is 11 and $ a_1\in\{7,8,9,10,11\} $. \underline{\textbf{If} $ \boldsymbol{a_1=7} $} then $ lk(4)=C_8(1,3,[0,6,5],[7,a_2,\allowbreak2]) $ and $ lk(7)=C_8(6,5,[4,2,a_2],8,1,0) $ which implies $ C(0,6,7,4)\in lk(5) $ which is a contradiction, therefore $ a_1\neq 7 $. \underline{\textbf{If} $ \boldsymbol{a_1=8} $} then $ lk(8)=C_8(10,5,[4,2,a_2],[7,0,1]) $. If $ a_2\neq 11 $ then $ 0,1,4,8\notin V(lk(11)) $ i.e. $ |v_{lk(11)}|\leq 7 $ which is not possible, therefore $ a_2=11 $ and then $ [9,10] $ will form an adjacent edge of two 4-gon, which is a contradiction. Therefore $ a_1\neq 9. $ Now all possible values of $ (a_1,a_2,a_3,a_4) $ are: (9,7,5,11), (9,7,6,11), (9,7,8,11), (9,7,11,5), (9,7,11,6), (9,7,11,8), (9,8,5,11), (9,8,6,11), (9,8,7,11), (9,8,11,5), (9,8,11,6), (9,11,5,7), (9,11,5,8), (9,11,\allowbreak6,7), (9,8,6,8), (9,11,7,5), (9,11,7,6), (9,11,8,5), (9,11,8,6), (10,7,5,11), (10,7,6,11), (10,7,8,11), (10,7,11,5), (10,7,11,6), (10,7,11,8), (10,8,5,11), (10,8,6,11), (10,8,\allowbreak7,11), (10,8,11,5), (10,8,11,\allowbreak6), (10,8,11,7), (10,11,5,7), (10,11,5,8), (10,11,6,7), (10,11,6,8), (10,11,7,5), (10,11,7,6), (10,\allowbreak11,8,5), (10,11,8,6), (11,7,5,8), (11,7,6,8), (11,7,8,5), (11,7,8,6), (11,8,5,7), (11,8,6,7), (11,8,\allowbreak7,5), (11,8,7,6), (11,9,5,7), (11,9,5,8), (11,9,6,7), (11,9,6,8), (11,9,7,5), (11,9,7,6), (11,9,8,5), (11,9,8,6), (11,\allowbreak10,5,7), (11,10,5,8), (11,10,6,7), (11,10,6,8), (11,10,7,5), (11,10,7,6), (11,10,8,5). \par Let $ P=\{5,6,7,8,9,10,11\} $, then for any $ (a_1,a_2,a_3,a_4) $, vertices of a 4-gon are $ P\setminus\{a_1,a_2,\allowbreak a_3,a_4\}\cup\{11\} $, therefore the followings can not be the value of $ (a_1,a_2,a_3,a_4) $: (9,7,8,11), (9,7,11,8), (9,8,7,11), (10,7,8,11), (10,8,7,11), (10,8,11,7), (10,7,11,8), (11,7,5,8), (11,7,6,8), (11,7,8,5), (11,7,8,6), (11,8,5,7), (11,8,6,7), (11,8,7,5), (11,8,7,6). \\
		If $ (a_1,a_2,a_3,a_4)=(9,7,5,11) $ then after completing $ lk(7) $, we see that $ 0,1,4,7\notin V(lk(11)) $ i,e, $ |v_{lk(11)}|\leq 7 $ which is not possible. If $ (a_1,a_2,a_3,a_4)=(9,7,6,11) $ then face sequence will not follow in $ lk(7) $. If $ a_1=9,a_4=5 $ then face sequence will not follow in $ lk(5) $, therefore $ (a_1,a_2,a_3,a_4)\neq (9,7,11,5),(9,11,8,5),(9,117,5) $. If $ (a_1,a_2,a_3,a_4)=(9,7,11,6) $ then $ 11\notin V(lk(9)) $, therefore $ |v_{lk(11)}|\leq 7 $ which is not possible. For $ (a_1,a_2,a_3,a_4)=(9,10,a_3,a_4) $, face sequence will not follow in $ lk(10) $, therefore $ (a_1,a_2,a_3,a_4)\neq$ (9,8,5,11), (9,8,6,11), (9,8,11,5), (9,8,11,6) . For $ (a_1,a_2,a_3,a_4)=(9,11,5,7),(9,11,5,8) $, after completing $ lk(9) $, we see that face sequence will not follow in $ lk(5) $. If $ (a_1,a_2,a_3,a_4)=(9,11,6,a_4) $ then $ lk(9)=C_8(5,a_4,[3,1,10],[11,2,4]) $, so if $ a_4 =7$ then $ 2,3,4,9\notin V(lk(8)) $ i.e. $ |v_{lk(8)}|\leq 7 $, similarly, if $ a_4=8 $ then $ |v_{lk(7)}|\leq 7 $, which is a contradiction, therefore $ (a_1,a_2,a_3,a_4)\neq (9,11,6,7),(9,11,6,8) $. For $ a_1=9,a_4=6 $, after completing $ lk(9) $, face sequence will not follow in $ lk(6) $, therefore $ (a_1,a_2,a_3,a_4)\neq (9,11,7,6),(9,11,8,6) $. If $ (a_1,a_2,a_3,a_4)=(10,7,5,11) $ then after completing $ lk(10) $, face sequence will not follow in $ lk(5) $. If $ a_1=10 $ and al least one of $ a_2,a_3,_4 $ is not 8, then after completing $ lk(10) $, face sequence will not follow in $ lk(5) $, therefore $ (a_1,a_2,a_3,a_4)\neq $ (10,7,6,11), (10,7,11,5), (10,7,11,6), (10,11,5,7), (10,11,6,7), (10,11,7,5), (10,11,7,6). If $ (a_1,a_2,a_3,a_4)=(10,8,a_3,a_4) $ then $ lk(10)=C_8(b_1,5,[4,2,8],[1,3,9]) $ which implies $ b_1\neq 11 $. Therefore $ 0,1,4,10\notin V(lk(11)) $ i.e. $ |v_{lk(11)}|\leq 7 $ which is a contradiction. Therefore $ (a_1,a_2,a_3,a_4)\neq$ (10,8,5,11), (10,8,6,11), (10,8,11,6), (10,8,11,5). If $ (a_1,a_2,a_3,a_4)=(10,11,a_3,a_4) $ then after completing $ lk(10) $, we see that face sequence will not follow in $ lk(8) $, therefore $ (a_1,a_2,a_3,a_4)\neq $ (10,11,5,8), (10,11,8,5), (10,11,6,8), (10,11,8,6). If $ (a_1,a_2,a_3,a_4)=(11,9,5,7) $ then $ lk(9)=C_8(5,7,[3,1,10],[11,4,2]) $ which implies face sequence in $ lk(7) $ is not followed. Similarly $ (a_1,a_2,a_3,a_4)\neq$ (11,9,5,8), (11,9,6,7), (11,9,6,8). If $ (a_1,a_2,a_3,a_4)=(11,9,a_3,a_4) $ then $ lk(9)=C_8(a_3,a_4,[3,1,10],[11,4,2]) $ which implies $ C(2,9,a_4,3)\in lk(a_3) $, which make contradiction. Therefore $ (a_1,a_2,a_3,a_4)\neq $ (11,9,7,5), (11,9,7,6), (11,9,8,5), (11,9,8,6). If $ (a_1,a_2,a_3,\allowbreak a_4)=(11,10,5,7) $ then $ lk(7)=C_8(9,6,[0,1,\allowbreak8],[5,2,3]) $ which implies $ 0,1,3,7\notin V(lk(11)) $ i.e. $ |v_{lk(11)}|\leq 7 $ which is a contradiction. Similarly $ (a_1,a_2,a_3,a_4)\neq (11,10,7,5),\allowbreak(11,10,7,6),(11,10,5,8),(11,10,6,7),(11,10,6,8) $. If $ (a_1,a_2,a_3,a_4)=(11,10,6,7) $ then $ lk(7)=C_8(b_1,9,[3,2,6],[0,1,8]) $ which implies $ b_1\neq 11 $, therefore $ |v_{lk(11)}|\leq 7 $, which is not possible. If $ (a_1,a_2,a_3,a_4)=(11,10,8,5) $ then face sequence will not follow in $ lk(10) $.
	\end{case}
	If $ (a,b,c,d)=(5,4,c,d) $ then face sequence will not follow in $ lk(5) $. Similarly if $ (a,b,c,d)=(6,5,c,d) $ then face sequence will not follow in $ lk(6) $. Therefore $ (a,b,c,d)\neq$ (5,4,3,9), (5,4,10,9), (6,5,9,3), (6,5,9,10). If $ (a,b,c,d)=(5,9,4,3) $ then in $ lk(3) $, face sequence will not follow.
	
	When $lk(1)=C_{8}(c,d,[8,7,0],[2,a,b])$, then after isomorphism,  all possible values of $(a,b,c,d)$ are as follows:\\
	(4,5,6,9)$\approx$ $\lbrace$ (4,5,6,10), (4,5,6,11) $\rbrace$, (4,5,9,6) $\approx$ $\lbrace$ (4,5,10,6), (4,5,11,6) $\rbrace$, (4,5,9,10) $\approx$ $\lbrace$ (4,5,11,9), (4,5,10,9), (4,5,11,10), (4,5,10,11), (4,5,9,11) $\rbrace$, (4,6,5,9) $\approx$ $\lbrace$ (4,6,5,10), (4,6,5,11) $\rbrace$, (4,6,9,5) $\approx$ $\lbrace$ (4,6,10,5), (4,6,11,5) $\rbrace$, (4,6,9,10) $\approx$ $\lbrace$ (4,6,9,11), (4,6,10,9), (4,6,10,11), (4,6,11,9), (4,6,11,10) $\rbrace$, (4,9,5,6) $\approx$ $\lbrace$ (4,10,5,6), (4,11,5,6) $\rbrace$, (4,9,5,10) $\approx$ $\lbrace$ (4,9,5,11), (4,10,5,9), (4,10,5,11), (4,11,5,9), (4,11,5,10) $\rbrace$, (4,9,6,5) $\approx$ $\lbrace$ (4,10,6,5), (4,11,6,5) $\rbrace$, (4,9,6,\allowbreak10) $\approx$ $\lbrace$ (4,9,6,11), (4,10,6,9), (4,10,6,11), (4,11,6,9), (4,11,6,10) $\rbrace$, (4,9,10,5) $\approx$ $\lbrace$ (4,9,11,5), (4,10,9,5), (4,10,11,5), (4,11,9,5), (4,11,10,5) $\rbrace$, (4,9,10,6) $\approx$ $\lbrace$ (4,9,11,6), (4,10,9,6), (4,10,11,\allowbreak6), (4,11,9,6), (4,11,10,6) $\rbrace$, (4,9,10,11) $\approx$ $\lbrace$ (4,9,11, 10), (4,10,9,11), (4,10,11,9), (4,11,9,10), (4,11,10,9) $\rbrace$, (5,4,6,9) $\approx$ $\lbrace$ (5,4,6,10), (5,4,6, 11) $\rbrace$, (5,4,9,6) $\approx$ $\lbrace$ (5,4,10,6), (5,4,11,6) $\rbrace$, (5,4,9,10) $\approx$ $\lbrace$ (5,4,9,11), (5,4,10,9), (5,4,10,11), (5,4,11,9), (5,4,11,10) $\rbrace$, (5,6,4,9) $\approx$ $\lbrace$ (5,6,4,11), (5,6,4,10) $\rbrace$, (5,6,9,4) $\approx$ $\lbrace$ (5,6,10,4), (5,6,11,4) $\rbrace$, (5,6,9, 10) $\approx$ $\lbrace$ (5,6,9,11), (5,6,10,9), (5,6,10,11), (5,6,11,9), (5,6,11,10) $\rbrace$, (5,9,4,6) $\approx$ $\lbrace$ (5,10,4,6), (5,11,4,6) $\rbrace$, (5,9,4,\allowbreak10) $\approx$ $\lbrace$ (5,9,4,11), (5,10,4,9), (5,10,4,11), (5,11,4,9), (5,11,4,10) $\rbrace$, (5,9,6,4) $\approx$ $\lbrace$ (5,10,6,4), (5,11,6,4) $\rbrace$, (5,9,6,10) $\approx$ $\lbrace$ (5,9,6,11), (5,10,6,9), (5,10,6,11), (5,11,6,9), (5,11,6,10) $\rbrace$, (5,9,10,\allowbreak4) $\approx$ $\lbrace$ (5,9,11,4), (5,10,9,4), (5,10,11,4), (5,11,9,4), (5,11,10,4) $\rbrace$, (5,9,10,6) $\approx$ $\lbrace$ (5,9,11,6), (5,10,9,6), (5,10,11,6), (5,11,9,6), (5,11,10,6) $\rbrace$, (5,9,10,11) $\approx$ $\lbrace$ (5,9,11,10), (5,10,9,11), (5,10,\allowbreak11,9), (5,11,9,10), (5,11,10,9) $\rbrace$, (6,4,5,9) $\approx$ $\lbrace$ (6,4,5,10), (6,4,5,11) $\rbrace$, (6,4,9,5) $\approx$ $\lbrace$ (6,4,10,5), (6,4,11,5) $\rbrace$, (6,4,9,10)$ \approx\lbrace $ (6,4,9,11), (6,4,10,9), (6,4,10,11), (6,4,11,\allowbreak9), (6,4,11,10) $\rbrace$, (6,5,4,9) $\approx$ $\lbrace$ (6,5,4,10), (6,5,4,11) $\rbrace$, (6,5,9,4)$ \approx\lbrace $ (6,5,10,4), (6,5,11,4) $\rbrace$, (6,5,9,10) $\approx$ $\lbrace$ (6,5,9,11), (6,5,10,9), (6,5,10,11), (6,5,11,9), (6,5,11,10) $\rbrace$, (6,9,4,5) $\approx$ $\lbrace$ (6,10,4,5), (6,11,4,5)$\rbrace$, (6,9,4,10) $\approx$ $\lbrace$ (6,9,4,11), (6,10,4,9), (6,10,4,11), (6,11,4,9), (6,11,4,\allowbreak10) $\rbrace$, (6,9,5,4) $\approx$ $\lbrace$ (6,10,5,4), (6,11,5,4) $\rbrace$, (6,9,5,10) $\approx$ $\lbrace$ (6,9,5,11), (6,10,5,9), (6,10,5,11), (6,11,5,9), (6,11,5,10) $\rbrace$, (6,9,\allowbreak10,4) $\approx$ $\lbrace$ (6,9,11,4), (6,10,9,4), (6,10,11,4), (6,11,9,4), (6,11,10,\allowbreak4) $\rbrace$, (6,9,10,5) $\approx$ $\lbrace$ (6,9,11,5), (6,10,9,5), (6,10,11,5), (6,11,9,5), (6,11,10,5) $\rbrace$, (6,9,10,11) $\approx$ $\lbrace$ (6,9,11,10), (6,10,9,11), (6,10,\allowbreak11,9), (6,11,9,10), (6,11, 10,9) $\rbrace$, (9,4,5,6) $\approx$ $\lbrace$ (10,4,5,6), (11,4,5,6) $\rbrace$, (9,4,5,10) $\approx$ $\lbrace$ (9,4,5,11), (10,4,5,9), (10,4,5,11), (11,4,5,9), (11,4,5,10) $\rbrace$, (9,4,6,5) $\approx$ $\lbrace$ (10,4,6,5), (11,4,6,5) $\rbrace$, (9,4,6,10) $\approx$ $\lbrace$ (9,4,6,11), (10,4,6,9), (10,4,6,11), (11,4,6,9), (11,4,6,10) $\rbrace$, (9,4,10,5) $\approx$ $\lbrace$ (9,4,11,5), (10,4,9,5), (10,4,11,5), (11,4,9,5), (11,4,\allowbreak10,5) $\rbrace$, (9,4,10,6) $ \approx\lbrace $ (9,4,11,6), (10,4,9,6), (10,4,11,\allowbreak6), (11,4,9,6), (11,4,10,6)$ \rbrace $. (9,4,10,11) $\approx$ $\lbrace$ (9,4,11,10), (10,4,9,11), (10,4,\allowbreak11,9), (11,4,9,10), (11,4,10,9) $\rbrace$, (9,5,4,6) $\approx$ $\lbrace$ (10,5,4,6), (11,5,4,6) $\rbrace$, (9,5,4,10) $\approx$ $\lbrace$ (9,5,4,11), (10,5,4,9), (10,5,4,11), (11,5,4,9), (11,5,4,10) $\rbrace$, (9,5,6,4) $\approx$ $\lbrace$ (10,5,6,4), (11,5,6,4) $\rbrace$, (9,5,6,10) $\approx\allowbreak\lbrace$ (9,5,6,11), (10,5,6,9), (10,5,6,11), (11,5,6,9), (11,5,6,10) $\rbrace$, (9,5,10,4) $\approx$ $\lbrace$ (9,5,11,4), (10,5,9,4), (10,5,11,4), (11,5,9,4), (11,5,\allowbreak10,4) $\rbrace$, (9,5,10,6) $\approx$ $\lbrace$ (9,5,11,6), (10,5,9,6), (10,5,11,\allowbreak6), (11,5,9,\allowbreak6), (11,5,10,6) $\rbrace$, (9,5,10,11) $\approx$ $\lbrace$ (9,5,11, 10), (10,5,9,11), (10,5,11,9), (11,5,9,10), (11,5,10,9) $\rbrace$, (9,6,4,5) $\approx$ $\lbrace$ (10,6,4,5), (11,6, 4,5) $\rbrace$, (9,6,4,10) $\approx$ $\lbrace$ (9,6,4,11), (10,6,4,9), (10,6,4,11), (11,6,4,9), (11,6,4,10) $\rbrace$, (9,6,5,4) $\approx$ $\lbrace$ (10,6,5,4), (11,6,5,4) $\rbrace$, (9,6,5,10) $\approx$ $\lbrace$ (9,6,5,11), (10,6,5,9), (10,6,5,11), (11,6,5,9), (11,6,5,10) $\rbrace$, (9,6,10,4) $\approx$ $\lbrace$ (9,6,11,4), (10,6,9,4), (10,6,11,\allowbreak4), (11,6,9,4), (11,6,\allowbreak10,4) $\rbrace$, (9,6,10,5) $\approx$ $\lbrace$ (9,6,11,5), (10,6,9,5), (10,6,11,5), (11,6,9,\allowbreak5), (11,6,10,5) $\rbrace$, (9,6,10,11) $\approx$ $\lbrace$ (9,6,11,10), (10,6,9,11), (10,6,11,9), (11,6,9,10), (11,6,10,9) $\rbrace$, (9,10,4,5) $\approx$ $\lbrace$ (9,11,4,5), (10,9,4,5), (10,11,4,5), (11,9,4,5), (11,10,4,5) $\rbrace$, (9,10,4,6) $\approx$ $\lbrace$ (9,11,4,6), (10,9,4,6), (10,11,\allowbreak4,6), (11,9,4,6), (11,10,4,6) $\rbrace$, (9,10,5,4)$ \approx\lbrace $ (9,11,5,4), (10,9,5,4), (10,11,5,4), (11,9,5,4), (11,10,5,4)$ \rbrace $. (9,10,5,6)$ \approx\lbrace $ (9,11,5,6), (10,9,5,6), (10,11,5,6), (11,9,5,6), (11,10,5,6)$ \rbrace $. (9,10,\allowbreak4,11) $ \approx\lbrace $(9,11,4,10), (10,9,4,11), (10,11,4,9), (11,9,4,10), (11,10,4,9)$ \rbrace $. (9,10,5,11) $\approx$ $\lbrace$ (9,11,5, 10), (10,9,5,11), (10,11,5,9), (11,9,5,10), (11,10,5,9) $\rbrace$, (9,10,6,4) $\approx$ $\lbrace$ (9,11,6,4), (10,9,6,4), (10,11,\allowbreak6,4), (11,9,\allowbreak6,4), (11,10,6,4) $\rbrace$, (9,10,6,5) $\approx$ $\lbrace$ (9,11,6,5), (10,9,6,5), (10,11,\allowbreak6,5), (11,9,6,5), (11,10,6,5) $\rbrace$, (9,10,6,11) $\approx$ $\lbrace$ (9,11,6,10), (10,9,6,11), (10,\allowbreak11,6,9), (11,9,6,10), (11,10,6,9) $\rbrace$, (9,10,11,4) $\approx$ $\lbrace$ (9,11,10,4), (10,9,11,4), (10,11,9,4), (11,9,\allowbreak10,4), (11,10,9,4) $\rbrace$, (9,10,11,5) $\approx$ $\lbrace$ (9,11,10,5), (10,9,11,5), (10,11,9,5), (11,9,10, 5), (11,10,9,\allowbreak5) $\rbrace$, (9,10,11,6) $\approx$ $\lbrace$ (9,11,10,6), (10,9,11,6), (10,11,9,6), (11,9,10,6), (11,10,9,6) $\rbrace$.
	
	We see that if $a\in \lbrace 4,5,6\rbrace$, then $b\notin \lbrace 4,5,6\rbrace$ as then $(a,b)$ will appear in two 4-gon. Therefore following the following values of $ (a,b,c,d) $ are not possible: (4,5,6,9), (4,5,9,6), (4,5,9,10), (4,6,5,9), (4,6,9,5), (4,6,9,10), (5,4,6,9), (5,4,9,6), (5,4,9,10), (5,6,4,9), (5,6,9,4), (5,6,9,10), (6,4,5,9), (6,4,9,5), (6,4,9,10), (6,5,4,9), (6,5,9,4), (6,5,9,10).
	
	Since $[0,4,5,6]$ is a 4-gon, therefore $ (a,b,c,d)\neq $ (5,9,4,6), (5,9,6,4), (9,4,6,10), (9,5,4,6), (9,5,6,4), (9,6,4,5), (9,10,5,4), (9,6,4,10), (9,10,4,6), (9,10,6,4). (As in these cases [4,6] form an edge.)
	
	For $ (a,b,c,d)= $(4,9,5,6), (4,9,5,10), (4,9,6,10), (4,9,10,5), (4,9,10,6), (4,9,10,11), (9,4,5,\allowbreak6), (9,4,5,10), (9,4,10,5), (9,4,10,6), (9,4,10,11), face sequence will not follow in $ lk(4) $. For $ (a,b,c,d)=(6,9,10,11) $, we can not find eight vertices for $ lk(6) $.
	
	For (4,9,6,5), (6,9,5,4), (9,4,6,5), (9,6,5,4), (9,10,6,5), (9,6,5,10), (9,5,4,10), SEM will not oriantable.
	\begin{case}
		When $(a,b,c,d)=(5,9,4,10)$, then either $lk(4)=C_{8}(10,1,[9,b_{1},3],[0,6,\allowbreak 5])$ or $lk(4)=C_{8}(9,1,[10,c_{1},3],[0,6,5])$. For $lk(4)=C_{8}(10,1,[9,b_{1},3],[0,6,5])$, $lk(10)=C_{8}(1,4,[5,b_{2},b_{3}],[b_{4},b_{5},$ $8])$, which implies [5,10] form an edge in a 4-gon, which is not possible. For $lk(4)=C_{8}(9,1,[10,c_{1},3],$ $[0,6,5])$, 5 will appear two times in $lk(9)$.
	\end{case}
	\begin{case}
		When $(a,b,c,d)=(5,9,6,10)$, then either $lk(6)=C_{8}(10,1,[9,b_{1},7],[4,0,\allowbreak 5])$ or $lk(6)=C_{8}(9,1,[10,c_{1},7],[4,0,5])$. For $lk(6)=C_{8}(10,1,[9,b_{1},7],[4,0,5])$, $lk(10)=C_{8}(1,6,[5,b_{2},b_{3}],[b_{4},b_{5},$ $8])$, which implies [5,10] form an edge in a 4-gon, which is not possible. For $lk(6)=C_{8}(9,1,[10,c_{1},7],$ $[4,0,5])$, 5 will occur two times in $lk(9)$.
	\end{case}
	\begin{case}
		When $(a,b,c,d)=(5,9,10,4)$, then either $lk(4)=C_{8}(10,1,[8,a_{1}],[0,6,5])$ or $lk(4)=C_{8}(8,1,[10,c_{1},3],[0,6,5])$. For $lk(4)=C_{8}(10,1,[8,a_{1}],[0,6,5])$, we see that $[1,10]$ is not an edge in a 4-gon, therefore [5,10] will form an edge in a 4-gon, which is not possible.
		
		For $lk(4)=C_{8}(8,1,[10,c_{1},3],[0,6,5])$, remaining three  4-gon are $[3,4,10,c_{2}]$, $[2,11,c_{3},\allowbreak c_{4}]$, $[3,11,c_{5},c_{6}]$. Now $lk(8)=C_{8}(4,5,[a_{1},a_{2},a_{3}],[7,0,1])$ and $lk(5)=C_{8}(8,a_{1},[b_{1},b_{2},b_{3}],\allowbreak[6,0,4])$, where $a_{1}\in \lbrace 10,11\rbrace$. If $a_{1}=10$, then $lk(10)$ is not possible. If $a_{1}=11$, then either $lk(5)=C_{8}(8,11,[9,1,2],[6,0,4])$ or $lk(5)=C_{8}(8,11,[2,1,9],[6,0,4])$. For $lk(5)=C_{8}(8,11,[9,1,2],\allowbreak[6,0,4])$, then $lk(9)=C_{8}(10,\allowbreak d_{1},[d_{2},d_{3},11],[5,2,1])$. From $lk(8)$ we see that [8,11] is an edge in a 4-gon and from $lk(9)$ we see that [9,11] is an edge in a 4-gon, therefore from $lk(5)$ we see that face sequence is not followed in $lk(11)$. For $lk(5)=C_{8}(8,11,[2,1,9],[6,0,4])$, from $lk(8)$ we see that $[8,11]$ is an edge in a 4-gon and we have [2,11] is an edge in a 4-gon, therefore from $lk(5)$ we see that sace sequence is not followed in $lk(11)$.
	\end{case}
	\begin{case}
		When $(a,b,c,d)=(5,9,10,6)$, then $lk(6)=C_{8}(8,1,[10,a_{1},7],[0,4,5])$, where $a_{1}\in \lbrace 2,3,9,11\rbrace$. Remaining three 4-gons are $[3,4,c_{1},c_{2}]$, $[2,a_{1},c_{3},c_{4}]$, $[3,a_{1},c_{5},\allowbreak c_{6}]$. Therefore $a_{1}\notin \lbrace 9,11\rbrace$ as then above three 4-gons are not hold. Now $lk(5)=C_{8}(8,b_{1},[b_{2},b_{3},b_{4}],\allowbreak[4,0,6])$ and $lk(8)=C_{8}(6,5,[b_{1},b_{5},b_{6}],[7,0,1])$, where $b_{1}\in \lbrace 10,\allowbreak 11\rbrace$. If $b_{1}=10$ then face sequence is not followed in $lk(10)$. If $b_{1}=11$, then from $lk(8)$ we see that [8,11] form an edge in a 4-gon. From the above three 4-gon, we see that one of $c_{1},c_{2}$ is 11 as otherwise $10,11$ will occur two times in two 4-gon, and one of $c_{1},c_{2}$ is 9 as we have [6,10] is an edge in a 4-gon, therefore [9,10] will not form an edge in a 4-gon and therefore $9,10$ can not appear in same 4-gon. Now $b_{2}\in \lbrace 2,9\rbrace$. If $b_{2}=2$, then $c_{5},c_{6}\in \lbrace 8,11\rbrace$ and therefore $c_{3}=6$, $c_{4}=7$ and then  $lk(2)=C_{8}(3,0,[1,9,5],[7,6,10])$, which implies $lk(7)=C_{8}(5,11,[8,1,0],[6,10,2])$. From $lk(8)$ we see that [5,8,11] is a 3-gon, therefore from $lk(7)$, we see that 8 will occur two times in $lk(11)$. If $b_{2}=9$, then [9,11] form an edge in a 4-gon, and we have [8,11] is an edge in a 4-gon, therefore from $lk(5)$ we see that face sequence is not followed in $lk(11)$.
	\end{case}
	\begin{case}
		When $ (a,b,c,d)=(5,9,10,11) $ then $ [2,3,a_1] $ will be a face where $ a_1\in \{10,11\} $. If $ \boldsymbol{a_1=10} $ then three 4-gons are $ [3,4,c_1,c_2] $, $ [2,10,c_3,c_4] $, $ [3,10,c_5,c_6] $. Therefore one of $ c_1,c_2 $ is 11 and one of $ c_3,c_4 $ is 11 which implies [2,10], [3,10], [10,11] are adjacent edges of a 3-gon and a 4-gon. [8,11] and [9,10] are adjacent edges of two 3-gon, therefore $ c_1,c_2\in\{9,11\} $, $ c_5,c_5\in\{6,8\} $ and $ c_3=11,c_4=7 $. Now $ lk(2)=C_8(0,3,[10,11,7],[5,9,1]) $ and $ lk(7)=(5,6,[0,1,8],[11,10,2]) $ which implies $ C(0,1,11,7)\in lk(8) $ which is a contradiction. If $ \boldsymbol{a_1=11} $ then three 4-gons are $ [3,4,c_1,c_2] $, $ [2,11,c_3,c_4] $, $ [3,11,c_5,c_6] $. Therefore one of $ c_1,c_2 $ is 10 and one of $ c_3,c_4 $ is 10 i.e. [2,11], [3,11] and [10,11] are adjacent edges of a 3-on and a 4-gon. As [8,11] and [9,10] are adjacent edges od two 3-gon, therefore, $ c_1,c_2\in\{8,10\} $ and one of $ c_5,c_6 $ is 9. From $ lk(2) $, $ [2,5,c_4] $ is a face, so if $ c_4=7 $ then $ [5,6,7] $ will be a face and then $ C(4,0,7,5)\in lk(6) $ and if $ c_4=6 $ form $ lk(6) $, we get $ c_5=9,c_6=7 $ i.e. $ C(7,3,11,6)\in lk(9) $, which is a contradiction.
	\end{case}
	\begin{case}
		When $(a,b,c,d)=(6,9,4,5)$, then $lk(4)=C_{8}(1,9,[a_{1},b_{1},3],[0,6,5])$, where $a_{1},b_{1}\in \lbrace 7,8,10,11\rbrace$.
		
		If $a_{1}=7$, then $lk(7)=C_{8}(9,6,[0,1,8],[b_{1},3,4])$. Since [0,6] and [6,9] are edges of two 4-gon, therefore from $lk(7)$, face sequence is not followed in $lk(6)$.
		
		If $a_{1}=8$, then $lk(8)=C_{8}(9,5,[1,0,7],[b_{1},3,4])$, which implies $lk(9)=C_{8}(4,8,[5,\allowbreak a_{2},a_{3}],[6,2,1])$ and  $lk(5)=C_{8}(1,8,[9,a_{3},a_{2}],[6,0,4])$ and then $lk(6)=C_{8}(7,a_{4},[a_{5}, a_{6},a_{2}],\allowbreak[5,4,0])$. From $lk(6)$, we have $a_{2}\in \lbrace 2,9\rbrace$, but it is not possible in $lk(9)$.
		
		If $a_{1}=10$, then either $lk(6)=C_{8}(7,a_{2},[9,1,2],[5,4,0])$ or $lk(6)=C_{8}(7,a_{2},[2,1,9],\allowbreak [5,4,0])$.
		\begin{subcase}
		When $lk(6)=C_{8}(7,a_{2},[9,1,2],[5,4,0])$, remaining three 4-gons are $[3,4,10,\allowbreak c_{1}]$, $[2,11,c_{2},c_{3}]$, $[3,11,c_{4},c_{5}]$. Now $lk(5)=C_{8}(1,8,[b_{2},b_{3},2],[6,0,4])$, therefore, $c_{3}=5$, $b_{3}=11$ and $b_{2}=c_{2}$. Now $lk(9)=C_{8}(4,10,[b_{4},b_{5},a_{2}],[6,2,1])$. From $lk(6)$ and $lk(9)$ we see that $a_{2}=11$. Since $3,10$ occur in same 4-gon, therefore $c_{2}=10$ and then $c_{4}=9$. Since [0,7] is an edge in a 4-gon, therefore from $lk(6)$ we see that [7,11] will not form an edge in a 4-gon, therefore $c_{1}=7$ and hence $c_{5}=8$. Therefore $lk(4)=C_{8}(1,9,[10,7,3],[0,6,5]), lk(6)=C_{8}(7,11,[9,1,2],[5,4,0]), lk(5)=C_{8}(1,8,\allowbreak [10,11,2],[6,0,4]), lk(9)=C_{8}(4,10,[8,3,11],[6,2,\allowbreak1])$, which implies $lk(8)=C_{8}(10,5,\allowbreak [1,0,7],[3,11,9])\Rightarrow lk(7)=C_{8}(6,11,[10,4,3],[8,1,0])\allowbreak
		\Rightarrow lk(3)=C_{8}(0,2,[11,9,8],\allowbreak [7,10,4])
		\Rightarrow lk(2)=C_{8}(3,0,[1,9,6],[5,10,11])
		\Rightarrow lk(10)=C_{8}(9,8,\allowbreak[5,2,11],[7,3,4])\allowbreak
		\Rightarrow lk(11)=C_{8}(7,6,[9,8,3],[2,5,$ $10])$. This map isomorphic to the map $ \boldsymbol{KO_{1[(3^3,4,3,4)]}} $, given in figure \ref{33434or}, under the map $ (0,8,3,11,5)(1,2,6,7,10,4) $.
	\end{subcase}
		\begin{subcase}
		When $lk(6)=C_{8}(7,a_{2},[2,1,9],[5,4,0])$ where $a_{2}\in \lbrace 10,11\rbrace$, then $lk(5)=C_{8}(1,8,$ $[b_{2},b_{3},9],[6,0,4])$. Now remaining three 4-gon are $[3,4,10,c_{1}]$, $[2,11,c_{2},c_{3}]$, $[3,11,c_{3},\allowbreak c_{4}]$. From $lk(5)$ we see that $b_{2},b_{3}\in \lbrace 3,11\rbrace$ and $c_{4},c_{5}\in \lbrace 5,9\rbrace$. Therefore one of $c_{2},c_{3}$ will be 10. Since [0,7] is an edge in a 4-gon, therefore $[2,a_{2}]$ will corm an edge in a 4-gon. Now if $a_{2}=10$, then $c_{3}=10$, which implies one of $c_{1},c_{2}$ will be 7, but in both cases face sequence will not followed in $lk(10)$. If $a_{2}=11$, and we have [2,3,11] is a 3-gon, therefore face sequence is not followed in $lk(11)$.
	\end{subcase}
		
		If $a_{1}=11$, then remaining three 4-gons are $[3,4,11,b_{1}]$, $[2,10,c_{2},c_{3}]$, $[3,10,c_{4},c_{5}]$. Since $3,11$ appear in a 4-gon, therefore one of $c_{2},c_{3}$ must be 11. From $lk(4)$, we see that [1,9] is an edge of a 4-gon, therefore [9,11] will not form an edge in a 4-gon and therefore one of $c_{4},c_{5}$ must be 9. Now $lk(6)$ is either $lk(6)=C_{8}(7,b_{2},[2,1,9],[5,4,0])$ or $lk(6)=C_{8}(7,b_{2},[9,1,2],[5,4,0])$. If $lk(6)=C_{8}(7,b_{2},[2,1,9],[5,4,0])$ then $lk(9)=C_{8}(4,11,[10,3,5],\allowbreak [6,2,1])$, but we have $10,11$ occur in same 4-gon, which implies face sequence is not followed in $lk(11)$. If  $lk(6)=C_{8}(7,b_{2},[9,1,2],[5,4,0])$ where $b_{2}\in \lbrace 10,11\rbrace$, we see that [0,7] is an edge in a 4-gon, therefore $[9,b_{2}]$ will form an edge in a 4-gon, which implies $b_{2}=10=c_{4}$. $c_{5}\neq 7$ as then 7 will occur two times in $lk(10)$, therefore $c_{5}=8$ and hence $b_{1}=7$. Now $lk(5)=C_{8}(1,8,[b_{3},b_{4},2],[6,0,4])$, from here we see that [2,5] is an edge in a 4-gon, therefore $c_{3}=5$, which implies $b_{4}=10$ and $b_{3}=c_{2}=11$. Now $lk(9)=C_{8}(4,11,[8,3,10],[6,2,1])
		\Rightarrow lk((8)=C_{8}(5,11,[9,10,3],[7,0,1])
		\Rightarrow lk(7)=C_{8}(6,10,[11,4,5],[8,1,$ $0])
		\Rightarrow lk(10)=C_{8}(6,7,\allowbreak [11,5,2],[3,8,9])
		\Rightarrow lk(11)=C_{8}(9,8,[5,2,10],[7,3,4])
		\Rightarrow lk(2)=C_{8}(3,0,[1,9,6],[5,\allowbreak 11,10])
		\Rightarrow lk(3)=C_{8}(0,2,[10,9,8],\allowbreak[7,11,4])$.This map isomorphic to the map $ \boldsymbol{KO_{2[(3^3,4,3,4)]}} $, given in figure \ref{33434or}, under the map $ (0,8,3,11,4,1,2,6,7,10,5) $.
	\end{case}
	\begin{case}
		When $(a,b,c,d)=(6,9,5,10)$, then either $lk(6)=C_{8}(7,11,[9,1,2],[5,4,0])$ or $lk(6)=C_{8}(7,a_{1},[2,1,9],[5,4,0])$. If $lk(6)=C_{8}(7,a_{1},[2,1,9],[5,4,0])$, then 5 will appear two times in $lk(9)$. If $lk(6)=C_{8}(7,11,[9,1,2],[5,4,0])$, remaining three 4-gons are $[3,4,c_{1},c_{2}]$, $[2,a_{1},c_{3},c_{4}]$, $[3,a_{1},c_{5},c_{6}]$. Since [0,7] is an edge in a 4-gon, therefore [9,11] will form an edge in a 4-gon, and by considering $lk(2)$ we see that [2,5] will form an edge in a 4-gon, therefore $c_{3}=10, c_{4}=5$ and therefore $a_{1}=11$, which implies $c_{5}=9$. Since $10,11$ can not appear in two 4-gon, therefore one of $c_{1},c_{2}$ must be 10. Since [5,10] is an edge in a 4-gon, therefore [8,10] will not form in a 4-gon, which implies $c_{6}=8$ and hence one of $c_{1},c_{2}$ must be 7. Now $lk(2)=C_{8}(3,0,[1,9,6],[5,10,11])\Rightarrow lk(11)=C_{8}(6,7,[10,5,2],[3,8,9])\Rightarrow lk(9)=C_{8}(b_{1},5,[1,2,6],[11,3,8])$. Since [1,5] does not form an edge in a 4-gon, therefore $[5,b_{1}]$ will form an edge in a 4-gon, therefore $b_{1}\in \lbrace 4,6\rbrace$. But $b_{1}\neq 6$, therefore $b_{1}=4$. Therefore $lk(5)=C_{8}(9,1,[10,11,2],[6,0,4])\Rightarrow lk(8)=C_{8}(10,4,[9,11,3],[7,0,1])$, which implies $c_{2}=7$ which implies $c_{1}=10$. Now $lk(3)=C_{8}(0,2,[11,9,8],[7,10,4])\allowbreak\Rightarrow lk(4)=C_{8}(9,8,[10,7,3],[0,6,5])\Rightarrow lk(7)=C_{8}(6,11,[10,4,3],[8,1,0])\Rightarrow lk(10)\allowbreak =C_{8}(1,8,[4,3,7],[11,2,5])$. This map isomorphic to the map $ \boldsymbol{KO_{2[(3^3,4,3,4)]}} $, given in figure \ref{33434or}, under the map $ (0,7,8,1)(2,6,5,3,11,10)(4,9) $.
	\end{case}
	\begin{case}
		When $(a,b,c,d)=(9,5,6,10)$, then $lk(6)=C_{8}(1,10,[a_{1},a_{2},7],[0,4,5])$ which implies $lk(2)=C_{8}(0,3,[b_{1},b_{2},b_{3}],[9,5,1])$. Since [1,5] is an edge in a 4-gon, therefore from $lk(6)$ we see that $[6,a_{1}]$ and $[10,a_{1}]$ form an edge in two distinct 4-gon, theefore $a_{1}=11$. Therefore remaining three 4-gons are $[6,7,a_{2},11]$, $[10,11,c_{1},c_{2}]$ and $[10,8,c_{3},c_{4}]$. Since [0,1] is an edge in a 4-gon, therefore $[2,b_{1}]$ and $[3,b_{1}]$ form an edge in two distinct 4-gon. Therefore [3,4], $[2,b_{1}]$ and $[3,b_{1}]$ form an edge in three distince 4-gon, which implies $b_{1}=11$ and then $a_{2}=2$, $c_{1}=3$ and $c_{2}=9$ and $c_{3},c_{4}\in \lbrace 3,4\rbrace$. Now $lk(7)=C_{8}(4,9,[2,11,6],[0,1,8])\Rightarrow lk(4)=C_{8}(7,9,[5,6,0],[3,10,8])$ which implies 9 will occur two times in $lk(5)$.
	\end{case}
	\begin{case}
		When $(a,b,c,d)=(9,5,10,4)$, then $lk(4)=C_{8}(8,1,[10,c_{1},3],[0,6,5])\Rightarrow lk(5)=C_{8}(10,8,[4,0,6],[9,2,1])\Rightarrow lk(10)=C_{8}(1,5,[8,a_{1},a_{2}],[c_{1},3,4])$. Now incomplete three 4-gons are $[3,4,10,c_{1}]$, $[2,11,c_{2},c_{3}]$ and $[3,11,c_{4},c_{5}]$. Since $3,10$ can not be occur in two 4-gon, therefore one of $c_{2},c_{3}$ must be 10. Considering $lk(10)$ we will see that [8,10] will form an edge in a 4-gon, therefore one of $c_{2},c_{3}$ must be 8. Now $4,6$ can not be appear in two 4-gon, therefore one of $c_{4},c_{5}$ must be 6. Now $c_{1}\in \lbrace 7,9\rbrace$. If $c_{1}=7$, then $lk(7)=C_{8}(6,b_{1},[10,4,3],[8,1,0])$. Since [0,6] form an edge in a 4-gon, therefore $[10,b_{1}]$ will form an edge in a 4-gon, which implies $b_{1}\in \lbrace 2,11\rbrace$, but form any $b_{1}$ is not possible as then face sequence is not followed in $lk(b_{1})$. If $c_{1}=9$ then $lk(9)=C_{8}(6,11,[3,4,10],[2,1,5])$, but we have [3,11] is an edge in a 4-gon and $6,11$ appear in same 4-gon, therefore face sequence is not followed in $lk(11)$.
	\end{case}
	\begin{case}
		When $(a,b,c,d)=(9,5,10,6)$ then three incomplete 4-gons are $[3,4,c_{1},c_{2}]$, $[2,a_{1},\allowbreak c_{3},c_{4}]$ and $[3,a_{1},c_{5},c_{6}]$. Now $lk(6)=C_{8}(8,1,[10,b_{1},7],[0,4,5])$. As $[10,6,7,b_{1}]$ is a 4-gon, therefore $a_{1}=10$. Now $(4,5)$ and $(4,6)$ can not appear in same 4-gon, therefore $(5,10)$ and $(6,10)$ will appear in two different 4-gon, which implies face sequence is not followed in $lk(10)$.
	\end{case}
	\begin{case}
		When $(a,b,c,d)=(9,5,10,11)$, we see that [10,11] form an edge in a 4-gon, therefore [8,11] will not form an edge in a 4-gon. Now $lk(5)$ is either $lk(5)=C_{8}(10,b_{1},[4,0,6],[9,2,\allowbreak1])$ or $lk(5)=C_{8}(10,b_{1},[6,0,4],[9,2,1])$, where for both cases $b_{1}\in \lbrace 7,8\rbrace$. In both cases, $b_{1}=8$ is not possible as then $lk(8)$ not possible. If $b_{1}=7$, then $lk(7)=C_{8}(5,4,[8,1,0],[6,d_{1},10])$. Now remaining incomplete 4-gons are $[3,4,c_{1},c_{2}]$, $[2,a_{1},c_{3},c_{4}]$ and $[3,a_{1},c_{5},c_{6}]$. As $[10,7,6,d_{1}]$ is a 4-gon, therefore $a_{1}=10$. Now $d_{1}\in \lbrace 2,3\rbrace$. Since $10,11$ can not appear in two distinct 4-gon, therefore one of $c_{1},c_{2}$ must be 11 and since $3,11$ can not appear in two different 4-gon, therefore one of $c_{3},c_{4}$ must be 11. If $d_{1}=3$ then [8,11] will form an egde in a 4-gon, which is not possible. If $d_{1}=2$ then $[10,7,6,d_{1}]$ is not possible.
	\end{case}
	\begin{case}
		When $(a,b,c,d)=(9,6,10,d)$, then $lk(6)=C_{8}(7,10,[1,2,9],[5,4,0])$ and then 1 will appears two times in $ lk(8) $, is a contradiction. Therefore $ (a,b,c,d)\neq (9,6,10,4),(9,6,\allowbreak10,5),(9,6,10,11) $.
	\end{case}
	\begin{case}
		When $(a,b,c,d)=(9,10,4,5)$, then remaining three 4-gons are $[3,4,b_{1},b_{2}],\allowbreak [2,11,\allowbreak b_{3},b_{4}],[3,$ $11,b_{5},b_{6}]$. Now $lk(4)=C_{8}(1,10,[b_{1},b_{2},3],[0,6,5])$, therefore $b_{1},b_{2}\notin \lbrace 5,6,\allowbreak10\rbrace$. 2 and 10 can not be appear in same 4-gon, therefore one of $b_{5},b_{6}$ must be 10 and therefore one of $b_{1},b_{2}$ must be 9. Since [1,10] is an edge in a 4-gon, therefore $[10,b_{1}]$ will not form an edge in a 4-gon and hence $b_{2}=9$. Now $b_{1}\in \lbrace 7,10\rbrace$. If $b_{1}=7$ then $lk(7)=C_{8}(10,6,[0,1,8],[9,3,4])$, from here we see that [6,10] form an edge in a 4-gon and we know [0,6] is an edge in a 4-gon, this implies face sequence will not followed in $lk(6)$. If $b_{1}=8$ then $lk(8)=C_{8}(10,5,[1,0,7],[9,3,4])$, this implies [5,10] will form an edge in a 4-gon, which implies one of $b_{5},b_{6}$ will be 5 and therefore $b_{3},b_{4}\in \lbrace 6,7\rbrace$. Now $lk(3)=C_{8}(0,2,[11,b_{6},b_{5}],[9,8,4])$ and $lk(9)=C_{8}(7,c_{1},[c_{2},c_{3},b_{5}],[3,4,8])$. From $lk(3)$, $b_{5}\in \lbrace 5,10\rbrace$ and from $lk(9)$, $b_{5}\in \lbrace 1,2,10\rbrace$. Therefore $b_{5}=10$ and then $b_{6}=5, c_{1}=1,c_{2}=2$. Since [7,8] is an edge in a 4-gon, therefore $[2,c_{1}]$ form an edge in a 4-gon, which implies $c_{1}\in \lbrace 6,7,11\rbrace$, but for any $c_{1}$ face sequence will not follow in $lk(7)$.
	\end{case}
	\begin{case}
		When $(a,b,c,d)=(9,10,4,11)$, then remaining three 4-gons are $[3,4,b_{1},b_{2}]$, $[2,11,b_{3},b_{4}]$ and $[3,11,b_{5},b_{6}]$. As [4,11] is not an edge in a 4-gon, therefore [4,10] will form an edge in a 4-gon, therefore $lk(4)=C_{8}(11,1,[10,b_{2},3],[0,6,5])$, where $b_{2}\in \lbrace 7,8,\rbrace$ and $b_{1}=10$. Now since [4,11] is not an edge in a 4-gon, therefore [8,11] will form an edge in a 4-gon and hence $b_{2}=7$. Now we see that [5,11] and [8,11] will form an edge in two distinct 4-gon. Now $lk(7)=C_{8}(6,c_{1},[c_{2},4,c_{3}],[8,1,0])$, where $c_{3}\in \lbrace 3,10\rbrace$. If $c_{3}=3$ then $lk(3)=C_{8}(0,2,[11,c_{4},3],[7,10,4])$ which implies [8,11] appear in same 4-gon but not form an edge, which is not possible. If $c_{3}=10$ then $lk(10)=C_{8}(8,6,[9,2,1],[4,3,7])$. We have [7,8] is an edge in a 4-gon, therefore [6,9] will form an edge in a 4-gon and then [5,8] will form an edge which is not possible. Similarly $ (a,b,c,d)\neq (9,10,5,11) $.
	\end{case}
	\begin{case}
		When $(a,b,c,d)=(9,10,11,4)$ then either $lk(4)=C_{8}(8,1,[11,c_{1},3],[0,6,\allowbreak 5])$ or $lk(4)=C_{8}(11,1,[8,c_{1},3],[0,6,5])$. Now remaining three 4-gons are $[3,4,b_{1},b_{2}]$, $[2,11,b_{3},b_{4}]$, $[3,11,b_{5},b_{6}]$. If $lk(4)=C_{8}(8,1,[11,c_{1},3],[0,6,5])$ then one of $b_{1},b_{2}$ is 11 which is not possible. If $lk(4)=C_{8}(11,1,[8,c_{1},3],[0,6,5])$, since [1,11] is not an edge in a 4-gon, therefore [5,11] and [10,11] are make an edge in two distinct 4-gon. Now $b_{3}\neq 10$ then (2,10) will appear in two 4-gon. Therefore $b_{5}=10$ and $b_{3}=5$. Since (5,6) and (4,6) can not appear in two 4-gon, therefore $b_{6}=6$. (2,9) can not appear in two 4-gon, therefore $b_{2}=9$ and therefore $b_{4}=7$. Therefore $lk(4)=C_{8}(11,1,[8,9,3],[0,6,5])\Rightarrow
		lk(3)=C_{8}(0,2,[11,10,6],[9,8,4])\Rightarrow
		lk(2)=C_{8}(0,3,[11,5,7],[9,10,1])\Rightarrow
		lk(9)=C_{8}(7,6,[3,4,8],$ $[10,1,2])\Rightarrow
		lk(6)=C_{8}(7,9,[3,11,10],[5,4,0])\Rightarrow
		lk(10)=C_{8}(5,8,[9,2,1],[11,3,6])\Rightarrow
		lk(8)=C_{8}(10,5,\allowbreak [7,0,1],[4,3,9])\Rightarrow
		lk(7)=C_{8}(6,9,[2,11,5],[8,1,0])\Rightarrow
		lk(5)=C_{8}(8,10,\allowbreak[6,\allowbreak 0,4],[11,\allowbreak 2,7])$ $\Rightarrow
		lk(11)=C_{8}(1,4,[5,7,2],[3,6,10])$. This map is identical to the map $ \boldsymbol{KNO_{3[(3^3,4,3,4)]}} $, given in figure \ref{33434non}.
	\end{case}
	\begin{case}
		When $(a,b,c,d)=(9,10,11,6)$ then $lk(6)=C_{8}(8,1,[11,a_{1},7],[0,4,5])$. Now remaining three incomplete 4-gons are $[3,4,b_{1},b_{2}]$, $[2,11,b_{3},b_{4}]$ and $[3,11,b_{5},b_{6}]$. Since (4,6) can not appear in two 4-gon, therefore 6 will be one of $b_{3},b_{4},b_{5},b_{6}$ and we have [6,11] is an edge, therefore [6,11] form an edge in a 4-gon and then [10,11] will not form an edge in a 4-gon, which implies one of $b_{1},b_{2}$ must be 10. Since (2,9) and (9,10) can not appear in two 4-gon, therefore one of $b_{5},b_{6}$ must be 9. Now from $lk(6)$ we see that $a_{1}=2$. Therefore $lk(2)=C_{8}(0,3,[11,6,7],[9,10,1])$ which implies 6 will appear two times in $lk(7)$.
	\end{case}
	\begin{case}
		When $(a,b,c,d)=(6,9,4,10)$ then  Now either $lk(4)=C_{8}(10,1,[9,a_{1},3],[0,\allowbreak 6,5])$ or $lk(4)=C_{8}(9,1,[10,a_{1},3],[0,6,5])$ where $a_{1}\in \lbrace 7,8,11\rbrace$.
		\begin{subcase}
		If $lk(4)=C_{8}(10,1,[9,a_{1},3],[0,6,5])$ then $a_{1}\neq 7,8$ as then (10,11) will occur in  two 4-gon, therefore $a_{1}=11$. Now two incomplete 4-gons are $[2,10,b_{1},b_{2}]$ and $[3,10,b_{3},b_{4}]$. Now since (3,11) will not appear in two 4-gon, therefore one of $b_{1},b_{2}$ is 11. If remaining one of $b_{1},b_{2}$ is 5 then (7,8) will occur in two 4-gon, therefore one of $b_{3},b_{4}$ is 5. From $lk(4)$ we see that [1,10] is not an edge in a 4-gon, therefore [5,10] will form an edge in a 4-gon and then [8,10] will form an edge in a 4-gon, therefore $b_{1}=8$, $b_{2}=11$, $b_{3}=5$ and $b_{4}=7$. Therefore $lk(10)=C_{8}(1,4,[5,7,3],[2,11,8])\Rightarrow
		lk(3)=C_{8}(0,2,[10,5,7],[11,9,4])\Rightarrow
		lk(2)=C_{8}(0,3,[10,8,11],[6,9,1])\Rightarrow
		lk(11)=C_{8}(6,7,[3,4,9],[8,10,2])\Rightarrow
		lk(7)=C_{8}(6,\allowbreak11,[3,10,\allowbreak 5],[8,1,0])\Rightarrow
		lk(8)=C_{8}(5,9,\allowbreak [11,2,10],[1,0,7])\Rightarrow
		lk(9)=C_{8}(5,8,[11,3,4],[1,\allowbreak2,6])\Rightarrow
		lk(5)=C_{8}(9,8,[7,3,10],\allowbreak [4,0,6])\Rightarrow
		lk(6)=C_{8}(7,11,[2,1,9],[5,4,0])$. This map isomorphic to the map $ \boldsymbol{KNO_{3[(3^3,4,3,4)]} }$, given in figure \ref{33434non}, under the map $ (0,3,2)(1,4,11,7,9,\allowbreak5,10) $.
	\end{subcase}
		\begin{subcase}
		If $lk(4)=C_{8}(9,1,[10,a_{1},3],[0,6,5])$ then three incomplete 4-gons are $[3,4,10,a_{1}]$, $[2,11,b_{1},b_{2}]$ and $[3,11,b_{3},b_{4}]$. Since (3,10) can not appear in two 4-gon, therefore one of $b_{1},b_{2}$ is 10. Since [4,10] is an edge in a 4-gon, therefore [8,10] will not form an edge in a 4-gon, therefore one of $b_{3},b_{4}$ will be 8 and hence $a_{1}=7$. Now $lk(7)=C_{8}(6,a_{2},[10,4,3],[8,1,0])$. Since [0,6] is an edge in a 4-gon, therefore $[6,a_{2}]$ will not form an edge in a 4-gon which implies $a_{2}=11$. Now $lk(5)=C_{8}(9,a_{3},[a_{4},a_{5},a_{6}],[6,0,4])$ which implies $lk(6)=C_{8}(7,11,[a_{7},a_{8},a_{6}],[5,4,0])$. From $lk(6)$ we see that $a_{6}\in \lbrace 2,9\rbrace$, but if $a_{6}=9$ then 9 will occur two times in $lk(5)$, therefore $a_{6}=2$ and then $a_{7}=9,a_{8}=1$. Now from $lk(5)$ we see that [2,5] form an edge in a 4-gon, therefore $b_{2}=5$ and then $b_{1}=10$ and $b_{3}=9,b_{4}=8$ which implies $a_{4}=10,a_{5}=11$. Now since [4,9] is not an edge in a 4-gon, therefore $[8,a_{3}]$ will form an edge in a 4-gon, which implies $a_{3}=8$. Therefore $lk(3)=C_{8}(0,2,[11,9,8],[7,10,4])\Rightarrow
		lk(10)=C_{8}(1,8,[5,2,11],[7,3,4])\Rightarrow
		lk(8)=C_{8}(10,5,[9,11,3],[7,0,1])\Rightarrow
		lk(2)=C_{8}3,0,[1,9,6],[5,10,11]\Rightarrow
		lk(9)=C_{8}(5,4,\allowbreak [1,2,6],[11,3,8])\Rightarrow
		lk(11)=C_{8}(6,7,[10,5,2],[3,8,9])$. This map isomorphic to the map $ \boldsymbol{KNO_{1[(3^3,4,3,4)]}} $, given in figure \ref{33434non}, under the map $ (0,7,8,1)(2,6,5,9,4,3,11,10) $.
	\end{subcase}
	\end{case}
	\begin{case}
		When $(a,b,c,d)=(6,9,10,4)$ then either $lk(4)=C_{8}(1,8,[5,6,0],[3,b_{1},\allowbreak 10])$ or $lk(4)=C_{8}(1,10,[5,6,0],[3,b_{1},8])$.
		\begin{subcase}
		If $lk(4)=C_{8}(1,8,[5,6,0],[3,b_{1},10])$ then remaining three incomplete 4-gons are $[3,4,10,b_{1}]$, $[2,11,b_{2},b_{3}]$ and $[3,11,b_{4},b_{5}]$. Now since (3,10) can not appear in two 4-gon, therefore one of $b_{2},b_{3}$ is 10 and as [4,10] is an edge in a 4-gon, therefore [9,10] will not form an edge in a 4-gon which implies one of $b_{4},b_{5}$ is 9. Noe from $lk(4)$ we see that $b_{1}=7$. Now either $lk(7)=C_{8}(6,a_{1},[3,4,10],[8,1,0])$ or $lk(7)=C_{8}(6,a_{1},[10,4,3],[8,1,0])$. If $lk(7)=C_{8}(6,a_{1},[3,4,10],[8,1,0])$, since [0,6] is an edge in a 4-gon, therefore $[6,a_{1}]$ will not form an edge in a 4-gon which implies $a_{1}=11$ which is not possible as then 11 will occur two times in $lk(3)$. If $lk(7)=C_{8}(6,a_{1},[10,4,3],[8,1,0])$ then similarly $[6,a_{1}]$ will not form an edge in a 4-gon, therefore $a_{1}=11$ which implies $b_{2}=10$ and then from $lk(4)$ and $lk(7)$ we see that [3,8] will form an edge in a 4-gon, therefore $b_{5}=8$ and therefore $b_{6}=9$ and $b_{3}=5$. Now
		$lk(10)=C_{8}(1,9,[5,2,11],[7,3,4])\Rightarrow
		lk(8)=C_{8}(5,4,[1,0,7],[3,11,9])\Rightarrow
		lk(9)=C_{8}(5,10,[1,2,6],[11,3,8])\Rightarrow
		lk(6)=C_{8}(7,11,[9,1,2],[5,4,0])\Rightarrow
		lk(11)=C_{8}(6,7,[10,5,2],\allowbreak[3,8,9])\Rightarrow
		lk(2)=C_{8}(0,3,[11,10,5],[6,9,1])\Rightarrow
		lk(3)=C_{8}(0,2,\allowbreak [11,9,8],[7,10,4])\Rightarrow
		lk(5)=C_{8}(9,8,[4,0,6],[2,11,10])$. This map isomorphic to the map $ \boldsymbol{KO_{2[(3^3,4,3,4)]}} $, given in figure \ref{33434or}, under the map $ (0,5)(1,9,2,7,11,8,3,10) $.
	\end{subcase}
		\begin{subcase}
		If $lk(4)=C_{8}(1,10,[5,6,0],[3,b_{1},8])$ then $b_{1}\in \lbrace 10,11\rbrace$, but $b_{1}=10$ is not possible, therefore $b_{1}=11$. Now remaining two incomplete 4-gons are $[2,10,b_{2},b_{3}]$ and $[3,10,b_{4},b_{5}]$. Since (3,11) and (2,9) can not appear in two 4-gon, therefore one of $b_{2},b_{3}$ is 11 and one of $b_{4},b_{5}$ is 9 and since [9,10] is an edge therefore $b_{4}=9$. From $lk(4)$ we see that [1,10] is not an edge in a 4-gon, therefore [5,10] will form an edge in a 4-gon which implies $b_{2}=5$ and then $b_{3}=11$ and $b_{5}=7$. Now $lk(10)=C_{8}(4,1,[9,7,3],[2,11,5])\Rightarrow
		lk(2)=C_{8}(0,3,[10,5,11],[6,9,1])\Rightarrow
		lk(3)=C_{8}(0,2,[10,9,7],[11,8,4])\Rightarrow
		lk(6)=C_{8}(7,11,[2,1,9],\allowbreak[5,4,0])\Rightarrow
		lk(7)=C_{8}(6,11,[3,10,9],[8,1,0])\Rightarrow
		lk(11)=C_{8}(6,7,[3,4,8],[5,10,2])\Rightarrow
		lk(8)=C_{8}(9,5,\allowbreak [11,3,4],[1,0,7])\Rightarrow
		lk(5)=C_{8}(8,9,[6,0,4],[10,2,11])\Rightarrow
		lk(9)=C_{8}(8,5,\allowbreak[6,2,1],\allowbreak [10,3,7])$. This map isomorphic to the map $ \boldsymbol{KNO_{3[(3^3,4,3,4)]}} $, given in figure \ref{33434non}, under the map $ (0,2)(4,11,6,7,9,8,10) $.
	\end{subcase}
	\end{case}
	\begin{case}
		When $(a,b,c,d)=(6,9,10,5)$ then either $lk(6)=C_{8}(7,11,[2,1,9],[5,4,0])$ or $lk(6)=C_{8}(7,11,[9,1,2],[5,4,0])$.
		\begin{subcase}
		If  $lk(6)=C_{8}(7,11,[2,1,9],[5,4,0])$ then $lk(2)=C_{8}(0,3,[10,b_{3},11],\allowbreak [6,9,1])$. Now remaining three incomplete 4-gons are $[3,4,b_{1},b_{2}]$, $[2,10,b_{3},11]$ and $[3,10,b_{4},\allowbreak b_{5}]$. Therefore one of $b_{1},b_{2}$ is 11. Now either [5,10] will form an edge in a 4-gon or not. If [5,10] form an edge in a 4-gon then $lk(5)=C_{8}(1,8,[4,0,6],[9,a_{1},10])$ which implies (9,10) occur in same 4-gon but not form an edge which make contradiction. If [5,10] will not form an edge in a 4-gon, then one of $b_{1},b_{2}$ is 5 and [5,8] will form an edge in a 4-gon, which is not possible.
	\end{subcase}
		\begin{subcase}
		If $lk(6)=C_{8}(7,11,[9,1,2],[5,4,0])$ then either [5,10] form an edge in a 4-gon or not. If [5,10]  is not an edge in a 4-gon then $lk(5)=C_{8}(10,1,[8,11,2],[6,0,4])$. Now remaining two incomplete 4-gons are $[3,4,b_{1},b_{2}]$ and $[3,11,b_{3},b_{4}]$ and then 10 will occur in both 4-gon which is not possible as (3,10) can not appear in two 4-gon. Therefore [5,10] will form an edge in a 4-gon and then $lk(5)=C_{8}(1,8,[4,0,6],[2,11,\allowbreak 10])$ which implies $lk(2)=C_{8}(0,3,[11,10,5],[6,9,1])$. Now remaining two incomplete 4-gons are $[3,4,b_{1},b_{2}]$ and $[3,11,b_{3},b_{4}]$. As [10,11] is an edge in a 4-gon, therefore one of $b_{1},b_{2}$ is 10. Since [9,10] is not an edge in a 4-gon, therefore one of $b_{3},b_{4}$ is 9. Since [1,8] is an edge in a 4-gon, therefore from $lk(5)$ we get [4,8] will not form an edge in a 4-gon which implies one of $b_{3},b_{4}$ is 8 and then one of $b_{1},b_{2}$ is 7. Since [0,7] is an edge in a 4-gon, therefore from $lk(6)$ we see that [9,11] will form an edge in a 4-gon i.e. $b_{3}=9$ and $b_{4}=8$. Now $lk(8)=C_{8}(5,4,[9,11,3],[7,0,1])$ which implies $b_{2}=7$ and then $b_{1}=10$. Now $lk(3)=C_{8}(0,2,[11,9,8],[7,10,4])\Rightarrow lk(4)=C_{8}(8,9,[10,7,3],[0,6,5])\Rightarrow lk(10)=C_{8}(9,1,[5,2,11],[7,3,4])\Rightarrow lk(7)=C_{8}(6,11,[10,4,3],[8,1,0])\Rightarrow lk(9)=C_{8}(10,4,[8,3,11],\allowbreak[6,2,1])$. This map isomorphic to the map $\boldsymbol{ KNO_{1[(3^3,4,3,4)]}} $, given in figure \ref{33434non}, under the map $ (0,10)(1,2,8,9,3,5)(6,11,7) $.
	\end{subcase}
	\end{case}
	\begin{case}
		When $(a,b,c,d)=(9,10,5,6)$ then remaining three incomplete 4-gons are $[3,4,b_{1},\allowbreak b_{2}]$, $[2,11,b_{3},b_{4}]$ and $[3,11,b_{5},b_{6}]$. Now $lk(6)=C_{8}(1,8,[c_{1},c_{2},7],[0,4,5])$. If $b_{5},b_{6}\in \lbrace 6,7\rbrace$ then one of $b_{3},b_{4}$ must be 9 or 10 which is not possible. Therefore $b_{3},b_{4}\in \lbrace 6,7\rbrace$. Now $lk(2)=C_{8}(0,3,[11,b_{3},b_{4}],[9,10,1])$, from this and $lk(6)$ we see that $b_{4}\neq 6$, therefore $b_{4}=7,b_{3}=6$ and then $c_{1}=11,c_{2}=2$. Since [1,8] is an edge in a 4-gon, therefore [8,11] will not form an edge in a 4-gon, therefore one og $b_{1},b_{2}$ is 8. Now $lk(7)=C_{8}(c_{3},9,[2,11,6],[0,1,8])$ where $c_{3}\in \lbrace 3,4\rbrace$, which implies one of $b_{5},b_{6}$ is 9 and then remaining one of $b_{1},b_{2}$ is 10 and therefore remaining one of $b_{5},b_{6}$ is 5. Now $c_{3}\neq 3$ as then face sequence will not followed in $lk(3)$, therefore $c_{3}=4$. Therefore $lk(4)=C_{8}(9,7,[8,10,3],[0,6,5])\Rightarrow lk(8)=C_{8}(11,6,[1,0,7],[4,3,10])\Rightarrow lk(5)=C_{8}(1,10,[11,3,9],[4,0,6])$ which implies $b_{5}=5,b_{6}=9$. Now $lk(9)=C_{8}(4,7,[2,1,10],[3,11,\allowbreak5])\Rightarrow lk(3)=C_{8}(0,2,[11,5,9],[10,8,4])\Rightarrow lk(10)=C_{8}(5,11,\allowbreak [8,4,3],[9,2,1])\Rightarrow lk(11)=C_{8}(10,8,[6,7,2],[3,9,5])$. This map is isomorphic to the map $ \boldsymbol{KO_{2[(3^3,4,3,4)]}} $, given in figure \ref{33434or}, under the map $ (0,5,3,10,2,1,6,9,8)(4,11,7) $.
	\end{case}
	\begin{case}
		When $(a,b,c,d)=(9,10,6,11)$ then remaining three incomplete 4-gons are $[3,4,b_{1},\allowbreak b_{2}]$, $[2,11,b_{3},b_{4}]$ and $[3,11,b_{5},b_{6}]$. (10,11) can not appear in two 4-gon, therefore one of $b_{1},b_{2}$ is 10, and (2,9), (9,10) can not appear in two 4-gon, which implies one of $b_{5},b_{6}$ is 9. Since (4,6) will not appear in two 4-gon, therefore (6,11) will appear in same 4-gon and we see that [6,11] is an edge, therefore [6,11] will form an edge in a 4-gon. Now $lk(6)=C_{8}(10,1,[11,a_{1},7],[0,4,5])$ which implies [6,7] form an edge in a 4-gon, therefore $a_{1}=2$ and $b_{3}=6,b_{4}=7$. Since (4,5) can not appear in two 4-gon, therefore one of $b_{5},b_{6}$ is 5 and hence one of $b_{1},b_{2}$ is 8. Now $lk(2)=C_{8}(0,3,[11,6,7],[9,10,1])$ which implies $lk(7)=C_{8}(c_{1},9,[2,11,6],[0,1,8])$. From here we see that $[8,c_{1}]$ will form an edge in a 4-gon which implies $c_{1}\in \lbrace 3,4,10\rbrace$. $c_{1}\neq 3,10$ as then face sequence is not followed in $lk(c_{1})$, therefore $c_{1}=4$ which implies $b_{1}=8,b_{2}=10$. Now $lk(4)=C_{8}(7,9,[5,6,0],[3,10,8])$ which implies $lk(8)=C_{8}(c_{2},11,[1,0,7],[4,3,10]$, where $[11,c_{2}]$ form an edge in a 4-gon and $[10,c_{2}]$ will not form an edge in a 4-gon which implies $c_{2}=5$. Now $lk(11)=C_{8}(8,1,[6,7,2],[3,9,5])\Rightarrow lk(5)=C_{8}(8,10,[6,0,4],[9,3,11])\Rightarrow lk(9)=C_{8}(4,7,[2,1,10],[3,11,5])\Rightarrow lk(3)=C_{8}(0,2,[11,5,9],[10,8,4])\Rightarrow lk(10)=C_{8}(5,6,\allowbreak [1,2,9],\allowbreak [3,4,8])$. This map isomorphic to the map $ \boldsymbol{KO_{1[(3^3,4,3,4)]}} $, given in figure \ref{33434or}, under the map $ (0,10)(1,4,8)(2,5)(3,7,11,9,6) $.
	\end{case}
	\begin{case}
		When $(a,b,c,d)=(9,10,11,5)$ then remaining three incomplete 4-gons are $[3,4,b_{1},\allowbreak b_{2}]$, $[2,11,b_{3},b_{4}]$ and $[3,11,b_{5},b_{6}]$. If one of $b_{3},b_{4}$ is 5, then one of $b_{5},b_{6}$ will be 6. Since [5,11] form an edge, therefore it will form an edge in a 4-gon, which implies [10,11] and [5,8] will not form an edge in 4-gon, therefore one of $b_{1},b_{2}$ is 10. Since (2,9) and (9,10) can not appear in two 4-gon, therefore one of $b_{5},b_{6}$ is 9. Since [5,8] is not an edge in a 4-gon, therefore one of $b_{1},b_{2}$ is 8 which implies one of $b_{3},b_{4}$ is 7, therefore $b_{3}=5,b_{4}=7$. Now $lk(7)=C_{8}(6,10,[2,11,5],[8,1,0])$ which implies 7 will occur two times in $lk(8)$, which is not possible. Similarly if $b_{5}=5$ then one of $b_{3},b_{4}$ is 6 one one of $b_{1},b_{2}$ is 10 and $b_{6}=9$. Now either $lk(5)=C_{8}(1,8,[4,0,6],[9,3,11])$ or $lk(5)=C_{8}(1,8,[6,0,4],[9,3,11])$. If $lk(5)=C_{8}(1,8,[4,0,6],[9,3,11])$ then $lk(6)=C_{8}(9,c_{1},[c_{2},c_{3},7],[0,4,5])$. From $lk(5)$ we see that one of $b_{3},b_{4}$ is 8 and from $lk(6)$ we see that $b_{3},b_{4}$ is 7, which is not possible. If $lk(5)=C_{8}(1,8,[6,0,4],[9,3,11])$ then $lk(9)=C_{8}(4,c_{1},[2,1,10],[3,11,5])$. Since [4,5] form an edge in a 4-gon, therefore $[2,c_{1}]$ will form an edge in a 4-gon, which implies $c_{1}=7$. Now
		$lk(7)=C_{8}(9,4,[8,1,0],[6,11,2])\Rightarrow
		lk(11)=C_{8}(1,10,[6,7,2],[3,9,5])\Rightarrow
		lk(6)=C_{8}(8,10,[11,2,7],[0,4,3])\Rightarrow
		lk(8)=C_{8}(5,6,\allowbreak [10,3,4],[7,0,1])$ $\Rightarrow
		lk(4)=C_{8}(9,7,[8,10,3],\allowbreak[0,6,5])\Rightarrow
		lk(2)=C_{8}(0,3,[11,6,7],\allowbreak [9,10,1])\Rightarrow
		lk(3)=C_{8}(0,2,[11,5,9],[10,8,4])\Rightarrow
		lk(10)=C_{8}(6,11,[1,2,9],[3,4,8])$. This map isomorphic to the map $ \boldsymbol{KO_{1[(3^3,4,3,4)]}} $, given in figure \ref{33434or}, under the map $ (0,7,6)(1,9,11,4,8,2,5)(3,10) $.
	\end{case}
	\textbf{Non-orientable cases:}
	
	If $ (a,b,c,d)=(4,9,6,5) $ then $ lk(4) $ will not possible.
	\begin{case}
		If $ (a,b,c,d)=(6,9,5,4) $ then $ lk(6)=C_8(7,a_1,[9,1,2],[5,4,0]) $ which implies $ a_1=10 $. Now $ lk(5)=C_8(1,9,[a_2,a_3,2],[6,0,4]) $ and $ lk(2)=C_8(0,3,[a_3,a_2,5],[6,9,1] $. Considering $ lk(9) $, from $ lk(5) $, we see that $ a_2\neq 10 $. So, if $ a_3\neq 10 $ then $ |v_{lk(10)}|\leq7 $ which is not possible. Therefore $ a_3=10 $ and then $ a_2=11 $ as otherwise $ |v_{lk(11)}|\leq 6 $ which make contradiction. Now $ lk(10)=C_8(6,7,[11,5,2],[3,8,9])\Rightarrow lk(9)=C_8(11,5,[1,2,6],[10,3,8]) \Rightarrow lk(8)=C_8(11,4,[1,0,7],[3,10,9])\Rightarrow lk(4)=C_8(1,8,[11,7,3],[0,6,5])\Rightarrow lk(3)=C_8(2,0,\allowbreak[4,11,7],\allowbreak [8,9,10])\Rightarrow lk(7)=C_8(6,10,[11,4,3],[8,1,0])\Rightarrow lk(11)=C_8(8,9,[5,2,10],[7,3,\allowbreak4])$. This map is isomorphic to the map $ KNO_{2[(3^3,4,3,4)]} $, under the map $ (0,6,5,4)(2,10,11,\allowbreak1,3,7,8) $.
	\end{case}
	For $ (a,b,c,d)=(9,4,6,5),(9,6,5,4) $, $ C(5,1,4)\in lk(6) $ and $ C(1,6,0,4)\in lk(5) $, respectively.
	\begin{case}
		When $ (a,b,c,d)=(9,6,5,10) $ then $ lk(6)=C_8(7,11,[9,2,1],[5,4,0]) $. Now $ [5,10,a_1] $ and $ [2,3,b_1] $ will be two faces of the polygon for some $ a_1,b_1\in V $. Therefore $ [2,b_1],[3,b_1],[5,a_1],[10,a_1],[9,11] $ all are adjacent egde of a 3-gon and a 4-gon, therefore $ a_1\in\{3,11\}, b_1\in\{10,11\} $. If $ b_1=10 $ then $ a_1 $ can not be 3, therefore $ a_1=11 $ and then $ \{3,4,5,11\} $ will form a 4-gon, which make contradiction. If $ b_1=11 $ then $ a_1\neq 11 $ as [5,11], [9,11], [10,11] all are adjacent edges of a 3-gon and a 4-gon, which make contradiction. Therefore $ b_1=3 $ and then $ [3,4,7,10],[3,11,9,5] $ will be two faces and $ \{2,11,8,10\} $ will form a 4-gon. Now after completing $ lk(3,lk(4)) $ and $ lk(9) $, we see that $ lk(7) $ will not possible.
	\end{case}
	\begin{case}
		When $ (a,b,c,d)=(9,5,4,10) $ then $ [2,3,10] $ and $ [4,10,11] $ will be two faces. Now $ lk(4)=C_8(10,1,[5,6,0],[3,a_1,11])\Rightarrow lk(3)=C_8(0,2,[10,8,b_1],[a_1,11,4])\Rightarrow lk(10)\allowbreak=C_8(4,1,[8,b_1,3],[2,c_1,11])\Rightarrow lk(2)=C_8(0,3,[10,11,c_1],[9,5,1]) $. From here we see that $ b_1\in \{6,9\} $. If $ b_1=9 $ then $ lk(9)=C_8(7,6,[2,1,5],[8,10,3]) $, $ lk(6)=C_8(7,9,[2,10,11],[5,\allowbreak4,0]) $, $ lk(7)=C_8(6,9,[3,4,11],[8,1,0]) $, $ lk(8)=C_8(11,5,[9,3,10],[1,\allowbreak0,7]) $, $ lk(5)=C_8(11,\allowbreak8,[9,2,\allowbreak1],[4,0,6]) $, $ lk(11)=C_8(5,8,[7,3,4],[10,2,6]) $. This map isomorphic to the map $ \boldsymbol{KNO_{3[(3^3,4,3,4)]}} $, given in figure \ref{33434non}, under the map $ (0,2,3)(1,11,8,5,10,\allowbreak4)(6,9) $. If $ b_1=6 $ then $ lk(9)=C_8(6,7,[2,1,5],[11,4,3]) $, $ lk(6)=C_8(9,7,[0,4,5],[8,10,3]) $, $ lk(7)=C_8(6,9,\allowbreak[2,10,11],[8,1,\allowbreak 0]) $, $ lk(8)=C_8(11,5,[6,3,10],[1,0,7]) $, $ lk(5)=C_8(8,11,[9,\allowbreak2,1],[4,0,6]) $, $ lk(11)=C_8(5,8,\allowbreak[7,2,10],[4,3,9]) $. This map isomorphic to the map $ \boldsymbol{KNO_{3[(3^3,4,3,4)]}} $, given in figure \ref{33434non}, under the map $ (1,4)(2,3)(5,8)(6,7)(10,11) $.
	\end{case}
	\begin{case}
		When $ (a.b,c,d)=(9,10,6,5) $ then $ lk(6)=C_8(1,10,[a_1,a_2,7],[0,4,5]) $ where $ a_1\in\{2,3,11\} $ and $ [2,3,11] $ will be a face. If $ a_1=2 $ then face sequence will not follow in $ lk(2) $ and if $ a_1=3 $ then after completing $ lk(3) $, face sequence will not follow in $ lk(10) $, therefore $ a_1=11 $. Now three incomplete 4-gons are $ [3,4,c_1,c_2] $, $ [2,11,c_3,c_4] $, $ [3,11,c_5,c_6] $. As [1,10] and [2,9] are an edge of a 4-gons, therefore [10,11] will be an adjacent edge of two 3-gon, therefore $ c_1,c_2\in\{8,10\} $ and one of $ c_5,c_6 $ will be 9 and then $ [2,11,c_3,c_4]=[a_2,11,6,7] $. Then one of $ c_5,c_6 $ is 5. Now $ lk(2)=C_8(3,0,[1,10,9],[7,6,11])\Rightarrow lk(5)=C_8(1,8,[11,3,9],[4,0,6])\Rightarrow lk(3)=C_8(2,0,[4,8,10],[9,5,11])\Rightarrow lk(10)=C_8(6,11,[8,4,\allowbreak3],[9,2,1])\Rightarrow lk(8)=C_8(5,11,[10,3,\allowbreak 4],[7,0,1])\Rightarrow lk(4)=C_8(9,7,[8,10,3],[0,6,5])\Rightarrow lk(7)=C_8(4,9,[2,11,6],[0,1,8])\Rightarrow lk(9)=C_8(7,4,[5,11,3],[10,1,2])\Rightarrow lk(11)=C_8(10,\allowbreak8,[5,9,3],[2,7,6]) $. This map is isomorphic to the map $ \boldsymbol{KNO_{1[(3^3,4,3,4)]}} $, given in figure \ref{33434non}, under the map $ (0,8,3,5)(1,2,10,9,6)(4,\allowbreak7,11) $.
	\end{case}
\end{proof}
\section{Acknowledgment}

 The work reported here forms part of the PhD thesis of DB.

\end{document}